\documentclass[letter,12pt,table]{article}

\usepackage[english]{babel}
\usepackage[utf8]{inputenc}

\usepackage[margin=1 in, top=1 in, bottom= 1.2 in]{geometry}

\makeatletter
\g@addto@macro\normalsize{%
  \setlength\abovedisplayskip{7pt}
  \setlength\belowdisplayskip{7pt}
  \setlength\abovedisplayshortskip{7pt}
  \setlength\belowdisplayshortskip{7pt}
}
\makeatother

\usepackage{tocloft}

\usepackage{amsfonts}
\usepackage{mathrsfs}
\usepackage{bbm}
\usepackage{latexsym}
\usepackage{math dots}
\usepackage{amssymb}
\usepackage{mathtools}
\usepackage{relsize}
\usepackage{todonotes}
\usepackage{enumitem}
\setlist{nolistsep} 	%[label=(\roman{enumi}),ref=(\roman{enumi}),leftmargin=*]
\usepackage{amsthm}
\usepackage{array}
\usepackage{tikz}
\usetikzlibrary{decorations.pathreplacing}

\usepackage{xcolor}
\definecolor{Color1}{rgb}{0.0, 0.42, 0.47}%Blue
\definecolor{Color2}{rgb}{0.78, 0.11, 0.0}%Scarlet
\definecolor{Color3}{RGB}{209, 3, 115}
\definecolor{Color4}{RGB}{3, 141, 88}

\usepackage{titlesec}
\titleformat{\section}
  {\large\center\bfseries}
  {\thesection.}{.7em}{}
\titlespacing*{\section}{0pt}{3.5ex plus 0ex minus 0ex}{1.5ex plus 0ex}
\titleformat{\subsection}
  {\center\bfseries}
  {\thesubsection.}{.7em}{}
\titlespacing*{\subsection}{0pt}{3.5ex plus 0ex minus 0ex}{1.5ex plus 0ex}
\titleformat{\subsubsection}
  {\center\bfseries}
  {\thesubsubsection.}{.7em}{}
\titlespacing*{\subsubsection}{0pt}{3.5ex plus 0ex minus 0ex}{1.5ex plus 0ex}

\addto\captionsenglish{}

\usepackage{titling}
\makeatletter
\renewenvironment{abstract}{
\begin{center}
{\bfseries \large\abstractname\vspace{\z@}}
\end{center}
\quotation
}
\makeatother

\usepackage[linktocpage=true]{hyperref}
\usepackage[capitalize]{cleveref}
\hypersetup{citecolor = Color2,colorlinks,
			linkcolor = black,
			urlcolor = Color2}

\newtheoremstyle{plain}{3mm}{3mm}{\slshape}{}{\bfseries}{.}{.5em}{}
\newtheoremstyle{definition}{2mm}{2mm}{}{}{\bfseries}{.}{.5em}{}
\theoremstyle{plain}
	
\newtheorem{Theorem}{Theorem}

\newtheorem{Lemma}[Theorem]{Lemma}
\newtheorem{Proposition}[Theorem]{Proposition}
\newtheorem{Corollary}[Theorem]{Corollary}

\newtheorem{Question}[Theorem]{Question}
\newtheorem{Problem}[Theorem]{Problem}
\theoremstyle{definition}
\newtheorem{Definition}[Theorem]{Definition}
\newtheorem{Remark}[Theorem]{Remark}
\newtheorem{Example}[Theorem]{Example}

\theoremstyle{plain} 
\newcounter{MainTheoremCounter}

\theoremstyle{plain}
\newtheorem*{namedthm}{\namedthmname}
\newcounter{namedthm}
\makeatletter
	\newenvironment{named}[2]
	{\def\namedthmname{#1}
	\refstepcounter{namedthm}
	\namedthm[#2]\def\@currentlabel{#1}}
	{\endnamedthm}
\makeatother
\usepackage{chngcntr}
\counterwithin{Theorem}{section}

\numberwithin{equation}{section}

\allowdisplaybreaks

\newcommand{\Erdos}{Erd\H{o}s}

\newcommand{\Turan}{Tur{\'a}n}

\newcommand{\Oh}{{\rm O}}
\newcommand{\oh}{{\rm o}}
\newcommand{\N}{\mathbb{N}}
\newcommand{\Z}{\mathbb{Z}}
\newcommand{\R}{\mathbb{R}}
\newcommand{\C}{\mathbb{C}}
\newcommand{\Q}{\mathbb{Q}}
\newcommand{\T}{\mathbb{T}}

\newcommand{\define}[1]{{\itshape #1}}

\renewcommand{\epsilon}{\varepsilon}
\renewcommand{\leq}{\leqslant}
\renewcommand{\geq}{\geqslant}
\renewcommand{\setminus}{\backslash}
\renewcommand{\Re}{{\rm Re}}

\renewcommand{\subset}{\subseteq}

\newcommand{\E}{\mathbb{E}}

\newcommand{\1}{1}
\usepackage[normalem]{ulem}

\usepackage[makeroom]{cancel}

\DeclareMathOperator*{\Leb}{Leb}

\newcommand{\A}{\mathcal{A}}
\newcommand{\innprod}[2]{\left\langle #1, #2 \right\rangle}

\newcommand{\str}{\mathrm{str}}
\newcommand{\unf}{\mathrm{unf}}

\newcommand{\inv}{\mathrm{inv}}
\newcommand{\erg}{\mathrm{erg}}

\newcommand{\Bohr}{\mathrm{Bohr}}
\usepackage{bm}

\usepackage[refpage]{nomencl}
\renewcommand*\pagedeclaration[1]{\leaders\hbox to 0.8em{\hss.\hss}\hfill \makebox[2em][r]{\hyperlink{page.#1}{#1}}}

\makenomenclature

\author{By~~{\scshape Ethan Ackelsberg}~~and~~{\scshape Florian~K.~Richter}}
\date{\small \today}
\title{\bfseries An inverse theorem for sumsets of sets of positive density in the integers}

\begin{document}

\maketitle
\begin{abstract}
Let $d(\cdot)$ denote the natural density on the positive integers. We characterize all sets $A,B$ with positive density satisfying $d(A+B)=d(A)+d(B)$, under the assumption that the two sets are not both contained in a proper finite union of residue classes. This gives a new inverse theorem for Kneser's sumset inequality in the integers, and provides a partial answer to a long-standing open question of \Erdos{} and Graham.
\end{abstract}

\tableofcontents
\thispagestyle{empty}

%==========================================================
%==========================================================

%SECTION
\section{Introduction}

A central topic in additive combinatorics is the study of \define{sumsets}
\[
A+B=\{a+b:a\in A,\,b\in B\},
\]
where $A$ and $B$ are subsets of an abelian group $(G,+)$.
In the setting where the ambient group is compact, Kneser established the following landmark result.

\begin{Theorem}[Kneser's sumset inequality in compact abelian groups,~{\cite[Satz~1]{Kneser56}}]
\label{thm_Kneser_for_compact_groups}
Let $G$ be a compact abelian group with Haar measure $m$.
If $A, B \subseteq G$ are non-empty Borel sets then either
\begin{equation}
\label{eqn_Kneser_for_compact_groups}
m(A+B)\geq m(A)+m(B) %\min\{m(G),\, m(A)+m(B)\},
\end{equation}
or there exists a %proper
finite index subgroup $H$ of $G$ such that $A+B$ is a finite union of cosets of $H$ and $m(A+B)=m(A+H)+m(B+H)-m(H)$.
\end{Theorem}

Note that if $A$ and $B$ are Borel measurable, then their sumset $A+B$ is analytic and therefore Haar measurable, which ensures that $m(A+B)$ is well defined.
If the assumption on the sets $A$ and $B$ is relaxed and one only assumes Haar measurability, then the sumset $A+B$ is no longer guaranteed to be measurable, but the statement of \cref{thm_Kneser_for_compact_groups} remains true if one replaces $m(A+B)$ by $m_*(A+B)$, where $m_*(C)=\sup\{m(K): K\subset C,~K~\text{compact}\}$ is the inner measure corresponding to $m$.

Kneser's theorem contains several other classical sumset inequalities as special cases. For example, when $G=\Z/p\Z$ for some prime $p$, we obtain from \cref{thm_Kneser_for_compact_groups} the Cauchy--Davenport inequality, which says that for all non-empty sets $A,B\subset\Z/p\Z$,
\begin{equation*}
\tag{Cauchy--Davenport}
\label{eqn_cauchy-davenport}
|A+B|\geq  \min\{p,\,|A|+|B|-1\}.
\end{equation*}

Another classical result on sumsets is the Brunn--Minkowski inequality: for any bounded, non-empty, Borel sets $A,B\subset \R$ we have
\begin{equation}
\tag{Brunn--Minkowski}
\label{eqn_brunn-minkowski}
\operatorname{Leb}(A+B)\geq \operatorname{Leb}(A)+\operatorname{Leb}(B),
\end{equation}
where $\operatorname{Leb}$ denotes the Lebesgue measure on $\R$.
By embedding $A$, $B$, and $A+B$ into the compact connected group $\R/(4N\Z)$, where $N$ is sufficiently large so that $A,B\subset [-N,N]$, we see that the Brunn–Minkowski inequality also follows from \cref{thm_Kneser_for_compact_groups}. 

Of great importance in additive combinatorics are \emph{inverse theorems}, which aim to classify the extremal configurations for which a given arithmetic inequality, such as \eqref{eqn_Kneser_for_compact_groups}, is sharp or nearly sharp, usually affirming strong structural constraints.
For the Cauchy-Davenport inequality, an inverse theorem was proved by Vosper \cite{Vosper56a,Vosper56b}, who showed that two sets $A,B\subset\Z/p\Z$ with $|A|,|B| > 1$ and $|A|+|B|<p$ satisfy the equality
$|A+B|=|A|+|B|-1$ if and only if they are arithmetic progressions of the same step-size.
An inverse theorem for the Brunn–Minkowski inequality states that the equality 
$
\Leb(A+B)= \Leb(A)+\Leb(B)
$
happens if and only if there exist two compact intervals $I,J\subset\R$ such that
\[
A\subset I, \quad B\subset J,\quad\text{and}\quad \Leb(I\setminus A)=\Leb(J\setminus B)=0.
\]    

There is also an inverse theorem for \cref{thm_Kneser_for_compact_groups}.
To avoid the degenerate case where \eqref{eqn_Kneser_for_compact_groups} fails due to the interference of a finite index subgroup, it is convenient to impose the hypothesis that one of the sets, say the set $B$, has non-empty intersection with every coset of every finite index subgroup of $G$, so that no local obstructions can arise.
Under these assumptions, the circle group $\T = \R/\Z$ \nomenclature{$\T$}{circle group $\R/\Z$} plays a fundamental role in the inverse theorem.

\begin{Theorem}[Inverse theorem for sumsets in compact abelian groups, {\cite{Kneser56,Griesmer14}}]
\label{thm_inverse_theorem_sumsets_compact_groups}
Suppose $A, B \subseteq G$ are non-empty Borel sets with $m(A), m(B) > 0$, $m(A) + m(B) < m(G)$, and suppose that $B$ has non-empty intersection with every coset of every finite index subgroup of $G$.
If
\[
m(A+B) = m(A) + m(B)
\]
then there exists a finite index subgroup $H$ of $G$ and a decomposition
\begin{equation}
\label{eqn_inverse_theorem_sumsets_compact_groups_lifting_representation}
A=A_0+a_0\qquad\text{and}\qquad B = (B_0 \cup B_1)+b_0,
\end{equation}
where $A_0,B_0\subset H$, $a,b\in G$, $B_1 \subseteq G \setminus H$ with $m((G\setminus H)\setminus B_1) = 0$, and one of the following two conditions is satisfied: 
\begin{enumerate}[label=(\arabic*), leftmargin=*]
\item
\label{itm_inverse_theorem_sumsets_compact_groups_i}
There exists a continuous surjective homomorphism $\phi\colon H\to\T$ and closed intervals $I,J\subset\T$ such that
\[
A_0\subset \phi^{-1}(I), \quad B_0\subset \phi^{-1}(J),\quad\text{and}\quad m(\phi^{-1}(I)\setminus A_0)=m(\phi^{-1}(J)\setminus B_0)=0.
\]
\item
\label{itm_inverse_theorem_sumsets_compact_groups_ii}
$m(B_0) = 0$ and $m(A\triangle (A-t))=m(B\triangle (B-t))=0$ for all $t$ in the group generated by $B_0-B_0$.
\end{enumerate}
\end{Theorem}

In the case when $G$ is connected, \cref{thm_inverse_theorem_sumsets_compact_groups} was proved by 
Kneser in \cite[Satz~2]{Kneser56}; the general case is a corollary of \cite[Theorem 1.3]{Griesmer14}.
Note that if $B$ does not necessarily meet every coset of every finite index subgroup then there are more complicated situations that can arise; we refer the reader to \cite{Griesmer14} for more details. 

Sets of the form $\phi^{-1}(I)$ and $\phi^{-1}(J)$, where $I$ and $J$ are closed intervals in $\T$ and $\phi\colon G\to \T$ is a surjective homomorphism, are called \emph{parallel Bohr intervals} in $G$.
Case~\ref{itm_inverse_theorem_sumsets_compact_groups_i} of \cref{thm_inverse_theorem_sumsets_compact_groups} says that either $A$ and $B$ are, up to zero measure, parallel Bohr intervals in $G$ (which corresponds to the situation when $H=G$ and $B_1=\emptyset$), or they come from parallel Bohr intervals of a proper finite index subgroup $H$ lifted to $G$ according to the representation in \eqref{eqn_inverse_theorem_sumsets_compact_groups_lifting_representation}.

Case~\ref{itm_inverse_theorem_sumsets_compact_groups_ii} also imposes strict structural constraints on the sets $A$ and $B$, but of a different nature than those arising in case~\ref{itm_inverse_theorem_sumsets_compact_groups_i}.
Note, in particular, that case~\ref{itm_inverse_theorem_sumsets_compact_groups_ii} can only happen when the set $B$ satisfies $m(B)=m(G)(1-\frac{1}{h})$ for some $h\in\N$.

Kneser's theorem on compact abelian groups, \cref{thm_Kneser_for_compact_groups}, is complemented by an analogous result on sumsets in the positive integers $\N=\{1,2,3,\ldots\}$.\nomenclature{$\N$}{positive integers}
The natural way to measure the size of (infinite) subsets of $\N$ is to use the notion of density.
The \define{lower density} and \define{upper density} of $A\subset\N$ are defined respectively as\nomenclature{$\underline{d},\overline{d}$}{lower and upper density}
\[
\underline{d}(A)=\liminf_{N\to\infty}\frac{|A\cap[N]|}{N} \qquad\text{and}\qquad
\overline{d}(A)=\limsup_{N\to\infty}\frac{|A\cap[N]|}{N},
\]
where $[N]=\{1,\ldots,N\}$.\nomenclature{$[N]$}{$\{1,\ldots,N\}$}
If $\underline{d}(A)= \overline{d}(A)$ then the limit
\[
d(A)=\lim_{N\to\infty}\frac{|A\cap [N]|}{N}
\]
exists and we call this number the \define{density} of $A$.\nomenclature{$d$}{density}

The next theorem, which is also due to Kneser, indicates that sumsets of positive density sets in $\N$ behave surprisingly similarly to sumsets of positive measure sets in compact abelian groups. 

\begin{Theorem}[Kneser's sumset inequality in the integers,~{\cite{Kneser53}}]
\label{thm_Kneser_for_integers}
If $A, B \subseteq \N$ are non-empty sets then either
\begin{equation}
\label{eqn_Kneser_in_integers}
\underline{d}(A+B)\geq \underline{d}(A)+\underline{d}(B),
\end{equation}
or there exists $H=h\N$ for some $h\in\N$ such that $A+B$ equals, up to finitely many elements, a finite union of translates of $H$ and $\underline{d}(A+B)=d(A+H)+d(A+H)-d(H)$.
\end{Theorem}

The main result of our paper establishes an inverse theorem for \cref{thm_Kneser_for_integers}.
In fact, our main result can be viewed as an integer analogue of  \cref{thm_inverse_theorem_sumsets_compact_groups} in the same way that \cref{thm_Kneser_for_integers} is an integer analogue of \cref{thm_Kneser_for_compact_groups}. Given a sequence of natural numbers $\mathbf{N} = (N_s)_{s \in \N}$ with $N_s \to \infty$, we define\nomenclature{$d_{\mathbf{N}}$}{density along a sequence $\mathbf{N}$}
\[
d_{\mathbf{N}}(A)=\lim_{s\to\infty} \frac{|A\cap [N_s]|}{N_s}
\]
whenever this limit exists. Similar to \cref{thm_inverse_theorem_sumsets_compact_groups}, we restrict to the case when one of the sets avoids all local obstructions: we say a set $B\subset\N$ \emph{meets every residue class in $\N$} if $B\cap(a\N+b)\neq\emptyset $ for all $a,b\in\N$.

\begin{Theorem}[Inverse theorem for sumsets in the integers]
\label{thm: main}
Let $A, B \subset \N$ with $d(A) > 0$, $d(A) + d(B) < 1$, and suppose that $B$ meets every residue class in $\N$.
If $\mathbf{N} = (N_s)_{s \in \N}$ is a sequence of natural numbers with $N_s \to \infty$ and
\[
d_{\mathbf{N}}(A+B) = d(A) + d(B),
\]
then there exists $H=h\N$ for some $h\in\N$ and a decomposition
\[
A=A_0-a_0\qquad\text{and}\qquad B = (B_0 \cup B_1)-b_0,
\]
where $A_0,B_0\subset H$, $a_0,b_0\in \{0,1,\ldots,h-1\}$, $B_1 \subseteq \N \setminus H$, and one of the following two conditions is satisfied: 
\begin{enumerate}[label=(\arabic*), leftmargin=*]
\item
\label{itm_main_thm_1}
$d((\N\setminus H)\setminus B_1) = 0$ and there exists an irrational number $\theta\in\T$ and closed intervals $I,J\subset\T$ such that if $\phi\colon H\to\T$ is the map $\phi(n)=n\theta\bmod 1$ for all $n\in H$ then
\[
A_0\subset \phi^{-1}(I), \quad B_0\subset \phi^{-1}(J),\quad\text{and}\quad d(\phi^{-1}(I)\setminus A_0)=d(\phi^{-1}(J)\setminus B_0)=0.
\]
\item
\label{itm_main_thm_2}
$d_{\mathbf{N}}((\N\setminus H)\setminus B_1) = 0$, $d_{\mathbf{N}}(B_0)=0$ and $d_{\mathbf{N}}(A\triangle(A-t))=d_{\mathbf{N}}(B\triangle(B-t))=0$ for all $t\in H$.
\end{enumerate}
\end{Theorem}

If $I$ and $J$ are closed intervals in $\T$ and $\phi(n)=n\theta$ for some irrational $\theta\in\T$ then the sets $A=\phi^{-1}(I)$ and $B=\phi^{-1}(J)$ are called \emph{parallel Bohr intervals} in $\N$, and it is easy to verify that parallel Bohr intervals satisfy $d(A+B)=d(A)+d(B)$.
Similarly to \cref{thm_inverse_theorem_sumsets_compact_groups}, case~\ref{itm_main_thm_1} of \cref{thm: main} asserts that, up to zero density modifications, $A$ and $B$ are either parallel Bohr intervals in $\N$ (if $h=1$) or lifts of parallel Bohr intervals from a proper subsemigroup $H=h\N$ to $\N$ (if $h\geq 2$).

It is not immediately obvious that there actually exist positive density sets $A,B\subset\N$ that fall into case~\ref{itm_main_thm_2} of \cref{thm: main}; we construct explicit examples in \cref{prop_ctmn}.
We also note that in case~\ref{itm_main_thm_2} of \cref{thm: main}, it is in general not possible to replace $d_{\mathbf{N}}((\N\setminus H)\setminus B_1) = 0$ with $d((\N\setminus H)\setminus B_1) = 0$ as illustrated by \cref{prop_dens_sumset_DNE}. 
A peculiar conclusion that we draw from this observation is that while in case~\ref{itm_main_thm_1} the density $d(A+B)$ is guaranteed to exist, in case~\ref{itm_main_thm_2} the asymptotic density of $A+B$ might not exist. Indeed, in \cref{sec_non-existence_of_density_of_sumset} we provide a pair of sets satisfying $\underline{d}(A+B)=d(A)+d(B)<1$ and $\overline{d}(A+B)=1$.

\cref{thm: main} is closely related to a long-standing question of \Erdos{} and Graham \cite{EG80} with an updated formulation appearing as Problem \#335 on the \Erdos{} problem website \cite{ErdosProblems335}.

\begin{Problem}[cf.~{\cite[Problem~\#335]{ErdosProblems335}}]
\label{proplem_335}
Characterize all sets $A,B\subset\N$ with $d(A)>0$, $d(B)>0$, and
$d(A+B)=d(A)+d(B)$.
\end{Problem}

Note that our main result, \cref{thm: main}, resolves this problem under the additional assumption that $B$ meets every residue class in $\N$.
\Erdos{} and Graham speculated in \cite[p.~51]{EG80} that all sets of positive density that satisfy the conclusion of \cref{proplem_335} arise in a similar fashion to parallel Bohr intervals, perhaps using homomorphisms onto other groups. %However, this is not the case, as the following simple example illustrates.
However, there are several reason why this is not the case. Firstly, case~\ref{itm_main_thm_2} of \cref{thm: main} leads to counterexamples to this claim. Also, one can construct counterexamples by exploiting simple divisibility obstructions, as the following example shows.

\begin{Example}
\label{ex_even_coinflip_set}
Flip a fair coin infinitely often, and define $A=\{2n: n^{\mathrm{th}}~\text{coin flip is heads}\}$. By the law of large numbers, we have almost surely that $d(A)=\frac{1}{4}$ and $d(A+A)=\frac{1}{2}$. 
\end{Example}

%SECTION
\section*{Acknowledgments}

This work was supported by the Swiss National Science Foundation grant TMSGI2-211214.
We thank Marius Tiba for providing the reference \cite{Shao19}, which led to a streamlined proof of \cref{thm: delta-sum almost contains a sum}.
We also thank John Griesmer for comments on an earlier draft that enhanced the discussion of open problems in \cref{sec: questions}.

%SECTION
\section{Notation}

We use standard number theoretic asymptotic notation throughout the paper.
For sequences $(a_s)_{s \in \N}$ and $(b_s)_{s \in \N}$, we write
\begin{equation*}
    a_s = \Oh(b_s)
\end{equation*}
or
\begin{equation*}
    a_s \ll b_s
\end{equation*}
to mean that there exists a constant $C > 0$ such that $|a_s| \leq C |b_s|$ for all sufficiently large $s \in \N$.
If the asymptotic variable $s$ is not clear from context, we include it in the subscript, for example $a_s = \Oh_{s \to \infty}(b_s)$.
The notation $a_s = \oh(b_s)$ (or, more explicitly, $a_s = \oh_{s \to \infty}(b_s)$) means that $\lim_{s \to \infty} \frac{|a_s|}{|b_s|} = 0$.
For example, $a_s = \oh(1)$, means that $\lim_{s \to \infty} a_s = 0$.

For other notation used in the paper, we provide the following list with references to the page where the notation is defined.

{\footnotesize \printnomenclature}

%SECTION
\section{Outline of the proof} \label{sec: proof outline}

In this section we outline in broad strokes the proof of \cref{thm: main}.
Fix a sequence of natural numbers $\mathbf{N} = (N_s)_{s \in \N}$ with $N_s \to \infty$.
We start with a definition.

\begin{Definition} \label{def: locally ergodic}
We say a bounded function $f \colon \N \to \C$ is \emph{locally ergodic along $\mathbf{N}$} if
\begin{equation*}
	\limsup_{H \to \infty} \limsup_{s \to \infty} \frac{1}{N_s} \sum_{n=1}^{N_s} \left| \frac{1}{H} \sum_{h=1}^H f(n+h) \right| = 0.
\end{equation*}
\end{Definition}

In our proof of \cref{thm: main}, we treat separately the cases in which the functions $\1_A-d(A)$ and $\1_B-d(B)$ are locally ergodic along $\mathbf{N}$ and in which they are not, since our methods differ substantially in these two situations.
This naturally divides the overall structure of our argument into two parts.
The first part consists of showing the following theorem.
For convenience, given sets $C,D\subset\N$, we will henceforth write\nomenclature{$\sim,\sim_{\mathbf{N}}$}{equality up to zero density (along a sequence $\mathbf{N})$}
\[
C\sim_{\mathbf{N}} D\iff d_{\mathbf{N}}(C\triangle D)=0
\qquad\text{and}\qquad
C\sim D\iff d(C\triangle D)=0.
\]

\begin{Theorem} \label{thm: ergodic or h}
Let $A, B \subseteq \N$ with $d(A) = \alpha > 0$ and $d(B) = \beta > 0$ such that $\alpha + \beta < 1$.
Suppose $B$ meets every residue class in $\N$.
Let $\mathbf{N} = (N_s)_{s \in \N}$ be an increasing sequence with $\lim_{s \to \infty} N_s = \infty$ such that $d_{\mathbf{N}}(A+B) = \alpha + \beta$.
Then there exists a subsequence $\mathbf{N}'$ of $\mathbf{N}$ such that either
\begin{enumerate}[label=(\arabic*), leftmargin=*]
    \item\label{itm ergodic or h 1} $\1_A - \alpha$ and $\1_B - \beta$ are locally ergodic along $\mathbf{N}'$ and there exists $h \in \N$ with $h < \frac{1}{1-\beta}$ such that $n \mapsto \1_A(n+t_1) \ldots \1_A(n+t_k) e(nq)$ is locally ergodic along $\mathbf{N}'$ for every $q \in \Q \setminus \frac{\Z}{h}$, every $k \in \N$, and every $t_1, \ldots, t_k \in \N$, or
    \item\label{itm ergodic or h 2} there exist %a subsequence $\mathbf{N}'$ of $\mathbf{N}$, and 
    integers $h \geq 2$ and $a_0, b_0 \in \{0,1,\ldots,h-1\}$ such that $A \subseteq h\N - a_0$, $A + h \sim_{\mathbf{N}'} A$, and $B \sim_{\mathbf{N}'} (\N \setminus h\N) - b_0$.
\end{enumerate}
\end{Theorem}

The second part of the argument is devoted to proving the following theorem.

\begin{Theorem}
\label{thm: 1_A is locally ergodic}
If we are in case \ref{itm ergodic or h 1} of \cref{thm: ergodic or h} then there exist $a_0, b_0 \in \{0,1,\ldots,h-1\}$, an irrational $\theta\in\T$, and closed intervals $I,J\subset\T$ such that if
    \[
    A_0=A+a_0,\qquad B_0= (B+b_0)\cap h\N,\qquad\text{and}\qquad B_1= (B+b_0)\setminus h\N,
    \]
    then $A_0\subset h\N$, $B_1\sim\N\setminus h\N$, $A_0\sim \{n\in h\N: n\theta\in I\}$, and $B_0\sim \{n\in h\N: n\theta\in J\}.$
\end{Theorem}

Let us show how Theorems~\ref{thm: ergodic or h} and~\ref{thm: 1_A is locally ergodic} combined imply \cref{thm: main}.

\begin{proof}[Proof of \cref{thm: main}]
Assume $d_{\mathbf{N}}(A+B) = \alpha + \beta$.
If the density of $B$ is not of the form $1-\frac{1}{h}$ for some positive integer $h$, then case~\ref{itm ergodic or h 2} of \cref{thm: ergodic or h} is impossible, and it follows from \cref{thm: 1_A is locally ergodic} that case~\ref{itm_main_thm_1} of \cref{thm: main} holds.
So suppose $d(B)=1-\frac{1}{h}$ for some $h\in\N$.

Next, suppose there exists some subsequence $\mathbf{N}'$ of $\mathbf{N}$ such that the density $d_{\mathbf{N}'}((A+h)\triangle A)$ exists and is positive. This means for any subsequence $\mathbf{N}''$ of $\mathbf{N}'$ we have $A+h\not\sim_{\mathbf{N}''} A$. Applying \cref{thm: ergodic or h} with $\mathbf{N}'$ in place of $\mathbf{N}$, we see that case~\ref{itm ergodic or h 2} of \cref{thm: ergodic or h} is again impossible, which by \cref{thm: 1_A is locally ergodic} implies that we are in case~\ref{itm_main_thm_1} of \cref{thm: main}.

We have therefore reduced to the situation in which, for every subsequence $\mathbf{N}'$ of $\mathbf{N}$
such that the density $d_{\mathbf{N}'}((A+h)\triangle A)$ exists, it necessarily holds that $d_{\mathbf{N}'}((A+h)\triangle A)=0$.
This implies that $d_{\mathbf{N}}((A+h)\triangle A)=0$, or in other words,
\[
A+h\sim_{\mathbf{N}} A.
\]

A similar argument shows that if there exists some subsequence $\mathbf{N}'$ such that $d_{\mathbf{N}'}(\triangle (N\setminus h\N-b_0))$ exists and is positive then we are again in case~\ref{itm_main_thm_1} of \cref{thm: main}. Hence we can assume that $d_{\mathbf{N}'}(\triangle (N\setminus h\N-b_0))$ is zero for any subsequence $\mathbf{N}'$ for which it is defined, implying that
\[
B \sim_{\mathbf{N}} (\N \setminus h\N) - b_0.
\]
We conclude that case~\ref{itm ergodic or h 2} of \cref{thm: ergodic or h} holds with $\mathbf{N}'=\mathbf{N}$, which is equivalent to case~\ref{itm_main_thm_2} of \cref{thm: main}.
\end{proof}

The methods involved in proving Theorems~\ref{thm: ergodic or h} and~\ref{thm: 1_A is locally ergodic} are rather different from one another, and therefore we discuss our approaches to each of them separately.

\paragraph{Outline of the proof of \cref{thm: ergodic or h}.}
A key part of the proof of \cref{thm: ergodic or h} consists of showing that there exists a sequence $\mathbf{H}_G=(H_{G,s})_{s\in\N}$ of natural numbers satisfying 
\[
\lim_{s\to\infty} H_{G,s} = \infty \qquad\text{and}\qquad  \lim_{s\to\infty} \frac{H_{G,s}}{N_s}=0,
\]
such that we have convenient structural control over the intersections of the form $A\cap [n,n+H_{G,s})$, $B\cap [0,H_{G,s})$, and $(A+B)\cap [n,n+H_{G,s})$.
That is, we show that for ``most'' $n\in \{N_{s-1}+1,\ldots,N_s\}$ there exist a number $Q_n$ and sets $A_n,B_n\subset \{0,1,\ldots,Q_n-1\}$ such that
\begin{equation}
\label{eqn_approx_local_strucutre}
\begin{aligned}
&Q_n=(1+\oh_{s\to\infty}(1)) H_{G,s}
\\
&|A_n|=\alpha+\oh_{s\to\infty}(1)
\\
&A_n~\text{is ``approximately equal'' to}~(A-n)\cap \{0,1,\ldots,Q_n-1\} 
\\
&|B_n|=\beta+\oh_{s\to\infty}(1)
\\
&B_n~\text{is ``approximately equal'' to}~B\cap \{0,1,\ldots,Q_n-1\}   
\\
&((A-n)+B)\cap  \{0,1,\ldots,Q_n-1\}~\text{contains}~\text{``most'' of}~(A_n+_{\bmod{Q_n}} B_n),
\end{aligned}
\end{equation}
where $A_n+_{\bmod{Q_n}} B_n = \{0\leq m< Q_n: m\equiv a+b\bmod Q_n~\text{for some}~(a,b)\in A_n\times B_n\}$;
for the precise statement, see \cref{thm: local sumset}.
It then follows that
\begin{align*}
\alpha+\beta &= \frac{|(A+B)\cap [N_s]|}{N_s} +\oh_{s\to\infty}(1)
\\
&= \frac{1}{N_s}\sum_{n=1}^N \frac{|((A-n)+B)\cap \{0,1,\ldots,H_{G,s}-1\}|}{H_{G,s}} +\oh_{s\to\infty}(1)
\\
&= \frac{1}{N_s}\sum_{n=1}^N \frac{|((A-n)+B)\cap \{0,1,\ldots,Q_n-1\}|}{Q_n}+\oh_{s\to\infty}(1)
\\
&\geq  \frac{1}{N_s}\sum_{n=1}^N \frac{|A_n+_{\bmod{Q_n}} B_n|}{Q_n} +\oh_{s\to\infty}(1).
\end{align*}
This reduces the ``global'' inverse problem for $(A+B)\cap [N_s]$ to a ``local'' inverse problem for $A_n+_{\bmod{Q_n}} B_n$.
The latter can be viewed as a sumset in $\Z/Q_n\Z$. Thus, using a quantitative inverse theorem for sumsets in finite cyclic groups (\cref{thm: cyclic sum without obstructions}) as a black box, we find that for most $n\in \{N_{s-1}+1,\ldots,N_s\}$ the sets $A_n$ and $B_n$ fall into one of only two categories:
\begin{enumerate}[label=(\Roman*), leftmargin=*]
\item\label{itm outline fin 1}
either $A_n$ and $B_n$ ``resemble'' Bohr sets,
\item\label{itm outline fin 2}
or there exists a subgroup of $\Z/Q_n\Z$ of uniformly bounded index such that most of $A_n$ is contained in a coset of this subgroup and $B_n$ is approximately equal to a finite union of cosets of this subgroup.
\end{enumerate}
For the details, see \cref{sec: local inverse theorem}, specifically \cref{thm: local inverse theorem}.
Since by \eqref{eqn_approx_local_strucutre} we have $A_n+n\approx A\cap [n,n+H_{G,s})$ and $B_n\approx B\cap [0,H_{G,s})$, the above provides insight into the local structure of the sets $A$ and $B$ on scale $\mathbf{H}_G=(H_{G,s})_{s\in\N}$. It remains to convert this into information about the global structure of $A$ and $B$ on scale $\mathbf{N}=(N_s)_{s\in\N}$.
This is done in Sections~\ref{sec_loc_to_glob_residue_classes}, \ref{sec_eliminating_rational_frequencies}, and~\ref{sec: local ergodicity}, where we show that the occurrence of case~\ref{itm outline fin 2} implies case~\ref{itm ergodic or h 2} of \cref{thm: ergodic or h}, whereas case~\ref{itm outline fin 1} implies case~\ref{itm ergodic or h 1} of \cref{thm: ergodic or h}.

Note that this approach depends heavily on the existence of the scale $\mathbf{H}_G=(H_{G,s})_{s\in\N}$ for which \eqref{eqn_approx_local_strucutre} is possible.
To prove its existence, we introduce a new class of uniformity seminorms that lie in an intermediate regime between the $U^2$~Gowers seminorms and the $U^2$~Host–Kra seminorms and we establish a structure theorem for these new seminorms; see \cref{sec: intermediate scale seminorms} for details.

\paragraph{Outline of the proof of \cref{thm: 1_A is locally ergodic}.} 
Our proof of \cref{thm: 1_A is locally ergodic} is based on techniques from dynamical systems, and we summarize the necessary prerequisites from ergodic theory in \cref{sec_ergodic_theory}.
The paper \cite{Griesmer13} provides an earlier example of the use of ergodic theory in the study of sumsets in the integers and serves as a source of inspiration for some of our arguments.

The basic idea of this approach is to replace sumsets of sets of integers by ``dynamical sumsets'' which are defined as follows:
Given an invertible measure-preserving system $(X, \mu, T)$, a measurable subset $E \subseteq X$, and a set of natural numbers $D \subseteq \N$, we define the sumset $D + E$ by
\begin{equation*}
	D + E = \bigcup_{d \in D} T^dE.
\end{equation*}
The link between sumsets in the integers and their dynamical counterparts is established via a variant of Furstenberg’s correspondence principle, which we prove in \cref{sec_dyn_model}. 
It asserts that for any sets $A,B\subset\N$ with $d(A) = \alpha > 0$, $d(B) = \beta > 0$ and $\alpha + \beta < 1$, there exist ergodic invertible measure-preserving systems $(X_A, \mu_A, T_A)$ and $(X_B, \mu_B, T_B)$, transitive points $x_A \in X_A$ and $x_B \in X_B$, clopen sets $E_A \subseteq X_A$ and $E_B \subseteq X_B$, and continuous factor maps $\pi_A : X_A \to G$ and $\pi_B : X_B \to G$, where $(G, m, \theta)$ is the maximal common rotational factor of the systems $(X_A, \mu_A, T_A)$ and $(X_B, \mu_B, T_B)$, such that:
\begin{enumerate}[label=(\arabic*),leftmargin=*]
	\item\label{itm_FC_i}
    $\pi_A(x_A) = \pi_B(x_B) = 0$,
	\item\label{itm_FC_ii}
    $A = \{n \in \N : T_A^n x_A \in E_A\}$ and $B = \{n \in \N : T_B^n x_B \in E_B\}$,
	\item\label{itm_FC_iii}
    $\mu_A(E_A) = \alpha$ and $\mu_B(E_B) = \beta$,
	\item\label{itm_FC_vi}
    $\max\{\mu_A(B + E_A), \mu_B(A + E_B)\} \leq d_{\mathbf{N}}(A+B)=\alpha + \beta$, and
    \item\label{itm_FC_v}
    the number $C_G$ of connected components of $G$ satisfies $C_G < \frac{1}{1-\beta}$.
\end{enumerate}
In light of property~\ref{itm_FC_ii}, one can regard the sets $E_A$ and $E_B$ as ``dynamical models'' of the sets $A$ and $B$, respectively, and  property~\ref{itm_FC_v} connects the size of the dynamical sumsets $B + E_A$ and $A + E_B$ to the size of the integer sumset $A+B$. We also note that the ergodicity of the systems $(X_A, \mu_A, T_A)$ and $(X_B, \mu_B, T_B)$ can be assumed because of local ergodicity of $\1_A - \alpha$ and $\1_B - \beta$, and property~\ref{itm_FC_v} is a consequence of the correlation condition on $\1_A$ assumed in part~\ref{itm ergodic or h 1} of \cref{thm: ergodic or h}.

The final important component in this correspondence principle is the role of the maximal common rotational factor system $(G, m, \theta)$. This factor turns out to be \underline{characteristic} for the dynamical sumsets $B + E_A$ and $A + E_B$, meaning that their structure is largely governed by their projection onto this factor (cf.~\cref{prop: correspondence with convolution}). This is particularly convenient, since sumsets in rotation systems can be studied using basic methods from Fourier analysis, such as the convolution operator.
More precisely, letting $\varphi_A = \E_{\mu_A}[\1_{E_A} \mid G]$ and $\varphi_B = \E_{\mu_B}[\1_{E_B} \mid G]$ denote the projections of $\1_{E_A}$ and $\1_{E_B}$ onto $G$, we show in \cref{sec_reducing_to_rotationl_factor} that the set $S = \{x \in G : (\varphi_A * \varphi_B)(x) > 0\}$ satisfies $m(S)\leq \alpha+\beta$, where $\varphi_A * \varphi_B$ is the convolution of $\varphi_A$ and $\varphi_B$ in $G$.

Finally, in \cref{sec_finishing_the_proof}, we show that since $m(S)\leq \alpha+\beta$, there exist a set $C_0\subset G$ that is $m$-a.e.~equal to $\{x\in G:\varphi_A(x)>0\}$ and another set $D_0\subset G$ that is $m$-a.e.~equal to $\{x\in G:\varphi_B(x)>0\}$
such that
\begin{itemize}
\item $m(C_0)=\alpha$, $m(D_0)=\beta$, and $m(C_0+D_0)=\alpha+\beta$;
\item $\pi_A^{-1}(C_0)$ and $\pi_A^{-1}(C_0+D_0)$ are $\mu_A$-a.e.~equal to $E_A$ and $B+E_A$, respectively;
\item $\pi_B^{-1}(D_0)$ and $\pi_B^{-1}(C_0+D_0)$ are $\mu_B$-a.e.~equal to $E_B$ and $A+E_B$, respectively.
\end{itemize}
Form there, we can use \cref{thm_inverse_theorem_sumsets_compact_groups} to deduce that $A$ and $B$ are parallel Bohr intervals. Note that case~\ref{itm_inverse_theorem_sumsets_compact_groups_ii} of \cref{thm_inverse_theorem_sumsets_compact_groups}  cannot happen because of the bound $C_G < \frac{1}{1-\beta}$.

%SECTION
\section{Sumsets in abelian groups}

As described in \cref{sec: proof outline} above, a key step in the proof of \cref{thm: ergodic or h} relates the inverse problem for sumsets of sets of positive density in $\N$ to a problem about sumsets in finite cyclic groups.
In this section, we compile useful results about sumsets to be used in the proof of \cref{thm: ergodic or h}.

%SUBSECTION
\subsection{Schnirelmann density}

The \emph{Schnirelmann density} of a set $A \subseteq \N$ is defined by\nomenclature{$\sigma$}{Schnirelmann density}
\begin{equation*}
    \sigma(A) = \inf_{N \in \N} \frac{|A \cap [N]|}{N}.
\end{equation*}
We similarly define the Schnirelmann density of a set $A$ on an interval $I = \{a, a+1,\ldots,b\}$ by
\begin{equation*}
    \sigma(A, I) = \min_{x \in I} \frac{|A \cap \{a,a+1,\ldots,x\}|}{x-a+1}.
\end{equation*}
Note that $\sigma(A) = \lim_{N\to\infty} \sigma(A, [N]) = \inf_{N \in \N} \sigma(A, [N])$.

In the article \cite{Schnirelmann33} in which he introduced his eponymous density, Schnirelmann established the following inequality for the Schnirelmann density of a sumset in terms of the densities of its summands.\footnote{For an exposition of Schnirelmann's theorem, we refer the reader to \cite[Chapter I, Theorem 1]{HR83}.
A later result of Mann \cite{Mann42} established the related inequality $\sigma(A \cup B \cup (A+B)) \geq \sigma(A) + \sigma(B)$.}

\begin{Theorem}[Schnirelmann \cite{Schnirelmann33}] \label{thm: Schnirelmann}
    Let $A, B \subseteq \N$.
    If $1 \in A$, then
    \begin{equation*}
        \sigma(A \cup (A+B)) \geq \sigma(A) + \sigma(B) - \sigma(A) \sigma(B).
    \end{equation*}
\end{Theorem}

\begin{Corollary} \label{cor: Schnirelmann up to N}
    Let $A, B \subseteq [N]$ for some $N \in \N$.
    Suppose $\sigma(A, [N]) = \alpha > 0$ and $\sigma(B, [N]) = \beta$.
    Then
    \begin{equation*}
        d(A \cup (A+B), [N]) \geq \alpha + \beta \left( 1 - \alpha \right).
    \end{equation*}
\end{Corollary}

\begin{proof}
    Note that the assumption $\sigma(A,[N]) > 0$ implies that $1 \in A$.
    Let $A' = A \cup \{N+1,N+2,\ldots\}$ and $B' = B \cup \{N+1,N+2,\ldots\}$.
    Then $\sigma(A') = \alpha$ and $\sigma(B') = \beta$.
    Moreover, $\left( A' \cup (A'+B') \right) \cap [N] = (A \cup (A+B)) \cap [N]$, so
    \begin{equation*}
        d(A \cup (A+B), [N]) = d(A' \cup (A'+B'), [N]) \geq \sigma(A' \cup (A'+B')),
    \end{equation*}
    and the latter quantity is bounded below by $\alpha + \beta - \alpha\beta$ by Theorem~\ref{thm: Schnirelmann}.
\end{proof}

The next lemma shows that if a set has positive density in an interval in $\N$, then it has Schnirelmann density in a subinterval.

\begin{Lemma} \label{lem: subinterval with Schnirelmann density}
    Let $A \subseteq [N]$ for some $N \in \N$, and suppose $|A| \ge \delta N$.
    Then for every $\epsilon > 0$, there exists $x \leq (1-\epsilon)N$ such that
    \begin{equation} \label{eq: large Schnirelmann density}
        \sigma(A, \{x+1,\ldots,N\}) = \min_{y \in \{x+1,\ldots,N\}} \frac{|A \cap \{x+1,\ldots,y\}|}{y-x} > \delta - \epsilon.
    \end{equation}
\end{Lemma}

\begin{proof}
    If $|A \cap [M]| > (\delta - \epsilon)M$ for every $M \in [N]$, then \eqref{eq: large Schnirelmann density} holds for $x = 0$, so suppose this is not the case.
    Let $x = \max\{M \in [N] : |A \cap [M]| \leq (\delta - \epsilon)M\}$.
    Then
    \begin{equation*}
        \delta - \epsilon \geq \frac{|A \cap [x]|}{x} \geq \frac{|A \cap [N]| - (N-x)}{N} \geq \delta - \left( 1 - \frac{x}{N} \right).
    \end{equation*}
    Rearranging the inequality gives $x \leq (1-\epsilon)N$.
    Moreover, if $y \in \{x+1,\ldots,N\}$, then
    \begin{equation*}
        |A \cap \{x+1,\ldots,y\}| = \underbrace{|A \cap [y]|}_{> (\delta-\epsilon)y} - \underbrace{|A \cap [x]|}_{\leq (\delta-\epsilon)x} > (\delta-\epsilon)(y-x)
    \end{equation*}
    by maximality of $x$.
\end{proof}

%SUBSECTION
\subsection{An inverse theorem for sumsets in finite cyclic groups}

In the context of finite cyclic groups, Kneser's theorem (\cref{thm_Kneser_for_compact_groups}) has the following consequence.

\begin{Theorem} \label{thm: cyclic group Kneser}
    For every $\epsilon > 0$ and every $Q \in \N$, if $A,B \subseteq \Z/Q\Z$ are nonempty subsets and $B$ intersects every coset of every subgroup of index at most $\epsilon^{-1}$, then either
    \begin{equation*}
        |A+B| > |A| + |B| - \epsilon Q
    \end{equation*}
    or $A+B = \Z/Q\Z$.
\end{Theorem}

\begin{proof}
    Suppose that $|A+B| \leq |A| + |B| - \epsilon Q$.
    By \cref{thm_Kneser_for_compact_groups} (see also \cite[Theorem 4.2]{Nathanson96_inverse} for the discrete case used here), there is a subgroup $H \leq \Z/Q\Z$ such that $A+B$ is a union of cosets of $H$ and $|A+B| = |A+H| + |B+H| - |H|$.
    Hence,
    \begin{equation*}
        |H| = |A+H| + |B+H| - |A+B| \geq |A| +|B| - |A+B| \geq \epsilon Q,
    \end{equation*}
    so $H$ is a subgroup of index at most $\epsilon^{-1}$.
    But then $B$ intersects every coset of $H$, so $A+B = A+B+H = A+(B+H) = \Z/Q\Z$.
\end{proof}

We now wish to obtain a corresponding inverse theorem.
Namely, if $A, B \subseteq \Z/Q\Z$, $B$ intersects every coset of every subgroup of small index, and $|A+B|$ is not much larger than $|A|+|B|$, what can be said about the structure of $A$ and $B$?
A more general inverse theorem was established by Griesmer \cite{Griesmer19}, and we will use his result to deduce strong structural information about $A$ and $B$ in our specific setting of interest.

First, we reproduce \cite[Theorem 1.15]{Griesmer19} below, specialized to finite cyclic groups.

\begin{Theorem} \label{thm:stability_cyclic_groups}
    For every $\epsilon > 0$ and $k \in \N$, there exists $\delta > 0$ and $D \in \N$ such that for every $Q \in \N$, if $A, B \subseteq \Z/Q\Z$ with $|A|, |B| \geq \epsilon Q$ and  $|A+B| \leq |A| + |B| + \delta Q$, then there is a subgroup $H$ of $\Z/Q\Z$ of index at most $D$ such that at least one of the following holds:
    \begin{enumerate}[label=(\Roman{enumi}),ref=(\Roman{enumi}),leftmargin=*]
        \item\label{itm stability_cyclic_groups I} $A+B$ is $\epsilon$-periodic with respect to $H$, meaning $|(A+B+H)\setminus(A+B)| < \epsilon |H|$, and $|A+B+H| \leq |A+H| + |B+H|$.
        \item\label{itm stability_cyclic_groups II} $A+B$ is not $\epsilon$-periodic with respect to $H$, while $A$ and $B$ have partitions $A = (A_0 \cup A_1) + a_0$ and $B = (B_0 \cup B_1) + b_0$ for some $a_0 \in A$ and $b_0 \in B$ such that
        \begin{enumerate}[label=(II.\alph*),leftmargin=*]
            \item\label{itm stability_cyclic_groups II.a} $A_0, B_0 \ne \emptyset$,
            \item\label{itm stability_cyclic_groups II.b} $A_0, B_0 \subseteq H$,
            \item\label{itm stability_cyclic_groups II.c} $|A_0+B_0| \leq |A_0| + |B_0| + \epsilon |H|$,
            \item\label{itm stability_cyclic_groups II.d} at least one of $A_1$ or $B_1$ is nonempty,
            \item\label{itm stability_cyclic_groups II.e} $(A_1+H) \cap A_0 = \emptyset$ and $(B_1 + H) \cap B_0 = \emptyset$, and
            \item\label{itm stability_cyclic_groups II.f} $|(A_1 + H) \setminus A_1| \leq \epsilon |H|$ and $|(B_1+H) \setminus B_1| \leq \epsilon |H|$.
        \end{enumerate}
        \item\label{itm stability_cyclic_groups III} $|A+B| < (1-\epsilon)|H|$ and there exist $A',B' \subseteq H$ and $x,y \in \Z/Q\Z$ such that
        \begin{enumerate}[label=(III.\alph*),leftmargin=*]
            \item\label{itm stability_cyclic_groups III.a} $A \subseteq x+A'$ and $B \subseteq y + B'$,
            \item\label{itm stability_cyclic_groups III.b} $|(x+A')\setminus A| < \epsilon Q$ and $|(y+B') \setminus B| < \epsilon Q$, and
            \item\label{itm stability_cyclic_groups III.c} there exists a surjective homomorphism $\phi : H \to \Z/N\Z$ such that $A' = \phi^{-1}(I)$ and $B' = \phi^{-1}(J)$ for some $N > k$ and intervals $I,J\subseteq\Z/N\Z$.
        \end{enumerate}
    \end{enumerate}
\end{Theorem}

\begin{proof}
    This is the special case of \cite[Theorem 1.15]{Griesmer19} applied to cyclic groups $\Z/Q\Z$.
    The general version of \cite[Theorem 1.15]{Griesmer19} applies to arbitrary compact abelian groups and has as an additional possibility in item \ref{itm stability_cyclic_groups III} that the sets $A'$ and $B'$ arise from surjective homomorphisms to $\T$.
    Since a discrete group cannot have $\T$ as a homomorphic image, we can rule out this case when specializing to cyclic groups.

    In \cite[Theorem 1.15]{Griesmer19}, there is a single parameter $d = D$, and $k$ in \ref{itm stability_cyclic_groups III.c} is also replaced by $d$ rather than having a separate quantifier.
    To deduce the present statement from \cite[Theorem 1.15]{Griesmer19}, we consider two cases.
    If $d \geq k$, then we get $N > k$ in \ref{itm stability_cyclic_groups III.c} from the conclusion $N > d$ appearing in \cite[Theorem 1.15]{Griesmer19}.
    If $d < k$, then we can take a smaller $\epsilon'$ so that the corresponding $d'$ satifies $d' \geq k$, reducing to the previously handled case.
\end{proof}

Under the additional assumption that $B$ has non-empty intersection with every coset of every subgroup of small index, we prove the following inverse theorem, which bears a strong resemblance to \cref{thm_inverse_theorem_sumsets_compact_groups} and \cref{thm: main}.

\begin{Theorem} \label{thm: cyclic sum without obstructions}
    For every $\alpha, \beta, \epsilon, \eta > 0$ and $k \in \N$, there exist $\delta > 0$ and $D \in \N$ such that if $Q \in \N$ and $A, B \subseteq \Z/Q\Z$ satisfy $|A| \geq \alpha Q$, $|B| \geq \beta Q$, and $B$ has non-empty intersection with every coset of every subgroup of index at most $D$, then at least one of the following holds:
    \begin{enumerate}[label=(\roman*), leftmargin=*]
        \item\label{itm cyclic sum without obstructions i} $|A+B| \geq (\alpha+\beta+\delta)Q$.
        \item\label{itm cyclic sum without obstructions ii} $|A+B| \geq (1-\epsilon)Q$.
        \item\label{itm cyclic sum without obstructions iii} There exists a subgroup $H \leq \Z/Q\Z$ of index at most $D$, elements $a_0 \in A$ and $b_0 \in B$, and a decomposition $A = A_0 + a_0$ and $B = (B_0 \cup B_1) + b_0$ such that
        \begin{enumerate}[label=(iii.\alph*), leftmargin=*]
            \item\label{itm cyclic sum without obstructions iii.a} $A_0, B_0 \subseteq H$,
            \item\label{itm cyclic sum without obstructions iii.b} $B_1 \subseteq (\Z/Q\Z) \setminus H$ and $|B_1| > Q - |H| - \epsilon |H|$, and
            \item\label{itm cyclic sum without obstructions iii.c} there exists a surjective homomorphism $\phi : H \to \Z/N\Z$ for some $N > k$ and intervals $I,J \subseteq \Z/N\Z$ such that
            \[
                A_0\subset \phi^{-1}(I), \quad B_0\subset \phi^{-1}(J),\quad\text{and}\quad |\phi^{-1}(I)\setminus A_0|, |\phi^{-1}(J)\setminus B_0| < \epsilon Q.
            \]
        \end{enumerate}
        \item\label{itm cyclic sum without obstructions iv} There exists a subgroup $H \leq \Z/Q\Z$ of index at most $D$, elements $a_0 \in A$ and $b_0 \in B$, and a decomposition $A = A_0 + a_0$ and $B = (B_0 \cup B_1) + b_0$ such that
        \begin{enumerate}[label=(iv.\alph*), leftmargin=*]
            \item\label{itm cyclic sum without obstructions iv.a} $A_0, B_0 \subseteq H$,
            \item\label{itm cyclic sum without obstructions iv.b} $B_1 \subseteq (\Z/Q\Z) \setminus H$ and $|B_1| > Q - |H| - \eta |H|$
            \item\label{itm cyclic sum without obstructions iv.c} $|B_0| < \eta |H|$.
        \end{enumerate}
    \end{enumerate}
\end{Theorem}

\begin{proof}    
    We carry out a density increment argument on $\alpha$.
    Let $P(\alpha)$ be the statement:
    ``Theorem \ref{thm: cyclic sum without obstructions} holds for $\alpha$ and arbitrary $\beta, \epsilon, \eta > 0$ and $k \in \N$.''
    We will prove $P(\alpha)$ for $\alpha > \frac{1}{2}$ and the implication $P(2\alpha) \implies P(\alpha)$ for $\alpha \in (0,\frac{1}{2}]$.
    These two parts combine to prove Theorem \ref{thm: cyclic sum without obstructions} for all $\alpha \in (0,1)$: there exists $j \in \N$ such that $2^j\alpha \in (\frac{1}{2},1]$, so $P(2^j\alpha)$ holds and by induction we conclude $P(\alpha)$.

    \bigskip

    \textbf{Step 1. $P(\alpha)$ for $\alpha > \frac{1}{2}$.}
    
    Let us first prove $P(\alpha)$ for $\alpha > \frac{1}{2}$.
    Suppose $\alpha > \frac{1}{2}$.
    Let $\delta > 0$ and $D \in \N$ be given by Theorem \ref{thm:stability_cyclic_groups} for $\epsilon' = \min\{\epsilon,\beta,\frac{1}{2}\}$.
    Suppose $Q \in \N$, $A, B \subseteq \Z/Q\Z$, $|A| \geq \alpha Q$, $|B| \geq \beta Q$, and $B$ intersects every coset of every subgroup of index at most $D$.
    If $|A+B| \geq (\alpha+\beta+\delta)Q$, then there is nothing to show, so assume $|A+B| < (\alpha+\beta+\delta)Q$.
    Then by the choice of $\delta$, there exists a subgroup $H \leq \Z/Q\Z$ of index at most $D$ such that one of \ref{itm stability_cyclic_groups I}, \ref{itm stability_cyclic_groups II}, or \ref{itm stability_cyclic_groups III} holds from Theorem \ref{thm:stability_cyclic_groups}.

    Suppose \ref{itm stability_cyclic_groups I} holds.
    Since $B$ intersects every coset of $H$, we have $A+B+H = A+(B+H) = \Z/Q\Z$.
    Hence, by $\epsilon'$-periodicity of $A+B$, we have
    \begin{equation*}
        |A+B| \geq |A+B+H| - \epsilon' |H| \geq Q - \epsilon' Q \geq (1-\epsilon)Q,
    \end{equation*}
    so \ref{itm cyclic sum without obstructions ii} holds.
    
    Suppose now that \ref{itm stability_cyclic_groups II} holds.
    We partition $A = (A_0 \cup A_1) + a_0$ and $B = (B_0 \cup B_1) + b_0$ to satisfy \ref{itm stability_cyclic_groups II.a}--\ref{itm stability_cyclic_groups II.f}.
    Note that by properties \ref{itm stability_cyclic_groups II.d} and \ref{itm stability_cyclic_groups II.e}, the group $H$ is a proper subgroup of $\Z/Q\Z$.
    By \ref{itm stability_cyclic_groups II.b}, we have $|A_0| \leq |H| \leq \frac{Q}{2}$.
    Therefore, since $|A| \geq \alpha Q > \frac{Q}{2}$, we have $A_1 \ne \emptyset$.
    Let $a \in A_1$.
    Then by \ref{itm stability_cyclic_groups II.f}, $|(a+H) \setminus A| \leq \epsilon |H|$.
    Let $\tilde{A} = A \cap (a+H)$.
    Since $B$ intersects every coset of $H$, we conclude
    \begin{equation*}
        |A+B| \geq |\tilde{A} + B| \geq (1 - \epsilon)Q.
    \end{equation*}
    That is, \ref{itm cyclic sum without obstructions ii} holds.

    Finally, suppose \ref{itm stability_cyclic_groups III} holds.
    Property \ref{itm stability_cyclic_groups III.a} says that $B$ is contained in a single coset of $H$, but $B$ intersects every coset of $H$ by assumption, so we must have $H = \Z/Q\Z$.
    Hence, there exist sets $A', B' \subseteq \Z/Q\Z$ such that $A \subseteq A'$, $B \subseteq B'$, $|A' \setminus A| < \epsilon Q$, $|B' \setminus B| < \epsilon Q$, and $A' = \varphi^{-1}(I)$, $B' = \varphi^{-1}(J)$ for a surjective homomorphism $\Z/Q\Z \to \Z/N\Z$ for some $N > k$ and intervals $I,J \subseteq \Z/N\Z$.
    Taking $H = \Z/Q\Z$ and $a_0 = b_0 = 0$, property \ref{itm cyclic sum without obstructions iii} is satified.

    \bigskip

    \textbf{Step 2. $P(2\alpha) \implies P(\alpha)$ for $\alpha \leq \frac{1}{2}$.}
    
    We now show that for $\alpha \leq \frac{1}{2}$, we have the implication $P(2\alpha) \implies P(\alpha)$.
    Let $\alpha \in (0, \frac{1}{2}]$, and assume $P(2\alpha)$.
    Let $\delta_0 > 0$ and $D_0 \in \N$ be given by $P(2\alpha)$ for the parameters $\beta'=\eta$, $\epsilon'= \min\{\frac{\epsilon}{2},2\alpha,\frac{\eta}{2}\}$, $\eta' = \eta$, and $k' = k$.
    Let $\delta_1 > 0$ and $D_1$ be given by Theorem \ref{thm:stability_cyclic_groups} for $\epsilon'' = \min\left\{ \delta_0, \frac{\epsilon'}{D_0} \right\}$.
    Then put $\delta = \delta_1$ and $D = D_0D_1$.
    Suppose $Q \in \N$, $A, B \subseteq \Z/Q\Z$, $|A| \geq \alpha Q$, $|B| \geq \beta Q$, and $B$ intersects every coset of every subgroup of index at most $D$.
    As above (in the case $\alpha > \frac{1}{2}$), we may assume $|A+B| < (\alpha+\beta+\delta)Q$ and there exists a subgroup $H \leq \Z/Q\Z$ of index at most $D_1$ such that one of \ref{itm stability_cyclic_groups I}, \ref{itm stability_cyclic_groups II}, or \ref{itm stability_cyclic_groups III} holds from Theorem \ref{thm:stability_cyclic_groups} with $\epsilon''$.

    \bigskip

    If \ref{itm stability_cyclic_groups I} holds, then we may apply the same argument from above (which did not use any information about $\alpha$) to conclude that $|A+B| \geq (1-\epsilon)Q$.

    \bigskip

    Suppose \ref{itm stability_cyclic_groups II} holds.
    Decompose $A = (A_0 \cup A_1) + a_0$ and $B = (B_0 \cup B_1) + b_0$ accordingly.
    As above, if $A_1 \ne \emptyset$, then $|A+B| \geq (1-\epsilon'')Q$.
    Suppose $A_1 = \emptyset$.
    Since $B$ intersects every coset of $H$, we have $B_1+H = (\Z/Q\Z) \setminus H$ and by \ref{itm stability_cyclic_groups II.f}, $|B_1| > Q - |H| - \epsilon'' |H|$.
    If $|B_0| < \eta|H|$, then property \ref{itm cyclic sum without obstructions iv} is satisfied and we are done.
    
    Assume $|B_0| \geq \eta |H|$.
    Note that $A = A_0+a_0$, so $|A_0| = |A| = [\Z/Q\Z:H] \cdot \alpha|H| \geq 2\alpha |H|$.
    The group $H$ is also a cyclic group and $A_0, B_0 \subseteq H$ satisfy
    \begin{itemize}
        \item $|A_0| \geq 2\alpha|H|$,
        \item $|B_0| \geq \eta |H|$,
        \item $B_0$ intersects every coset of every subgroup of index at most $\frac{D}{[\Z/Q\Z:H]} \geq \frac{D}{D_1} = D_0$, and
        \item $|A_0+B_0| \leq |A_0| + |B_0| + \epsilon'' |H| \leq |A_0| + |B_0| + \delta_0 |H|$.
    \end{itemize}
    We then apply $P(2\alpha)$ (with $\beta'$, $\epsilon'$, and $\eta'$ as defined above) and consider the different cases.

    The sets $A_0$ and $B_0$ do not satisfy \ref{itm cyclic sum without obstructions i}.
    
    If \ref{itm cyclic sum without obstructions ii} holds for $A_0$ and $B_0$, i.e. $|A_0+B_0| \geq (1-\epsilon')|H|$, then
    \begin{multline*}
        |A+B| = |A_0+B_0| + |A_0+B_1| \\
         \geq (1-\epsilon')|H| + (Q - |H|- \epsilon''|H|)
         = Q - (\epsilon' + \epsilon'')|H| \geq (1-\epsilon) Q,
    \end{multline*}
    since $\epsilon', \epsilon'' \leq \frac{\epsilon}{2}$.
    Hence, \ref{itm cyclic sum without obstructions ii} holds for $A$ and $B$.

    Suppose \ref{itm cyclic sum without obstructions iii} holds for $A_0$ and $B_0$.
    That is, there is a subgroup $K \leq H$ of index at most $D_0$, elements $a' \in A_0$ and $b' \in B_0$, and a decomposition $A_0 = A' + a'$ and $B_0 = (B' \cup B'') + b'$ such that
    \begin{enumerate}[label=(\alph*), leftmargin=*]
        \item $A', B' \subseteq K$,
        \item $B'' \subseteq H \setminus K$ and $|B''| > |H| - |K| - \epsilon' |K|$, and
        \item there exists a surjective homomorphism $\phi : K \to \Z/N\Z$ for some $N > k$ and intervals $I,J \subseteq \Z/N\Z$ such that
        \[
            A' \subset \phi^{-1}(I), \quad B' \subset \phi^{-1}(J),\quad\text{and}\quad |\phi^{-1}(I)\setminus A'|, |\phi^{-1}(J)\setminus B'| < \epsilon' |H|.
        \]
    \end{enumerate}
    Observe that $K$ is then a subgroup of index at most $D_0D_1 = D$ in $\Z/Q\Z$ and properties \ref{itm cyclic sum without obstructions iii.a}--\ref{itm cyclic sum without obstructions iii.c} are satisfied for $A$ and $B$ in relation to the decomposition $A = A' + (a' + a_0)$ and $B = (B' \cup (B'' \cup (B_1-b'))) + (b'+b_0)$:
    \begin{enumerate}[label=(iii.\alph*), leftmargin=*]
        \item $A', B' \subseteq K$,
        \item $B'' \cup (B_1-b') \subseteq (H \setminus K) \cup ((\Z/Q\Z) \setminus H) = (\Z/Q\Z) \setminus K$ and
            \begin{multline*}
                |B'' \cup (B_1 - b')| = |B''| + |B_1| > |H| - |K| - \epsilon' |K| + Q - |H| - \epsilon'' |H| \\
            = Q - |K| - \epsilon' |K| - \epsilon'' |H| \geq Q - |K| - \epsilon |K|,
            \end{multline*}
            where in the last inequality we have used the choice of parameters $\epsilon' \leq \frac{\epsilon}{2}$ and $\epsilon'' \leq \frac{\epsilon'}{D_0}$.
        \item there exists a surjective homomorphism $\phi : K \to \Z/N\Z$ for some $N > k$ and intervals $I,J \subseteq \Z/N\Z$ such that
        \[
            A' \subset \phi^{-1}(I), \quad B' \subset \phi^{-1}(J),\quad\text{and}\quad |\phi^{-1}(I)\setminus A'|, |\phi^{-1}(J)\setminus B'| < \epsilon' |H| < \epsilon Q.
        \]
    \end{enumerate}

    Finally, suppose \ref{itm cyclic sum without obstructions iv} holds for $A_0$ and $B_0$.
    That is, there exists a subgroup $K \leq H$ of index at most $D_0$, elements $a' \in A_0$ and $b' \in B_0$, and a decomposition $A_0 = A' + a'$ and $B_0 = (B' \cup B'') + b'$ such that
    \begin{enumerate}[label=(\alph*), leftmargin=*]
        \item $A', B' \subseteq K$,
        \item $B'' \subseteq H \setminus K$ and $|B'| > |H| - |K| - \eta |K|$
        \item $|B'| < \eta |K|$.
    \end{enumerate}
    Then $K$ has index at most $D$ in $\Z/Q\Z$ and the sets $A$ and $B$ satisfy \ref{itm cyclic sum without obstructions iv} with the decomposition $A = A' + (a'+a_0)$ and $B = (B' \cup (B'' \cup (B_1-b'))) + (b'+b_0)$:
    \begin{enumerate}[label=(iv.\alph*), leftmargin=*]
        \item $A', B' \subseteq K$,
        \item $B'' \cup (B_1-b') \subseteq (H \setminus K) \cup ((\Z/Q\Z) \setminus H) = (\Z/Q/Z) \setminus K$ and
            \begin{multline*}
                |B'' \cup (B_1 - b')| = |B''| + |B_1| > |H| - |K| - \epsilon' |K| + Q - |H| - \epsilon'' |H| \\
            = Q - |K| - \epsilon' |K| - \epsilon'' |H| \geq Q - |K| - \eta |K|,
            \end{multline*}
            where in the last inequality we have used the choice of parameters $\epsilon' \leq \frac{\eta}{2}$ and $\epsilon'' \leq \frac{\epsilon'}{D_0}$.
        \item $|B'| < \eta |K|$.
    \end{enumerate}

    \bigskip

    The final case to consider is when \ref{itm stability_cyclic_groups III} holds.
    The argument above handling the case $\alpha > \frac{1}{2}$ did not directly use any assumption on $\alpha$ so applies equally well for $\alpha \leq \frac{1}{2}$.
    Hence, property \ref{itm cyclic sum without obstructions iii} holds with $H = \Z/Q\Z$.
\end{proof}

%SUBSECTION
\subsection{Popular sumsets in finite abelian groups}

Let $G$ be a finite abelian group.
Given two sets $A, B \subseteq G$ and $\delta \in [0,1]$, the \emph{$\delta$-popular sumset} is\nomenclature{$+_{\delta}$}{$\delta$-popular sumset}
\begin{equation}
\label{def_popular_sumset}
    A +_{\delta} B = \{x \in G : |A \cap (x-B)| > \delta |G|\}.
\end{equation}

We make a couple of preliminary observations about $\delta$-popular sumsets.
First, the $\delta$-popular sumset can be expressed equivalently as $A +_{\delta} B = \{x \in G : (\1_A * \1_B)(x) > \delta\}$, where $f*g$ denotes the \emph{convolution}\nomenclature{$*$}{convolution}
\begin{equation*}
    (f*g)(x) = \frac{1}{|G|} \sum_{u+v=x} f(u) g(v).
\end{equation*}
Second, the $0$-popular sumset $A +_0 B$ is equal to the ordinary sumset $A+B$.
As $\delta$ decreases to 0, one may therefore expect that $A+_{\delta}B$ increasingly resembles a genuine sumset.
The following theorem makes this relationship precise.
Importantly here, $\delta$ does not depend on the group $G$ or the sets $A, B \subseteq G$, only on the degree $\epsilon$ to which $A+_{\delta}B$ should resemble a sumset.

\begin{Theorem}\label{thm: delta-sum almost contains a sum}
    For every $\epsilon > 0$, there exists $\delta > 0$ such that if $G$ is a finite abelian group and $A, B \subseteq G$, then there exist $A' \subseteq A$ and $B' \subseteq B$ such that
    \begin{itemize}
        \item $|A \setminus A'| < \epsilon |G|$,
        \item $|B \setminus B'| < \epsilon |G|$, and
        \item $|(A'+B') \setminus (A+_{\delta}B)| < \epsilon |G|$.
    \end{itemize}
\end{Theorem}

\cref{thm: delta-sum almost contains a sum} is closely related to the ``almost all'' version of the Balog--Szemer\'{e}di--Gowers theorem due to Shao \cite[Theorem 1.1]{Shao19}.
We follow a similar strategy of proof, deducing \cref{thm: delta-sum almost contains a sum} from the following special case of the arithmetic removal lemma of Green.

\begin{Lemma}[cf. {\cite[Theorem 1.5]{Green05}}] \label{lem: arithmetic removal}
    For every $\epsilon > 0$, there exists $\delta > 0$ such that if $G$ is a finite abelian group and $A, B, C \subseteq G$ are subsets such that the equation $a+b=c$ has fewer than $\delta|G|^2$ solutions $(a,b,c) \in A \times B \times C$, then one can remove fewer than $\epsilon|G|$ elements from each of the sets $A,B,C$ to obtain sets $A', B', C'$ such that $a'+b'=c'$ has no solutions $(a',b',c') \in A' \times B' \times C'$.
\end{Lemma}

\begin{proof}[Proof of \cref{thm: delta-sum almost contains a sum}]
    Let $\epsilon > 0$ be given.
    Take $\delta > 0$ as given by Lemma \ref{lem: arithmetic removal}.
    Let $G$ be a finite abelian group and $A, B \subseteq G$.
    Put $C = (A+B) \setminus (A+_{\delta}B)$.
    By definition of the $\delta$-popular sumset, if $c \in G$ such that $a+b=c$ has more than $\delta|G|$ solutions with $(a,b) \in A \times B$, then $c \notin C$.
    Therefore, the triple $(A, B, C)$ satisfies the hypothesis of the arithmetic removal lemma, so there are subsets $A' \subseteq A$, $B' \subseteq B$, and $C' \subseteq C$ such that
    \begin{itemize}
        \item $|A \setminus A'| < \epsilon |G|$,
        \item $|B \setminus B'| < \epsilon |G|$,
        \item $|C \setminus C'| < \epsilon |G|$, and
        \item $(A'+B') \cap C' = \emptyset$.
    \end{itemize}
    Reinterpreting this final property, we have
    \begin{equation*}
        A' + B' \subseteq (A+B) \setminus C' = (A+_{\delta}B) \cup (C \setminus C'),
    \end{equation*}
    so
    \begin{equation*}
        |(A'+B') \setminus (A+_{\delta}B)| \le |C\setminus C'| < \epsilon |G|.
    \end{equation*}
\end{proof}

%SECTION
\section{Uniform distribution and discrepancies of sequences mod 1} \label{sec: discrepancy}

A classical notion in number theory is the phenomenon of uniform distribution.
A sequence of real numbers $(x_n)_{n \in \N}$ is \emph{uniformly distributed mod 1} if for every interval $I \subseteq \T$,
\begin{equation*}
    \lim_{N \to \infty} \frac{\left| n \in [N] : \{x_n\} \in I\right|}{N} = |I|,
\end{equation*}
where $|I|$ is the length (Haar measure) of the interval $I$ and $\{x_n\}$ denotes the fractional part of $x_n$.
Uniform distribution was first systematically studied by Weyl \cite{Weyl16}, who showed that $(n\alpha)_{n \in \N}$ is uniformly distributed mod 1 for irrational $\alpha \in \R \setminus \Q$ and, more generally, the same is true for a polynomial sequence $(P(n))_{n \in \N}$ if $P(x) \in \R[x]$ has an irrational coefficient other than the constant term.

In this paper, we will need to estimate the rate at which sequences such as $(n\alpha)_{n \in \N}$ equidistribute.
A useful measure for such quantitative purposes is the discrepancy of a sequence.
For $N \in \N$, the \emph{discrepancy} of a sequence of real numbers $(x_n)_{n \in \N}$ is given by\nomenclature{$\textup{Disc}$}{discrepancy}
\begin{equation*}
    \textup{Disc}_N \left( (x_n)_{n \in \N} \right) = \sup_{I \subseteq \T~\text{interval}} \left| \frac{\left| \{n \in [N] : x_n \in I\}\right|}{N} - |I| \right|.
\end{equation*}
A sequence is uniformly distributed mod 1 if and only if its discrepancy tends to 0 as $N \to \infty$ (see \cite[Chapter 2, Theorem 1.1]{KN74}).
The rate of convergence to 0 provides for quantitative refinements of uniform distribution.

In this section, we recall general estimates on discrepancies of sequences (the inequalities of \Erdos{}--\Turan{} and \Erdos{}--\Turan{}--Koksma) and then apply the general estimates to specific sequences of interest.

%SUBSECTION
\subsection{The \Erdos{}--\Turan{} and \Erdos{}--\Turan{}--Koksma inequalities}

A useful tool for estimating discrepancies of sequences is the \Erdos{}--\Turan{} inequality, relating the discrepancy of a sequence to the behavior of associated exponential sums.
We use the standard number theoretic notation $e(x)$ to denote $e^{2\pi ix}$.\nomenclature{$e(\cdot)$}{complex exponential function}
In many computations, it is useful to estimate the difference $|e(x)-1|$ in terms of the distance of $x$ to the nearest integer\nomenclature{$\lVert\cdot\rVert_{\T}$}{distance to nearest integer}
\begin{equation*}
    \|x\|_{\T} = \min_{n \in \Z} |x-n|.
\end{equation*}
The quantities $|e(x)-1|$ and $\|x\|_{\T}$ are related by the inequalities
\begin{equation*}
    4\|x\|_{\T} \leq |e(x)-1| \leq 2\pi \|x\|_{\T}.
\end{equation*}

\begin{Theorem}[\Erdos{}--\Turan{} inequality {\cite[Theorem III]{ET48}}] \label{thm: Erdos-Turan}
    Let $t_1, \ldots, t_m$ be real numbers.
    There is a universal constant $C_0$ such that
    \begin{equation*}
        \textup{Disc}_m = \sup_{I \subseteq \T~\text{interval}} \left| \frac{|\{j \in [m] : t_j \in I \}|}{m} - |I| \right| \leq C_0 \left( \frac{1}{n} + \frac{1}{m} \sum_{k=1}^n \frac{1}{k} \left| \sum_{j=1}^m e(kt_j) \right|\right)
    \end{equation*}
    for every $n \in \N$.
\end{Theorem}

A multidimensional extension of the \Erdos{}--\Turan{} inequality is the \Erdos{}--\Turan{}--Koksma inequality, established independently by Koksma \cite{Koksma50} and Sz\"{u}sz \cite{Szusz52}.
In order to state the inequality, we need a higher dimensional notion of discrepancy.
By a \emph{box} in $\T^d$, we mean a set of the form $B = \prod_{i=1}^d I_i$, where each $I_i \subseteq \T$ is an interval.
For $N \in \N$, the \emph{discrepancy} of a sequence of $d$-dimensional vectors $(\bm{x}_n)_{n \in \N}$ is given by
\begin{equation*}
    \textup{Disc}_N \left( (\bm{x}_n)_{n \in \N} \right) = \sup_{B \subseteq \T^d~\text{box}} \left| \frac{\left| \{n \in [N] : \bm{x}_n \in B\}\right|}{N} - |B| \right|,
\end{equation*}
where $|B|$ is the volume (Haar measure) of $B$.
We now state the \Erdos{}--\Turan{}--Koksma inequality, as formulated in \cite[p. 116]{KN74}.

\begin{Theorem}[\Erdos{}--\Turan{}--Koksma inequality] \label{thm: Erdos-Turan-Koksma}
    Let $\bm{t}_1, \ldots, \bm{t}_m \in \R^d$.
    Then for every $n \in \N$,
    \begin{multline*}
        \textup{Disc}_m = \sup_{B \subseteq \T~\text{box}} \left| \frac{|\{j \in [m] : \bm{t}_j \in B \}|}{m} - |B| \right| \\
         \leq 6d^2 3^d\left( \frac{1}{n} + \frac{1}{m} \sum_{\bm{h} \in \Z^d, 0 < \|\bm{h}\|_{\infty} \leq n} \frac{1}{\prod_{i=1}^d \max\{1, |h_i|\}} \left| \sum_{j=1}^m e\left( \langle \bm{h}, \bm{t_j} \rangle \right) \right| \right).
    \end{multline*}
\end{Theorem}

%SUBSECTION
\subsection{Local densities of Bohr intervals}

For $\theta \in \T$ and an interval $I \subseteq \T$, let\nomenclature{$\Bohr(\cdot,\cdot)$}{Bohr interval in $\N$}
\begin{equation*}
    \Bohr(\theta,I) = \{n \in \N : n\theta \in I\}.
\end{equation*}

If $\theta \notin \Q$, then $(n\theta)_{n \in \N}$ is uniformly distributed mod 1, so $d(\Bohr(\theta,I)) = |I|$.
For $\theta = \frac{p}{q} \in \Q$, we can pick an interval $J = \left[ \frac{a}{q}, \frac{b}{q} \right)$ so that $\Bohr(\theta,I) = \Bohr(\theta,J)$, and $d(\Bohr(\theta,I)) = \frac{b-a}{q} = |J|$.
Thus, we are always free to assume that $I$ is chosen so that $d(\Bohr(\theta,I)) = |I|$.

In this subsection, we produce estimates on the quantity
\begin{equation} \label{eq: discrepancy}
    \left| \frac{|\Bohr(\theta,I) \cap \{x+1,\ldots,x+M\}|}{M} - |I| \right|
\end{equation}
at different scales $M$ depending on diophantine properties of $\theta$ and $I$.
In line with many classical applications of the Hardy--Littlewood circle method, the relevant property of $\theta$ is its distance from rational numbers with small denominators.
Interestingly, as we will see, there are additional, less conventional features that arise when estimating \eqref{eq: discrepancy}: the behavior of the rational points themselves is somewhat distinct from the behavior of other ``major arc'' points (numbers well approximated by rationals with small denominator); moreover, the behavior of ``major arcs'' depends on whether the length of the interval $I$ is also major arc.

A general observation that will be useful for establishing upper bounds on \eqref{eq: discrepancy} is that the quantity in \eqref{eq: discrepancy} is equal to
\begin{equation*}
    \left| \frac{|\{n \in [M] : (x+n)\theta \in I\}|}{M} - |I| \right|,
\end{equation*}
which can be bounded above by the discrepancy of the sequence $(n\theta)_{n \in \N}$.

%SUBSUBSECTION
\subsubsection{Minor arcs}

Our first estimate bounds the discrepancy \eqref{eq: discrepancy} when $\theta$ is far away from rational numbers with small denominator (the ``minor arcs'').

\begin{Lemma} \label{lem: discrepancy for irrational Bohr set}
    Let $\alpha, \delta > 0$ and $Q \in \N$.
    If $\Bohr(\theta,I)$ has density $\alpha$ and $\|q\theta\|_{\T} \geq \delta$ for all $1 \leq q \leq Q$, then
    \begin{equation*}
        \left| \frac{|\Bohr(\theta,I) \cap \{x+1,\ldots,x+M\}|}{M} - \alpha \right| \leq C \left( \frac{1}{Q} + \frac{\log{Q}}{M\delta} \right)
    \end{equation*}
    for every $x \in \N$ and every $M \in \N$, where $C$ is a universal constant.
\end{Lemma}

\begin{proof}
    Without loss of generality, we may assume $|I| = \alpha$.
    Fix $x, M \in \N$.
    By the \Erdos{}--\Turan{} inequality (Theorem \ref{thm: Erdos-Turan}),
    \begin{align*}
        \left| \frac{|\Bohr(\theta,I) \cap \{x+1,\ldots,x+M\}|}{M} - \alpha \right| 
         & \leq \textup{Disc}_M \left( (n\theta)_{n \in \N} \right) \\
         & \leq C_0 \left( \frac{1}{Q} + \frac{1}{M} \sum_{q=1}^Q \frac{1}{q} \left| \sum_{m=1}^M e(mq\theta) \right| \right).
    \end{align*}
    Now for each $q \leq Q$, we may compute the geometric sum
    \begin{equation*}
        \left| \sum_{m=1}^M e(mq\theta) \right| = \left| \sum_{m=1}^M e(q\theta)^m \right| = \frac{\left| e(Mq\theta)-1 \right|}{\left| e(q\theta)-1 \right|} \leq \frac{2}{4\|q\theta\|_{\T}} \leq \frac{1}{2\delta}.
    \end{equation*}
    Therefore,
    \begin{equation*}
        \left| \frac{|\Bohr(\theta,I) \cap \{x+1,\ldots,x+M\}|}{M} - \alpha \right|  \leq C_0 \left( \frac{1}{Q} + \frac{1}{2M\delta} \sum_{q=1}^Q \frac{1}{q} \right) \ll \frac{1}{Q} + \frac{\log{Q}}{M\delta}.
    \end{equation*}
\end{proof}

%SUBSUBSECTION
\subsubsection{Rational frequencies}

For the complementary case, when $\theta$ is well approximated by a rational number with small denominator, there are three distinct possible behaviors for the discrepancy \eqref{eq: discrepancy}.
We first deal with the case that $\theta$ is rational, where we may leverage the periodicity of the sequence $(n\theta)_{n \in \N}$ to control the discrepancy whenever the scale $M$ is large compared to the denominator of $\theta$.

\begin{Lemma} \label{lem: small discrepancy for rational Bohr set}
    If $\theta = \frac{p}{q}$ and $I = \left[\frac{a}{q}, \frac{b}{q} \right)$, then
    \begin{equation*}
        \left| \frac{|\Bohr(\theta, I) \cap \{x+1,\ldots,x+M\}|}{M} - \frac{b-a}{q} \right| < \frac{q}{M}
    \end{equation*}
    for every $x \in \N$ and $M \in \N$.
\end{Lemma}

\begin{proof}
    Write $M = kq + r$ with $r < q$.
    We can decompose
    \begin{equation*}
        \{x+1,\ldots,x+M\} = \bigcup_{j=0}^{k-1} \underbrace{\{x+jq+1, \ldots, x + (j+1)q)\}}_{J_j} \cup \underbrace{\{x+kq+1,\ldots,x+M\}}_{E}.
    \end{equation*}
    Each of the intervals $J_j$ of length $q$ contains exactly $b-a$ elements of $\Bohr(\theta, I)$.
    Therefore,
    \begin{equation*}
        k(b-a) \leq |\Bohr(\theta,I) \cap \{x+1,\ldots,x+M\}| \leq k(b-a) + r
    \end{equation*}
    Dividing through by $M = kq+r$ gives the desired inequality.
\end{proof}

%SUBSUBSECTION
\subsubsection{Major arcs}

In the case that $\theta$ is close but not equal to a rational number $\frac{p}{q}$, say with $\left| \theta - \frac{p}{q} \right| = \epsilon$, then the behavior \eqref{eq: discrepancy} depends critically on the scale $M$ relative to $\epsilon^{-1}$, so long as the length of the interval $I$ is not also well-approximated by a rational number of denominator $q$.

\begin{Lemma} \label{lem: large discrepancy for rational Bohr set}
    Let $\theta = \frac{p}{q}$, and let $\alpha \in (0,1)$.
    If $\epsilon \in (0,1)$, $K > 1$, and $KM \leq \epsilon^{-1} \leq K^{-1}H$, then
    \begin{equation*}
        \sup_{x \in \N} \left| \frac{|\Bohr(\theta + \epsilon, I) \cap \{x+1,\ldots,x+H\}|}{H} - \alpha \right| \ll q\epsilon + \frac{1}{K}
    \end{equation*}
    and
    \begin{equation*}
        \frac{1}{H} \sum_{x=1}^H \left| \frac{|\Bohr(\theta + \epsilon, I) \cap \{x+1,\ldots,x+M\}|}{M} - \alpha \right| \geq \frac{\|q\alpha\|_{\T}}{q} - \Oh \left( q\epsilon + \frac{1}{K} + \frac{q}{M} \right)
    \end{equation*}
    for every $I \subseteq \T$ with $d(\Bohr(\theta+\epsilon,I)) = \alpha$.
\end{Lemma}

\begin{proof}
    For convenience in computations, we will assume $|I| = \alpha$.
    Fix $x \in \N$ and $L \in \N$.
    The number $|\Bohr(\theta+\epsilon,I) \cap \{x+1,\ldots,x+L\}|$ counts how many terms of the sequence $(x+n)(\theta+\epsilon)$, $n \in [L]$, land in $I$ mod $1$.
    Decomposing into residue classes mod $q$, we can write
    \begin{multline} \label{eq: decompose by residues}
        |\Bohr(\theta+\epsilon,I) \cap \{x+1,\ldots,x+L\}| = \sum_{r=0}^{q-1} \sum_{n=1}^{L/q} \1_I \left( \underbrace{(x+r)(\theta+\epsilon)}_{a_r(x)} + nq\epsilon \right) + \Oh(q) \\
        = \sum_{r=0}^{q-1} \left| \Bohr(q\epsilon, I - a_r(x)) \cap [L/q] \right| + \Oh(q).
    \end{multline}
    Taking $L = H \geq K\epsilon^{-1}$ and applying the Erd\H{o}s--Tur\'{a}n inequality,
    \begin{multline*}
        \left| \frac{\left| \Bohr(q\epsilon, I - a_r(x)) \cap [H/q] \right|}{H/q} - \alpha \right| \leq C_0 \left( 2q\epsilon + \frac{q}{H} \sum_{k=1}^{(2q\epsilon)^{-1}} \frac{1}{k} \underbrace{\left| \sum_{h=1}^{H/q} e(hq\epsilon k) \right|}_{\leq \frac{1}{2kq\epsilon}} \right) \\
         \leq C_0 \left( 2q\epsilon + \frac{1}{2H\epsilon} \cdot \frac{\pi^2}{6} \right)
         \ll q\epsilon + \frac{1}{H\epsilon} \leq q\epsilon + \frac{1}{K}.
    \end{multline*}
    Thus, averaging over $r \in \{0, 1, \ldots, q-1\}$, we have
    \begin{equation*}
        \sup_{x \in \N} \left| \frac{|\Bohr(\theta + \epsilon, I) \cap \{x+1,\ldots,x+H\}|}{H} - \alpha \right| \ll q \epsilon + \frac{1}{K} + \Oh \left( \frac{q}{H} \right).
    \end{equation*}
    Noting that $\frac{q}{H} \leq \epsilon \frac{q}{K} \leq q\epsilon$, this proves the first part of the lemma.

    \bigskip

    Now we wish to estimate
    \begin{equation*}
        \left| \frac{\left| \Bohr(q\epsilon, I - a_r(x)) \cap [M/q] \right|}{M/q} - \alpha \right|
    \end{equation*}
    for $x \in [H]$.
    Note that $mq\epsilon \in (0,1)$ for every $m \in [M/q]$, so
    \begin{equation*}
        \left| \Bohr(q\epsilon, I-a_r(x)) \cap [M/q] \right| = \sum_{m=1}^{M/q} \1_{I-a_r(x)}(mq\epsilon) = \frac{|(0,M\epsilon) \cap (I - a_r(x))|}{q\epsilon} + \Oh(1).
    \end{equation*}
    Hence,
    \begin{equation} \label{eq: approximate by intersection of intervals}
        \frac{\left| \Bohr(q\epsilon, I-a_r(x)) \cap [M/q] \right|}{M/q} = \frac{\left| I \cap (a_r(x), a_r(x)+M\epsilon) \right|}{M\epsilon} + \Oh \left( \frac{q}{M} \right).
    \end{equation}

    We first give a heuristic argument to show that
    \begin{equation*}
        \frac{1}{H} \sum_{x=1}^H \left| \frac{|\Bohr(\theta + \epsilon, I) \cap \{x+1,\ldots,x+M\}|}{M} - \alpha \right|
    \end{equation*}
    is large and then make it precise.
    Over the range $x \in [H]$, the value of $a_0(x)$ is nearly uniformly distributed in $\T$ by the first part of the lemma, so
    \begin{multline*}
        \frac{1}{H} \sum_{x=1}^H \left| \frac{|\Bohr(\theta + \epsilon, I) \cap \{x+1,\ldots,x+M\}|}{M} - \alpha \right| \\
        \approx \int_0^1 \left| \frac{1}{q} \sum_{r=0}^{q-1} \frac{\left| I \cap \left( t + \frac{rp}{q},t + \frac{rp}{q} +M\epsilon \right) \right|}{M\epsilon} - \alpha \right|~dt \\
        = \int_0^1 \left| \frac{1}{q} \sum_{r=0}^{q-1} \frac{\left| (0,\alpha) \cap \left( t + \frac{rp}{q},t + \frac{rp}{q} +M\epsilon \right) \right|}{M\epsilon} - \alpha \right|~dt.
    \end{multline*}
    Since $M\epsilon$ is small (recall $M\epsilon \leq K^{-1}$), the function $\frac{\left| (0,\alpha) \cap \left( t,t+M\epsilon \right) \right|}{M\epsilon}$ is approximately equal to $\1_{(0,\alpha)}(t)$, so we get the further approximation
    \begin{equation*}
        \int_0^1 \left| \frac{1}{q} \sum_{r=0}^{q-1} \1_{(0,\alpha)}\left( t + \frac{rp}{q} \right) - \alpha\right|~dt \geq \min_{0 \leq a \leq q} \left| \alpha - \frac{a}{q} \right| = \frac{\|q\alpha\|_{\T}}{q}.
    \end{equation*}

    Now let us track the errors introduced by each approximation.
    Define $F_r : [H] \to [0,1]$ by
    \begin{equation*}
        F_r(x) = \frac{\left| I \cap (a_r(x), a_r(x)+M\epsilon) \right|}{M\epsilon}.
    \end{equation*}
    Then let
    \begin{equation*}
        F(t) = \frac{\left| (0,\alpha) \cap (t, t+M\epsilon) \right|}{M\epsilon}
        = \begin{cases}
            1, & 0 \leq t < \alpha - M\epsilon; \\
            \frac{1}{M\epsilon}(\alpha - t), & \alpha - M\epsilon \leq t < \alpha; \\
            0, & \alpha \leq t < 1 - M\epsilon; \\
            \frac{1}{M\epsilon}(t-1+M\epsilon), & 1 - M\epsilon \leq t < 1
        \end{cases}
    \end{equation*}
    so that $F_r(x) = F(a_r(x)-c)$, where $c$ is the left endpoint of $I$.
    Using the estimate from the first part of the lemma,\footnote{Strictly speaking, the first part of the lemma bounds the discrepancy only for intervals.
    However, the function $t \mapsto \frac{1}{q} \sum_{r=0}^{q-1} F(r(\theta+\epsilon)+t)$ is $[0,1]$-valued and Riemann integrable, so the same estimates apply.}
    \begin{multline} \label{eq: approximate by integral}
        \frac{1}{q} \sum_{r=0}^{q-1} \frac{1}{H} \sum_{x=1}^{H} F_r(x) = \frac{1}{q} \sum_{r=0}^{q-1} \frac{1}{H} \sum_{x=1}^{H} F(-c+r(\theta+\epsilon) + x(\theta+\epsilon)) \\
        = \frac{1}{q} \sum_{r=0}^{q-1} \int_0^1 F(r(\theta+\epsilon) + t)~dt + \Oh\left( q\epsilon + \frac{1}{K} \right).
    \end{multline}
    The function $F$ is piecewise linear with maximal slope $\pm \frac{1}{M\epsilon}$.
    Hence, for $0 \leq r \leq q-1$,
    \begin{equation*}
        \max_{t \in \T} |F(t+r\epsilon) - F(t)| \leq \frac{r\epsilon}{M\epsilon} \leq \frac{q}{M},
    \end{equation*}
    so
    \begin{equation} \label{eq: integral over equally spaced points}
        \frac{1}{q} \sum_{r=0}^{q-1} \int_0^1 F(r(\theta+\epsilon) + t)~dt = \frac{1}{q} \sum_{r=0}^{q-1} \int_0^1 F \left( \frac{rp}{q} + t \right)~dt + \Oh \left( \frac{q}{M} \right).
    \end{equation}

    Combining \eqref{eq: decompose by residues}, \eqref{eq: approximate by intersection of intervals}, \eqref{eq: approximate by integral}, and \eqref{eq: integral over equally spaced points},
    \begin{equation*}
        \frac{|\Bohr(\theta+\epsilon,I) \cap \{x+1,\ldots,x+M\}|}{M} = \frac{1}{q} \sum_{r=0}^{q-1} \int_0^1 F \left( \frac{rp}{q} + t \right)~dt + \Oh \left( q\epsilon + \frac{1}{K} + \frac{q}{M} \right).
    \end{equation*}
    Finally,
    \begin{align*}
        \left| \frac{1}{q} \sum_{r=0}^{q-1} \int_0^1 F \left( \frac{rp}{q} + t \right)~dt \right. & \left. -~\frac{1}{q} \sum_{r=0}^{q-1} \int_0^1 1_{(0,\alpha)} \left( \frac{rp}{q} + t \right)~dt \right| \\
         & \leq \int_0^1 \left| F(t) - \1_{(0,\alpha)}(t) \right|~dt \\
         & = \int_{\alpha-M\epsilon}^{\alpha} \left( 1 - \frac{1}{M\epsilon}(\alpha-t) \right)~dt + \int_{1-M\epsilon}^1 \frac{1}{M\epsilon} (t-1+M\epsilon)~dt \\
         & = \frac{1}{M\epsilon} \leq \frac{1}{K}.
    \end{align*}
    This completes the proof.
\end{proof}

\begin{Corollary} \label{cor: local periodicity for nearly rational frequency}
    Let $\theta = \frac{p}{q}$ expressed in lowest terms and $\alpha \in (0,1)$ with $\{q\alpha\} = q\alpha - \lfloor q\alpha \rfloor \geq \delta > 0$.
    Let $\epsilon \in \left( 0, \frac{\delta}{q} \right)$, $K > 1$, and $KM \leq \epsilon^{-1} \leq K^{-1}H$.
    Let $I \subseteq \T$ be an interval of length $|I| = \alpha$.
    \begin{enumerate}[label=(\arabic*), leftmargin=*]
        \item There is a set $E \subseteq [H]$ with $|E| \geq \left( \delta - \Oh \left( q^2\epsilon + \frac{q}{K} \right) \right) |H|$ such that for every $x \in E$,
            \begin{equation*}
                \Bohr(\theta+\epsilon,I) \cap \{x+1,\ldots,x+M\} = \bigcup_{r \in R_x} (q\Z + r) \cap \{x+1,\ldots,x+M\}
            \end{equation*}
            for some set $R_x \subseteq \{0,1, \ldots, q-1\}$ with $|R_x| = \lceil q\alpha \rceil$.
        \item If $\alpha > \frac{1}{q}$ and $q\epsilon + \frac{1}{K} < \frac{\delta}{q}$, then for every $x \in \N$, there exists $r_x \in \{0, 1, \ldots, q-1\}$ such that
            \begin{equation*}
                (q\Z + r_x) \cap \{x+1,\ldots,x+M\} \subseteq \Bohr(\theta+\epsilon,I) \cap \{x+1,\ldots,x+M\}.
            \end{equation*}
    \end{enumerate}
\end{Corollary}

\begin{proof}
    Write $I = [c,c+\alpha)$.

    \bigskip
    
    (1) Suppose $x(\theta+\epsilon) \in \left[ c + \frac{s}{q}, c + \frac{s}{q} + \frac{\delta}{q} - \frac{1}{K} \right)$ for some $s \in \{0, 1, \ldots, q-1\}$.
    Then for $m \in [M]$, we have
    \begin{equation} \label{eq: approximation of orbit by rationals}
        (x+m)(\theta+\epsilon) - \left( c + \frac{s}{q} + m\frac{p}{q} \right) = x(\theta+\epsilon) - c - \frac{s}{q} + m \epsilon \in \left[ \epsilon, \frac{\delta}{q} \right),
    \end{equation}
    since $M\epsilon \leq \frac{1}{K}$.
    We claim that
    \begin{multline} \label{eq: periodic expression for Bohr set}
        \Bohr(\theta+\epsilon,I) \cap \{x+1,\ldots,x+M\} \\
        = \left\{ n \in \{x+1,\ldots,x+M\} : \frac{s}{q} + (n-x) \frac{p}{q} \in \left[ 0,\alpha - \frac{\delta}{q} \right] \right\}.
    \end{multline}
    Indeed, if $m \in [M]$ and $\frac{s}{q} + m \frac{p}{q} \in \left[ 0,\alpha - \frac{\delta}{q} \right]$, then by \eqref{eq: approximation of orbit by rationals},
    \begin{equation*}
        (x+m)(\theta+\epsilon) \in c + \left[ 0,\alpha - \frac{\delta}{q} \right] + \left[ \epsilon, \frac{\delta}{q} \right) = \left[c + \epsilon, c+ \alpha \right) \subseteq I.
    \end{equation*}
    Conversely, if $m \in [M]$ and $(x+m)(\theta+\epsilon) \in I = [c,c+\alpha)$, then by \eqref{eq: approximation of orbit by rationals},
    \begin{equation*}
        \frac{s}{q} + m \frac{p}{q} \in [0, \alpha) - \left[ \epsilon, \frac{\delta}{q} \right) = \left( -\frac{\delta}{q}, \alpha - \epsilon \right).
    \end{equation*}
    But $\frac{s}{q} + m\frac{p}{q}$ is a rational number with denominator $q$, so it must lie in the smaller interval
    \begin{equation*}
        \frac{s}{q} + m\frac{p}{q} \in \left[0, \alpha - \frac{\delta}{q} \right].
    \end{equation*}

    From \eqref{eq: periodic expression for Bohr set}, we see that $\Bohr(\theta+\epsilon,I) \cap \{x+1,\ldots,x+M\}$ is a $q$-periodic set expressible in the form
    \begin{equation*}
        \bigcup_{r \in R_x} (q\Z + r) \cap \{x+1,\ldots,x+M\}
    \end{equation*}
    for
    \begin{equation*}
        R_x = \left\{ r \in \{0,1, \ldots, q-1\} : \frac{s}{q} + (r-x)\frac{p}{q} \in \left[ 0,\alpha - \frac{\delta}{q} \right] \right\},
    \end{equation*}
    which has $\lceil q \alpha \rceil$ elements.
    Let
    \begin{multline*}
        E_s = \left\{x \in [H] : x(\theta + \epsilon) \in \left[ c + \frac{s}{q}, c + \frac{s}{q} + \frac{\delta}{q} - \frac{1}{K} \right) \right\} \\
        = \Bohr \left( \theta+\epsilon, \left[ c + \frac{s}{q}, c + \frac{s}{q} + \frac{\delta}{q} - \frac{1}{K} \right) \right) \cap [H].
    \end{multline*}
    By Lemma \ref{lem: large discrepancy for rational Bohr set},
    \begin{equation*}
        |E_s| \geq \left( \frac{\delta}{q} - \Oh \left( q\epsilon + \frac{1}{K} \right) \right) |H|,
    \end{equation*}
    and we may take $E = \bigsqcup_{s=0}^{q-1} E_s$.

    \bigskip

    (2) Suppose now that $\alpha > \frac{1}{q}$ and $q\epsilon + \frac{1}{K} < \frac{\delta}{q}$.
    Then
    \begin{equation*}
        \alpha - \frac{1}{q} = \frac{q\alpha - 1}{q} \geq \frac{\{q\alpha\}}{q} \geq \frac{\delta}{q} > q\epsilon + \frac{1}{K}.
    \end{equation*}
    Let $J_0 = [0, \frac{1}{q} + q\epsilon)$ and $J = c + J_0 \subseteq I$.
    Then $J_0 - s(\theta+\epsilon) \supseteq \left[ -\frac{sp}{q}, -\frac{sp}{q} + \frac{1}{q} \right)$ for $s \in \{0, \ldots, q-1\}$, so $\bigcup_{s=0}^{q-1} (J - s(\theta+\epsilon)) = \T$.
    Hence, for every $x \in \N$, there exists $s_x \in \{0,1,\ldots,q-1\}$ such that $(x+s_x)(\theta+\epsilon) \in J$.
    Then for $m \in [M/q]$,
    \begin{equation*}
        (x+s_x+qm)(\theta+\epsilon) \in J + qm(\theta+\epsilon) = J + qm\epsilon \subseteq \left[ c, c + \frac{1}{q} + q\epsilon + \frac{1}{K} \right) \subseteq [c, c+\alpha) = I.
    \end{equation*}
    Thus,
    \begin{equation*}
        (q\Z + r_x) \cap \{x+1,\ldots,x+M\} \subseteq \Bohr(\theta+\epsilon,I) \cap \{x+1,\ldots,x+M\}
    \end{equation*}
    for $r_x = x + s_x \bmod q$.
\end{proof}

In the final case, when both $\theta$ and the interval length $|I|$ are approximately rational numbers with denominator $q$, the discrepancy is again small as soon as $M$ is large compared with $q$.

\begin{Lemma} \label{lem: rational frequency and interval length}
    Let $\theta = \frac{p}{q}$ with $\gcd(p,q) = 1$.
    If $\alpha \in (0,1)$ and $q\alpha \in \Z$, then for every $\epsilon \in \left( 0, \frac{1}{q^2} \right)$, every interval $I \subseteq \T$ such that $d(\Bohr(\theta+\epsilon,I)) = \alpha$, and every $M \in \N$,
    \begin{equation*}
        \sup_{x \in \N} \left| \frac{|\Bohr(\theta+\epsilon,I) \cap \{x+1,\ldots,x+M\}|}{M} - \alpha \right| \ll q \sqrt{\epsilon} + \frac{q}{M}.
    \end{equation*}
\end{Lemma}

\begin{proof}
    Assume $|I| = \alpha$.
    Since
    \begin{equation*}
        |\Bohr(\theta+\epsilon,I) \cap \{x+1,\ldots,x+M\}| = \sum_{m=1}^M \1_I((x+m)(\theta+\epsilon)),
    \end{equation*}
    we have the bound (taking $t = x(\theta+\epsilon)$)
    \begin{equation*}
        \sup_{x \in \N} \left| \frac{|\Bohr(\theta+\epsilon,I) \cap \{x+1,\ldots,x+M\}|}{M} - \alpha \right| \leq \sup_{t \in \T} \left| \frac{1}{M} \sum_{m=1}^M \1_I(t + m(\theta+\epsilon)) - \alpha \right|.
    \end{equation*}
    The right hand side is now independent of the choice of interval of length $\alpha$, so it suffices to estimate
    \begin{equation} \label{eq: discrepancy over [M]}
        \sup_{t \in \T} \left| \frac{1}{M} \sum_{m=1}^M \1_{[0,\alpha)}(t + m(\theta+\epsilon)) - \alpha \right|.
    \end{equation}

    As a preliminary step, let us estimate
    \begin{equation*}
        \sup_{t \in \T} \left| \frac{1}{L} \sum_{l=1}^L \1_{[0,\alpha)}(t + l(\theta+\epsilon)) - \alpha \right|
    \end{equation*}
    for $L \leq (q\epsilon)^{-1}$.
    Fix $t \in \T$.
    We can split $l$ into residue classes mod $q$ to obtain
    \begin{equation*}
        \sum_{l=1}^L \1_{[0,\alpha)}(t + l(\theta+\epsilon)) = \sum_{r=0}^{q-1} \sum_{n=1}^{L/q} \1_{[0,\alpha)}(t + r(\theta + \epsilon) + nq\epsilon) + \Oh(q),
    \end{equation*}
    so
    \begin{multline*}
        \left| \frac{1}{L} \sum_{l=1}^L \1_{[0,\alpha)}(t + l(\theta+\epsilon)) - \alpha \right| \\
        \leq \left| \frac{1}{q} \sum_{r=0}^{q-1} \frac{1}{L/q} \sum_{n=1}^{L/q} \1_{[0,\alpha)}(t + r\theta + (nq+r)\epsilon) - \alpha \right| + \Oh \left( \frac{q}{L} \right) \\
        \leq \left| \frac{1}{q} \sum_{s=0}^{q-1} \frac{1}{L/q} \sum_{n=1}^{L/q} \1_{[0,\alpha)} \left( t + \frac{s}{q} + (nq+r(s))\epsilon \right) - \alpha \right| + \Oh \left( \frac{q}{L} \right),
    \end{multline*}
    where in the last step we have reindexed by taking $s = rp \bmod{q}$ and $r(s) = p^{-1}s \bmod{q}$.
    Then
    \begin{equation*}
        \frac{1}{q} \sum_{s=0}^{q-1} \frac{1}{L/q} \sum_{n=1}^{L/q} \1_{[0,\alpha)} \left( t + \frac{s}{q} + (nq+r(s))\epsilon \right)
        = \frac{1}{L/q} \sum_{n=1}^{L/q} F(t+nq\epsilon),
    \end{equation*}
    where
    \begin{equation*}
        F(x) = \frac{1}{q} \sum_{s=0}^{q-1} \1_{[0,\alpha)} \left( x + \frac{s}{q} + r(s)\epsilon \right).
    \end{equation*}
    Note that $F(x) = \alpha$ for $\frac{a}{q} \leq x < \frac{a+1}{q} - (q-1)\epsilon$, $a \in \{0, 1, \ldots, q-1\}$, and $\alpha - \frac{1}{q} \leq F(x) \leq \alpha + \frac{1}{q}$ for all $x \in \T$.
    Therefore,
    \begin{equation*}
        \left| \frac{1}{L/q} \sum_{n=1}^{L/q} F(t+nq\epsilon) - \alpha \right|
         \leq \frac{1}{q} \sum_{a=0}^{q-1} \frac{1}{L/q} \sum_{n=1}^{L/q} \1_{\left[ \frac{a}{q} - (q-1)\epsilon, \frac{a}{q} \right)}(t+nq\epsilon).
    \end{equation*}
    The sequence $t + nq\epsilon$, $n \in [L/q]$, is contained in the interval $(t,t+L\epsilon]$ of length $L\epsilon \leq \frac{1}{q}$, so it can meet at most two of the intervals $\left[ \frac{a}{q} - (q-1)\epsilon, \frac{a}{q} \right)$.
    Moreover, since the sequence is $q\epsilon$-separated, at most one term can belong to a given interval of length $(q-1)\epsilon$, so
    \begin{equation*}
        \frac{1}{q} \sum_{a=0}^{q-1} \frac{1}{L/q} \sum_{n=1}^{L/q} \1_{\left[ \frac{a}{q} - (q-1)\epsilon, \frac{a}{q} \right)}(t+nq\epsilon) \leq \frac{2}{q} \cdot \frac{1}{L/q} = \frac{2}{L}.
    \end{equation*}
    Thus,
    \begin{equation} \label{eq: discrepancy over short intervals}
        \sup_{t \in \T} \left| \frac{1}{L} \sum_{l=1}^L \1_{[0,\alpha)}(t + l(\theta+\epsilon)) - \alpha \right| \ll \frac{1}{L} + \frac{q}{L} \ll \frac{q}{L}.
    \end{equation}

    Now we return to estimating \eqref{eq: discrepancy over [M]}.
    If $M \leq (q\epsilon)^{-1}$, then we take $L = M$ in \eqref{eq: discrepancy over short intervals}.
    Suppose $M > (q\epsilon)^{-1}$.
    Take $L = \lfloor \epsilon^{-1/2} \rfloor$.
    Since $q^2\epsilon < 1$ by assumption, we have
    \begin{equation*}
        L \leq \frac{1}{\sqrt{\epsilon}} < \frac{1}{\sqrt{\epsilon}} \cdot \frac{1}{q\sqrt{\epsilon}} = \frac{1}{q\epsilon}.
    \end{equation*}
    Let $t \in \T$.
    By decomposing the interval $[M]$ into intervals of length $L$ and applying \eqref{eq: discrepancy over short intervals}, we have
    \begin{multline*}
        \sum_{m=1}^M \1_{[0,\alpha)}(t + m(\theta+\epsilon)) = \sum_{k=1}^{M/L} \sum_{l=1}^L \1_{[0,\alpha)}(t + kL + l(\theta+\epsilon)) + \Oh \left( L \right) \\
         = \frac{M}{L} \left( L\alpha + \Oh \left( q \right) \right) + \Oh(L)
         = M\alpha + \Oh \left( \frac{Mq}{L}\right) + \Oh(L)
    \end{multline*}
    Therefore, dividing by $M$, we have
    \begin{equation*}
        \sup_{t \in \T} \left| \frac{1}{M} \sum_{m=0}^{M-1} \1_{[0,\alpha)}(t + m(\theta+\epsilon)) - \alpha \right| \ll \frac{q}{L} + \frac{L}{M} \ll q\sqrt{\epsilon}.
    \end{equation*}
\end{proof}

%SUBSECTION
\subsection{Multidimensional discrepancy estimates and a result on simultaneous approximation}

\begin{Lemma}
\label{lem_almost_denseness_condition}
For every $\epsilon>0$ and $d\in\N$ there exists $C_0=C_0(d,\epsilon)>0$ such that for all $H\in \N$ with $H\geq C_0$ the following holds: If $(\alpha_1, \ldots, \alpha_d) \in \T^d$ is such that for all $(w_1, \ldots, w_d) \in [-C_0,C_0]^d \cap \Z^d$ with $(w_1,\ldots,w_d)\neq (0,\ldots,0)$ we have
\begin{align}
\label{eqn_almost_denseness_condition}
    \left\| \sum_{i=1}^d w_i \alpha_i \right\|_{\T}  \geq \frac{C_0}{H},
\end{align}
then the sequence $\{(k\alpha_1,\ldots,k\alpha_d): k=0,1,\ldots,\lfloor \epsilon H\rfloor\}$ is $\epsilon$-dense in $\T^d$.
\end{Lemma}

\begin{proof}
Define $D=D(d,\epsilon)=2 d^{\frac{d}{2}+2} 3^{d+2}\epsilon^{-d}$ and set 
\[
C_0=C_0(d,\epsilon)=\frac{d^{\frac{d}{2}+2} 3^{d+2} (2D+1)^d}{\epsilon^{d+1}}.
\]
Let $B=[a_1,b_1)\times\ldots\times[a_d,b_d)\subset \T^d$ be a box whose sides all have length $\frac{\epsilon}{\sqrt{d}}$, that is, $b_i-a_i=\frac{\epsilon}{\sqrt{d}}$ for all $i=1,\ldots,d$.
Let $\mathcal{N}_{B}$ denote the number of points from the set $\{(k\alpha_1,\ldots,k\alpha_d): k=0,1,\ldots,\lfloor \epsilon H\rfloor\}$ that are contained in the box $B$.
If we can show that $\mathcal{N}_{B}>0$ for any such box $B$, then this proves that the sequence $\{(k\alpha_1,\ldots,k\alpha_d): k=0,1,\ldots,\lfloor \epsilon H\rfloor\}$ is $\epsilon$-dense in $\T^d$.

According to the \Erdos{}--\Turan{}--Koksma inequality (\cref{thm: Erdos-Turan-Koksma}),
we have
\begin{align*}
\bigg|\frac{\mathcal{N}_{B}}{\lfloor \epsilon H\rfloor+1}&-
\frac{\epsilon^d}{d^{d/2}}\bigg|
\\
& \leq 6d^23^d \left(\frac{1}{D} + \sum_{\substack{(w_1,\ldots,w_d)\in[-D,D]^d\cap\Z^d\\ (w_1,\ldots,w_d)\neq (0,\ldots,0)}} %\frac{1}{\prod_{i=1}^d \max\{1,|w_i|\}}
\Bigg|\frac{1}{\lfloor \epsilon H\rfloor+1}\sum_{k=0}^{\lfloor \epsilon H\rfloor}e(k(w_1\alpha_1+\ldots+w_d\alpha_d))\Bigg|\right),
\end{align*}
where we have used the bound $\frac{1}{\prod_{i=1}^d \max\{1,|w_i|\}} \leq 1$.
In view of \eqref{eqn_almost_denseness_condition}, we have for all $(w_1,\ldots,w_d)\in[-D,D]^d\cap\Z^d$ with $(w_1,\ldots,w_d)\neq (0,\ldots,0)$ that
\begin{align*}
\Bigg|\sum_{k=0}^{\lfloor \epsilon H\rfloor}e(k(w_1\alpha_1+\ldots+w_d\alpha_d))\Bigg|
%&\leq \frac{2}{|1-e(w_1\alpha_1+\ldots+w_d\alpha_d)|}
%\\
\leq \frac{1}{2\|w_1\alpha_1+\ldots+w_d\alpha_d\|_{\T}}
%\\
\leq \frac{H}{2C_0}.
\end{align*}
Together, this implies that
\begin{align*}
\bigg|\frac{\mathcal{N}_{B}}{\lfloor \epsilon H\rfloor+1}-\frac{\epsilon^{d}}{d^{d/2}}\bigg|
&\leq 6d^23^d \bigg(\frac{1}{D}+
\frac{(2D+1)^d}{2C_0\epsilon}\bigg)
\\
&\leq \frac{6d^23^d}{D} + \frac{d^2 3^{d+1} (2D+1)^d }{\epsilon C_0}
\\
&= \frac{\epsilon^{d}}{3d^{d/2}} + \frac{\epsilon^{d}}{3d^{d/2}}< \frac{\epsilon^{d}}{d^{d/2}},
\end{align*}
where the last line follows from the definitions of $D$ and $C_0$. This implies $\mathcal{N}_{B}>0$, completing the proof.
\end{proof}

\begin{Lemma}
\label{lem_finding_quantitative_almost_periods}
For every $d\in\N$ and $\epsilon>0$ there exists some $C=C(d,\epsilon)>0$ such that for all $H \in \N$ with $H\geq C$ the following holds:
If $(\alpha_1, \ldots, \alpha_d) \in \T^d$ and for all $(w_1, \ldots, w_d) \in [-C,C]^d \cap \Z^d$ either
\begin{align*}
    \left\| \sum_{i=1}^d w_i \alpha_i \right\|_{\T} & \leq \frac{1}{CH}
\intertext{or}
    \left\| \sum_{i=1}^d w_i \alpha_i \right\|_{\T} & \geq \frac{C}{H},
\end{align*}
then for every $x\in [\epsilon,1-\epsilon]$ there exists $m \in \N$ such that $(x - \epsilon) H \leq m \leq (x+\epsilon)H$ and $\|m\alpha_i\|_{\T} \leq \epsilon$ for every $i \in \{1, \ldots, d\}$.
\end{Lemma}

\begin{proof}
Fix $\epsilon>0$. We proceed by induction on $d$, starting with the base case $d=1$. We define $C=C(1,\epsilon)=\epsilon^{-2}$.
We distinguish two cases.
First, assume there exists some $w\in [-\epsilon C,\epsilon C]\cap \Z$ with $w\neq 0$ such that
\[
\|w\alpha\|_{\T}\leq \frac{1}{CH}.
\]
By replacing $w$ with $-w$, we can assume without loss of generality that $w$ is positive.
Then we take $m= \lfloor \frac{xH}{w}\rfloor w$ and note that $\|m\alpha\|_{\T}\leq \frac{1}{C}\leq \epsilon$ and $|m- xH|\leq w\leq \epsilon C\leq \epsilon H$.

If we are not in the first case, then 
for all $w\in [-\epsilon C,\epsilon C]\cap \Z$ with $w\neq 0$ we have
\[
\|w\alpha\|_{\T}> \frac{1}{CH}.
\]
According to our assumptions, this implies that 
for all such $w$ we actually have 
\[
\|w\alpha\|_{\T}\geq \frac{C}{H}.
\]
Using Dirichlet's approximation theorem, we can find 
some $w\in [-\epsilon C,\epsilon C]\cap \Z$ with $w\neq 0$ and
\[
\|w\alpha\|_{\T}\leq \frac{1}{\epsilon C}\leq \epsilon.
\]
Together, this gives
\[
\frac{C}{H}\leq \|w \alpha\|_{\T}\leq \epsilon.
\]
We conclude that the sequence $0,w\alpha, 2w\alpha,\ldots, \lfloor \frac{H}{C}\rfloor w\alpha$ is $\epsilon$-dense in $\T$.
Translating the sequence by $\lfloor xH \rfloor$, we get that
$\lfloor xH\rfloor,(\lfloor xH\rfloor+w)\alpha, (\lfloor xH\rfloor+2w)\alpha,\ldots, (\lfloor xH\rfloor+\lfloor \epsilon H\rfloor w)\alpha$ is $\epsilon$-dense in $\T$. It follows that there exists $m\in\{\lfloor xH\rfloor,\lfloor xH\rfloor+w, \ldots, \lfloor xH\rfloor+\lfloor \frac{H}{C}\rfloor w\}$ such that
\[
\|m\alpha\|_{\T}\leq\epsilon.
\]
Noting that $(x-\epsilon)H\leq \lfloor xH\rfloor$ and  $\lfloor xH\rfloor+\lfloor \frac{H}{C}\rfloor w\leq (x+\epsilon)H$, we get $(x-\epsilon)H\leq m\leq (x+\epsilon)H$ and we are done.

Next, assume $d\geq 2$ and the claim has already been proved for $d-1$.
Let $C_0=C_0(d,\epsilon)$ be as guaranteed by \cref{lem_almost_denseness_condition}, let $\epsilon'=\frac{\epsilon}{2dC_0}$, let $C'=C(d-1,\epsilon')$ be as guaranteed by the induction hypothesis, and define $C=C(d,\epsilon)=\max\{C_0,C', 2\epsilon^{-1}\}$. 
Again, we distinguish two cases. First, assume there exists some $(w_1,\ldots,w_d)\in [-C_0,C_0]^d\cap \Z^d$ with $(w_1,\ldots,w_d)\neq (0,\ldots,0)$ and 
\begin{equation}
\label{eqn_c1_ann}
\|w_1\alpha_1+\cdots+w_d\alpha_d\|_{\T}\leq \frac{1}{CH}.
\end{equation}
By reordering and replacing $w_1,\ldots,w_d$ with $-w_1,\ldots,-w_d$ if necessary, we can assume without loss of generality that $w_d$ is positive.

Take $y=\frac{x}{w_d}$.
By the induction hypothesis applied to $(\alpha_1,\ldots,\alpha_{d-1})\in\T^{d-1}$, 
there exists $m' \in \N$ such that $  (y - \epsilon') H \leq m' \leq (y + \epsilon') H$ and $\|m'\alpha_i\|_{\T} \leq \epsilon'$ for every $i \in \{1, \ldots, d-1\}$.
If we now take $m=w_dm'$ and use $w_d\epsilon'\leq \epsilon$, then we have
\[
(x - \epsilon) H \leq m \leq (x + \epsilon) H
\]
and $\|m\alpha_i\|_{\T} \leq  \epsilon$ for every $i \in \{1, \ldots, d-1\}$. Finally, using \eqref{eqn_c1_ann}, we get
\begin{align*}
\|m\alpha_d\|_{\T}
&=\|m' w_d\alpha_d\|_{\T}
\\
&\leq \|m' (w_1\alpha_1+\ldots+w_{d-1}\alpha_{d-1}+)\|_{\T}+ \frac{1}{CH}
\\
&\leq \underbrace{dC_0 \epsilon'}_{\leq \frac{\epsilon}{2}} + \underbrace{\frac{1}{CH}}_{\leq \frac{\epsilon}{2}}\leq \epsilon
\end{align*}
as desired.

If we are in the second case then
for all $(w_1,\ldots,w_d)\in [-C_0,C_0]^d\cap \Z^d$ with $(w_1,\ldots,w_d)\neq (0,\ldots,0)$ we have
\[
\|w_1\alpha_1+\cdots+w_d\alpha_d\|_{\T}> \frac{1}{CH},
\]
which, in light of our assumption, implies
\[
\|w_1\alpha_1+\cdots+w_d\alpha_d\|_{\T}\geq \frac{C}{H}\geq \frac{C_0}{H}.
\]
From \cref{lem_almost_denseness_condition}, it now follows that the sequence $\{(k\alpha_1,\ldots,k\alpha_d): k=0,1,\ldots,\lfloor \epsilon H\rfloor\}$ is $\epsilon$-dense in $\T^d$.
To finish the proof, we can repeat the same argument already used in the proof of the base case $d=1$ to find $m\in\N$ with $(x-\epsilon)H\leq m\leq (x+\epsilon)H$ and 
$
\|m\alpha_i\|_{\T}\leq\epsilon
$ for all $i\in\{1,\ldots,d\}$.
\end{proof}

%SECTION
\section{Gowers norms and arithmetic regularity} \label{sec: Gowers norms}

The \emph{Gowers uniformity norms}, introduced by Gowers in his work on Szemer\'{e}di's theorem \cite{Gowers01}, are an indispensable tool in additive combinatorics and lie at the foundation of \emph{higher order Fourier analysis}.

We first define the norms in the context of finite abelian groups and then extend to $\N$.
Let $G$ be a finite abelian group.
For a function $f : G \to \C$, we define the \emph{Gowers uniformity norms} by\nomenclature{$\lVert\cdot\rVert_{U^k(G)}$}{Gowers norm on a group $G$}
\begin{align*}
    \|f\|_{U^1(G)} & = \left( \frac{1}{|G|^2} \sum_{x, h \in G} f(x) \overline{f(x+h)} \right)^{1/2} = \left| \frac{1}{|G|} \sum_{x \in G} f(x) \right|, \\
    \|f\|_{U^2(G)} & = \left( \frac{1}{|G|^3} \sum_{x,h_1,h_2 \in G} f(x) \overline{f(x+h_1)} \overline{f(x+h_2)} f(x+h_1+h_2) \right)^{1/4}, \\
     & \vdots \\
    \|f\|_{U^k(G)} & = \left( \frac{1}{|G|^{k+1}} \sum_{x \in G, \bm{h} \in G^k} \prod_{\bm{\omega} \in \{0,1\}^k} C^{|\bm{\omega}|} f(x+\bm{\omega}\cdot\bm{h}) \right)^{1/2^k},
\end{align*}
where $C$ is the complex conjugation map and $|\bm{\omega}| = |\{i\in [k] : \omega_i=1\}|$.
The $U^1$ Gowers norm is in fact only a seminorm, but the higher order norms are genuine norms.
Some useful properties of the Gowers norms are their recursive relationship and monotonicity:
\begin{equation*}
    \|f\|_{U^{k+1}(G)}^{2^{k+1}} = \frac{1}{|G|} \sum_{h \in G} \|\Delta_hf\|_{U^k(G)}^{2^k},
\end{equation*}
where $\Delta_hf(x) = f(x) \overline{f(x+h)}$, and
\begin{equation*}
    \|f\|_{U^1(G)} \leq \|f\|_{U^2(G)} \leq \ldots.
\end{equation*}
The $U^2$ Gowers norm is closely related to the Fourier transform.
Indeed,
\begin{equation*}
    \|f\|_{U^2(G)} = \left\| \hat{f} \right\|_{\ell^4}.
\end{equation*}

One can lift the definition of the Gowers norms from cyclic groups to finite discrete intervals in the positive integers.
Given a function $f\colon \N\to\C$, the $U^k$ Gowers seminorms over an interval $\{x+1,\ldots,x+N\} \subseteq \N$ are defined as\nomenclature{$\lVert\cdot\rVert_{U^k(\{x+1,\ldots,x+N\})}$}{Gowers norm on a discrete interval}
\begin{equation*}
\|f\|_{U^k(\{x+1,\ldots,x+N\})} = \frac{\big\|\tilde{f}\big\|_{U^k(\Z/\tilde{N}\Z)}}{\|1_{[N]}\|_{U^k(\Z/\tilde{N}\Z)}},
\end{equation*}
where $\tilde{f}\colon \Z/\tilde{N}\Z\to\C$ is the function given by $\tilde{f}(n)= \1_{[N]}(n) f(x+n)$ and $\tilde{N} \geq 2^kN$.

In this paper, we utilize only the $U^1$ and $U^2$ Gowers seminorms, so the remaining discussion focuses on these cases.
For a more comprehensive account, see \cite{Tao12}.

The $U^1$ Gowers seminorm has the simple expression
\begin{equation*}
    \|f\|_{U^1(\{x+1,\ldots,x+N\})} = \left| \frac{1}{N} \sum_{n=1}^{N} f(x+n) \right|,
\end{equation*}
so for example, $\lim_{N \to \infty} \|\1_A\|_{U^1([N])}$, if it exists, is equal to the density of the set $A$ for $A \subseteq \N$.

Closely related to the $U^2$ Gowers seminorm is the $u^2$ Fourier seminorm.
Given a function $f\colon \N\to\C$ and a discrete interval $\{x+1,\ldots,x+N\}\subset\N$, consider\nomenclature{$\lVert\cdot\rVert_{u^2(\{x+1,\ldots,x+N\})}$}{Fourier norm on a discrete interval}
\begin{equation}
\label{eqn_def_fourier_U2_norm_interval}
\|f\|_{u^2(\{x+1,\ldots,x+N\})}= \sup_{\alpha\in\R} \Bigg|\frac{1}{N} \sum_{n=1}^N f(x+n)e(\alpha n)\Bigg|.
\end{equation}
One can show (cf. \cite[Eq.~(11.9),~p.~422]{TV06}) that 
\begin{equation}
\label{eqn_U2-Gowers-Fourier_relation}    
C^{-1}\|f\|_{u^2(\{x+1,\ldots,x+N\})}
\leq \|f\|_{U^2(\{x+1,\ldots,x+N\})} \leq C \|f\|_{u^2(\{x+1,\ldots,x+N\})}^{\frac{1}{2}}
\end{equation}
for some universal constant $C\geq 1$, showing that the $U^2$ Gowers seminorms and the $u^2$ Fourier seminorms are closely intertwined.

Given $N\in\N$ and functions $f,g\colon[N]\to\C$, let us also define\nomenclature{$\innprod{\cdot}{\cdot}_{[N]}$}{inner product over $[N]$} \nomenclature{$\lVert\cdot\rVert_{2,[N]}$}{$L^2$ norm over $[N]$}
\begin{align*}
\innprod{f}{g}_{[N]}&=\frac{1}{N} \sum_{n=1}^{N} f(n)\overline{g(n)}
\intertext{and}
\| f \|_{2, [N]}&=\sqrt{\innprod{f}{f}_{[N]}}= \Bigg(\frac{1}{N} \sum_{n=1}^{N} |f(n)|^2\Bigg)^{\frac{1}{2}}.
\end{align*}

The following variant of the arithmetic regularity lemma follows from {\cite[Theorem~1.2.11]{Tao12}}.

\begin{Theorem}[Arithmetic regularity lemma -- a variant]
\label{thm_arithmetic_regularity_lemma}
For every $\epsilon>0$ there exists $d^\star=d^\star(\epsilon)\in\N$ such that for any $N\in\N$ and any
$
f \colon [N] \to [0,1],
$
we can decompose $f = f_{\mathrm{str}} + f_{\mathrm{psd}}$ such that:
\begin{enumerate}[label=(\roman{enumi}),ref=(\roman{enumi}),leftmargin=*]
    \item (Nonnegativity)
    $f_{\mathrm{str}}$ takes values in $[0,1]$, and $\frac{1}{N}\sum_{n=1}^N f_{\mathrm{psd}}(n)=0$;
    \item (Structure) There exist $c_1,\ldots,c_{d^\star}\in\C$ with $|c_i|\leq 1$, and $\alpha_1,\ldots,\alpha_{d^\star}\in\R$ such that if $\mathcal{P}^\star(n)=\sum_{i=1}^{d^\star} c_i e(n\alpha_i)$ then
    \[
    \| f_{\mathrm{str}}-\mathcal{P}^\star \|_{2, [N]} \leq \epsilon;
    \]
    \item (Pseudorandomness) $\|f_{\mathrm{psd}}\|_{u^2([N])}\leq \epsilon$;
    \item (Orthogonality) $\innprod{f_{\mathrm{str}}}{f_{\mathrm{psd}}}_{[N]}=0$.
\end{enumerate}    
\end{Theorem}

\begin{Remark}
The formulation of the arithmetic regularity lemma given in {\cite[Theorem~1.2.11]{Tao12}}, from which \cref{thm_arithmetic_regularity_lemma} is derived, does not mention that $|c_i|\leq 1$ or that $\innprod{f_{\mathrm{str}}}{f_{\mathrm{psd}}}_{[N]}=0$. Nonetheless, these extra properties follow from the proof given there, and we state them explicitly since they will be needed later.
\end{Remark}

%SECTION
\section{Ergodic theory}
\label{sec_ergodic_theory}

We recall some basic notions from ergodic theory that will be used in the proof of \cref{thm: main} and introduce some new tools for analyzing the ergodic decomposition of Furstenberg systems.

%SUBSECTION
\subsection{Measure-preserving systems, factors, and the ergodic decomposition}

A \emph{measure-preserving system} is a triple $(X, \mu, T)$, where $(X, \mu)$ is a probability space and $T : X \to X$ is a measurable map preserving the measure $\mu$, i.e. $\mu(T^{-1}E) = \mu(E)$ for every measurable $E \subseteq X$.
The map $T$ is called a \emph{measure-preserving transformation}.
For our purposes, we will always assume that $X$ is a compact metric space, $T$ is continuous, and $\mu$ is defined on the Borel $\sigma$-algebra.

Given a measure-preserving system $(X, \mu, T)$, the map $f \mapsto f \circ T$ on $L^2(\mu)$ is an isometry, called the \emph{Koopman operator}.
In a standard abuse of notation, we also denote the Koopman operator by $T$.
Thus, $Tf = f \circ T$.

A measure-preserving system $(Y, \nu, S)$ is a \emph{factor} of $(X, \mu, T)$ if there is a measurable map $\pi : X \to Y$ such that $\pi \circ T = S \circ \pi$ $\mu$-almost everywhere and $\pi_*\mu = \nu$.
We write $\mathcal{Y}$ for the $\sigma$-algebra generated by functions $f \circ \pi$ for measurable $f : Y \to \C$.
In this way, factors correspond to $T$-invariant sub-$\sigma$-algebras, and we will also refer to such $\sigma$-algebras $\mathcal{Y}$ as factors.
Pre-composing with the factor map $\pi$, $f \mapsto f \circ \pi$, provides an isomorphism between $L^2(\nu)$ and $L^2(\mu|_{\mathcal{Y}}) \subseteq L^2(\mu)$.
The \emph{conditional expectation} of a function $f \in L^2(\mu)$ refers to two closely related objects:\nomenclature{$\E[\cdot\mid\cdot]$}{conditional expectation}
\begin{itemize}
    \item $\E[f \mid \mathcal{Y}]$ is the $\mathcal{Y}$-measurable function on $X$ obtained as the orthogonal projection of $f$ onto $L^2(\mu|_{\mathcal{Y}})$, and
    \item $\E[f \mid Y]$ is the measurable function on $Y$ satisfying $\E[f \mid Y] \circ \pi = \E[f \mid \mathcal{Y}]$.
\end{itemize}

For a measure-preserving system $(X, \mu, T)$, we denote by $\mathcal{I}$ the $\sigma$-algebra of measurable sets $E \subseteq X$ such that $\mu(E \triangle T^{-1}E) = 0$.
A measure-preserving system is \emph{ergodic} if $\mathcal{I}$ is the trivial $\sigma$-algebra $\mathcal{I} = \{E : \mu(E) \in \{0,1\}\}$.
A general (potentially non-ergodic) measure-preserving system $(X, \mu, T)$ admits a decomposition into ergodic systems.

\begin{Definition} \label{def: ergodic decomp}
    Let $(X, \mu, T)$ be a measure-preserving system.
    A family of Borel probability measures $(\mu_x)_{x \in X}$ is an \emph{ergodic decomposition} of $(X, \mu, T)$ if
    \begin{itemize}
        \item $\int_X \mu_x~d\mu(x) = \mu$, and
        \item for every $f \in L^2(\mu)$,
            \begin{equation*}
                \E[f \mid \mathcal{I}](x) = \int_X f~d\mu_x
            \end{equation*}
            for $\mu$-almost every $x \in X$.
    \end{itemize}
\end{Definition}

\begin{Theorem}[Ergodic decomposition theorem]
    Let $(X, \mu, T)$ be a measure-preserving system, where $X$ is a compact metric space, $T$ is continuous, and $\mu$ is a Borel probability measure.
    Then $(X, \mu, T)$ admits an ergodic decomposition.
    Moreover, if $(\mu_x)_{x \in X}$ and $(\mu'_x)_{x \in X}$ are two ergodic decompositions, then $\mu_x = \mu'_x$ for $\mu$-almost every $x \in X$.
\end{Theorem}

The mean ergodic theorem of von Neumann provides a relationship between the factor $\mathcal{I}$ of invariant sets and ergodic averages.

\begin{Theorem}[Mean ergodic theorem]
    Let $(X, \mu, T)$ be a measure-preserving system.
    For every $f \in L^2(\mu)$,
    \begin{equation*}
        \lim_{N - M \to \infty} \frac{1}{N-M} \sum_{n=M+1}^N T^nf = \E[f \mid \mathcal{I}]
    \end{equation*}
    in $L^2(\mu)$.
\end{Theorem}

%SUBSECTION
\subsection{Topological dynamical systems and invariant measures}

Dropping measure-theoretic considerations for a moment, we call $(X, T)$ a \emph{topological dynamical system} if $X$ is a compact metric space and $T : X \to X$ is continuous.
We will typically assume that $T$ is invertible and hence a homeomorphism.

A point $x \in X$ is \emph{transitive} if its orbit $\{T^nx : n \in \Z\}$ is dense in $X$.
Given a $T$-invariant Borel probability measure $\mu$ on $X$, the triple $(X,\mu,T)$ becomes a measure-preserving system.
One way of generating invariant measures is by taking limits of averages along orbits.
In order for an averaging scheme to produce an invariant measure in the end, the method of averaging should possess an asymptotic invariance property.
A sequence $(\Phi_N)_{N \in \N}$ of finite subsets of $\Z$ is a \emph{F{\o}lner sequence} if for every $t \in \Z$,
\begin{equation*}
    \lim_{N \to \infty} \frac{|(\Phi_N +t) \cap \Phi_N|}{|\Phi_N|} = 1.
\end{equation*}
If $(\Phi_N)_{N \in \N}$ is a F{\o}lner sequence and $x \in X$, then the weak* limit points of the sequence $\frac{1}{|\Phi_N|} \sum_{n \in \Phi_N} \delta_{T^nx}$ (which exist by the Banach--Alaoglu theorem) are $T$-invariant measures.
Given a F{\o}lner sequence $\Phi = (\Phi_N)_{N \in \N}$, we say that a point $x \in X$ is \emph{generic} for $\mu$ along $\Phi$, written $x \in \textup{gen}(\mu,\Phi)$, if $\frac{1}{|\Phi_N|} \sum_{n \in \Phi_N} \delta_{T^nx}$ converges to $\mu$ in the weak* topology.
That is, for every continuous function $f \in C(X)$,
\begin{equation*}
    \lim_{N \to \infty} \frac{1}{|\Phi_N|} \sum_{n \in \Phi_N} f(T^nx) = \int_X f~d\mu.
\end{equation*}
The pointwise ergodic theorem implies that if $\mu$ is an ergodic measure, then $\mu$-almost every $x \in X$ is generic for $\mu$ along the F{\o}lner sequence $\Phi = ([N])_{N \in \N}$.
Moreover, if $\mu$ is ergodic and $x \in X$ is transitive, then there exists a F{\o}lner sequence $\Phi$ such that $x \in \textup{gen}(\mu,\Phi)$ (see \cite[Proposition 3.9]{Furstenberg81a}).

Within the class of topological dynamical systems, we are particularly interested in families of systems with additional structure.
A system $(X,T)$ is called \emph{distal} if for every pair of points $x,y \in X$ with $x\ne y$, one has $\inf_{n \in \Z} d_X(T^nx, T^ny) > 0$, where $d_X$ is the metric on $X$.
The main property of distal systems that will be convenient for us is the following lemma.

\begin{Lemma} \label{lem: extension}
	Let $(Y, \nu, S)$ be an ergodic system and $y_0 \in Y$ a transitive point.
	Assume that $(Z, \lambda, R)$ is a measurable factor of $(Y, \nu, S)$ that is topologically distal.
	Then there exists an ergodic system $(X, \mu, T)$, a transitive point $x_0 \in X$, and a continuous factor map $\rho : X \to Y$ such that $\rho(x_0) = y_0$, $\rho$ is a measurable isomorphism, and there is a continuous factor map $\pi : X \to Z$.
\end{Lemma}

\begin{proof}
	A proof of this statement is given in \cite[Lemma 5.8]{KMRR24a} in the case that $(Z,\lambda,R)$ is a pro-nilfactor.
	The proof is exactly the same for general distal factors, but we include a short argument for completeness.
	
	Denote by $\tau : Y \to Z$ the given measurable factor map.
	By \cite[Proposition 6.1]{HK09}, there is a point $z_0 \in Z$ and a F{\o}lner sequence $\Phi$ such that
	\begin{equation} \label{eq: convergence from (y_0,z_0)}
		\lim_{N \to \infty} \frac{1}{|\Phi_N|} \sum_{n \in \Phi_N} f(S^ny_0) g(R^nz_0) = \int_Y f \cdot (g \circ \tau)~d\nu
	\end{equation}
	for all $f \in C(Y)$ and $g \in C(Z)$.
	Let $X = \overline{\left\{ (S^ny_0, R^nz_0) : n \in \Z \right\}} \subseteq Y \times Z$, and let $T : X \to X$ be the tranformation $T = S \times R$.
	Put $x_0 = (y_0, z_0) \in X$, and note that $x_0$ is transitive by construction.
	We define a $T$-invariant measure $\mu$ on $X$ by
	\begin{equation*}
		\mu = \lim_{N \to \infty} \frac{1}{|\Phi_N|} \sum_{n \in \Phi_N} \delta_{T^nx_0} = \lim_{N \to \infty} \frac{1}{|\Phi_N|} \sum_{n \in \Phi_N} \delta_{(S^ny_0, R^nz_0)},
	\end{equation*}
	where the limit is taken in the weak* topology and exists by \eqref{eq: convergence from (y_0,z_0)}.
	Namely,
	\begin{equation*}
		\int_X f \otimes g~d\mu = \int_Y f \cdot (g \circ \tau)~d\nu
	\end{equation*}
	for $f \in C(Y)$ and $g \in C(Z)$.
	We define the factor map $\rho : X \to Y$ by $\rho(y,z) = y$ and the map $\pi : X \to Z$ by $\pi(y,z) = z$.
	It is easy to check that this are indeed factor maps, and they are continuous, since they are the coordinate projections.
	Finally, to see that $\rho$ is a measurable isomorphism, we note that $y \mapsto (y, \tau(y))$ is an almost sure inverse of $\rho$.
\end{proof}

%SUBSECTION
\subsection{Host--Kra uniformity seminorms}

As a consequence of the mean ergodic theorem, we may define a seminorm $\|\cdot\|_{U^1}$ on $L^{\infty}(\mu)$ by
\begin{equation*}
    \|f\|_{U^1} = \left( \lim_{H \to \infty} \frac{1}{H} \sum_{h=1}^H \int_X f \cdot T^h\overline{f}~d\mu \right)^{1/2},
\end{equation*}
which is in fact equal to $\left\| \E[f \mid \mathcal{I}] \right\|_{L^2}$.
The $U^1$-seminorm is part of a family of seminorms capturing higher-order structures in $(X, \mu, T)$.
For $k \in \N$, we define the \emph{Host--Kra uniformity seminorm} of order $k$ on $L^{\infty}(\mu)$ by\nomenclature{$\lVert\cdot\rVert_{U^k}$}{Host--Kra seminorm}
\begin{equation*}
    \|f\|_{U^k}^{2^k} = \lim_{H \to \infty} \frac{1}{H^k} \sum_{\bm{h} \in [H]^k} \int_X \prod_{\bm{\omega} \in \{0,1\}^k} T^{\bm{\omega} \cdot \bm{h}} C^{|\bm{\omega}|} f~d\mu,
\end{equation*}
where again $C$ is the complex conjugation map and $|\bm{\omega}| = |\{i \in [k] : \omega_i = 1\}|$.
The Host--Kra seminorms, introduced by Host and Kra in \cite{HK05a}, are ergodic theoretic analogues of the Gowers norms and satisfy similar relations:
\begin{equation*}
    \|f\|_{U^{k+1}}^{2^{k+1}} = \lim_{H \to \infty} \frac{1}{H} \sum_{h=1}^H \left\| f \cdot T^h \overline{f} \right\|_{U^k}^{2^k}
\end{equation*}
and
\begin{equation*}
    \|f\|_{U^1} \leq \|f\|_{U^2} \leq \ldots.
\end{equation*}
Host and Kra proved that for each $k \in \N$, there exists a factor $\mathcal{Z}_{k-1}$ such that $\|f\|_{U^k} = 0$ if and only if $\E[f \mid \mathcal{Z}_{k-1}] = 0$.
By the observation above that $\|f\|_{U^1} = \left\| \E[f \mid \mathcal{I}] \right\|_{L^2}$, we see that $\mathcal{Z}_0 = \mathcal{I}$.
Moreover, by monotonicity of the seminorms, the factors are nested: $\mathcal{I} = \mathcal{Z}_0 \subseteq \mathcal{Z}_1 \subseteq \ldots$.

The Host--Kra structure theorem provides a full description of the factors $\mathcal{Z}_k$, $k \geq 0$, in ergodic systems.
We will not use the full structure theorem so focus only on the factors $\mathcal{Z}_0$ (which we have already described) and $\mathcal{Z}_1$, which will be relevant for us later on.
For a full description of the Host--Kra structure theorem, see \cite{HK05a, HK18} for the ergodic case and \cite{JM26arXiv} for an extension to non-ergodic systems.

%SUBSUBSECTION
\subsubsection{The $U^1$ seminorm and the $\mathcal{Z}_0$ factor}

We have more or less fully addressed the factor $\mathcal{Z}_0 = \mathcal{I}$.
However, in order to compare with other results of the paper, we reinterpret the mean ergodic theorem as a decomposition result for the $U^1$-seminorm.

\begin{Theorem} \label{thm: ergodic theorem}
    Let $(X, \mu, T)$ be a measure-preserving system, and let $f : X \to [0,1]$ be measurable.
    Then there is a decomposition $f = f_{\inv} + f_{\erg}$ such that
    \begin{enumerate}[label=(\roman*),leftmargin=*]
        \item (Nonnegativity) $0 \leq f_{\inv} \leq 1$ almost everywhere, and $\int_X f_{\erg}~d\mu = 0$.
        \item (Structure) $f_{\inv}$ is $T$-invariant.
        \item (Uniformity) $\|f_{\erg}\|_{U^1} = 0$.
        \item (Orthogonality) $\innprod{f_{\inv}}{f_{erg}} = 0$.
    \end{enumerate}
\end{Theorem}

\begin{proof}
    Take $f_{\inv} = \E[f\mid\mathcal{I}]$ and $f_{\erg} = f - f_{\inv}$.

    The projection of a $[0,1]$ valued function is again $[0,1]$-valued, and $f_{\inv}$ is $T$-invariant by construction.
    Moreover, since $f_{\inv}$ is obtained by means of an orthgonal projection, we have $\innprod{f_{\inv}}{f_{\erg}} = 0$.

    Now, since $f_{\erg}$ is orthogonal to $\mathcal{I}$, we have
    \begin{equation*}
        \|f_{\erg}\|_{U^1} = \|\E[f_{\erg} \mid \mathcal{I}]\|_{L^2} = 0
    \end{equation*}
    and
    \begin{equation*}
        \int_X f_{\erg}~d\mu = \innprod{f_{\erg}}{\1} = 0,
    \end{equation*}
    since the constant function $\1$ is $T$-invariant.
\end{proof}

%SUBSUBSECTION
\subsubsection{The $U^2$ seminorm and the $\mathcal{Z}_1$ factor}

We now move to the factor $\mathcal{Z}_1$.
If $(X, \mu, T)$ is ergodic, then $\mathcal{Z}_1$ is the \emph{Kronecker factor}, which has several equivalent characterizations:
\begin{itemize}
    \item $L^2(\mu|_{\mathcal{Z}_1})$ is the closed linear span of the eigenfunctions of $T$, i.e.
        \begin{equation*}
            L^2(\mu|_{\mathcal{Z}_1}) = \overline{\textup{span}\left\{ f \in L^2(\mu) : \exists \lambda \in \C, Tf = \lambda f\right\}}
        \end{equation*}
    \item $L^2(\mu|_{\mathcal{Z}_1})$ is the subspace of function with pre-compact orbit
        \begin{equation*}
            L^2(\mu|_{\mathcal{Z}_1}) = \left\{ f \in L^2(\mu) : \overline{\{T^nf:n\in \N\}}~\text{is compact} \right\}
        \end{equation*}
    \item As a measure-preserving system, $Z_1$ is the maximal factor of $(X, \mu, T)$ that is isomorphic to a rotation on a compact abelian group, i.e. a system of the form $(G, m, \theta)$, where $G$ is a compact abelian group, $m$ is the Haar measure on $G$, and $\theta \in G$ such that $\{n\theta : n \in \N\}$ is dense in $G$.
\end{itemize}

In the non-ergodic case, the factor $\mathcal{Z}_1$ takes on a more intricate form.
A prototypical example of a non-ergodic system that is isomorphic to its $Z_1$ factor is the skew-product $T(x,y) = (x, y+x)$ on $\T^2$.
Here, each ergodic component is isomorphic to a group rotation but the system as a whole is not.
A structure theorem of Frantzikinakis and Host \cite{FH18b} describes the $\mathcal{Z}_1$ factor of a general non-ergodic system in terms of ``relative'' eigenfunctions.

\begin{Definition}
    A \emph{relative orthonormal system} with respect to $\mathcal{I}$ is a countable family of functions $(\phi_j)_{j \in \N}$ in $L^2(\mu)$ such that
    \begin{itemize}
        \item $\E\left[ |\phi_j|^2 \mid \mathcal{I} \right] \in \{0,1\}$ almost everywhere for every $j \in \N$, and
        \item $\E \left[ \phi_j \overline{\phi}_k \mid \mathcal{I} \right] = 0$ almost everywhere for $j,k \in \N$, $j \ne k$.
    \end{itemize}
    The family $(\phi_j)_{j \in \N}$ is a \emph{relative orthonormal basis} if additionally
    \begin{equation*}
        L^2(\mu) = \overline{\textup{span} \left\{ \phi_j \psi : j \in \N, \psi \in L^{\infty}~\text{is}~\mathcal{I}\text{-measurable}\right\}}.
    \end{equation*}
\end{Definition}

\begin{Definition}
    A function $\phi \in L^{\infty}(\mu)$ is a \emph{relative eigenfunction} with respect to $\mathcal{I}$ if there exists a $T$-invariant function $\lambda \in L^{\infty}(\mu)$ such that
    \begin{itemize}
        \item $|\phi| \in \{0,1\}$ almost everywhere,
        \item $\lambda(x) = 0$ for almost every $x \in X$ such that $\phi(x) = 0$, and
        \item $T\phi = \lambda \phi$ almost everywhere.
    \end{itemize}
\end{Definition}

\begin{Theorem}[{\cite[Theorem 5.2]{FH18b}}] \label{thm: nonergodic Z_1}
    Let $(X, \mu, T)$ be a measure-preserving system.
    Then $L^2(\mu|_{\mathcal{Z}_1})$ admits a relative orthonormal basis of relative eigenfunctions.
\end{Theorem}

The upshot of Theorem \ref{thm: nonergodic Z_1} is that for $f \in L^2(\mu)$, the conditional expectation $\E[f \mid \mathcal{Z}_1]$ can be approximated by a linear combination of the form
\begin{equation*}
    \sum_{i=1}^d c_i \phi_i,
\end{equation*}
where $c_i$ is $T$-invariant and $\phi_i(Tx) = e(\alpha_i(x)) \phi_i(x)$ for some $T$-invariant function $\alpha_i : X \to \T$.
To emphasize the relationship with the arithmetic regularity lemma above (\cref{thm_arithmetic_regularity_lemma}), we may restate \cref{thm: nonergodic Z_1} in the following form.

\begin{Theorem} \label{thm: Z_1 structure theorem}
    Let $(X, \mu, T)$ be a measure-preserving system, and let $f : X \to [0,1]$ be measurable.
    Then there is a decomposition $f = f_{\str} + f_{\unf}$ such that
    \begin{enumerate}[label=(\roman*),leftmargin=*]
        \item (Nonnegativity) $0 \leq f_{\str} \leq 1$ almost everywhere, and $\int_X f_{\unf}~d\mu = 0$.
        \item (Structure) For every $\epsilon > 0$, there exists $d \in \N$, $T$-invariant functions $c_1, \ldots, c_d : X \to \C$ with $\|c_i\|_{\infty} \leq 1$, $T$-invariant functions $\alpha_1, \ldots, \alpha_d : X \to \T$, and relative eigenfunctions $\phi_1, \ldots, \phi_d$ satisfying $\phi_i(Tx) = e(\alpha_i(x))\phi_i(x)$ for almost every $x \in X$ such that if $\mathcal{P} = \sum_{i=1}^d c_i \phi_i$, then
            \begin{equation*}
                \left\| f_{\str} - \mathcal{P} \right\|_{L^2(\ u)} < \epsilon.
            \end{equation*}
        \item (Uniformity) $\|f_{\unf}\|_{U^2} = 0$.
        \item (Orthogonality) $\innprod{f_{\str}}{f_{\unf}} = 0$.
    \end{enumerate}
\end{Theorem}

%SUBSECTION
\subsection{Furstenberg systems} \label{sec: Furstenberg systems}

Suppose $u : \N \to \C$ is a bounded function.
One may associate to $u$ certain measure-preserving systems, called \emph{Furstenberg systems}, that capture statistical properties of $u$.

In order to define Furstenberg systems, we introduce additional notation.
Given a bounded function $u : \N \to \C$, let $\mathcal{A}_u \subseteq \ell^{\infty}(\Z)$ be the translation-invariant *-algebra generated by $u$.
That is, $\mathcal{A}_u$ consists of linear combinations of functions of the form
\begin{equation*}
    n \mapsto C^{\omega_1}u(n+h_1) \cdot \ldots \cdot C^{\omega_k}u(n+h_k)
\end{equation*}
for $k \in \N \cup \{0\}$, $h_1, \ldots, h_k \in \Z$, and $\omega_1, \ldots, \omega_k \in \{0,1\}$.
In order to make sense of such expressions for negative values of $h_i$, we extend $u$ to $\Z$ by defining $u(n) = 0$ for $n \leq 0$.
The construction in this section can be carried out within $\ell^{\infty}(\N)$, but we choose to work in $\ell^{\infty}(\Z)$ in order to produce invertible Furstenberg systems.

Suppose $\mathcal{A} \subseteq \ell^{\infty}(\Z)$ is a translation-invariant *-algebra (for example, $\mathcal{A}_u$ for some $u \in \ell^{\infty}(\N)$).
For a sequence $\mathbf{N} = (N_s)_{s \in \N}$ with $N_s \to \infty$, we say that $\mathcal{A}$ \emph{admits averages} along $\mathbf{N}$ if the limit\nomenclature{$\E_{n \in \mathbf{N}}$}{average along a sequence $\mathbf{N}$}
\begin{equation*}
    \E_{n \in \mathbf{N}} a(n) = \lim_{s \to \infty} \frac{1}{N_s} \sum_{n=1}^{N_s} a(n)
\end{equation*}
exists for every $a \in \mathcal{A}$.
A bounded function $u : \N \to \C$ \emph{admits correlations} along $\mathbf{N}$ if $\mathcal{A}_u$ admits averages along $\mathbf{N}$.
If $\mathcal{A}$ is separable, then a standard diagonalization argument shows that there exists some sequence $\mathbf{N}$ along which $\mathcal{A}$ admits averages.
In particular, for a bounded function $u : \N \to \C$, there exists a sequence $\mathbf{N}$ such that $u$ admits correlations.

Let $\mathcal{A} \subseteq \ell^{\infty}(\Z)$ be a separable translation-invariant *-algebra, and suppose $\mathcal{A}$ admits averages along a sequence $\mathbf{N}$.
The closure $\overline{\A}$ of $\A$ in $\ell^{\infty}(\Z)$ is a C*-algebra, so by the Gelfand--Naimark representation theorem, there is a compact metric space $X$ and an isomorphism $\Phi : \overline{\A} \to C(X)$.
(The space $X$, which can be identified with the space of C*-algebra homomorphisms from $\overline{\mathcal{A}}$ to $\C$, will be metrizable because of our assumption that $\A$ is separable.)
The isomorphism $\Phi$ induces a map $\tilde{T} : C(X) \to C(X)$ defined by $\Phi \circ \tau = \tilde{T} \circ \Phi$, where $(\tau a)(n) = a(n+1)$.
Let $T : X \to X$ then be the transformation defined by $\tilde{T}f = f \circ T$.
The map $a \mapsto \E_{n\in\mathbf{N}}a(n)$ is a positive linear functional on $\mathcal{A}$, so it induces a positive linear functional $L$ on $C(X)$.
Hence, by the Riesz--Markov--Kakutani representation theorem, there exists a Borel probability measure $\mu$ on $X$ such that
\begin{equation*}
    \int_X \Phi(a)~d\mu = L(\Phi(a)) = \E_{n \in \mathbf{N}} a(n)
\end{equation*}
for every $a \in \mathcal{A}$.
Since $\E_{n \in \mathbf{N}}a(n+1) = \E_{n \in \mathbf{N}}a(n)$, the measure $\mu$ is $T$-invariant.
That is, $(X, \mu, T)$ is a measure-preserving system, which we call the \emph{Furstenberg system} of $\mathcal{A}$ with respect to $\mathbf{N}$.
If $\mathcal{A} = \mathcal{A}_u$ for a bounded function $u : \N \to \C$, we will also refer to $(X, \mu, T)$ as the \emph{Furstenberg system} of $u$ with respect to $\mathbf{N}$.

%SUBSECTION
\subsection{Hilbert spaces of functions on $\N$}

Interestingly, the Hilbert space $L^2(\mu)$ for a Furstenberg system $(X, \mu, T)$ of $\mathcal{A}$ can be interpreted as a space of equivalence classes of functions defined on the integers.
Let $\A \subseteq \ell^{\infty}(\Z)$ be a separable translation-invariant *-algebra, and let $\mathbf{N} = (N_s)_{s \in \N}$ be a sequence such that $\A$ admits averages along $\mathbf{N}$.
Define\nomenclature{$\mathcal{L}^2(\mathcal{A},\mathbf{N})$}{$L^2$ space of functions associated to an algebra $\mathcal{A}$ along a sequence $\mathbf{N}$}
\begin{equation*}
    \mathcal{L}^2(\A, \mathbf{N}) = \left\{ f : \Z \to \C : \forall \epsilon > 0~\exists a \in \A, \limsup_{s \to \infty} \frac{1}{N_s} \sum_{n=1}^{N_s} |f(n) - a(n)|^2 < \epsilon \right\}/\sim_{\mathbf{N}},
\end{equation*}
where $f \sim_{\mathbf{N}} g$ if and only if $\E_{n \in \mathbf{N}} |f(n) - g(n)|^2 = 0$.

\begin{Theorem}[cf. {\cite[Theorem 2.1]{Farhangi}}]
\label{thm_farhangi_2.1}
    The space $\mathcal{L}^2(\A, \mathbf{N})$ is a Hilbert space with inner product\nomenclature{$\innprod{\cdot}{\cdot}_{\mathbf{N}}$}{inner product along a sequence $\mathbf{N}$}
    \begin{equation*}
        \innprod{f}{g}_{\mathbf{N}} = \E_{n \in \mathbf{N}} f(n)\overline{g(n)}.
    \end{equation*}
    Moreover, the map $\Phi : \A \to C(X)$ extends to an isometric isomorphism $\Phi : \mathcal{L}^2(\A, \mathbf{N}) \to L^2(X, \mu)$.
\end{Theorem}

\begin{proof}
    Since the algebra $\A$ is a vector space, we immediately have that $\mathcal{L}^2(\A, \mathbf{N})$ is also a vector space.
    Sesquilinearity of $\innprod{\cdot}{\cdot}_{\mathbf{N}}$ is built into the definition, and $\innprod{\cdot}{\cdot}_{\mathbf{N}}$ is positive definite on $\mathcal{L}^2(\A, \mathbf{N})$, since we have identified functions related by $\sim_{\mathbf{N}}$.

    Let us check that $\mathcal{L}^2(\A, \mathbf{N})$ is complete.
    Suppose $(f_n)_{n \in \N}$ is a Cauchy sequence in $\mathcal{L}^2(\A, \mathbf{N})$.
    Let $\epsilon_i \to 0$ such that for every $j \geq i$,
    \begin{equation*}
        \E_{n \in \mathbf{N}}|f_j(n) - f_i(n)|^2 < \epsilon_i.
    \end{equation*}
    Construct a sequence $(S_j)_{j \in \N}$ inductively to satisfy:
    \begin{itemize}
        \item for $i \in \N$, $j \geq i$, and $s \geq S_j$, then
            \begin{equation*}
                \frac{1}{N_s} \sum_{n=1}^{N_s} |f_j(n) - f_i(n)|^2 < \epsilon_i;
            \end{equation*}
        \item for $i \in \N$ and $j \geq i$,
            \begin{equation*}
                \frac{1}{N_{S_j} - N_{S_{j-1}}} \sum_{n=N_{S_{j-1}}+1}^{N_{S_j}} |f_j(n) - f_i(n)|^2 < \epsilon_i;
            \end{equation*}
            and
        \item for $i \in \N$,
            \begin{equation*}
                \frac{1}{N_{S_i}} \sum_{j=1}^{i-1} \sum_{n=N_{S_{j-1}}+1}^{N_{S_j}} |f_j(n) - f_i(n)|^2 < \epsilon_i.
            \end{equation*}
    \end{itemize}
    Define $f(n) = f_i(n)$ for $N_{S_{i-1}} + 1 \leq n \leq N_{S_i}$.
    Then for $k > i$ and $S_{k-1} < s < S_k$, we have
    \begin{align*}
        \sum_{n=1}^{N_s} |f(n) - f_i(n)|^2 = &~ \sum_{j=1}^{i-1} \sum_{n=N_{S_{j-1}}+1}^{N_{S_j}} |f_j(n) - f_i(n)|^2 \\
         & + \sum_{j=i}^{k-1} \sum_{n=N_{S_{j-1}}+1}^{N_{S_j}} |f_j(n) - f_i(n)|^2 + \sum_{n=N_{S_{k-1}}+1}^{N_s} |f_{k-1}(n) - f_i(n)|^2 \\
         & < \epsilon_i N_{S_i} + \sum_{j=i}^{k-1} \epsilon_i (N_{S_j} - N_{S_{j-1}}) + \epsilon_i N_s \leq 2 \epsilon_i N_s.
    \end{align*}
    Therefore, $f_i \to f$ in $\mathcal{L}^2(\A, \mathbf{N})$.

    To see that $\Phi$ extends to an isometric isomorphism $\Phi : \mathcal{L}^2(\A, \mathbf{N}) \to L^2(X, \mu)$, it suffices to note that $\Phi$ is an isometric isomorphism between $\A$ and $C(X)$, which are dense subspaces of $\mathcal{L}^2(\A, \mathbf{N})$ and $L^2(X, \mu)$ respectively.
\end{proof}

We denote the norm on $\mathcal{L}^2(\A, \mathbf{N})$ by $\|f\|_{2,\mathbf{N}} = \left( \innprod{f}{f}_{\mathbf{N}} \right)^{1/2} = \left( \E_{n \in \mathbf{N}}|f(n)|^2 \right)^{1/2}$.\nomenclature{$\lVert\cdot\rVert_{2,\mathbf{N}}$}{$L^2$ norm along a sequence $\mathbf{N}$}

%SUBSECTION
\subsection{Host--Kra seminorms on $\N$} \label{sec: discrete ergodic thm}

Let $\A \subseteq \ell^{\infty}(\Z)$ be a separable translation-invariant *-algebra.
Since we may view $L^2(\mu)$ for a Furstenberg system $(X, \mu, T)$ of $\A$ as the space $\mathcal{L}^2(\A,\mathbf{N})$, we may also lift the Host--Kra seminorms to $\A$.
Define the \emph{Host--Kra $U^k$-seminorm} $\|f\|_{U^k(\mathbf{N})}$ by\nomenclature{$\lVert\cdot\rVert_{U^k(\mathbf{N})}$}{Host--Kra seminorm along a sequence $\mathbf{N}$}
\begin{equation*}
    \|f\|_{U^k(\mathbf{N})} = \|\Phi(f)\|_{U^k},
\end{equation*}
where $\Phi : \A \to C(X)$ is the Gelfand representation.
More explicitly,
\begin{equation*}
    \|f\|_{U^k(\mathbf{N})} = \left( \lim_{H \to \infty} \frac{1}{H^k} \sum_{\bm{h} \in [H]^k} \lim_{s \to \infty} \frac{1}{N_s} \sum_{n=1}^{N_s} \prod_{\bm{\omega}\in\{0,1\}^k} C^{|\omega|} f(n+\bm{\omega}\cdot\bm{h})\right)^{1/2^k}.
\end{equation*}

We will now interpret the structure theorems for the Host--Kra seminorms in the context of functions on $\N$ and describe the relationship between the $U^1(\mathbf{N})$-seminorm and the ergodic decomposition of the Furstenberg system $(X, \mu, T)$.

%SUBSUBSECTION
\subsubsection{The $U^1(\mathbf{N})$-seminorm and local ergodicity}

Call a bounded function $f : \N \to \C$ \emph{locally invariant} along $\mathbf{N}$ if
\begin{equation*}
    \limsup_{s \to \infty} \frac{1}{N_s} \sum_{n=1}^{N_s} |f(n+1) - f(n)| = 0.
\end{equation*}
As defined in \cref{def: locally ergodic}, a bounded function $f : \N \to \C$ is \emph{locally ergodic} along $\mathbf{N}$ if
\begin{equation*}
    \limsup_{H \to \infty} \limsup_{s \to \infty} \frac{1}{N_s} \sum_{n=1}^{N_s} \left| \frac{1}{H} \sum_{h=1}^H f(n+h)\right| = 0.
\end{equation*}
Note that if $f$ admits correlations along $\mathbf{N}$, then $f$ is locally ergodic if and only if $\|f\|_{U^1(\mathbf{N})} = 0$.
Reinterpreting \cref{thm: ergodic theorem}, we have the following decomposition theorem showing that local invariance and local ergodicity are complementary notions.

\begin{Theorem} \label{thm: discrete ergodic theorem}
    Let $f : \N \to [0,1]$, and suppose $\mathbf{N} = (N_s)_{s \in \N}$ is a sequence with $\lim_{s \to \infty} N_s = \infty$ such that $f$ admits correlations along $\mathbf{N}$.
    Then there exists a decomposition $f = f_{\inv} + f_{\erg}$ such that
    \begin{enumerate}[label=(\roman*),leftmargin=*]
        \item\label{itm disc ET i} (Nonnegativity) $0 \leq f_{\inv} \leq 1$ and $\E_{n \in \mathbf{N}} f_{\erg}(n) = 0$.
        \item\label{itm disc ET ii} (Structure) $f_{\inv}$ is locally invariant.
        \item\label{itm disc ET iii} (Uniformity) $f_{\erg}$ is locally ergodic.
        \item\label{itm disc ET iv} (Orthogonality) $\innprod{f_{\inv}}{f_{\erg}}_{\mathbf{N}} = 0$.
    \end{enumerate}
\end{Theorem}

\begin{proof}
    Let $\A_f$ be the translation-invariant *-algebra generated by $f$.
    Let $(X, \mu, T)$ be the Furstenberg system of $f$ along $\mathbf{N}$ and $\Phi : \mathcal{L}^2(\A_f, \mathbf{N}) \to L^2(\mu)$ the extension of the Gelfand representation provided by \cref{thm_farhangi_2.1}.
    We apply \cref{thm: ergodic theorem} to $g = \Phi(f)$ to obtain a decomposition $g = g_{\inv} + g_{\erg}$.
    Taking $f_{\inv} = \Phi^{-1}(g_{\inv})$ and $f_{\erg} = \Phi^{-1}(g_{\erg})$ provides the desired decomposition of $f$.
\end{proof}

The proof of \cref{thm: discrete ergodic theorem} that we have just presented is rather abstract, relying on the Furstenberg correspondence principle and the mean ergodic theorem to produce $f_{\inv}$ as a projection of $f$ onto the subspace of locally invariant functions.
For some applications and for further refinements of \cref{thm: discrete ergodic theorem} such as \cref{thm:U^1 structure theorem} below, it is useful to give an alternative proof that gives a more concrete description of the function $f_{\inv}$.

The following technical lemma is key in our next proof of \cref{thm: discrete ergodic theorem}. Additionally, it will play an important role later in proving the structure theorems for the intermediate scale uniformity seminorms introduced in \cref{sec: intermediate scale seminorms}.

\begin{Lemma}
\label{lem_decomposition_in_the_limit}
Let $f\colon \N\to\C$ be a function, and let $\mathbf{N} = (N_s)_{s \in \N}$ be a sequence of natural numbers with $\lim_{s\to\infty}N_s=\infty$.
Suppose for every $k\in\N$ there exist $f_{k,1},f_{k,2}\colon \N\to\C$ such that
\begin{itemize}
\item $\|f_{k,1}\|_\infty,\|f_{k,2}\|_\infty\leq \|f\|_\infty$;
\item $f=f_{k,1}+f_{k,2}$;
\item $\innprod{f_{k,1}}{f_{\ell,1}}_{\mathbf{N}}$, $\innprod{f_{k,2}}{f_{\ell,2}}_{\mathbf{N}}$, and $\innprod{f_{k,1}}{f_{\ell,2}}_{\mathbf{N}}$ are well-defined for all $k,\ell\in\N$.
\item $\lim_{k\to\infty} \sup_{\ell\geq k} |\innprod{f_{k,1}}{f_{\ell,2}}_{\mathbf{N}}|=0$;
\end{itemize}
Then there exists an increasing sequence of natural numbers $(k_t)_{t\in\N}$ such that if
\[
f_1(n)=\sum_{t\in \N} \1_{(N_{t-1},N_t]}(n) f_{k_t,1}(n)
\qquad\text{and}\qquad
f_2(n)=\sum_{t\in \N} \1_{(N_{t-1},N_t]}(n) f_{k_t,2}(n),
\]
then we have:
\begin{itemize}
\item $f=f_1+f_2$;
\item $\innprod{f_{1}}{f_{2}}_{\mathbf{N}}=0$;
\item $\lim_{k\to\infty} \| f_{1} -  f_{k,1} \|_{2, \mathbf{N}}=
\lim_{k\to\infty} \| f_{2} -  f_{k,2} \|_{2, \mathbf{N}}=0$.
\end{itemize} 
\end{Lemma}

\begin{proof}
By multiplying $f$ with a positive constant if necessary, we can assume without loss of generality that $\|f\|_\infty\leq 1$.
Using $f=f_{k,1}+f_{k,2}$ and $\sup_{\ell\geq k}|\innprod{f_{k,1}}{f_{\ell,2}}_{\mathbf{N}}|=\oh_{k\to\infty}(1)$, we observe that uniformly over all $\ell\geq k$,
\begin{align*}
\|f_{k,1}\|_{2, \mathbf{N}}^2
&=\innprod{f_{k,1}}{f_{k,1}}_{\mathbf{N}}
\\
&=\innprod{f_{k,1}}{f}_{\mathbf{N}}+\oh_{k\to\infty}(1)
\\
&=\innprod{f_{k,1}}{f_{\ell,1}}_{\mathbf{N}}+\oh_{k\to\infty}(1)
\\
&\leq \|f_{k,1}\|_{2, \mathbf{N}} \cdot \|f_{\ell,1}\|_{2, \mathbf{N}} +\oh_{k\to\infty}(1),
\end{align*}
where the last inequality follows from the Cauchy-Schwarz inequality.
If $0\leq z, z' \leq \|f\|_\infty$ are two limit points of the sequence $\|f_{k,1}\|_{2,\mathbf{N}}$, then this inequality yields $z^2 \leq zz'$ and $(z')^2 \leq zz'$.
The only way for both inequalities to be satisfied is if $z = z'$, so the limit $\lim_{k\to\infty} \|f_{k,1}\|_{2, \mathbf{N}}$ exists.
Then, also uniformly over all $\ell\geq k$, we get
\begin{align*}
\|f_{\ell,1}-f_{k,1}\|_{2, \mathbf{N}}^2
&= \|f_{\ell,1}\|_{2, \mathbf{N}}^2 +\|f_{k,1}\|_{2, \mathbf{N}}^2 - 2\underbrace{\Re\big(\innprod{f_{k,1}}{f_{\ell,1}}_{\mathbf{N}}\big)}_{\|f_{k,1}\|_{2, \mathbf{N}}^2+\oh_{k\to\infty}(1)}
\\
&= \|f_{\ell,1}\|_{2, \mathbf{N}}^2 -\|f_{k,1}\|_{2, \mathbf{N}}^2 +\oh_{k\to\infty}(1)
\\
&= \oh_{k\to\infty}(1).
\end{align*}
This shows that $f_{k,1}$ is a Cauchy sequence with respect to $\|.\|_{2,\mathbf{N}}$.
Let $\epsilon_k \to 0$ such that for every $\ell \geq k$,
    \begin{equation*}
        \lim_{s\to\infty}\frac{1}{N_s}\sum_{n =1}^{N_s}|f_{\ell,1}(n) - f_{k,1}(n)|^2 < \epsilon_k.
    \end{equation*}
    Construct a sequence $(S_j)_{j \in \N}$ inductively to satisfy:
    \begin{itemize}
        \item for $k \in \N$, $\ell \geq k$, and $s \geq S_{\ell-1}$,
            \begin{equation*}
                \frac{1}{N_s} \sum_{n=1}^{N_s} |f_{\ell,1}(n) - f_{k,1}(n)|^2 < \epsilon_k;
            \end{equation*}
        \item for $k \in \N$ and $\ell \geq k$,
            \begin{equation*}
                \frac{1}{N_{S_\ell} - N_{S_{\ell-1}}} \sum_{n=N_{S_{\ell-1}}+1}^{N_{S_\ell}} |f_{\ell,1}(n) - f_{k,1}(n)|^2 < \epsilon_k;
            \end{equation*}
        \item for $k \in \N$,
            \begin{equation*}
                \frac{1}{N_{S_k}} \sum_{j=1}^{k-1} \sum_{n=N_{S_{j-1}}+1}^{N_{S_j}} |f_{\ell,1}(n) - f_{k,1}(n)|^2 < \epsilon_k.
            \end{equation*}
    \end{itemize}
    Now define $k_t=\min\{m\in\N: t\leq S_m\}$
    and consider
    \[
    f_1(n)=\sum_{t\in \N} \1_{(N_{t-1},N_t]}(n) f_{k_t,1}(n)
\quad\text{and}\quad
f_2(n)=\sum_{t\in \N} \1_{(N_{t-1},N_t]}(n) f_{k_t,2}(n).
    \]
    The sequence $(k_t)_{t\in\N}$ has the property that $k_t=m$ for all $t\in\N$ with $S_{m-1}<t\leq S_m$, or equivalently, for all $t\in\N$ we have
    $S_{k_t-1}< t\leq S_{k_t}$.
    For every $k\in\N$ and every sufficiently large $s\in\N$ we thus have
    \begin{align*}
        \sum_{n=1}^{N_s} |f_1(n) - f_{k,1}(n)|^2 = &~ 
        \sum_{1\leq t\leq s} \sum_{n=N_{t-1}+1}^{N_t}
        |f_{k_t,1}(n) - f_{k,1}(n)|^2
        \\
        = &~
        \sum_{j=1}^{k-1} \sum_{n=N_{S_{j-1}}+1}^{N_{S_j}} |f_{\ell,1}(n) - f_{k,1}(n)|^2 \\
         & + \sum_{j=k}^{k_s-1} \sum_{n=N_{S_{j-1}}+1}^{N_{S_j}} |f_{\ell,1}(n) - f_{k,1}(n)|^2 + \sum_{n=N_{S_{k_s-1}}+1}^{N_s} |f_{k_s}(n) - f_{k,1}(n)|^2 \\
         & < \epsilon_k N_{S_k} + \sum_{j=k}^{k_s-1} \epsilon_k (N_{S_j} - N_{S_{j-1}}) + \epsilon_k N_s \leq 2 \epsilon_k N_s.
    \end{align*}
This proves that $\lim_{k\to\infty} \| f_{1} -  f_{k,1} \|_{2, \mathbf{N}}=0$. Likewise, we get $\lim_{k\to\infty} \| f_{2} -  f_{2,1} \|_{2, \mathbf{N}}=0$. Since $\lim_{k\to\infty} \innprod{f_{k,1}}{f_{k,2}}_{\mathbf{N}}=0$, we obtain that 
 $\innprod{f_1}{f_2}_{\mathbf{N}}=0$ as desired. 
\end{proof}

We now give a second proof of \cref{thm: discrete ergodic theorem}.

\begin{proof}[Proof of \cref{thm: discrete ergodic theorem}]
Define for every $k\in\N$ the function 
\[
f_{k,\inv}(n)=\sum_{t\in\N} \1_{[2^kt, 2^k(t+1))}(n) \Bigg(\frac{1}{2^k}\sum_{h=0}^{2^k-1} f(2^kt+h)\Bigg)
\]
and let $f_{k,\erg}=f-f_{k,\inv}$. Then we have
\[
\innprod{f_{k,\erg}}{f_{\ell,\inv}}_{\mathbf{N}}=0,\qquad\forall k, \ell\in\N~\text{with}~k\leq \ell.
\]
Moreover, since $f$ admits correlations along $\mathbf{N}$, we have that
$\innprod{f_{k,\inv}}{f_{\ell,\inv}}_{\mathbf{N}}$, $\innprod{f_{k,\inv}}{f_{\ell,\erg}}_{\mathbf{N}}$, and $\innprod{f_{k,\erg}}{f_{\ell,\erg}}_{\mathbf{N}}$ are well-defined for all $k,\ell\in\N$. 
We can thus apply \cref{lem_decomposition_in_the_limit} with $f_{k,1}=f_{k,\inv}$ and $f_{k,2}=f_{k,\erg}$ to find two functions $f_{\inv}\colon \N\to [0,1]$ and $f_{\erg}\colon \N\to[-1,1]$ such that $f=f_{\inv}+f_{\erg}$, $\innprod{f_{\inv}}{f_{\erg}}_{\mathbf{N}}=0$, and
$\lim_{k\to\infty} \| f_{\inv} -  f_{k,\inv} \|_{2, \mathbf{N}}=
\lim_{k\to\infty} \| f_{\erg} -  f_{k,\erg} \|_{2, \mathbf{N}}=0$.
Note that for all $k\in\N$, $f_{k,\erg}$ is locally ergodic along $\mathbf{N}$, and hence $f_{\erg}$ is locally ergodic along $\mathbf{N}$. Moreover, for all $k\in\N$,
\[
\lim_{s\to\infty}\frac{1}{N_s}\sum_{n=1}^{N_s}|f_{k,\inv}(n+1)-f_{k,\inv}(n)|=\Oh(2^{-k}),
\]
hence $f_{\inv}$ is locally invariant along $\mathbf{N}$.
\end{proof}

Our next result uses this decomposition theorem to relate locally ergodicity to the ergodic decomposition of an associated Furstenberg system.

\begin{Theorem} \label{thm: U^1 uniformity implies constant measure in ergodic components}
    Let $A \subseteq \N$, and suppose $\mathbf{N} = (N_s)_{s \in \N}$ is a sequence with $\lim_{s \to \infty} N_s = \infty$ such that $\1_A$ admits correlations along $\mathbf{N}$.
    Let $(X, \mu, T)$ be the Furstenberg system of $A$ along $\mathbf{N}$.
    Let $\A \subseteq \ell^{\infty}(\Z)$ be the translation-invariant *-algebra generated by $\1_A$, and let $\Phi : \A \to C(X)$ be the Gelfand representation.
    Let $E \subseteq X$ be the clopen set defined by $\1_E = \Phi(\1_A)$.
    Suppose $\mu = \int_X \mu_x~d\mu(x)$ is an ergodic decomposition of $\mu$.
    If $\1_A - d_{\mathbf{N}}(A)$ is locally ergodic, where $d_{\mathbf{N}}(A) = \E_{n \in \mathbf{N}} \1_A(n)$, then $\mu_x(E) = d_{\mathbf{N}}(A)$ for almost every $x \in X$.
\end{Theorem}

\begin{proof}
    Let $f = \1_A$.
    The assumption that $\1_A - d_{\mathbf{N}}(A)$ is locally ergodic means that the decomposition $f = f_{\inv} + f_{\erg}$ takes the form $d_{\mathbf{N}} + (\1_A - d_{\mathbf{N}}(A))$.
    The Gelfand representation maps constants to constants, so
    \begin{equation*}
        \E[\1_E \mid \mathcal{I}] = \Phi(f_{\inv}) = \Phi(d_{\mathbf{N}}(A)) = d_{\mathbf{N}}(A)
    \end{equation*}
    almost everywhere.
    Moreover, by the definition of an ergodic decomposition (\cref{def: ergodic decomp}),
    \begin{equation*}
        \E[\1_E \mid \mathcal{I}](x) = \int_X \1_E~d\mu_x = \mu_x(E)
    \end{equation*}
    for almost every $x \in X$.
\end{proof}

Let $\mathcal{L}^{\infty}(\A,\mathbf{N})$ denote the space of bounded functions in $\mathcal{L}^2(\A, \mathbf{N})$, i.e. $\mathcal{L}^{\infty}(\A,\mathbf{N}) = \mathcal{L}^2(\A,\mathbf{N}) \cap \ell^{\infty}(\Z)$.
Note that $\A$ is $\mathcal{L}^2$-dense in $\mathcal{L}^{\infty}(\A, \mathbf{N})$.
Therefore, by continuity of $\Phi$ with respect to the $\mathcal{L}^2$ norm, the Gelfand representation $\Phi : \A \to C(X)$ extends to a C*-algebra isomorphism $\Phi : \mathcal{L}^{\infty}(\A, \mathbf{N}) \to L^{\infty}(\mu)$.
With this observation, we can prove the following theorem, which succinctly captures the sense in which the Furstenberg system encodes the statistical behavior of functions $f \in \A$.

\begin{Theorem} \label{thm: statistical representation}
    Let $\A \subseteq \ell^{\infty}(\Z)$ be a separable translation-invariant *-algebra that admits averages along a sequence $\mathbf{N}$.
    Let $(X, \mu, T)$ be the Furstenberg system of $\A$ along $\mathbf{N}$.
    Let $\Phi : \A \to C(X)$ be the Gelfand representation.
    Then for any $f \in \mathcal{L}^{\infty}(\A, \mathbf{N})$ and any continuous function $F : \C \to \C$,
    \begin{equation} \label{eq:ergodic decomp as random variable}
        \int_X F(\Phi(f))~d\mu = \E_{n \in \mathbf{N}}F(f(n)).
    \end{equation}

    In other words, if we view $f_s = f|_{\{1, \dots, N_s\}}$ as a random variable, where $\{1,\dots, N_s\}$ is given the uniform probability measure, then $f_s$ converges in distribution to the random variable $\Phi(f)$ as $s \to \infty$.
\end{Theorem}

\begin{proof}
    It clearly suffices to prove \eqref{eq:ergodic decomp as random variable} for continuous functions defined on the disk $\{z \in \C : |z| \leq \|f\|_{\infty}\}$.
    Then by a standard approximation argument and the Stone--Weierstrass theorem, it is enough to check \eqref{eq:ergodic decomp as random variable} for functions $F$ of the form
    \begin{equation} \label{eq:conjugate polynomial}
        F(z) = \sum_{j,k=0}^d a_{j,k} z^j \overline{z}^k \in \C[z,\overline{z}].
    \end{equation}
    
    The map $\Phi$ is a *-algebra homomorphism on $\mathcal{L}^{\infty}(\A, \mathbf{N})$, so given $F$ of the form \eqref{eq:conjugate polynomial}, we have $\Phi(F(f)) = F(\Phi(f))$.
    Taking the inner product with the constant function $\1$ and using the fact that $\Phi$ is an isometric isomorphism from $\mathcal{L}^2(\A, \mathbf{N})$ to $L^2(\mu)$, we conclude
    \begin{align*}
        \int_X F(\Phi(f))~d\mu & = \innprod{F(\Phi(f))}{\1}_{L^2(\mu)} \\
        & = \innprod{\Phi(F(f))}{\Phi(\1)}_{L^2(\mu)} \\
        & = \innprod{F(f)}{\1}_{\mathbf{N}} \\
        & = \E_{n \in \mathbf{N}} F(f(n)).
    \end{align*}
\end{proof}

%SUBSUBSECTION
\subsubsection{The $U^2(\mathbf{N})$-seminorm}

Recall that the $U^2(\mathbf{N})$ seminorm is defined by
\begin{equation*}
    \|f\|_{U^2(\mathbf{N})} = \left( \lim_{H \to \infty} \frac{1}{H^2} \sum_{h_1,h_2=1}^H \lim_{s \to \infty} \frac{1}{N_s} \sum_{n=1}^{N_s} f(n) \overline{f(n+h_1)} \overline{f(n+h_2)} f(n+h_1+h_2) \right)^{1/4},
\end{equation*}
assuming that the limits exist.
Applying the $\mathcal{Z}_1$ structure theorem of Frantzikinakis and Host (Theorem \ref{thm: Z_1 structure theorem}), we obtain the following structure theorem for the $U^2(\mathbf{N})$-seminorm.

\begin{Theorem} \label{thm: U2 HK structure theorem}
    Let $f : \N \to [0,1]$, and suppose $\mathbf{N} = (N_s)_{s \in \N}$ is a sequence with $\lim_{s \to \infty} N_s = \infty$ such that $f$ admits correlations along $\mathbf{N}$.
    Then there exists a decomposition $f = f_{\str} + f_{\unf}$ such that
    \begin{enumerate}[label=(\roman*),leftmargin=*]
        \item (Nonnegativity) $0 \leq f_{\str} \leq 1$ and $\E_{n \in \mathbf{N}} f_{\unf}(n) = 0$.
        \item (Structure) For every $\epsilon > 0$, there exists $d \in \N$ and locally invariant functions $c_1, \ldots, c_d : \N \to \C$ with $\|c_i\|_{\infty} \leq 1$ and locally invariant functions $\alpha_1, \ldots, \alpha_d : \N \to \T$ such that if $\mathcal{P}(n) = \sum_{i=1}^d c_i(n) e(n\alpha_i(n))$, then
            \begin{equation*}
                \left\| f_{\str} - \mathcal{P} \right\|_{2,\mathbf{N}} < \epsilon.
            \end{equation*}
        \item (Uniformity) $\|f_{\unf}\|_{U^2(\mathbf{N})} = 0$.
        \item (Orthogonality) $\innprod{f_{\str}}{f_{\unf}}_{\mathbf{N}} = 0$.
    \end{enumerate}
\end{Theorem}

%SECTION
\section{The intermediate-scale seminorms} \label{sec: intermediate scale seminorms}

\subsection{Notation}

Let $\mathbf{N} = (N_s)_{s \in \N}$ and $\mathbf{H}=(H_s)_{s\in\N}$ be sequences in $\N$.
We write $1\prec \mathbf{H}$ if $\lim_{s\to\infty} 1/H_s=0$,
$\mathbf{H}\preceq \mathbf{N}$ if $H_s\leq N_s$ for all large $s\in\N$, and $\mathbf{H}\prec \mathbf{N}$ if $\lim_{s\to\infty} H_s/N_s=0$.\nomenclature{$\prec,\preceq$}{asymptotic ordering of non-decreasing sequences in $\N$}
Given two $k$-tuples of sequences $(\mathbf{N}_1,\ldots, \mathbf{N}_k)$ and $(\mathbf{N}_1',\ldots,\mathbf{N}_k')$, where $\mathbf{N}_i = (N_{i,s})_{s \in \N}$ and $\mathbf{N}_i' = (N_{i,s}')_{s \in \N}$ are sequences in $\N$, we say that $(\mathbf{N}_1',\ldots,\mathbf{N}_k')$ is a subsequence of $(\mathbf{N}_1,\ldots,\mathbf{N}_k)$ if there exists a strictly increasing function $t\colon \N\to\N$ such that $N_{i,s}'=N_{i,t(s)}$ for all $s\in\N$ and $i\in\{1,\ldots,k\}$.

\subsection{The $U^1(\mathbf{N},\mathbf{H})$ seminorm}
\label{sec: intermediate U1 seminorm}

If $\mathbf{N} = (N_s)_{s \in \N}$ and $\mathbf{H}=(H_s)_{s\in\N}$ are sequences in $\N$ with
$1\prec \mathbf{H}\preceq \mathbf{N}$,
and $f\colon \N\to\C$ is a bounded function, we define\nomenclature{$\lVert\cdot\rVert_{U^k(\mathbf{N},\mathbf{H})}$}{intermediate scale seminorm}
\begin{equation}
\label{eqn_intermediate_scale_U_1_norm_def_1}
\|f\|_{U^1(\mathbf{N},\mathbf{H})}=\lim_{s \to \infty} \frac{1}{N_s-H_s+1} \sum_{n=1}^{N_s-H_s+1}  \|f\|_{U^1(\{n,n+1,\ldots,n+H_{s}-1\})}
\end{equation}
whenever this limit exists, where $\|f\|_{U^1(\{n,n+1,\ldots,n+H_{s}-1\})}$ is the $U^1$ Gowers norm of $f$ on the interval $\{n,n+1,\ldots,n+H_s-1\}$ (see \cref{sec: Gowers norms}).
If this limit does not exist, then we say that $\|f\|_{U^1(\mathbf{N},\mathbf{H})}$ is not well defined.

Note that if $ \mathbf{H}\prec \mathbf{N}$ then
\begin{equation}
\label{eqn_intermediate_scale_U_1_norm_def_2} 
\|f\|_{U^1(\mathbf{N},\mathbf{H})}=
\lim_{s \to \infty} \frac{1}{N_s} \sum_{n=1}^{N_s}  
\Bigg|\frac{1}{H_s}\sum_{h=1}^{H_s} f(n+h)\Bigg|,
\end{equation}
where the limit in \eqref{eqn_intermediate_scale_U_1_norm_def_1} exists if and only if the one in \eqref{eqn_intermediate_scale_U_1_norm_def_2} does. On the other hand, if $ \mathbf{H}= \mathbf{N}$ then $\|f\|_{U^1(\mathbf{N},\mathbf{N})}$ coincides with the mean of $f$ along the sequence $\mathbf{N}=(N_s)_{s\in\N}$, that is,
\[
\|f\|_{U^1(\mathbf{N},\mathbf{N})}=\lim_{s \to \infty} \Bigg|\frac{1}{N_s}\sum_{n=1}^{N_s} f(n)\Bigg|.
\]

The $\|.\|_{U^1(\mathbf{N},\mathbf{H})}$ seminorm sandwiches between the Host--Kra seminorm $\|.\|_{U^1(\mathbf{N})}$ and the (asymptotic) Gowers seminorm $\lim_{s\to\infty}\|.\|_{U^1([N_s])}$.
Indeed, a straightforward application of the triangle inequality reveals that if
$1\prec\mathbf{H}_1\prec\mathbf{H}_2\preceq\mathbf{N}$ then
\begin{equation}
\label{eqn: U1 hierarchy}
\underbrace{\|f\|_{U^1(\mathbf{N})}}_{\text{Host--Kra seminorm}}\hspace{-1em}\geq \|f\|_{U^1(\mathbf{N},\mathbf{H}_1)}
\geq \|f\|_{U^1(\mathbf{N},\mathbf{H}_2)}\geq \lim_{s\to\infty} \underbrace{\|f\|_{U^1([N_s])}}_{\text{Gowers norm}},
\end{equation}
where each of the inequalities is understood to hold if the involved seminorms are well defined.

\begin{Definition}[$U^1$~good scale]
\label{def: U1 good scale}
Let $f\colon\N\to [0,1]$, and suppose $\mathbf{N} = (N_s)_{s \in \N}$ and $\mathbf{H}=(H_s)_{s\in\N}$ are sequences in $\N$ with
$1\prec \mathbf{H}\preceq \mathbf{N}$.
We say that $(\mathbf{N},\mathbf{H})$ is a \define{$U^1$~good scale} for $f$ if there exists a decomposition $f = f_{\str} + f_{\unf}$ such that:
\begin{enumerate}[label=(\roman{enumi}),ref=(\roman{enumi}),leftmargin=*]
\item\label{itm U^1 structure theorem i}
(Nonnegativity) 
$0 \leq f_{\str} \leq 1$ and $\E_{n \in \mathbf{N}} f_{\unf} = 0$.
\item\label{itm U^1 structure theorem ii}
(Structure)
If
\[
    \Gamma_s = \left\{ n\in [N_s-H_s+1]: f_{\str}(n+m_1)=f_{\str}(n+m_2)~\forall m_1,m_2\in [H_s] \right\},
\]
then
\[
\lim_{s\to\infty} \frac{|\Gamma_s|}{N_s-H_s+1}=1.
\]
\item\label{itm U^1 structure theorem iii}
(Uniformity)
There is $\mathbf{H}'=(H_s')_{s\in\N}$ with $1\prec \mathbf{H}'\prec \mathbf{H}$ and $\|f_{\unf}\|_{U^1(\mathbf{N},\mathbf{H}')}=0.$
\end{enumerate}
\end{Definition}

Whenever $(\mathbf{N},\mathbf{H})$ is a $U^1$ good scale, the $U^1$ seminorm of $f$ is determined entirely by its structured component, that is, $\|f\|_{U^1(\mathbf{N},\mathbf{H})}=\|f_{\str}\|_{U^1(\mathbf{N},\mathbf{H})}$.
Parts \ref{itm U^1 structure theorem ii} and \ref{itm U^1 structure theorem iii} of the definition further imply
\begin{enumerate}[label=(\roman{enumi}),ref=(\roman{enumi}),leftmargin=*]
    \setcounter{enumi}{3}
    \item
    (Orthogonality)
    $\innprod{f_{\str}}{f_{\unf}}_{\mathbf{N}} = 0$.
\end{enumerate}
These properties together imply that the decomposition is unique up to modifications on a set of zero density: if $f = f_{\str} + f_{\unf}$ and $f = f_{\str}' + f_{\unf}'$ are two decompositions satisfying \ref{itm U^1 structure theorem i}, \ref{itm U^1 structure theorem ii}, and \ref{itm U^1 structure theorem iii}, then $\|f_{\str}-f_{\str}'\|_{2,\mathbf{N}}=\|f_{\unf}-f_{\unf}'\|_{2,\mathbf{N}}=0$. Finally, the definition of a $U^1$ good scale is stable under small modifications of the scale parameter $\mathbf{H}$: if $(\mathbf{N},\mathbf{H})$ is $U^1$ good and $\mathbf{H}'$ is as in part \ref{itm U^1 structure theorem iii}, then for any intermediate scale $\mathbf{H}''$ satisfying $\mathbf{H}'\prec\mathbf{H}''\preceq \mathbf{H}$ the pair $(\mathbf{N},\mathbf{H}'')$ is also $U^1$ good. Thus, the existence of a single good scale automatically provides an entire range of compatible good scales.

The next theorem serves as the structure theorem for the $U^1(\mathbf{N},\mathbf{H})$-seminorms. It ensures the existence of many $U^1$-good scales, and by the very definition of a good scale this, in turn, guarantees the existence of a decomposition at scale $\mathbf{H}$ into an order-$1$ structured component and an order-$1$ uniform component.

\begin{Theorem}[Structure theorem for $U^1(\mathbf{N},\mathbf{H})$-seminorm] \label{thm:U^1 structure theorem}
Let $f\colon\N\to [0,1]$, and suppose $\mathbf{N} = (N_s)_{s \in \N}$, $\mathbf{K}^+=(K^+_s)_{s\in\N}$, and $\mathbf{K}^-=(K^-_s)_{s\in\N}$ are sequences in $\N$ satisfying
\[
1\prec \mathbf{K}^-\prec \mathbf{K}^+\preceq \mathbf{N}.
\]
After replacing $(\mathbf{N},\mathbf{K}^+,\mathbf{K}^-)$ by a subsequence if necessary, there exist sequences $\mathbf{H}^+ = (H_s^+)_{s \in \N}$ and $\mathbf{H}^- = (H_s^-)_{s \in \N}$ with
\[
\mathbf{K}^-\preceq \mathbf{H}^-\prec \mathbf{H}^+\preceq \mathbf{K}^+
\]
such that for any $\mathbf{H}=(H_s)_{s\in\N}$ with $\mathbf{H}^-\prec \mathbf{H}\prec \mathbf{H}^+$ the pair $(\mathbf{N},\mathbf{H})$ is a $U^1$~good scale for $f$.
\end{Theorem}

\begin{proof}
By replacing $(\mathbf{N},\mathbf{K}^+,\mathbf{K}^-)$ with a subsequence of itself if necessary, we can assume without loss of generality that $\lim_{s\to\infty} N_s/N_{s+1}=0$.
For $k\in\N$, let $\mathbf{H}_k=(H_{k,s})_{s\in\N}$ be the sequence defined as
\[
H_{k,s}=\big\lfloor (K^-_s)^{\frac{1}{k}}(K^+_s)^{1-\frac{1}{k}}\big\rfloor,
\]
where $\lfloor .\rfloor$ denotes the floor function.
Clearly we have $\mathbf{K}^-= \mathbf{H}_1\prec \mathbf{H}_2\prec \mathbf{H}_3\prec \ldots\prec\mathbf{K}^+$.

Next, let $\mathcal{I}_{k,s}$ be a partition of $\{N_{s-1}+1,\ldots,N_s\}$ into intervals of length between $H_{k,s}/2$ and $2 H_{k,s}$. Since $\lim_{s\to\infty} H_s/(N_{s}-N_{s-1})=\lim_{s\to\infty} H_s/N_s=0$, we have that $\mathcal{I}_{k,s}$ is non-empty for all but at most finitely many $s\in\N$.
We now define
\begin{align*}
f_{k,\str}(n)&=\sum_{s\in\N} \sum_{I\in\mathcal{I}_{k,s}} \1_I(n) \Bigg(\frac{1}{|I|}\sum_{m\in I} f(m)\Bigg)
\\
f_{k,\unf}(n)&=\sum_{s\in\N} \sum_{I\in\mathcal{I}_{k,s}} \1_I(n) \Bigg(f(n)- \frac{1}{|I|}\sum_{m\in I} f(m)\Bigg).
\end{align*}
Replacing once more $(\mathbf{N},\mathbf{K}^+,\mathbf{K}^-)$ with a subsequence if necessary, we can assume that
$\innprod{f_{k,\str}}{f_{\ell,\str}}_{\mathbf{N}}$, $\innprod{f_{k,\str}}{f_{\ell,\unf}}_{\mathbf{N}}$, and $\innprod{f_{k,\unf}}{f_{\ell,\unf}}_{\mathbf{N}}$ are well defined for all $k,\ell\in\N$.
Moreover, by construction we have
\[
\innprod{f_{\ell,\str}}{f_{k,\unf}}_{\mathbf{N}}=0~~\text{whenever}~\ell\geq k.
\]
Invoking \cref{lem_decomposition_in_the_limit}, we can now find an increasing sequence of natural numbers $(k_t)_{t\in\N}$ such that if
\[
f_{\str}(n)=\sum_{t\in \N} \1_{(N_{t-1},N_t]}(n) f_{k_t,\str}(n)
\qquad\text{and}\qquad
f_{\unf}(n)=\sum_{t\in \N} \1_{(N_{t-1},N_t]}(n) f_{k_t,\unf}(n),
\]
then $f=f_{\str}+f_{\unf}$, $\innprod{f_{\str}}{f_{\unf}}_{\mathbf{N}}=0$, and $\lim_{k\to\infty} \| f_{\unf} -  f_{k,\unf} \|_{2, \mathbf{N}}=0$.

Now define $\mathbf{H}^*=(H^*_s)_{s\in\N}$ by
\[
H^*_s=H_{k_s,s},\qquad\forall s\in\N,
\]
and observe that $\mathbf{H}_k\prec \mathbf{H}^*\preceq \mathbf{K}^+$ holds for all $k\in\N$.
Finally, we let $\mathbf{H}^+=(H_s^+)_{s\in\N}$ and $\mathbf{H}^-=(H_s^-)_{s\in\N}$ be any sequences that grow faster than $\mathbf{H}_k$ for any $k\in\N$, but slower than $\mathbf{H}^*$. For example, we can take
\[
H^+_s= \bigg\lfloor\frac{H^*_s}{\log\log(K^-_s)}\bigg\rfloor, \quad
\text{and}\quad
H^-_s=\bigg\lfloor\frac{H^*_s}{\log(K^-_s)}\bigg\rfloor,
\]
as this gives $\mathbf{H}_k\prec \mathbf{H}^-\prec \mathbf{H}^+ \prec\mathbf{H}^*$ for all $k\in\N$ as desired.

By construction, the function $f_{\str}$ restricted to $\{N_{s-1}+1,\ldots,N_s\}$ is constant on intervals belonging to $\mathcal{I}_{k_s,s}$. Since intervals in $\mathcal{I}_{k_s,s}$ have length on the order of $H_s^*$, yet the ratio of $H_s^+$ to $H_s^*$ goes to zero as $s\to\infty$, we conclude that the set
$\Gamma_s=\{n\in [N_s-H_s^++1]: f_{\str}(n+m_1)=f_{\str}(n+m_2)~\forall m_1,m_2\in [H_s^+]
\},$
satisfies
\[
\lim_{s\to\infty} \frac{|\Gamma_s|}{N_s-H_s^++1}=1.
\]
This proves that condition \ref{itm U^1 structure theorem ii} of \cref{def: U1 good scale} is satisfied for the sequence $\mathbf{H}^+$. But if it holds for $\mathbf{H}^+$, then condition \ref{itm U^1 structure theorem ii} also holds for any sequence $\mathbf{H}=(H_s)_{s\in\N}$ with $1\prec \mathbf{H}\prec \mathbf{H}^+$.

Finally, notice that $f_{k,\unf}$ averages to $0$ over any interval in $\mathcal{I}_{k,s}$. Since intervals in $\mathcal{I}_{k,s}$ have length on the order of $H_{k,s}$, yet the ratio of $H_s^-$ to $H_{k,s}$ goes to $\infty$ as $s\to\infty$, we see that
\begin{align*}
\|f_{k,\unf}\|_{U^1(\mathbf{N},\mathbf{H}^-)}=
\lim_{s \to \infty} \frac{1}{N_s} \sum_{n=1}^{N_s}\Bigg|\frac{1}{H_s^-}\sum_{h=1}^{H_s^-} f_{k,\unf}(n+h)\Bigg|=0.
\end{align*}
Since $\lim_{k\to\infty} \| f_{\unf} -  f_{k,\unf} \|_{2, \mathbf{N}}=0$ and $\|f_{k,\unf}\|_{U^1(\mathbf{N},\mathbf{H}^-)}=0$ for all $k\in\N$, we conclude that $\|f_{\unf}\|_{U^1(\mathbf{N},\mathbf{H}^-)}=0$.
But if $\|f_{\unf}\|_{U^1(\mathbf{N},\mathbf{H}^-)}=0$ then 
by \eqref{eqn: U1 hierarchy} we get $\|f_{\unf}\|_{U^1(\mathbf{N},\mathbf{H})}=0$ for any sequence $\mathbf{H}=(H_s)_{s\in\N}$ with $\mathbf{H}^-\prec \mathbf{H}\preceq \mathbf{N}$. 
In particular, any $\mathbf{H}=(H_s)_{s\in\N}$ with $\mathbf{H}^-\prec \mathbf{H}\prec \mathbf{H}^+$ satisfies condition \ref{itm U^1 structure theorem iii} of \cref{def: U1 good scale}.
In conclusion, for any $\mathbf{H}=(H_s)_{s\in\N}$ with $\mathbf{H}^-\prec \mathbf{H}\prec \mathbf{H}^+$ the pair $(\mathbf{N},\mathbf{H})$ is a $U^1$~good scale for $f$.
\end{proof}

\cref{thm:U^1 structure theorem} shows the existence of $U^1$~good scales for a single function $f\colon\N\to[0,1]$.
Using a standard diagonalization argument, we can bootstrap this result to establish the existence of $U^1$-good scales that work simultaneously for every function in a given countable family; this is the content of the following corollary.

\begin{Corollary}
\label{cor:U^1 structure theorem}
Let $f_1,f_2,\ldots\colon\N\to [0,1]$ be a countable family of $[0,1]$-valued functions on $\N$, and suppose $\mathbf{N} = (N_s)_{s \in \N}$, $\mathbf{K}^+=(K^+_s)_{s\in\N}$, and $\mathbf{K}^-=(K^-_s)_{s\in\N}$ are sequences in $\N$ satisfying
\[
1\prec \mathbf{K}^-\prec \mathbf{K}^+\preceq \mathbf{N}.
\]
After replacing $(\mathbf{N},\mathbf{K}^+,\mathbf{K}^-)$ by a subsequence if necessary, there exist sequences $\mathbf{H}^+ = (H_s^+)_{s \in \N}$ and $\mathbf{H}^- = (H_s^-)_{s \in \N}$ with
\[
\mathbf{K}^-\preceq \mathbf{H}^-\prec \mathbf{H}^+\preceq \mathbf{K}^+
\]
such that for all $\mathbf{H}=(H_s)_{s\in\N}$ with $\mathbf{H}^-\prec \mathbf{H}\prec \mathbf{H}^+$ and all $k\in\N$ the pair $(\mathbf{N},\mathbf{H})$ is a $U^1$~good scale for $f_k$.
\end{Corollary}

\begin{proof}
We begin by finding for every $k,\ell=0,1,2,\ldots$ with $k\leq \ell$ an increasing sequence $t_\ell\colon \N\to\N$ and sequences $\mathbf{H}_{k,\ell}^-=(H_{k,t_\ell(s)}^-)_{s\in\N}$ and $\mathbf{H}_{k,\ell}^+=(H_{k,t_\ell(s)}^+)_{s\in\N}$ such that:
\begin{enumerate}[label=(\alph{enumi}),ref=(\alph{enumi}),leftmargin=*]
\item\label{itm_ctbl_U1_str_thm_a}
$(t_{\ell}(s))_{s\in\N}$ is a subsequence of $(t_{\ell-1}(s))_{s\in\N}$;
\item\label{itm_ctbl_U1_str_thm_b}
if $\mathbf{K}_\ell^+=(K_{t_\ell(s)}^+)_{s\in\N}$ and  $\mathbf{K}_\ell^-=(K_{t_\ell(s)}^-)_{s\in\N}$ then
\[
\mathbf{K}_\ell^-=\mathbf{H}_{0,\ell}^-\prec \mathbf{H}_{1,\ell}^-\prec \ldots\prec \mathbf{H}_{k,\ell}^-\prec \mathbf{H}_{k,\ell}^+\prec\ldots\prec \mathbf{H}_{1,\ell}^+\prec \mathbf{H}_{0,\ell}^+=\mathbf{K}_\ell^+;
\]
\item\label{itm_ctbl_U1_str_thm_c}
letting $\mathbf{N}_\ell=(N_{t_\ell(s)})_{s\in\N}$, then
for all $j\in\{1,\ldots,k\}$ and all 
$\mathbf{H}=(H_s)_{s\in\N}$ with $\mathbf{H}_{k,\ell}^-\prec \mathbf{H}\prec \mathbf{H}_{k,\ell}^+$ the pair $(\mathbf{N}_\ell,\mathbf{H})$ is a $U^1$~good scale for $f_j$.
\end{enumerate}
Set $t_0(s)=s$, $\mathbf{H}_{0,0}^-=\mathbf{K}_0^-=\mathbf{K}^-$, and $\mathbf{H}_{0,0}^+=\mathbf{K}_0^+=\mathbf{K}^+$.
Given $k\in\N$, if $t_0,\ldots,t_{k-1}$ and $\mathbf{H}_{0,0}^\pm,\mathbf{H}_{0,1}^\pm,\mathbf{H}_{1,1}^\pm,\ldots,\mathbf{H}_{k-1,k-1}^\pm$ have already been found, then we apply \cref{thm:U^1 structure theorem} to the triple $(\mathbf{N}_{k-1},\mathbf{H}_{k-1,k-1}^+,\mathbf{H}_{k-1,k-1}^-)$
to find a subsequence $(t_{k}(s))_{s\in\N}$ of $(t_{k-1}(s))_{s\in\N}$ and sequences $\mathbf{H}_{k,k}^+=(H_{k,t_k(s)}^+)_{s\in\N}$ and $\mathbf{H}_{k,k}^-=(H_{k,t_k(s)}^-)_{s\in\N}$ such that if $\mathbf{H}_{j,k}^\pm=(H_{j,t_k(s)}^\pm)_{s\in\N}$ for all $j=0,\ldots,k-1$ and $\mathbf{N}_k=(N_{t_k(s)})_{s\in\N}$ then
\[
\mathbf{H}_{k,k-1}^-\prec \mathbf{H}_{k,k}^-\prec \mathbf{H}_{k,k}^+\prec \mathbf{H}_{k,k-1}^+
\]
and all $\mathbf{H}=(H_s)_{s\in\N}$ with $\mathbf{H}_{k,k}^-\prec \mathbf{H}\prec \mathbf{H}_{k,k}^+$ have the property that the pair $(\mathbf{N}_k,\mathbf{H})$ is a $U^1$~good scale for $f_k$. By construction, any such $\mathbf{H}$ has the property that $(\mathbf{N}_k,\mathbf{H})$ is a $U^1$~good scale for the functions $f_1,\ldots,f_{k-1}$ too. Therefore \ref{itm_ctbl_U1_str_thm_a}, \ref{itm_ctbl_U1_str_thm_b}, and \ref{itm_ctbl_U1_str_thm_c} are satisfied.

Since $\mathbf{K}_k^-\prec \mathbf{H}_{k,k}^-\prec \mathbf{H}_{k,k}^+\prec \mathbf{K}_k^+$, there exists some $u_k\in\N$ such that
\[
H_{k,t_k(u_k)}^- \,-\, K_{t_k(u_k)}^- \geq k,
\quad
H_{k,t_k(u_k)}^+ \,-\, H_{k,t_k(u_k)}^- \geq k
\quad\text{and}\quad
K_{t_k(u_k)}^+ \,-\, H_{k,t_k(u_k)}^+ \geq k.
\]
Now define $\tilde{N}_s=N_{t_s(u_s)}$, $\tilde{K}_s^\pm=K_{t_s(u_s)}^\pm$, and $H_{s}^\pm=H_{s,t_s(u_s)}^\pm$. Then the sequences $\tilde{\mathbf{N}}=(\tilde{N}_s)_{s\in\N}$, $\tilde{\mathbf{K}}^\pm=(\tilde{K}_s^\pm)_{s\in\N}$, and $\mathbf{H}^\pm=(H_s^\pm)_{s\in\N}$ satisfy the properties that $(\tilde{\mathbf{N}},\tilde{\mathbf{K}}^+,\tilde{\mathbf{K}}^-)$ is a subsequence of $(\mathbf{N},\mathbf{K}^+,\mathbf{K}^-)$ and $\tilde{\mathbf{K}}^-\prec\mathbf{H}^-\prec\mathbf{H}^+\prec\tilde{\mathbf{K}}^+$.

Finally, if $\mathbf{H}=(H_s)_{s\in\N}$ is any sequence satisfying $\mathbf{H}^-\prec \mathbf{H}\prec \mathbf{H}^+$, then $(\tilde{\mathbf{N}},\mathbf{H})$  can be viewed as a subsequence of $(\tilde{\mathbf{N}},\mathbf{H}')$ for some sequence $\mathbf{H}'=(H_s')_{s\in\N}$ satisfying $\mathbf{H}_{k,k}^-\prec \mathbf{H}'\prec \mathbf{H}_{k,k}^+$. Since being a $U^1$~good scale is preserved under passing to subsequences, we see that 
$(\tilde{\mathbf{N}},\mathbf{H})$ is a $U^1$~good scale for $f_k$ as desired.
\end{proof}

The next result is a consequence of the preceding corollary. It shows that for any countable family of functions, one can find an entire range of scales along which their local averages exhibit especially regular behavior. This will be essential in what follows, as it guarantees that we can always choose a scale on which the relevant averages behave in an ``ergodic'' manner.

\begin{Corollary}
\label{cor:stable local averages for countable family}
Let $f_1,f_2,\cdots\colon\N\to\C$ be a countable collection of bounded functions on $\N$.
Suppose $\mathbf{N} = (N_s)_{s \in \N}$, $\mathbf{K}^+=(K^+_s)_{s\in\N}$, and $\mathbf{K}^-=(K^-_s)_{s\in\N}$ are sequences in $\N$ satisfying
\[
1\prec \mathbf{K}^-\prec \mathbf{K}^+\preceq \mathbf{N}.
\]
After replacing $(\mathbf{N},\mathbf{K}^+,\mathbf{K}^-)$ by a subsequence if necessary, there exist sequences $\mathbf{H}^+ = (H_s^+)_{s \in \N}$ and $\mathbf{H}^- = (H_s^-)_{s \in \N}$ with
$\mathbf{K}^-\preceq \mathbf{H}^-\prec \mathbf{H}^+\preceq \mathbf{K}^+$
and such that for any $\mathbf{H}=(H_s)_{s\in\N}$ with $\mathbf{H}^-\prec \mathbf{H}\prec \mathbf{H}^+$, and any $k,C\in\N$, if
\[
V_{s,k,C}=\Bigg\{n\in [N_s-H_s+1]: 
\Bigg|\frac{1}{CH_s}\sum_{h=1}^{CH_s} f_k(n+h)\,-\,\frac{1}{\lfloor H_s/C\rfloor }\sum_{h=1}^{\lfloor H_s/C\rfloor} f_k(n+h)\Bigg|<\frac{1}{C}\Bigg\}
\]
then
\[
\lim_{s\to\infty} \frac{|V_{s,k,C}|}{N_s-H_s+1}=1.
\]
\end{Corollary}

\begin{proof}
For any countable collection of bounded functions $f_1,f_2,\cdots\colon\N\to\C$ one can construct a countable collection of auxiliary functions $g_1,g_2,\cdots\colon\N\to[0,1]$ with the following approximation property:
for every $\epsilon>0$ and every function $f_i$ in the original family, there is a finite linear combination of the $g_j$’s that approximates $f_i$ uniformly up to an error of size $\epsilon$.
Consequently, if the conclusion of \cref{cor:stable local averages for countable family} is known to hold for the family $\{g_j\}_{j\in\N}$, then it must also hold for the original family $\{f_i\}_{i\in\N}$, simply because each $f_i$ can be uniformly approximated arbitrarily well by combinations of the $g_j$'s.
Therefore, we may assume without loss of generality from the outset that the functions under consideration take values in $[0,1]$.

Using \cref{cor:U^1 structure theorem}, we can now find sequences $\mathbf{H}^+ = (H_s^+)_{s \in \N}$ and $\mathbf{H}^- = (H_s^-)_{s \in \N}$ with
\[
\mathbf{K}^-\preceq \mathbf{H}^-\prec \mathbf{H}^+\preceq \mathbf{K}^+
\]
such that for all $\mathbf{H}=(H_s)_{s\in\N}$ with $\mathbf{H}^-\prec \mathbf{H}\prec \mathbf{H}^+$ and all $k\in\N$ the pair $(\mathbf{N},\mathbf{H})$ is a $U^1$~good scale for $f_k$. Hence, we can split $f_k=f_{k,\str}+f_{k,\unf}$ such that
properties \ref{itm U^1 structure theorem i}, \ref{itm U^1 structure theorem ii}, and \ref{itm U^1 structure theorem iii} from \cref{def: U1 good scale} are satisfied.
Note that by \ref{itm U^1 structure theorem iii} we have
\[
\lim_{s\to\infty}\frac{1}{N_s}\sum_{n=1}^{N_s}\Bigg|\frac{1}{CH_s}\sum_{h=1}^{CH_s} f_k(n+h) - \frac{1}{CH_s}\sum_{h=1}^{CH_s} f_{k,\str}(n+h)\Bigg|=0
\]
and 
\[
\lim_{s\to\infty}\frac{1}{N_s}\sum_{n=1}^{N_s}\Bigg|\frac{1}{\lfloor H_s/C\rfloor}\sum_{h=1}^{\lfloor H_s/C\rfloor} f_k(n+h) - \frac{1}{\lfloor H_s/C\rfloor}\sum_{h=1}^{\lfloor H_s/C\rfloor} f_{k,\str}(n+h)\Bigg|=0.
\]
So if we define
\[
V_{s,k,C}'=\Bigg\{n\in [N_s-H_s+1]: 
\Bigg|\frac{1}{CH_s}\sum_{h=1}^{CH_s} f_{k\str}(n+h)\,-\,\frac{1}{\lfloor H_s/C\rfloor }\sum_{h=1}^{\lfloor H_s/C\rfloor} f_{k,\str}(n+h)\Bigg|<\frac{1}{C}\Bigg\}
\]
then the difference between $V_{s,k,C}'$ and $V_{s,k,C}$ is a set of density zero.
The proof is completed by observing that \ref{itm U^1 structure theorem ii} implies
\[
\lim_{s\to\infty} \frac{|V_{s,k,C}'|}{N_s-H_s+1}=1.
\]
\end{proof}

Recall from \eqref{eqn: U1 hierarchy} that the intermediate scale seminorm $\|.\|_{U^1(\mathbf{N},\mathbf{H})}$ is bounded from below by the Host--Kra seminorm and from above by the (asymptotic) Gowers seminorm.
We conclude this subsection with a theorem that asserts that under natural regularity assumptions, if $\mathbf{H}$ grows sufficiently slowly then $\|.\|_{U^1(\mathbf{N},\mathbf{H})}$ coincides with the Host--Kra seminorm, whereas if $\mathbf{H}$ grows sufficiently fast then $\|.\|_{U^1(\mathbf{N},\mathbf{H})}$ coincides with the (asymptotic) Gowers seminorm.

\begin{Theorem}
\label{thm: existence of U1 G and HK scales}
Let $f\colon\N\to [0,1]$, and suppose $\mathbf{N}=(N_s)_{s\in\N}$ is a sequence of natural numbers satisfying $1\prec\mathbf{N}$.
\begin{enumerate}[label=(\roman{enumi}),ref=(\roman{enumi}),leftmargin=*]
\item 
\label{itm: existence of U1 HK scales}
(Existence of Host--Kra scale)
If $f$ admits correlations along $\mathbf{N}$ and $\lim_{s \to \infty} N_s/N_{s+1} = 0$, then there exists a sequence of natural numbers $\mathbf{H}_{\mathrm{HK}}=(H_{\mathrm{HK},s})_{s\in\N}$ with $1\prec\mathbf{H}_{\mathrm{HK}}\prec \mathbf{N}$ such that for all sequences $\mathbf{H}=(H_{s})_{s\in\N}$ with $1\prec \mathbf{H} \preceq \mathbf{H}_{\mathrm{HK}}$ the pair $(\mathbf{N},\mathbf{H})$ is a $U^1$~good scale for~$f$ with the property that the corresponding decomposition $f = f_{\str} + f_{\unf}$ satisfies
\begin{equation*}
    \|f_{\str} - f_{\inv}\|_{2,\mathbf{N}} = \|f_{\unf} - f_{\erg}\|_{2,\mathbf{N}} = 0
\end{equation*}
and
\begin{equation*}
    \|f\|_{U^1(\mathbf{N},\mathbf{H})} = \|f\|_{U^1(\mathbf{N})}.
\end{equation*}
\item 
\label{itm: existence of U1 G scales}
(Existence of Gowers scale)
If the mean of $f$ exits, i.e., the limit
\[
\lim_{N \to \infty} \frac{1}{N} \sum_{n=1}^N f(n)=\delta \quad\text{exists,}
\]
then there exists 
a sequence of natural numbers $\mathbf{H}_{\mathrm{G}}=(H_{\mathrm{G},s})_{s\in\N}$ with $1\prec\mathbf{H}_{\mathrm{G}}\prec \mathbf{N}$ such that for all sequences $\mathbf{H}=(H_{s})_{s\in\N}$ with $\mathbf{H}_{\mathrm{G}}\preceq\mathbf{H} \preceq \mathbf{N}$ we have
\[
\|f-\delta\|_{U^1(\mathbf{N},\mathbf{H})}=0.
\]
In particular, for any $\mathbf{H}=(H_{s})_{s\in\N}$ with $\mathbf{H}_{\mathrm{G}}\preceq\mathbf{H} \preceq \mathbf{N}$ the the pair $(\mathbf{N},\mathbf{H})$ is a $U^1$~good scale for~$f$, and the associated splitting $f=f_{\str}+f_{\unf}$ is given by
\[
f_{\str}=\delta,\qquad\text{and}\qquad f_{\unf}=f-\delta.
\]
\end{enumerate}
\end{Theorem}

\begin{proof}
    We first establish the existence of a Host--Kra scale.
    Assume $f$ admits correlations along $\mathbf{N}$ and $\lim_{s \to \infty} N_s/N_{s+1} = 0$.
    Let $f = f_{\inv} + f_{\erg}$ be the decomposition provided by \cref{thm: discrete ergodic theorem}.
    By property \ref{itm disc ET ii} in \cref{thm: discrete ergodic theorem}, we have
    \begin{equation*}
        \lim_{s \to \infty} \frac{1}{N_s} \sum_{n=1}^{N_s} |f_{\inv}(n+1)-f_{\inv}(n)| = 0.
    \end{equation*}
    Hence,
    \begin{equation*}
        \lim_{s \to \infty} \frac{1}{N_s} \sum_{n=1}^{N_s} |f_{\inv}(n+h)-f_{\inv}(n)| = 0
    \end{equation*}
    for every $h \in \N$.
    Pick $s_1 \leq s_2 \leq \ldots$ such that if $s \geq s_H$, then
    \begin{equation*}
        \frac{1}{N_s} \sum_{n=1}^{N_s} \max_{h \in [H]} \left| f_{\inv}(n+h)-f_{\inv}(n) \right| < 2^{-H}.
    \end{equation*}
    Put $H^+_s = \max\{H \in [\sqrt{N_s}] : s_H \leq s\}$.
    Then we have $1 \prec \mathbf{H}^+ \prec \mathbf{N}$ and
    \begin{equation} \label{eq: H^+ invariance}
        \lim_{s \to \infty} \frac{1}{N_s} \sum_{n=1}^{N_s} \max_{h \in [H^+_s]} \left| f_{\inv}(n+h)-f_{\inv}(n) \right| = 0.
    \end{equation}
    Define
    \begin{equation*}
        f_{\str}(n) = \frac{1}{H^+_s} \sum_{h=1}^{H^+_s} f_{\inv} \left( N_{s-1} + \left\lfloor \frac{n-(N_{s-1}+1)}{H^+_s} \right\rfloor H^+_s + h \right)
    \end{equation*}
    for $n \in \{N_{s-1}+1,\ldots,N_s\}$.
    Then
    \begin{multline*}
        \frac{1}{N_s} \sum_{n=1}^{N_s} |f_{\inv}(n) - f_{\str}(n)| \\
         = \frac{1}{N_s} \sum_{n=N_{s-1}+1}^{N_s} \left| f_{\inv}(n) - \frac{1}{H^+_s} \sum_{h=1}^{H^+_s} f_{\inv} \left( \underbrace{N_{s-1} + \left\lfloor \frac{n-(N_{s-1}+1)}{H^+_s} \right\rfloor H^+_s + h}_{\in \{n-H^+_s,\ldots,n+H^+_s\}} \right)\right|+ \oh(1) \\
         \leq \frac{1}{N_s} \sum_{n=N_{s-1}+1}^{N_s} \max_{-H^+_s \leq h \leq H^+_s} \left| f_{\inv}(n) - f_{\inv}(n+h) \right| + \oh(1),
    \end{multline*}
    so $\|f_{\inv} - f_{\str}\|_{2,\mathbf{N}} = 0$ by \eqref{eq: H^+ invariance}.
    Let $f_{\unf} = f - f_{\str}$.
    Then we also have
    \begin{equation*}
        \|f_{\erg} - f_{\unf}\|_{2,\mathbf{N}} = \|(f-f_{\inv}) + (f-f_{\str})\|_{2,\mathbf{N}} = \|f_{\inv} - f_{\str}\|_{2,\mathbf{N}} = 0.
    \end{equation*}

    Let $\mathbf{H}_{HK} = (H_{HK,s})_{s \in \N}$ be a sequence satisfying $1 \prec \mathbf{H}_{HK} \prec \mathbf{H}^+$.
    For example, we may take $H_{HK,s} = \lfloor \sqrt{H^+_s} \rfloor$.
    Suppose $\mathbf{H} = (H_s)_{s \in \N}$ and $1 \prec \mathbf{H} \preceq \mathbf{H}_{HK}$.
    We want to show that $f = f_{\str} + f_{\unf}$ as defined above satisfies conditions \ref{itm U^1 structure theorem i}, \ref{itm U^1 structure theorem ii}, and \ref{itm U^1 structure theorem iii} from \cref{def: U1 good scale}.

    Property \ref{itm U^1 structure theorem i} follows from \cref{thm: discrete ergodic theorem}\ref{itm disc ET i} and the observation from above that $\|f_{\inv} - f_{\str}\|_{2,\mathbf{N}} = \|f_{\erg} - f_{\unf}\| = 0$.

    By construction, $f_{\str}$ is constant on the intervals $\{N_{s-1}+mH^+_s+1,\ldots,N_{s-1}+(m+1)H^+_s\}$ of length $H^+_s$.
    Therefore,
    \begin{multline*}
        \Gamma_s = \left\{ n\in [N_s-H_s+1]: f_{\str}(n+m_1)=f_{\str}(n+m_2)~\forall m_1,m_2\in [H_s] \right\} \\
        \supseteq \bigcup_{m=0}^{(N_s-H_s-N_{s-1})/H^+_s} (N_{s-1} + mH^+_s + [H^+_s-H_s]),
    \end{multline*}
    whence
    \begin{equation*}
        |\Gamma_s| \geq \left\lfloor \frac{N_s-H_s-N_{s-1}}{H^+_s} \right\rfloor (H_s^+ - H_s) = \left( 1 - \oh(1) \right)N_s.
    \end{equation*}
    
    Finally, for any $\mathbf{H}' = (H'_s)_{s \in \N}$ with $1 \prec \mathbf{H}' \prec \mathbf{H}$, we have
    \begin{equation*}
        \|f_{\unf}\|_{U^1(\mathbf{N},\mathbf{H'})} = \|f_{\erg}\|_{U^1(\mathbf{N},\mathbf{H}')} \leq \|f_{\erg}\|_{U^1(\mathbf{N})} = 0
    \end{equation*}
    by \eqref{eqn: U1 hierarchy} and \cref{thm: discrete ergodic theorem}\ref{itm disc ET iii}.

    \bigskip

    Now we turn to producing a Gowers scale.
    We claim that for every $k \in \N$,
    \begin{equation} \label{eq:uniform convergence to mean}
        \lim_{s \to \infty} \max_{n \in [N_s]} \left| \frac{k}{N_s} \sum_{h=1}^{N_s/k} f(n+h) - \delta \right| = 0.
    \end{equation}
    Let $\epsilon > 0$.
    Fix $1 \leq n \leq N_s$, and let $\alpha = \frac{n}{N_s}$.
    If $\alpha < \epsilon$, then
    \begin{equation*}
        \frac{k}{N_s} \sum_{h=1}^{N_s/k} f(n+h) = \frac{k}{N_s} \sum_{h=1}^{N_s/k} f(h) + O \left( \epsilon \right) = \delta + o(1) + O(\epsilon).
    \end{equation*}
    On the other hand, if $\alpha \geq \epsilon$, then
    \begin{equation*}
        \frac{k}{N_s} \sum_{h=1}^{N_s/k} f(n+h) = \frac{k}{N_s} \left( \underbrace{\sum_{m=1}^{(\alpha+1/k)N_s} f(m)}_{(\alpha+1/k)N_s\delta+o(N_s)} - \underbrace{\sum_{m=1}^{\alpha N_s} f(m)}_{\alpha N_s \delta+o(N_s)} \right) = \delta + o(1).
    \end{equation*}
    Letting $\epsilon \to 0$ proves the claim.

    By \eqref{eq:uniform convergence to mean}, let $s_1 \leq s_2 \leq \ldots$ such that if $s \geq s_K$, then
    \begin{equation*}
        \max_{n \in [N_s]} \max_{k \in [K]} \left| \frac{k}{N_s} \sum_{h=1}^{N_s/k} f(n+h) - \delta \right| < 2^{-K}.
    \end{equation*}
    Let $K_s = \max\{K \in [\sqrt{N_s}] : s_K \leq s\}$.
    Then $1 \prec \mathbf{K} \prec \mathbf{N}$ and
    \begin{equation*}
        \lim_{s \to \infty} \max_{n \in [N_s]} \max_{k \in [K_s]} \left| \frac{k}{N_s} \sum_{h=1}^{N_s/k} f(n+h) - \delta \right| = 0.
    \end{equation*}
    Let $H^-_s = \lfloor \frac{N_s}{K_s} \rfloor$, and note that
    \begin{equation} \label{eq: H^- ergodicity}
        \lim_{s \to \infty} \max_{n \in [N_s]} \left| \frac{1}{H^-_s} \sum_{h=1}^{H^-_s} f(n+h) - \delta \right| = 0.
    \end{equation}

    We now let $\mathbf{H}_G = (H_{G,s})_{s \in \N}$ be an arbitrary sequence with $\mathbf{H}^- \prec \mathbf{H}_G \prec \mathbf{N}$.
    For example, we can put $H_{G,s} = \lfloor \sqrt{H^-_s N_s} \rfloor$.
    We want to show that the decomposition $f = \delta + (f-\delta)$ satisfies conditions \ref{itm U^1 structure theorem i}, \ref{itm U^1 structure theorem ii}, and \ref{itm U^1 structure theorem iii} in \cref{def: U1 good scale} for any sequence $\mathbf{H} = (H_s)_{s \in \N}$ with $\mathbf{H}_G \preceq \mathbf{H} \preceq \mathbf{N}$.
    Properties \ref{itm U^1 structure theorem i} and \ref{itm U^1 structure theorem ii} are trivially satisfied.
    Moreover, \eqref{eq: H^- ergodicity} shows that $\|f-\delta\|_{U^1(\mathbf{N},\mathbf{H}^-)} = 0$, so property \ref{itm U^1 structure theorem iii} is also satisfied by taking $\mathbf{H}' = \mathbf{H}^-$.
\end{proof}

\subsection{The $U^2(\mathbf{N},\mathbf{H})$ seminorm}
\label{sec: intermediate U2 seminorm}

In this section, we extend the intermediate-scale uniformity seminorms from order $1$ (introduced in \cref{sec: intermediate U1 seminorm}) to order $2$, and establish the corresponding structure theorem. While the definition of the  $U^2(\mathbf{N},\mathbf{H})$ seminorm parallels that of the $U^1(\mathbf{N},\mathbf{H})$ seminorm, the ensuing structure theorem, in particular the description of the structured component, is considerably more intricate than its order 1 counterpart (\cref{thm:U^1 structure theorem}). Nonetheless, the theorem provides a useful framework to study the behavior of sumsets $A+B$ for arbitrary subsets $A,B\subset\N$.

If $\mathbf{N} = (N_s)_{s \in \N}$ and $\mathbf{H}=(H_s)_{s\in\N}$ are sequences in $\N$ with
\[
1\prec \mathbf{H}\preceq \mathbf{N},
\]
we define the $U^2(\mathbf{N},\mathbf{H})$ seminorm of a bounded function $f\colon\N\to\C$ as
\begin{align}
\label{eqn_intermediate_scale_U_2_norm_def_1}
\|f\|_{U^2(\mathbf{N},\mathbf{H})}&=\lim_{s \to \infty} \frac{1}{N_s-H_s+1} \sum_{n=1}^{N_s-H_s+1}  \|f\|_{U^2(\{n,n+1,\ldots,n+H_{s}-1\})},
\end{align}
whenever this limit exists, where $\|.\|_{U^2(\{n,n+1,\ldots,n+H_{s}-1\})}$ is the Gowers norm on the interval $\{n,n+1,\ldots,n+H_{s}-1\}$.
If either $ \mathbf{H}\prec \mathbf{N}$ or $ \mathbf{H}= \mathbf{N}$ then \eqref{eqn_intermediate_scale_U_2_norm_def_1} simplifies and we have
\begin{equation}
\label{eqn_intermediate_scale_U_2_norm_def_2} 
\|f\|_{U^2(\mathbf{N},\mathbf{H})}=
\begin{cases}
\lim_{s \to \infty} \frac{1}{N_s} \sum_{n=1}^{N_s}  \|f\|_{U^2(\{n,n+1,\ldots,n+H_{s}-1\})},&\text{if}~\mathbf{H}\prec \mathbf{N},
\\
\lim_{s \to \infty} \|f\|_{U^2([N_s])},&\text{if}~ \mathbf{H}= \mathbf{N}.
\end{cases}
\end{equation}

We also define the $u^2(\mathbf{N},\mathbf{H})$ seminorm as
\begin{align}
\label{eqn_intermediate_scale_u_2_norm_def_2}
\|f\|_{u^2(\mathbf{N},\mathbf{H})}&=\lim_{s \to \infty} \frac{1}{N_s-H_s+1} \sum_{n=1}^{N_s-H_s+1}  \|f\|_{u^2(\{n,n+1,\ldots,n+H_{s}-1\})},
\end{align}
whenever this limit exists, where $\|.\|_{u^2(\{n,n+1,\ldots,n+H_{s}-1\})}$ is as given in \eqref{eqn_def_fourier_U2_norm_interval}. It follows from \eqref{eqn_U2-Gowers-Fourier_relation} and Jensen's inequality that
\begin{equation}
\label{eqn_U2-uniformity_iff_pseudorandom}    
C^{-1}\|f\|_{u^2(\mathbf{N},\mathbf{H})}\leq \|f\|_{U^2(\mathbf{N},\mathbf{H})}\leq C \|f\|_{u^2(\mathbf{N},\mathbf{H})}^{\frac{1}{2}},
\end{equation}
where $C$ is an absolute constant. In particular, it follows from \eqref{eqn_U2-uniformity_iff_pseudorandom} that
\begin{equation} \label{eq: U^2(N,H) uniformity}
\|f\|_{U^2(\mathbf{N},\mathbf{H})}=0 ~\iff~
\lim_{s \to \infty} \frac{1}{N_s-H_s+1} \sum_{n=1}^{N_s-H_s+1}
\sup_{\alpha\in\R} \Bigg|\frac{1}{H_s}\sum_{h=0}^{H_s-1} f(n+h)e(h\alpha)\Bigg|=0.
\end{equation}
Also, it follows from the triangle inequality that if
$\mathbf{H}_1=(H_{1,s})_{s\in\N}$ and $\mathbf{H}_2=(H_{_1s})_{s\in\N}$ are sequences of natural numbers with $1\prec\mathbf{H}_1\prec\mathbf{H}_2\preceq\mathbf{N}$ and such that $\|f\|_{u^2(\mathbf{H}_1,\mathbf{N})}$ and $\|f\|_{u^2(\mathbf{H}_2,\mathbf{N})}$ are well defined then
\begin{equation}
\label{eqn: u2 hierarchy}
\|f\|_{u^2(\mathbf{N},\mathbf{H}_1)}
\geq \|f\|_{u^2(\mathbf{N},\mathbf{H}_2)}.
\end{equation} 
An analogous inequality for the $U^2(\mathbf{N},\mathbf{H})$ seminorm instead of the $u^2(\mathbf{N},\mathbf{H})$ will not be needed in the forthcoming, but if true then it seems more difficult to derive.

We now introduce the order-2 analogue of a $U^1$-good scale from \cref{def: U1 good scale}. 

\begin{Definition}[$U^2$~good scale]
\label{def: U2 good scale}
Let $f\colon\N\to [0,1]$, and suppose $\mathbf{N} = (N_s)_{s \in \N}$ and $\mathbf{H}=(H_s)_{s\in\N}$ are sequences in $\N$ with
$1\prec \mathbf{H}\preceq \mathbf{N}$.
We say that $(\mathbf{N},\mathbf{H})$ is a \define{$U^2$~good scale} for $f$ if there exists a decomposition $f = f_{\str} + f_{\unf}$ such that:
\begin{enumerate}[label=(\roman{enumi}),ref=(\roman{enumi}),leftmargin=*]
\item\label{itm U^2 structure theorem i}
(Nonnegativity) 
$f_{\str}$ takes values in $[0,1]$ and $\E_{n \in \mathbf{N}} f_{\unf} = 0$.
\item\label{itm U^2 structure theorem ii}
(Structure)
For every $\epsilon>0$ there exist $d\in\N$, $c_1,\ldots,c_d\colon \N\to \C$ with $\|c_i\|_\infty\leq 1$, and $\alpha_1,\ldots,\alpha_d\colon \N\to\R$ such that if $\mathcal{P}(n)=\sum_{i=1}^d c_i(n) e(n \alpha_i(n))$ then
    \[
    \| \mathcal{P} -  f_{\str} \|_{2, \mathbf{N}}\leq \epsilon.
    \]
Moreover, given any $C\geq 1$, if we consider
\begin{align*}
\Gamma_s^{(1)}&=\{n\in [N_s-H_s+1]: \alpha_i(n+m_1)=\alpha_i(n+m_2)~\forall i\in [d],~\forall m_1,m_2\in [H_s]\},
\\
\Gamma_s^{(2)}&=\{n\in [N_s-H_s+1]: c_i(n+m_1)=c_i(n+m_2)~\forall i\in [d],~\forall m_1,m_2\in [H_s]\},
\\
W_{s,C}&=\Big\{n\in [N_s-H_s+1]: \text{for all}~(w_1, \ldots, w_d) \in [-C,C]^d\cap \Z^d,~\text{either}
\\
\|w_1&\alpha_1(n) +\ldots+w_d\alpha_d(n)\|_{\T}\leq \tfrac{1}{C H_s}
~\text{or}~\|w_1\alpha_1(n)+\ldots+w_d\alpha_d(n)\|_{\T}\geq \tfrac{C}{H_s}\Big\},
\end{align*}
then
\[
\lim_{s\to\infty} \frac{|\Gamma_s^{(1)}|}{N_s-H_s+1} = 
\lim_{s\to\infty} \frac{|\Gamma_s^{(2)}|}{N_s-H_s+1} =
\lim_{s\to\infty} \frac{|W_{s,C}|}{N_s-H_s+1}=1.
\]
\item\label{itm U^2 structure theorem iii}
(Uniformity)
There is $\mathbf{H}'=(H_s')_{s\in\N}$ with $1\prec \mathbf{H}'\prec \mathbf{H}$ and $\|f_{\unf}\|_{U^2(\mathbf{N},\mathbf{H}')}=0.$
\end{enumerate}
\end{Definition}

As was the case with $U^1$ good scales, the properties \ref{itm U^2 structure theorem i}, \ref{itm U^2 structure theorem ii}, and \ref{itm U^2 structure theorem iii} in \cref{def: U2 good scale} imply that
\begin{enumerate}[label=(\roman{enumi}),ref=(\roman{enumi}),leftmargin=*]
    \setcounter{enumi}{3}
    \item
    (Orthogonality)
    $\innprod{f_{\str}}{f_{\unf}}_{\mathbf{N}} = 0$
\end{enumerate}
also holds.
To see this, we first observe that is suffices to prove $\innprod{\mathcal{P}}{f_{\unf}}_{2,\mathbf{N}} = 0$ for the functions $\mathcal{P}$ appearing in \ref{itm U^2 structure theorem ii}, which then further reduces to establishing $\E_{n \in \mathbf{N}} e(n\alpha(n)) f_{\unf}(n) = 0$ whenever $\alpha : \N \to \T$ is a function satisfying that for every $s \in \N$ and all but $\oh_{s \to \infty}(N_s - H_s + 1)$ many $n \in [N_s - H_s + 1]$, one has
\begin{equation*}
    \alpha(n+m_1) = \alpha(n+m_2) \qquad (\forall m_1, m_2 \in [H_s]).
\end{equation*}
Applying \ref{itm U^2 structure theorem iii} and \eqref{eq: U^2(N,H) uniformity} then establishes the desired orthogonality.

Next we state the structure theorem for the $U^2(\mathbf{N},\mathbf{H})$ seminorm. The statement of the theorem is essentially the same as in the order-1 case (\cref{thm:U^1 structure theorem}), with the only change being that the notion of a $U^1$ good scale is replaced by that of a $U^2$ good scale. The real difference lies in what it means for a scale to be $U^2$~good. The theorem ensures the existence of such good scales, but this should really be understood as guaranteeing a decomposition $f = f_{\str} + f_{\unf}$ with the properties specified in \cref{def: U2 good scale}.

\begin{Theorem}[Structure theorem for $U^2(\mathbf{N},\mathbf{H})$-seminorm]
\label{thm:U^2 structure theorem}
Let $f\colon\N\to [0,1]$, and suppose $\mathbf{N} = (N_s)_{s \in \N}$, $\mathbf{K}^+=(K^+_s)_{s\in\N}$, and $\mathbf{K}^-=(K^-_s)_{s\in\N}$ are sequences in $\N$ satisfying
\[
1\prec \mathbf{K}^-\prec \mathbf{K}^+\preceq \mathbf{N}.
\]
After replacing $(\mathbf{N},\mathbf{K}^+,\mathbf{K}^-)$ by a subsequence if necessary, there exist sequences $\mathbf{H}^+ = (H_s^+)_{s \in \N}$ and $\mathbf{H}^- = (H_s^-)_{s \in \N}$ with
\[
\mathbf{K}^-\preceq \mathbf{H}^-\prec \mathbf{H}^+\preceq \mathbf{K}^+
\]
such that for any $\mathbf{H}=(H_s)_{s\in\N}$ with $\mathbf{H}^-\prec \mathbf{H}\prec \mathbf{H}^+$ the pair $(\mathbf{N},\mathbf{H})$ is a $U^2$~good scale for $f$.
\end{Theorem}

For the proof of \cref{thm:U^2 structure theorem}, we require a short technical lemma.
We endow the unit square $[0,1]^2$ with the standard lexigographic order:
given two points $(\lambda,\eta),(\lambda',\eta')\in [0,1]^2$, we write $(\lambda,\eta)\prec (\lambda',\eta')$ if either $\lambda<\lambda'$ or if $\lambda=\lambda'$ and $\eta<\eta'$. This yields a total ordering of $[0,1]^2$. 

\begin{Lemma}
\label{lem_decreasing_seq_on_unit_square}
Suppose $a\colon [0,1]^2\to [0,1]$ satisfies 
\begin{equation} \label{eq: decreasing on square}
    (\lambda,\eta)\prec (\lambda',\eta')~\implies~ a(\lambda,\eta)\geq a(\lambda',\eta').
\end{equation}
Then there exists $\lambda\in (0,1)$ such that
$a(\lambda,\eta)= a(\lambda,\eta')$ for all $\eta,\eta'\in[0,1]$.
\end{Lemma}

\begin{proof}
    Given $0 \leq \lambda_1 < \lambda_2 < \ldots \leq \lambda_k \leq 1$, we use \eqref{eq: decreasing on square} and telescoping to obtain
    \begin{align*}
        \sum_{j=1}^k (a(\lambda_j,0) - a(\lambda_j,1)) & = \sum_{j=1}^{k-1} (a(\lambda_j,0) - a(\lambda_j,1)) + a(\lambda_k,0) - a(\lambda_k,1) \\
         &\leq \sum_{j=1}^{k-1} (a(\lambda_j,0) - a(\lambda_{j+1},0)) + a(\lambda_k,0) - a(\lambda_k,1) \\
         & = a(\lambda_1,0) - a(\lambda_k,1) \leq a(0,0) - a(1,1).
    \end{align*}
    Therefore,
    \begin{equation*}
        \sum_{\lambda \in [0,1]} \left( a(\lambda,0)-a(\lambda,1) \right) = \sup_{F \subseteq [0,1]~\text{finite}} \sum_{\lambda \in F} (a(\lambda,0) - a(\lambda,1)) \leq a(0,0) - a(1,1) < \infty.
    \end{equation*}
    An uncountable sum of non-negative real numbers can only be finite if all but countably many summands are zero. 
    Hence there exists some $\lambda\in (0,1)$ such that $a(\lambda,0)-a(\lambda,1)=0$, and the claim follows.
\end{proof}

\begin{proof}[Proof of \cref{thm:U^2 structure theorem}]
Let $(\epsilon_k)_{k\in\N}$ be any sequence of positive real numbers such that $\lim_{k\to\infty} \epsilon_k =0$ and $d^\star(\epsilon_{k}) \epsilon_{k+1}\leq \epsilon_{k}$ for all $k$, where $d^\star$ is as in  \cref{thm_arithmetic_regularity_lemma}.
After passing to a subsequence of $(\mathbf{N},\mathbf{K}^+,\mathbf{K}^-)$ if needed, we may assume without loss of generality that $\lim_{s\to\infty} N_s / N_{s+1} = 0$.
Given $\lambda,\eta\in[0,1]$, define $\mathbf{H}_{\lambda,\eta}=(H_{\lambda,\eta,s})_{s\in\N}$ as 
\[
H_{\lambda,\eta,s}= \big\lfloor(K_{s}^-)^{1-\lambda} \cdot (K_{s}^+)^{\lambda}\cdot (\log(K_s^+))^\eta\big\rfloor. 
\]
Note that for any $(\lambda,\eta),(\lambda',\eta')\in [0,1]^2$,
\begin{equation}
\label{eqn_lexico_ordering_1}
(\lambda,\eta)\prec (\lambda',\eta')~\implies~
\mathbf{H}_{\lambda,\eta}\prec \mathbf{H}_{\lambda',\eta'}.
\end{equation}

Let $\mathcal{I}_{\lambda,\eta,s}$ be an arbitrary partition of $\{N_{s-1}+1,\ldots,N_s\}$ into intervals of length between $H_{\lambda,\eta,s}/2$ and $2H_{\lambda,\eta,s}$. For a fixed pair $(\lambda,\eta)\in[0,1]^2$, this is possible for all but finitely many $s\in\N$ because $\lim_{s\to\infty} N_s / N_{s+1} = 0$ and $\mathbf{H}_{\lambda,\eta}\prec \mathbf{N}$; this finite set of exceptions for which $\mathcal{I}_{\lambda,\eta,s}$ is not well defined will not affect the forthcoming argument.
Let $d_k^\star=d^\star(\epsilon_k)$. Using \cref{thm_arithmetic_regularity_lemma}, for every $\lambda,\eta\in[0,1]$, $k\in\N$, all sufficiently large $s\in\N$, and all $I\in \mathcal{I}_{\lambda,\eta,s}$,
we can find $g_{I,k,\mathrm{str}}\colon I\to [0,1]$, $g_{I,k,\mathrm{psd}}\colon I\to [-1,1]$, $c_{I,k,1},\ldots,c_{I,k,d_k^\star}\in\C$ with $|c_{I,k,i}|\leq 1$, and $\alpha_{I,k,1},\ldots,\alpha_{I,k,d_k^\star}\in\R$ such that if
\[
\mathcal{P}_{I,k}^\star=\sum_{i=1}^{d_k^\star} c_{I,k,i}\, e(n \alpha_{I,k,i})
\]
then we have
\begin{align}
\label{eqn_inf_arith_reg_lem_pr_1}
\frac{1}{|I|}\sum_{n\in I}\bigg|g_{I,k,\mathrm{str}}(n)-\mathcal{P}_{I,k}^\star(n)\bigg|^2\leq \epsilon_k^2,
\\
\label{eqn_inf_arith_reg_lem_pr_2}
\|g_{I,k,\mathrm{psd}}\|_{u^2(I)}\leq \epsilon_k.
\end{align}
Note that the choice of $d_k^\star$ depends only on $\epsilon_k$.
By the conclusion of \cref{thm_arithmetic_regularity_lemma}, we also have
\begin{equation}
\label{eqn_inf_arith_reg_lem_pr_3}
\bigg|\frac{1}{|I|}\sum_{n\in I} g_{I,k,\mathrm{str}}(n)g_{I,k,\mathrm{psd}}(n)\bigg|=0,
\end{equation}
and combining \eqref{eqn_inf_arith_reg_lem_pr_1} and \eqref{eqn_inf_arith_reg_lem_pr_2} with the assumption $d_k^\star \epsilon_{\ell}\leq d_k^\star \epsilon_{k+1}\leq \epsilon_{k}$, we get for all $\ell> k$ that
\begin{equation}
\label{eqn_inf_arith_reg_lem_pr_4}
\bigg|\frac{1}{|I|}\sum_{n\in I} g_{I,k,\mathrm{str}}(n)g_{I,\ell,\mathrm{psd}}(n)\bigg|\leq 2\epsilon_{k}.
\end{equation}
We now define
\begin{align*}
f_{\lambda,\eta,k,\mathrm{str}}(n)&=\sum_{s\in\N}~\sum_{I\in\mathcal{I}_{\lambda,\eta,s}} \1_{I}(n) g_{I,k,\mathrm{str}}(n),
\\
f_{\lambda,\eta,k,\mathrm{psd}}(n)&=\sum_{s\in\N}~\sum_{I\in\mathcal{I}_{\lambda,\eta,s}} \1_{I}(n) g_{I,k,\mathrm{psd}}(n),
\\
c_{\lambda,\eta,k,i}(n)&= \sum_{s\in\N}~\sum_{I\in\mathcal{I}_{\lambda,\eta,s}} \1_{I}(n) c_{I,k,i},
\\
\alpha_{\lambda,\eta,k,i}(n)&= \sum_{s\in\N}~\sum_{I\in\mathcal{I}_{\lambda,\eta,s}} \1_{I}(n) \alpha_{I,k,i},
\\
\mathcal{P}_{\lambda,\eta,k}^\star(n)&= \sum_{s\in\N}~\sum_{I\in\mathcal{I}_{\lambda,\eta,s}} \1_{I}(n) \mathcal{P}_{I,k}^\star(n)=\sum_{i=1}^{d_k^\star} c_{\lambda,\eta,k,i}\, e(n \alpha_{\lambda,\eta,k,i}).
\end{align*}
From \eqref{eqn_inf_arith_reg_lem_pr_1} we conclude
\begin{equation}
\label{eqn_inf_arith_reg_lem_pr_5}
\|f_{\lambda,\eta,k,\str}-\mathcal{P}_{\lambda,\eta,k}^\star\|_{2, \mathbf{N}}\leq \epsilon_k,
\end{equation}
from \eqref{eqn_inf_arith_reg_lem_pr_3} we get
\begin{equation}
\label{eqn_inf_arith_reg_lem_pr_8}
\innprod{f_{\lambda,\eta,k,\mathrm{str}}}{f_{\lambda,\eta,k,\mathrm{psd}}}_{\mathbf{N}}=0,\qquad\forall k\in\N,
\end{equation}
and from
\eqref{eqn_inf_arith_reg_lem_pr_4} we see that
\begin{equation}
\label{eqn_inf_arith_reg_lem_pr_9}
\innprod{f_{\lambda,\eta,k,\mathrm{str}}}{f_{\lambda,\eta,\ell,\mathrm{psd}}}_{\mathbf{N}}\leq 2\epsilon_k,\qquad\forall k<\ell\in\N.
\end{equation}
Furthermore, from \eqref{eqn: u2 hierarchy} and \eqref{eqn_inf_arith_reg_lem_pr_2} it follows that
\begin{equation}
\label{eqn_inf_arith_reg_lem_pr_7}
\|f_{\lambda,\eta,\ell,\mathrm{psd}}\|_{u^2(\mathbf{N},\mathbf{H})}\leq \epsilon_\ell,\quad \text{whenever}~\mathbf{H}_{\lambda,\eta}\prec \mathbf{H}\prec \mathbf{N}    
\end{equation}
which implies that
\begin{equation*}
%\label{eqn_inf_arith_reg_lem_pr_7}
\|f_{\lambda,\eta,\ell,\mathrm{psd}}\|_{u^2(\mathbf{N},\mathbf{H}_{\lambda',\eta'})}\leq \epsilon_\ell,\quad\text{whenever}~(\lambda,\eta)\prec (\lambda',\eta').
\end{equation*}
This gives
\begin{equation}
\label{eqn_inf_arith_reg_lem_pr_10}
\innprod{f_{\lambda',\eta',k,\mathrm{str}}}{f_{\lambda,\eta,\ell,\mathrm{psd}}}_{\mathbf{N}}\leq \epsilon_k+\epsilon_\ell d_k^\star\leq 2\epsilon_k,\quad\text{whenever}~(\lambda,\eta)\prec (\lambda',\eta').
\end{equation}

In light of \eqref{eqn_inf_arith_reg_lem_pr_5} and \eqref{eqn_inf_arith_reg_lem_pr_8}, we can invoke \cref{lem_decomposition_in_the_limit} to find $f_{\lambda,\eta,\mathrm{str}}\colon\N\to [0,1]$ and $f_{\lambda,\eta,\mathrm{unf}}\colon\N\to [-1,1]$ such that
$f=f_{\lambda,\eta,\mathrm{str}}+f_{\lambda,\eta,\mathrm{unf}}$, $\innprod{f_{\lambda,\eta,\mathrm{str}}}{f_{\lambda,\eta,\mathrm{unf}}}_{\mathbf{N}}=0$, $\lim_{k\to\infty} \| f_{\lambda,\eta,\mathrm{str}} -  f_{\lambda,\eta,k,\mathrm{str}} \|_{2, \mathbf{N}}=0$, and $\lim_{k\to\infty} \| f_{\lambda,\eta,\mathrm{unf}} -  f_{\lambda,\eta,k,\mathrm{psd}} \|_{2, \mathbf{N}}=0$.
By \eqref{eqn_inf_arith_reg_lem_pr_10}, we have
\[
\innprod{f_{\lambda',\eta',\mathrm{str}}}{f_{\lambda,\eta,\mathrm{unf}}}_{\mathbf{N}}=0,\quad\text{whenever}~(\lambda,\eta)\prec (\lambda',\eta').
\]
This implies
\begin{align*}
\|f_{\lambda',\eta',\mathrm{str}}\|_{2, \mathbf{N}}^2
&=\innprod{f_{\lambda',\eta',\mathrm{str}}}{f_{\lambda',\eta',\mathrm{str}}}_{\mathbf{N}}
\\
&=\innprod{f_{\lambda',\eta',\mathrm{str}}}{f}_{\mathbf{N}}
\\
&=\innprod{f_{\lambda',\eta',\mathrm{str}}}{f_{\lambda,\eta,\mathrm{str}}}_{\mathbf{N}}
\\
&\leq \|f_{\lambda,\eta,\mathrm{str}}\|_{2, \mathbf{N}} \cdot \|f_{\lambda',\eta',\mathrm{str}}\|_{2, \mathbf{N}},
\end{align*}
and hence
\[
(\lambda,\eta)\prec (\lambda',\eta')\implies \|f_{\lambda,\eta,\mathrm{str}}\|_{2, \mathbf{N}}\geq \|f_{\lambda',\eta',\mathrm{str}}\|_{2, \mathbf{N}}.
\]
According to \cref{lem_decreasing_seq_on_unit_square} 
there exists $\lambda\in(0,1)$ such that
$\|f_{\lambda,\eta,\mathrm{str}}\|_{2, \mathbf{N}}= \|f_{\lambda,\eta',\mathrm{str}}\|_{2, \mathbf{N}}$ for all $\eta,\eta'\in[0,1]$. In particular, we have for all $\eta,\eta'\in[0,1]$ that
\[
\|f_{\lambda,\eta,\mathrm{str}}-f_{\lambda,\eta',\mathrm{str}}\|_{2, \mathbf{N}}=\|f_{\lambda,\eta,\mathrm{unf}}-f_{\lambda,\eta',\mathrm{unf}}\|_{2, \mathbf{N}}=0.
\]
Now define $\mathbf{H}^-=\mathbf{H}_{\lambda,0}$ and $\mathbf{H}^+=\mathbf{H}_{\lambda,1}$, as well as
\[
f_{\str}=f_{\lambda,0,\mathrm{str}} \qquad\text{and}\qquad f_{\unf}=f_{\lambda,0,\mathrm{unf}}.
\]
Moreover, define
\begin{align*}
c_{k,i}(n)&= c_{\lambda,1,k,i}(n),
\\
\alpha_{k,i}(n)&= \alpha_{\lambda,1,k,i}(n),
\\
\mathcal{P}_k&=\mathcal{P}^\star_{\lambda,1,k}
=\sum_{i=1}^{d_k^\star} c_{k,i}\, e(n \alpha_{k,i}),
\end{align*}
and note that since $\| \mathcal{P}_k -  f_{\lambda,1,k,\mathrm{str}} \|_{2, \mathbf{N}}\leq \epsilon_k$, $\lim_{k\to\infty}\| f_{\lambda,1,\mathrm{str}} -  f_{\lambda,1,k,\mathrm{str}} \|_{2, \mathbf{N}}=0$, and $\| f_{\mathrm{str}} -  f_{\lambda,1,\mathrm{str}} \|_{2, \mathbf{N}}=0$, we have
\begin{equation}
\label{eqn_inf_arith_reg_lem_pr_11}
\lim_{k\to\infty}\| f_{\mathrm{str}} -  \mathcal{P}_k\|_{2, \mathbf{N}}=0.
\end{equation}
By construction, for all but finitely many $s$ we have that
\begin{equation}
\label{eqn_inf_arith_reg_lem_pr_12}
c_{k,i}(n)~\text{and}~\alpha_{k,i}(n)~\text{are constant on subintervals longer than $\geq H_{s}^+/2$ in $\{N_{s-1}+1,\ldots,N_s\}$.}
\end{equation}
Also, using \cref{cor:stable local averages for countable family}, after replacing $\mathbf{H}^-$ and $\mathbf{H}^+$ with a narrower pair of sequences if needed, we get that for all $C\in\N$ and all $(w_1, \ldots, w_{d_k^\star}) \in [-C,C]^{d_k^\star}\cap \Z^{d_k^\star}$,
if $\mathbf{H}^-\prec\mathbf{H}\prec\mathbf{H}^+$ and
\begin{align*}
V_{s,w_1, \ldots, w_{d_k^\star},C}=\Bigg\{n\in  [N_s&- H_s+1]:
\Bigg|\frac{1}{C^2H_s}\sum_{h=1}^{C^2H_s} e(h(w_1\alpha_{k,1}(n)+\ldots+w_{d_k^\star}\alpha_{k,d_k^\star}(n)))\,
\\
&-\,\frac{1}{\lfloor H_s/C^2\rfloor }\sum_{h=1}^{\lfloor H_s/C^2\rfloor} e(h(w_1\alpha_{k,1}(n)+\ldots+w_{d_k^\star}\alpha_{k,d_k^\star}(n)))\Bigg|<\frac{1}{C^2}\Bigg\}
\end{align*}
then
\begin{equation} \label{eq: V sets are large}
    \lim_{s\to\infty} \frac{|V_{s,w_1, \ldots, w_d,C}|}{N_s-H_s+1}=1.
\end{equation}

We claim that if we consider
\begin{align*}
&W_{s,k,C}=\Big\{n\in [N_s-H_s+1]: \text{for all}~(w_1, \ldots, w_{d_k^\star}) \in [-C,C]^{d_k^\star}\cap \Z^{d_k^\star},~\text{either}
\\
&\|w_1\alpha_{k,1}(n) +\ldots+w_{d_k^\star}\alpha_{k,d_k^\star}(n)\|_{\T}\leq \tfrac{1}{C H_s}
~\text{or}~\|w_1\alpha_{k,1}(n) +\ldots+w_{d_k^\star}\alpha_{k,d_k^\star}(n)\|_{\T}\geq \tfrac{C}{H_s}\Big\},
\end{align*}
then
\begin{equation}
\label{eqn_inf_arith_reg_lem_pr_13}
\lim_{s\to\infty} \frac{|W_{s,k,C}|}{N_s-H_s+1}=1.
\end{equation}
The sets $W_{s,k,C}$ satisfy $W_{s,k,1} \supseteq W_{s,k,2} \supseteq \ldots$, so it suffices to prove the claim for large $C$.
We will show that
\begin{equation*}
    W_{s,k,C} \supseteq \bigcap_{(w_1, \ldots, w_{d_k^\star}) \in [-C,C]^{d_k^\star} \cap \Z} V_{s,w_1,\ldots,w_{d_k^\star},C}
\end{equation*}
for $C \geq 4$ and $s$ sufficiently large (depending on $C$), from which the claim the follows by applying \eqref{eq: V sets are large}.
Fix $n \in V_{s,w_1,\ldots,w_{d_k^\star},C}$, and let $\beta = \sum_{i=1}^{d_k^{\star}} w_i \alpha_{k,i}$.
We want to show (assuming $s$ is sufficiently large) that either $\|\beta\|_{\T} \leq \frac{1}{CH_s}$ or $\|\beta\|_{\T} \geq \frac{C}{H_s}$.
Suppose for contradiction that $\frac{1}{CH_s} < \|\beta\|_{\T} < \frac{C}{H_s}$.
Then
\begin{equation*}
    \left| \frac{1}{C^2H_s} \sum_{h=1}^{C^2H_s} e(h\beta) \right|
    = \frac{|e(C^2H_s\beta) - 1|}{C^2 H_s |e(\beta)-1|}
    \leq \frac{2}{C^2 H_s \cdot 4\|\beta\|_{\T}}
    < \frac{1}{2C},
\end{equation*}
while
\begin{multline*}
    \left| \frac{1}{\lfloor H_s/C^2 \rfloor} \sum_{h=1}^{\lfloor H_s/C^2 \rfloor} e(h\beta) \right|
     = \frac{|e(\lfloor H_s/C^2 \rfloor\beta) - 1|}{\lfloor H_s/C^2 \rfloor |e(\beta)-1|}
     \geq \frac{4\|\lfloor H_s/C^2 \rfloor\beta\|_{\T}}{(H_s/C^2) \cdot 2\pi \|\beta\|_{\T}} \\
     > \frac{2 \left( \frac{H_s}{C^2} -1 \right) \frac{1}{CH_s}}{\pi \frac{H_s}{C^2} \frac{C}{H_s}}
     = \frac{2}{\pi} \frac{H_s - C}{C^2H_s} = \frac{2}{\pi} \cdot \frac{1}{C^2} - \oh_{s \to \infty}(1).
\end{multline*}
Therefore,
\begin{equation*}
    \left| \frac{1}{C^2H_s} \sum_{h=1}^{C^2H_s} e(h\beta) - \frac{1}{\lfloor H_s/C^2 \rfloor} \sum_{h=1}^{\lfloor H_s/C^2 \rfloor} e(h\beta) \right| \geq \frac{1}{2C} - \frac{2}{\pi} \frac{1}{C^2} - \oh_{s \to \infty}(1)
\end{equation*}
For $C \geq 4$, this contradicts the assumption $n \in V_{s,w_1,\ldots,w_{d_k^\star},C}$.
This proves the claim.

Now let $\epsilon>0$ and $\mathbf{H}=(H_s)_{s\in\N}$ with $\mathbf{H}^-\prec\mathbf{H}\prec\mathbf{H}^+$ be arbitrary.
Let $k$ be sufficiently large so that $\| f_{\mathrm{str}} -  f_{\lambda,1,k,\mathrm{str}} \|_{2, \mathbf{N}}\leq \epsilon/2$ and $\epsilon_k\leq \epsilon/2$, and define $d=d_k^\star$ and $\mathcal{P}=\mathcal{P}_k$.
From \eqref{eqn_inf_arith_reg_lem_pr_13}, \eqref{eqn_inf_arith_reg_lem_pr_13}, and \eqref{eqn_inf_arith_reg_lem_pr_13} it follows that with this choice of $d$ and $\mathcal{P}$ part~\ref{itm U^2 structure theorem ii} of \cref{def: U2 good scale} is satisfied.
Since $\mathbf{H}_{\lambda,0}=\mathbf{H}^-\prec \mathbf{H}$, it follows from \eqref{eqn_inf_arith_reg_lem_pr_7} that
\[
\lim_{k\to\infty}\|f_{\lambda,0,k,\mathrm{psd}}\|_{u^2(\mathbf{N},\mathbf{H})}=0.
\]
Combined with $\lim_{k\to\infty} \| f_{\lambda,0,\mathrm{unf}} -  f_{\lambda,0,k,\mathrm{unf}} \|_{2, \mathbf{N}}=0$, and since $f_{\unf}=f_{\lambda,0,\mathrm{unf}}$, we conclude that
\[
\|f_{\unf}\|_{u^2(\mathbf{N},\mathbf{H})}=0.
\]
It thus follows from \eqref{eqn_U2-uniformity_iff_pseudorandom} that $\|f_{\unf}\|_{U^2(\mathbf{N},\mathbf{H})}=0$ and so part~\ref{itm U^2 structure theorem iii} of \cref{def: U2 good scale} holds too.
\end{proof}

We showed in \cref{thm: existence of U1 G and HK scales}\ref{itm: existence of U1 HK scales} that if a function $f : \N \to [0,1]$ admits correlations along a sequence $\mathbf{N}$ and $\mathbf{H}$ is a sufficiently slowly growing scale, then the $U^1(\mathbf{N},\mathbf{H})$ seminorm of $f$ agrees with the Host--Kra $U^1$ seminorm of $f$ along $\mathbf{N}$.
The next lemma, which is needed for technical arguments in \cref{sec: local ergodicity} below, establishes a similar result for the $U^2$ seminorms.

\begin{Lemma} \label{lem: slow scales are U2 good}
    Let $f : \N \to [0,1]$ and suppose $\mathbf{N} = (N_s)_{s \in \N}$ is a sequence of natural numbers such that $\lim_{s \to \infty} N_s = \infty$ and $\lim_{s \to \infty} N_s/N_{s+1} = 0$.
    If $f$ admits correlations along $\mathbf{N}$, then there exists a sequence of natural numbers $\mathbf{H}_{HK,U^2} = (H_{HK,U^2,s})_{s \in \N}$ with $1 \prec \mathbf{H}_{HK,U^2} \prec \mathbf{N}$ such that for all sequences $\mathbf{H} = (H_s)_{s \in \N}$ with $1 \prec \mathbf{H} \preceq \mathbf{H}_{HK,U^2}$, the pair $(\mathbf{N},\mathbf{H})$ is a $U^2$ good scale for $f$ with the property that the corresponding decomposition $f = f_{\str} + f_{\unf}$ given by \cref{def: U2 good scale} for the $U^2(\mathbf{N},\mathbf{H})$ seminorm satisfies
    \begin{equation*}
        \|f_{\str} - f_1\|_{2,\mathbf{N}} = \|f_{\unf} - f_2\|_{2,\mathbf{N}} = 0,
    \end{equation*}
    where $f = f_1 + f_2$ is the decomposition given by \cref{thm: U2 HK structure theorem} for the $U^2(\mathbf{N})$ Host--Kra seminorm.
\end{Lemma}

\begin{proof}
    By \cref{thm: U2 HK structure theorem}, there exist locally invariant functions $c_{k,1}, \ldots, c_{k,d_k} : \N \to \C$ with $\|c_{k,i}\|_{\infty} \leq 1$ and locally invariant functions $\alpha_{k,1}, \ldots, \alpha_{k,d_k} : \N \to \T$ such that if $\mathcal{P}_k(n) = \sum_{i=1}^{d_k} c_{k,i}(n) e(n\alpha_{k,i}(n))$, then
    \begin{equation*}
        \|f_1 - \mathcal{P}_k\|_{2,\mathbf{N}} < \frac{1}{k}.
    \end{equation*}
    Then by \cref{thm: existence of U1 G and HK scales}, for each $k \in \N$, there exists $\mathbf{H}_{HK,k} = (H_{HK,k,s})_{s \in \N}$ with $1 \prec \mathbf{H}_{HK,k} \prec \mathbf{N}$ such that if $\mathbf{H} = (H_s)_{s \in \N}$ and $1 \prec \mathbf{H} \preceq \mathbf{H}_{HK,k}$, then the functions $c_{k,i}$ and $\alpha_{k,i}$ can be chosen to be $U^1$-structured in the sense of \cref{def: U1 good scale} for $i \in [d_k]$.
    We may pick a sequence $K_s \to \infty$ such that $\lim_{s \to \infty} \min_{k \in [K_s]} H_{HK,k,s} = \infty$ and define $H_{HK,U^2,s} = \min_{k \in [K_s]} H_{HK,k,s}$.
    Then $1 \prec \mathbf{H}_{HK,U^2}$ and $\mathbf{H}_{HK,U^2} \preceq \mathbf{H}_{HK,k}$ for every $k \in \N$.

    Let us check that $\mathbf{H}_{HK,U^2}$ has the desired properties.
    Suppose $\mathbf{H} = (H_s)_{s \in \N}$ and $1 \prec \mathbf{H} \preceq \mathbf{H}_{HK,U^2}$.
    We want to show that the decomposition $f = f_1 + f_2$ satisfies the nonnegativity, structure, and uniformity properties in \cref{def: U2 good scale}.
    Nonnegativity (property \ref{itm U^2 structure theorem i}) and orthogonality (property \ref{itm U^2 structure theorem iii}) follow from nonnegativity and orthogonality in \cref{thm: U2 HK structure theorem}.
    Given $\epsilon > 0$, we may pick $k \in \N$ such that $\frac{1}{k} \leq \epsilon$.
    Then
    \begin{equation*}
        \|f_1 - \mathcal{P}_k\|_{2,\mathbf{N}} < \frac{1}{k} \leq \epsilon,
    \end{equation*}
    and the sets
    \begin{align*}
        \Gamma_{s,k}^{(1)} & = \left\{ n \in [N_s - H_s + 1] : \alpha_{k,i}(n+m_1) = \alpha_{k,i}(n+m_2) ~\forall i \in [d_k],~\forall m_1, m_2 \in [H_s] \right\}
        \intertext{and}
        \Gamma_{s,k}^{(2)} & = \left\{ n \in [N_s - H_s + 1] : c_{k,i}(n+m_1) = c_{k,i}(n+m_2) ~\forall i \in [d_k],~\forall m_1, m_2 \in [H_s] \right\}
    \end{align*}
    satisfy
    \begin{equation*}
        \lim_{s \to \infty} \frac{|\Gamma_{s,k}^{(1)}|}{N_s-H_s+1} = \lim_{s \to \infty} \frac{|\Gamma_{s,k}^{(2)}|}{N_s - H_s + 1} = 1
    \end{equation*}
    since $\mathbf{H} \preceq \mathbf{H}_{HK,k}$.
    Moreover, if $C \in \N$ and $(w_1,\ldots,w_{d_k}) \in [-C,C]^{d_k} \cap \Z^{d_k}$, then the function $g(n) = e(w_1\alpha_{k,1}(n) + \ldots + w_{d_k} \alpha_{k,d_k}(n))$ is $U^1$ structured, so
    \begin{align*}
        V_{s,k,C} = \Bigg\{n \in [N_s&-H_s+1] : \Bigg|\frac{1}{C^2H_s}\sum_{h=1}^{C^2H_s} e(h(w_1\alpha_{k,1}(n)+\ldots+w_{d_k}\alpha_{k,d_k}(n)))\,
\\
&-\,\frac{1}{\lfloor H_s/C^2\rfloor }\sum_{h=1}^{\lfloor H_s/C^2\rfloor} e(h(w_1\alpha_{k,1}(n)+\ldots+w_{d_k}\alpha_{k,d_k}(n)))\Bigg|<\frac{1}{C^2}\Bigg\}
    \end{align*}
    satisfies
    \begin{equation*}
        \lim_{s \to \infty} \frac{|V_{s,k,C}|}{N_s-H_s+1} = 1.
    \end{equation*}
    Then arguing as in the proof of \cref{thm:U^2 structure theorem}, we have that the set
    \begin{align*}
&W_{s,k,C}=\Big\{n\in [N_s-H_s+1]: \text{for all}~(w_1, \ldots, w_{d_k}) \in [-C,C]^{d_k}\cap \Z^{d_k},~\text{either}
\\
&\|w_1\alpha_{k,1}(n) +\ldots+w_{d_k}\alpha_{k,d_k}(n)\|_{\T}\leq \tfrac{1}{C H_s}
~\text{or}~\|w_1\alpha_{k,1}(n) +\ldots+w_{d_k}\alpha_{k,d_k}(n)\|_{\T}\geq \tfrac{C}{H_s}\Big\},
\end{align*}
also satisfies
\begin{equation*}
\lim_{s\to\infty} \frac{|W_{s,k,C}|}{N_s-H_s+1}=1.
\end{equation*}
Thus, the decomposition $f = f_1 + f_2$ also satisfies \ref{itm U^2 structure theorem ii} from \cref{def: U2 good scale}.
\end{proof}

\section{Sumsets and $\delta$-popular sumsets in the integers at intermediate scales} \label{sec: intermediate scale sums and delta-sums}

In this section, we apply the structure theorems for the intermediate scale seminorms from \cref{sec: intermediate scale seminorms} to sumsets in $\N$.
Given $Q \in \N$ and two sets $E, F \subseteq \{0,1,\ldots,Q-1\}$, we write $E +_{\bmod{Q}} F$ to denote the sumset reduced mod $Q$:\nomenclature{$+_{\bmod{Q}}$}{sumset reduced mod $Q$}
\begin{align*}
    E +_{\bmod{Q}} F & = \{x + y \bmod{Q} : x \in E, y \in F\} \\
     & = \{x+y : x \in E, y \in F, x+y < Q\} \cup \{x+y-Q : x \in E, y \in F, x+y \geq Q\}.
\end{align*}

The main result of this section is the following.

\begin{Theorem} \label{thm: local sumset}
Let $A \subseteq \N$ such that $d(A) > 0$, and let $\mathbf{N} = (N_s)_{s \in \N}$ and $\mathbf{H}=(H_s)_{s\in\N}$ be sequences in $\N$ with
$1\prec \mathbf{H}\prec \mathbf{N}$ and $\lim_{s\to\infty} N_s/N_{s+1}=0$.
If $(\mathbf{N},\mathbf{H})$ is a $U^2$~good scale for $A$, then there exists a sequence $\epsilon_s \to 0$ such that for every $n \in \{N_{s-1}+1,\ldots, N_s\}$, there exists $Q_n\in\N$ with $H_s\leq Q_n\leq (1+\epsilon_s)H_s$ such that $Q_n$ is highly divisible in the sense that 
\[
\lim_{n \to \infty} \min\{m \in \N : m \nmid Q_n\} = \infty,
\]
and for any set $\tilde{B}_s \subseteq [H_s]$, for all but $\oh(N_s)$ many $n \in \{N_{s-1}+1,\ldots, N_s\}$, there exist sets $A_n\subset (A-n)\cap \{0,1,\ldots,Q_n-1\}$ and $B_n\subset \tilde{B}_s$ such that:
\begin{enumerate}[label=(\Roman{enumi}),ref=(\Roman{enumi}),leftmargin=*]
        \item\label{itm local sumset I} $|(A-n) \cap \{0,1,\ldots,Q_n-1\} \setminus A_n| \leq \epsilon_s Q_n$,
        \item\label{itm local sumset II} $|\tilde{B}_s \setminus B_n| \leq \epsilon_s Q_n$,
        \item\label{itm local sumset III} $|(A_n +_{\bmod{Q_n}} B_n) \setminus ((A+\tilde{B}_s-n)\cap \{0,1,\ldots,Q_n-1\})| \leq \epsilon_s Q_n$.
\end{enumerate}
Moreover, if $\tilde{B}_s = B \cap [H_s]$ for each $s \in \N$ and $B$ meets every residue class in $\N$, then there exists a sequence $D_s \to \infty$ such that each $B_n$ may be chosen to satisfy
\begin{enumerate}[label=(\Roman{enumi}),ref=(\Roman{enumi}),leftmargin=*]
    \setcounter{enumi}{3}
    \item\label{itm local sumset IV}
    $B_n$ meets every residue class mod $m \leq D_s$.
\end{enumerate}
\end{Theorem}

\subsection{A new proof of a corollary of Kneser's theorem}

One immediate application of \cref{thm: local sumset} is a new proof of the following corollary of Kneser's theorem (see \cref{thm_Kneser_for_integers}).

\begin{Corollary} \label{cor: Kneser without local obstructions}
    Let $A, B \subseteq \N$, and suppose $d(A), d(B) > 0$ (in particular, the density exists).
    Suppose moreover that $B$ meets every residue class in $\N$.
    Then
    \begin{equation*}
        \underline{d}(A+B) \geq \min\{1, d(A) + d(B)\}.
    \end{equation*}
\end{Corollary}

\begin{proof}
    Let $\alpha = d(A)$ and $\beta = d(B)$.
    Let $\mathbf{N} = (N_s)_{s \in \N}$ such that $\underline{d}(A+B) = d_{\mathbf{N}}(A+B)$ and $\lim_{s \to \infty} N_s/N_{s+1} = 0 $.
    Let $\mathbf{H}_G = (H_{G,s})_{s \in \N}$ be a Gowers scale for $A$ with $1 \prec \mathbf{H}_G \prec \mathbf{N}$ as given by \cref{thm: existence of U1 G and HK scales}\ref{itm: existence of U1 G scales}.
    Then by \cref{thm:U^2 structure theorem}, let $\mathbf{H} = (H_s)_{s \in \N}$ be a $U^2$ good scale for $A$ with $\mathbf{H}_G \prec \mathbf{H} \prec \mathbf{N}$.
    
    By \cref{thm: local sumset}, let $\epsilon_s \to 0$, $D_s \to \infty$, and $G_s \subseteq \{N_{s-1}+1,\ldots,N_s\}$ with $|G_s| = (1 - \oh_{s \to \infty}(1))N_s$ such that if $n \in G_s$, then there exists $Q_n \in \N$ with $H_s \leq Q_n \leq (1+\epsilon_s)H_s$ and sets $A_n\subset (A-n)\cap \{0,1,\ldots,Q_n-1\}$ and $B_n\subset B \cap \{0,1,\ldots,Q_n-1\}$ such that:
    \begin{itemize}
            \item $|(A-n) \cap \{0,1,\ldots,Q_n-1\} \setminus A_n| \leq \epsilon_s Q_n$,
            \item $|B \cap \{0,1,\ldots,Q_n-1\} \setminus B_n| \leq \epsilon_s Q_n$,
            \item $B_n$ meets every residue class mod $m$ for $m \leq D_s$, and
            \item $|(A_n +_{\bmod{Q_n}} B_n) \setminus ((A+B-n)\cap\{0,1,\ldots,Q_n-1\})| \leq \epsilon_s Q_n$.
    \end{itemize}
    Since $\mathbf{H} \succ \mathbf{H}_G$, we may additionally assume that $|(A-n) \cap \{0,1,\ldots,Q_n-1\}| = (\alpha + \oh_{s \to \infty}(1)) Q_n$ for $n \in G_s$.
    Therefore, the first two bullet points yield $|A_n| = (\alpha + \oh_{s\to\infty}(1))Q_n$ and $|B_n| = (\beta + \oh_{s\to\infty}(1))Q_n$.

    We now proceed by splitting the proof into cases.

    \bigskip

    \textbf{Case 1. $\alpha + \beta \geq 1$.}

    Let $n \in G_s$.
    By \cref{thm: cyclic group Kneser}, either $A_n +_{\bmod{Q_n}} B_n = \{0,1,\ldots,Q_n-1\}$ or
    \begin{multline*}
        |(A+B-n)\cap\{0,1,\ldots,Q_n-1\}| \geq |A_n +_{\bmod{Q_n}} B_n| - \epsilon_s Q_n \\
        \geq |A_n| + |B_n| - D_s^{-1}Q_n - \epsilon_s Q_n
        = \left( \alpha + \beta - \oh_{s \to \infty}(1) \right) Q_n \geq (1 - \oh_{s \to \infty}(1))Q_n.
    \end{multline*}
    Thus, for every $n \in G_s$,
    \begin{equation*}
        |(A+B-n)\cap\{0,1,\ldots,Q_n-1\}| = (1 - \oh_{s \to \infty}(1))Q_n.
    \end{equation*}
    Averaging over $n \leq N_s$,
    \begin{equation*}
        \underline{d}(A+B) = d_{\mathbf{N}}(A+B) = \lim_{s \to \infty} \frac{1}{N_s} \sum_{n=1}^{N_s} \frac{|(A+B-n)\cap\{0,1,\ldots,Q_n-1\}|}{Q_n} = 1.
    \end{equation*}

    \bigskip

    \textbf{Case 2. $\alpha + \beta < 1$.}

    In this case, for all large enough $s \in \N$, if $n \in G_s$, then $|A_n| + |B_n| < Q_n$.
    Therefore, by \cref{thm: cyclic group Kneser},
    \begin{align*}
        |(A+B-n)\cap\{0,1,\ldots,Q_n-1\}| & \geq |A_n +_{\bmod{Q_n}} B_n| - \epsilon_s Q_n \\
        & \geq |A_n| + |B_n| - D_s^{-1}Q_n - \epsilon_s Q_n \\
        & = \left( \alpha + \beta - \oh_{s \to \infty}(1) \right) Q_n.
    \end{align*}
    Averaging over $n \leq N_s$,
    \begin{equation*}
        \underline{d}(A+B) = d_{\mathbf{N}}(A+B) = \lim_{s \to \infty} \frac{1}{N_s} \sum_{n=1}^{N_s} \frac{|(A+B-n)\cap\{0,1,\ldots,Q_n-1\}|}{Q_n} \geq \alpha + \beta.
    \end{equation*}
\end{proof}

\subsection{Convolution at intermediate scales}

Given sequences $\mathbf{N} = (N_s)_{s \in \N}$ and $\mathbf{H}=(H_s)_{s\in\N}$ with $1\prec \mathbf{H}\prec \mathbf{N}$ and bounded functions $f,g\colon \N\to\C$,
we define the convolution of $f$ with $g$ at scale $(\mathbf{N},\mathbf{H})$ as\nomenclature{$*_{(\mathbf{N},\mathbf{H})}$}{convolution at scale $(\mathbf{N},\mathbf{H})$}
\[
f*_{(\mathbf{N},\mathbf{H})}g(n) = \sum_{t\in\N} \1_{(N_{t-1},N_t]}(n) \Bigg(\frac{1}{H_s}\sum_{m=1}^{H_s} f(n-m)g(m)\Bigg),
\]
with the convention that $N_0=0$.
This form of convolution is non-commutative. Its purpose is to capture the additive interaction between the values of $f(n)$ for $n\in [N_s]$ and those of $g(h)$ for $h\in [H_s]$.

The main result of this subsection is the following:

\begin{Theorem}
\label{thm_convolution_at_good_scale_1}
Let $f\colon \N\to [0,1]$, and suppose $\mathbf{N} = (N_s)_{s \in \N}$ and $\mathbf{H}=(H_s)_{s\in\N}$ are sequences in $\N$ with
$1\prec \mathbf{H}\prec \mathbf{N}$ and $\lim_{s\to\infty} N_s/N_{s+1}=0$. Suppose $(\mathbf{N},\mathbf{H})$ is a $U^2$~good scale for $f$, and let $f = f_{\str} + f_{\unf}$ be the associated splitting. Then
\[
\lim_{s \to \infty} \sup_{g : \N \to [0,1]} \|f*_{(\mathbf{N},\mathbf{H})}g\,-\,f_{\str}*_{(\mathbf{N},\mathbf{H})}g\|_{2,[N_s]}=0.
\]
\end{Theorem}

The main ingredient in the proof of \cref{thm_convolution_at_good_scale_1} is the following proposition.

\begin{Proposition}
\label{prop_relation_between_U2-norms}
For all sequences $\mathbf{N} = (N_s)_{s \in \N}$, $\mathbf{H} = (H_s)_{s \in \N}$, $\mathbf{H}_1 = (H_{1,s})_{s \in \N}$, and $\mathbf{H}_2 = (H_{2,s})_{s \in \N}$ with
\[
1\prec \mathbf{H}\prec \mathbf{H}_1, \mathbf{H}_2\prec \mathbf{N}
\]
and any bounded $f\colon \N\to\C$ we have that $\|f\|_{U^2(\mathbf{N},\mathbf{H})}=0$ implies
\[
\begin{split}
\lim_{s\to\infty} 
\frac{1}{N_s}\sum_{n=1}^{N_s} \sup_{u_1,u_2,g_1,g_2,g_3} \Bigg|
 \frac{1}{H_{1,s}}\sum_{h_1=-H_{1,s}}^{H_{1,s}} \frac{1}{H_{2,s}}\sum_{h_2=-H_{2,s}}^{H_{2,s}} u_1( &h_1) u_2(h_2)  g_1(n) g_2(n+h_1)
\\
&\cdot g_3(n+h_2)\overline{f(n+h_1+h_2)}\Bigg|=0,
\end{split}
\]
where the supremum is taken over all $1$-bounded functions $u_1,u_2,g_1,g_2,g_3\colon\Z\to\C$.
\end{Proposition}

Let us see how \cref{prop_relation_between_U2-norms} is used to prove \cref{thm_convolution_at_good_scale_1}.

\begin{proof}[Proof of \cref{thm_convolution_at_good_scale_1}]
Since $\lim_{s\to\infty} N_s/N_{s+1}=0$ and $f - f_{\str} = f_{\unf}$, we have for $g : \N \to [0,1]$
\begin{align*}
\|f*_{(\mathbf{N},\mathbf{H})} & g\,-\,f_{\str}*_{(\mathbf{N},\mathbf{H})}g\|_{2,[N_s]}^2
\\
&=\frac{1}{N_s}\sum_{n=1}^{N_s}
\Bigg|\Bigg(\frac{1}{H_s}\sum_{m=1}^{H_s} f(n-m)g(m)\Bigg)-\Bigg(\frac{1}{H_s}\sum_{m=1}^{H_s} f_\str(n-m)g(m)\Bigg)\Bigg|^2 + \oh(1)
\\
&=\frac{1}{N_s}\sum_{n=1}^{N_s}
\Bigg|\frac{1}{H_s}\sum_{m=1}^{H_s} f_\unf(n-m)g(m)\Bigg|^2 + \oh(1)
\\
&=\frac{1}{H_s}\sum_{m_1,m_2=1}^{H_s}\frac{1}{N_s}\sum_{n=1}^{N_s} g(m_1) g(m_2)f_\unf(n-m_1)f_\unf(n-m_2) + \oh(1)
\\
&=\frac{1}{H_s}\sum_{m_1,m_2=1}^{H_s}\frac{1}{N_s}\sum_{n=1}^{N_s} g(m_1) g(m_2)f_\unf(n)f_\unf(n+m_1-m_2) + \oh(1)
\\
&=\frac{1}{H_s}\sum_{m_1,m_2=-H_s}^{H_s}\frac{1}{N_s}\sum_{n=1}^{N_s} g_1(m_1)g_2(m_2)f_\unf(n)f_\unf(n+m_1+m_2) + \oh(1),
\end{align*}
where
\[
g_1(m)=
\begin{cases}
g(m),&\text{if}~m\geq 1,
\\
0,&\text{if}~m\leq 0,
\end{cases}
\qquad\text{and}\qquad
g_2(m)=
\begin{cases}
0,&\text{if}~m\geq 0,
\\
g(-m),&\text{if}~m\leq -1.
\end{cases}
\]
The claim now follows directly from \cref{prop_relation_between_U2-norms}.
\end{proof}

It remains to prove \cref{prop_relation_between_U2-norms}, for which we use the following lemmas.

\begin{Lemma}
\label{lem_weighted_cubic_averages_bound}
For any $f_1,f_2,f_3,f_4,u_1,u_2\colon \Z/N\Z\to \C$ with $\|f_i\|_\infty,\|u_j\|_\infty\leq 1$ we have
\[
\Bigg|\frac{1}{N^3} \sum_{n,h_1,h_2\in \Z/N\Z} u_1(h_1)u_2(h_2)f_1(n)\overline{f_2(n+h_1)}\overline{f_3(n+h_2)}f_4(n+h_1+h_2)\Bigg|\leq \|f_4\|_{U^2(\Z/N\Z)}.
\]
\end{Lemma}

\begin{proof}
% Let us assume without loss of generality that $\|f\|_\infty\leq 1$.
Using the Cauchy-Schwarz inequality, we have
\begin{align*}
&\Bigg|\frac{1}{N^3} \sum_{n,h_1,h_2\in \Z/N\Z} u_1(h_1)u_2(h_2)f_1(n)\overline{f_2(n+h_1) f_3(n+h_2)} f_4(n+h_1+h_2)\Bigg|
\\
&=\Bigg|\frac{1}{N^4} \sum_{n,h_1,h_2,h_3\in \Z/N\Z} u_1(h_1)u_2(h_2+h_3)f_1(n)\overline{f_2(n+h_1)}
\\
&\qquad \qquad \qquad \qquad \qquad \qquad \qquad \qquad \qquad \qquad \cdot \overline{f_3(n+h_2+h_3)} f_4(n+h_1+h_2+h_3)\Bigg|
\\
&\leq
\Bigg(\frac{1}{N^3} \sum_{n,h_1,h_2\in \Z/N\Z} \Bigg|\frac{1}{N} \sum_{h_3\in \Z/N\Z} u_2(h_2+h_3)\overline{f_3(n+h_2+h_3)}f_4(n+h_1+h_2+h_3)\Bigg|^2\Bigg)^{\frac{1}{2}}
\\
&=
\Bigg(\frac{1}{N^4} \sum_{n,h_1,h_2,h_3\in \Z/N\Z}  u_2(h_2)\overline{f_3(n+h_2)} f_4(n+h_1+h_2)
\overline{u_2(h_2+h_3)}f_3(n+h_2+h_3)
\\
&
\qquad \qquad \qquad \qquad \qquad \qquad \qquad \qquad \qquad \qquad \qquad \qquad \qquad \cdot \overline{f_4(n+h_1+h_2+h_3)}\Bigg)^{\frac{1}{2}}
\\
&=
\Bigg(\frac{1}{N^3} \sum_{n,h_1,h_3\in \Z/N\Z}  u(h_3)\overline{f_3(n)} f_4(n+h_1)
f_3(n+h_3)
\overline{f_4(n+h_1+h_3)}\Bigg)^{\frac{1}{2}},
\end{align*}
where $u(h_3)=\frac{1}{N}\sum_{h_2\in \Z/N\Z} u_2(h_2)\overline{u_2(h_2+h_3)}$. Repeating the same argument once more to remove the term $u(h_3)$, the conclusion follows.
\end{proof}

\begin{Lemma}
\label{lem_weighted_cubic_averages_bound_2}
Let $H,N\in\N$ and $\epsilon>0$ with $\epsilon^2 N \leq H \leq \epsilon N$. Then for any functions $f_1,f_2,f_3,f_4,u_1,u_2\colon \Z\to \C$ with $\|f_i\|_\infty,\|u_j\|_\infty\leq 1$ we have
\begin{align*}
\Bigg|\frac{1}{H^2} \sum_{h_1,h_2=1}^H\Bigg(\frac{1}{N} \sum_{n=1}^N  u_1(h_1)u_2(h_2)f_1(n) \overline{f_2(n+h_1)}\overline{f_3(n+h_2)} & f_4(n+h_1+h_2)\Bigg)\Bigg|
\\
&\leq 64\epsilon^{-4}\|f_4\|_{U^2([N])}+2\epsilon.
\end{align*}
\end{Lemma}

\begin{proof}
Let $\tilde{f}_i\colon \Z/4N\Z\to\C$ denote the function $\tilde{f}_i(n)=f_i(n)$ when $n\in[N]$ and $\tilde{f}_i(n)=0$ otherwise. Since $H\leq \epsilon N$, we have uniformly over all $1\leq h_1,h_2\leq H$, 
\begin{align*}
\Bigg|\frac{1}{N} \sum_{n=1}^N  u_1(h_1) & u_2(h_2)f_1(n) \overline{f_2(n+h_1)}\overline{f_3(n+h_2)}f_4(n+h_1+h_2)
\\
-
&\frac{1}{N} \sum_{n\in \Z/4N\Z}  u_1(h_1)u_2(h_2)\tilde{f}_1(n)  \overline{\tilde{f}_2(n+h_1)}\overline{\tilde{f}_3(n+h_2)}\tilde{f}_4(n+h_1+h_2)\Bigg|\leq 2\epsilon.
\end{align*}
Hence
\begin{align*}
\Bigg| & \frac{1}{H^2}  \sum_{h_1,h_2=1}^H\Bigg(\frac{1}{N} \sum_{n=1}^N  u_1(h_1)u_2(h_2)f_1(n)  \overline{f_2(n+h_1)}\overline{f_3(n+h_2)}f_4(n+h_1+h_2)\Bigg)\Bigg|
\\
&\leq 
\Bigg|\frac{1}{H^2} \sum_{h_1,h_2=1}^H\Bigg(\frac{1}{N} \sum_{n\in \Z/4N\Z}  u_1(h_1)u_2(h_2)\tilde{f}_1(n)  \overline{\tilde{f}_2(n+h_1)}\overline{\tilde{f}_3(n+h_2)}\tilde{f}_4(n+h_1+h_2)\Bigg)\Bigg|
+2\epsilon
\\
&=
\frac{64 N^2}{H^2} \Bigg|\frac{1}{(4N)^3} \sum_{n,h_1,h_2\in \Z/4N\Z}  1_{[H]}(h_1)u_1(h_1)1_{[H]}(h_2)u_2(h_2)\tilde{f}_1(n)
\\
&\qquad \qquad \qquad \qquad \qquad \qquad \qquad \qquad \cdot \overline{\tilde{f}_2(n+h_1)}\overline{\tilde{f}_3(n+h_2)}\tilde{f}_4(n+h_1+h_2)\Bigg|
+2\epsilon
\\
&\leq
\frac{64 N^2}{H^2} \big\|\tilde{f}_4\big\|_{U^2(\Z/4N\Z)}+2\epsilon,
\end{align*}
where the last inequality follows from \cref{lem_weighted_cubic_averages_bound}.
Using the assumption $\epsilon^2 N \leq H $ and the fact $\big\|\tilde{f}_4\big\|_{U^2(\Z/4N\Z)}\leq \big\|f_4\big\|_{U^2([N])}$, we obtain the desired bound and the proof is finished.
\end{proof}

\begin{proof}[Proof of \cref{prop_relation_between_U2-norms}]
Let $\epsilon>0$ and 
define the sequence $\mathbf{H}'=(H_s')_{s\in\N}$ via
\[
H_s'=\lfloor \epsilon H_s\rfloor,\qquad\forall s\in\N.
\]
Since $\mathbf{H}'\prec \mathbf{H}_1, \mathbf{H}_2$, by dividing the averages of length $2H_{1,s}+1$ and $2H_{2,s}+1$ into averages of length $H_s'$ and using the triangle inequality, we get
\[
\begin{split}
\limsup_{s\to\infty} 
\frac{1}{N_s}\sum_{n=1}^{N_s} \sup_{u_1,u_2,g_1,g_2,g_3} \Bigg|
 \frac{1}{H_{1,s}}\sum_{h_1=-H_{1,s}}^{H_{1,s}} \frac{1}{H_{2,s}}\sum_{h_2=-H_{2,s}}^{H_{2,s}} u_1( &h_1) u_2(h_2)  g_1(n) g_2(n+h_1)
\\
&\cdot g_3(n+h_2)\overline{f(n+h_1+h_2)}\Bigg|
\\
\leq \limsup_{s\to\infty} 
\frac{1}{N_s}\sum_{n=1}^{N_s} \sup_{u_1,u_2,g_1,g_2,g_3} \Bigg|
 \frac{1}{(H_{s}')^2}\sum_{h_1,h_2=1}^{H_{s}'} u_1( &h_1) u_2(h_2) g_1(n) g_2(n+h_1)
\\
&\cdot g_3(n+h_2)\overline{f(n+h_1+h_2)}\Bigg|.
\end{split}
\]
Finally, dividing the average of length $N_s$ into averages of length $H_s$ and using \cref{lem_weighted_cubic_averages_bound_2} (applied with $H_s$ in place of $N$ and $H_s'$ in place of $H$), the last expression is bounded by
\[
\leq  64\epsilon^{-4}\Bigg( \limsup_{s\to\infty} 
\frac{1}{N_s}\sum_{n=1}^{N_s}  \|f\|_{U^2(\{n,n+1,\ldots,n+H_s-1\})}\Bigg)+2\epsilon.
\]
Since $\epsilon$ was arbitrary, we obtain the desired conclusion.
\end{proof}

\subsection{Almost-periods for $U^2(\mathbf{N},\mathbf{H})$-structured functions}

\begin{Theorem}
\label{thm_almost-period_of_structured_component}
Let $f\colon \N\to [0,1]$, and suppose $\mathbf{N} = (N_s)_{s \in \N}$ and $\mathbf{H}=(H_s)_{s\in\N}$ are sequences in $\N$ with
$1\prec \mathbf{H}\prec \mathbf{N}$ and $\lim_{s\to\infty} N_s/N_{s+1}=0$.
Suppose $(\mathbf{N},\mathbf{H})$ is a $U^2$~good scale for $f$ and $f = f_{\str} + f_{\unf}$ is the associated splitting.
There exists a sequence $\epsilon_s \to 0$ such that for every $s\in\N$ and every $n\in \{N_{s-1}+1,\ldots,N_s\}$ we can find $Q_n\in\N$ with $H_s\leq Q_n\leq (1+\epsilon_s)H_s$ such that $Q_n$ is highly divisible in the sense that
\[
\lim_{n \to \infty} \min\{m \in \N : m \nmid Q_n\} = \infty,
\]
and
\[
\lim_{s\to\infty}\frac{1}{N_s}\sum_{n=1}^{N_s}\frac{1}{H_s}\sum_{h=0}^{H_s-1} |f_{\str}(n+h)-f_{\str}(n+h+Q_n)|=0.
\]
\end{Theorem}

\begin{proof}[Proof of \cref{thm_almost-period_of_structured_component}]
We shall prove that for every fixed $\epsilon>0$ and $M \in \N$, we can find for every $n\in \{N_{s-1}+1,\ldots,N_s\}$ a number $Q_n\in\N$ with $H_s\leq Q_n\leq (1+\epsilon)H_s$ such that
\[\lim_{n \to \infty} \min\{m \in \N : m \nmid Q_n\} \geq M,
\]
and
\[
\lim_{s\to\infty}\frac{1}{N_s}\sum_{n=1}^{N_s}\frac{1}{H_s}\sum_{h=0}^{H_s-1} |f_{\str}(n+h)-f_{\str}(n+h+Q_n)|\leq \epsilon.
\]
It then follows that we may replace the fixed $\epsilon$ and $M$ by sequences $\epsilon_s \to 0$ that tends to zero sufficiently slowly and $M_s \to \infty$ that tends to infinity sufficiently slowly such that the desired conclusion still holds.

As $(\mathbf{N},\mathbf{H})$ is a $U^2$~good scale for $f$, there exist $d\in\N$, $c_1,\ldots,c_d\colon \N\to \C$ with $\|c_i\|_\infty\leq 1$, and $\alpha_1,\ldots,\alpha_d\colon \N\to\R$ such that if $\mathcal{P}(n)=\sum_{i=1}^d c_i(n) e(n \alpha_i(n))$ then
\begin{equation}
\label{eqn_almost-period_of_structured_component_1}
    \| \mathcal{P} -  f_{\str} \|_{2, \mathbf{N}}\leq \frac{\epsilon}{3},
\end{equation}
and for all $C\geq 1$ if
\begin{align*}
\Gamma_s^{(1)}&=\{n\in [N_s-H_s+1]: \alpha_i(n+m_1)=\alpha_i(n+m_2)~\forall i\in [d],~\forall m_1,m_2\in [H_s]\},
\\
\Gamma_s^{(2)}&=\{n\in [N_s-H_s+1]: c_i(n+m_1)=c_i(n+m_2)~\forall i\in [d],~\forall m_1,m_2\in [H_s]\},
\\
W_{s,C}&=\Big\{n\in [N_s-H_s+1]: \text{for all}~(w_1, \ldots, w_d) \in [-C,C]^d\cap \Z^d,~\text{either}
\\
\|w_1&\alpha_1(n) +\ldots+w_d\alpha_d(n)\|_{\T}\leq \tfrac{1}{C H_s}
~\text{or}~\|w_1\alpha_1(n)+\ldots+w_d\alpha_d(n)\|_{\T}\geq \tfrac{C}{H_s}\Big\},
\end{align*}
then
\[
\lim_{s\to\infty} \frac{|\Gamma_s^{(1)}|}{N_s-H_s+1} = 
\lim_{s\to\infty} \frac{|\Gamma_s^{(2)}|}{N_s-H_s+1} =
\lim_{s\to\infty} \frac{|W_{s,C}|}{N_s-H_s+1}=1.
\]
Since $\lim_{s\to\infty} N_s/N_{s+1}=\lim_{s\to\infty} H_s/N_s=0$, we can slightly modify the sets $\Gamma_s^{(1)}$, $\Gamma_s^{(2)}$ and $W_{s,C}$ by replacing $[N_s-H_s+1]$ with $\{N_{s-1}+1,\ldots,N_s\}$ and $H_s$ with $3H_s$, so that if
\begin{align*}
\tilde{\Gamma}_s^{(1)}&=\{n\in \{N_{s-1}+1,\ldots,N_s\}: \alpha_i(n+m_1)=\alpha_i(n+m_2)~\forall i\in [d],~\forall m_1,m_2\in [3H_s]\},
\\
\tilde{\Gamma}_s^{(2)}&=\{n\in \{N_{s-1}+1,\ldots,N_s\}: c_i(n+m_1)=c_i(n+m_2)~\forall i\in [d],~\forall m_1,m_2\in [3H_s]\},
\\
\tilde{W}_{s,C}&=\Big\{n\in \{N_{s-1}+1,\ldots,N_s\}: \text{for all}~(w_1, \ldots, w_d) \in [-C,C]^d\cap \Z^d,~\text{either}
\\
\|w_1&\alpha_1(n) +\ldots+w_d\alpha_d(n)\|_{\T}\leq \tfrac{1}{3C H_s}
~\text{or}~\|w_1\alpha_1(n)+\ldots+w_d\alpha_d(n)\|_{\T}\geq \tfrac{C}{3H_s}\Big\},
\end{align*}
then we still have
\[
\lim_{s\to\infty} \frac{|\tilde{\Gamma}_s^{(1)}|}{N_s} = 
\lim_{s\to\infty} \frac{|\tilde{\Gamma}_s^{(2)}|}{N_s} =
\lim_{s\to\infty} \frac{|\tilde{W}_{s,C}|}{N_s}=1.
\]
Now take $C= C(d,\frac{\epsilon}{3d(M!)})$ from \cref{lem_finding_quantitative_almost_periods}.
If $n\in \tilde{\Gamma}_s^{(1)}\cap \tilde{\Gamma}_s^{(2)}\cap \tilde{W}_{s,C}$ then it follows from the conclusion of \cref{lem_finding_quantitative_almost_periods}, applied to $x=\frac{1}{3(M!)}+\frac{\epsilon}{3d(M!)}$, that there exists some $Q'_n\in \N$ with $\frac{H_s}{M!}\leq Q'_n\leq (1+\epsilon) \frac{H_s}{M!}$ such that
\[
\|Q'_n \alpha_{i}(n)\|_{\T}\leq \frac{\epsilon}{3d(M!)},\qquad\forall i=1,\ldots,d.
\]
Let $Q_n = M! \cdot Q'_n$.
Then $H_s \leq Q_n \leq (1+\epsilon)H_s$, $M! \mid Q_n$, and
\[
\|Q_n \alpha_{i}(n)\|_{\T}\leq \frac{\epsilon}{3d},\qquad\forall i=1,\ldots,d.
\]
Combined with the fact that $c_{i}(n+h)=c_{i}(n)$ and $\alpha_{i}(n+h)=\alpha_{i}(n)$ for all $h\in[3H_s]$, we conclude that
\begin{equation}
\label{eqn_almost-period_of_structured_component_2}
\max_{h\in [H_s]} |\mathcal{P}(n+h)-\mathcal{P}(n+h+Q_n)|\leq \frac{\epsilon}{3}.
\end{equation}
If $n$ does not belong to $\tilde{\Gamma}_s^{(1)}\cap \tilde{\Gamma}_s^{(2)}\cap \tilde{W}_{s,C}$ for some $s$, then it does not matter what $Q_n$ is, as the contribution of such $n$ is negligible. For simplicity, let us take $Q_n=M! \cdot \lceil \frac{H_s}{M!} \rceil$ in this case. 
To finish the proof, we combine \eqref{eqn_almost-period_of_structured_component_1} and \eqref{eqn_almost-period_of_structured_component_2} to deduce that
\begin{align*}
\lim_{s\to\infty}\frac{1}{N_s}\sum_{n=1}^{N_s}\frac{1}{H_s} & \sum_{h=0}^{H_s-1} |f_{\str}(n+h)-f_{\str}(n+h+Q_n)|
\\
&\leq
\lim_{s\to\infty}\frac{1}{N_s}\sum_{n=1}^{N_s}\frac{1}{H_s}\sum_{h=0}^{H_s-1} |\mathcal{P}(n+h)-\mathcal{P}(n+h+Q_n)| + \frac{2\epsilon}{3}
\\
&\leq \epsilon,
\end{align*}
arriving at the desired conclusion.
\end{proof}

\begin{Corollary}
\label{cor_convolution_at_good_scale_2}
Let $f\colon \N\to [0,1]$, and suppose $\mathbf{N} = (N_s)_{s \in \N}$ and $\mathbf{H}=(H_s)_{s\in\N}$ are sequences in $\N$ with
$1\prec \mathbf{H}\prec \mathbf{N}$ and $\lim_{s\to\infty} N_s/N_{s+1}=0$.
If $(\mathbf{N},\mathbf{H})$ is a $U^2$~good scale for $f$ then 
then there exists a sequence $\epsilon_s \to 0$ such that for all but $o(N_s)$ many $n\in \{N_{s-1}+1,\ldots,N_s\}$, we can find $Q_n\in\N$ with $H_s\leq Q_n\leq (1+\epsilon_s)H_s$ such that
\[
\lim_{n \to \infty} \min\{m \in \N : m \nmid Q_n\} = \infty
\]
and
\[
\sup_{g : \N \to [0,1]} \frac{1}{H_s}\sum_{h=0}^{H_s-1} |f*_{(\mathbf{N},\mathbf{H})}g(n+h)-f*_{(\mathbf{N},\mathbf{H})}g(n+h+Q_n)| \leq \epsilon_s.
\]
\end{Corollary}

\begin{proof}
First, we use \cref{thm_almost-period_of_structured_component} to find a sequence $\epsilon_s \to 0$ such that for every $s\in\N$ and every $n\in \{N_{s-1}+1,\ldots,N_s\}$ we can find $Q_n\in\N$ with $H_s\leq Q_n\leq (1+\epsilon_s)H_s$ such that
\[
\lim_{n \to \infty} \min\{m \in \N : m \nmid Q_n\} = \infty
\]
and
\begin{align}
\label{eqn_convolution_at_good_scale_2_1}
\lim_{s\to\infty}\frac{1}{N_s}\sum_{n=1}^{N_s}\frac{1}{H_s}\sum_{h=0}^{H_s-1} |f_{\str}(n+h)-f_{\str}(n+h+Q_n)|=0.
\end{align}
In light of \cref{thm_convolution_at_good_scale_1}, for every $g : \N \to [0,1]$, we have
\begin{align}
\label{eqn_convolution_at_good_scale_2_2}
\begin{split}
\frac{1}{N_s}\sum_{n=1}^{N_s} & \frac{1}{H_s} \sum_{h=0}^{H_s-1} |f*_{(\mathbf{N},\mathbf{H})}g(n+h)-f*_{(\mathbf{N},\mathbf{H})}g(n+h+Q_n)|
\\
&=
\frac{1}{N_s}\sum_{n=1}^{N_s} \frac{1}{H_s} \sum_{h=0}^{H_s-1} |f_{\str}*_{(\mathbf{N},\mathbf{H})}g(n+h)-f_\str*_{(\mathbf{N},\mathbf{H})}g(n+h+Q_n)|+\oh_{s\to\infty}(1),
\end{split}
\end{align}
where the $\oh_{s\to\infty}(1)$ error term is independent of $g$.
Using the definition of the convolution of $f$ with $g$ at scale $(\mathbf{N},\mathbf{H})$, we get for $n\in \{N_{s-1}+1,\ldots, N_s-3H_s\}$,
\begin{align*}
&\frac{1}{H_s} \sum_{h=0}^{H_s-1} |f_{\str}*_{(\mathbf{N},\mathbf{H})}g(n+h)-f_{\str}*_{(\mathbf{N},\mathbf{H})}g(n+h+Q_n)|
\\
&=\frac{1}{H_s}\sum_{h=0}^{H_s-1} \Bigg|\Bigg(\frac{1}{H_s}\sum_{m=1}^{H_s} f_{\str}(n+h-m)g(m)\Bigg)-\Bigg(\frac{1}{H_s}\sum_{m=1}^{H_s} f_{\str}(n+h+Q_n-m)g(m)\Bigg)\Bigg| 
\\
&\leq \frac{1}{H_s^2}\sum_{h=0}^{H_s-1}\sum_{m=1}^{H_s} \big|f_{\str}(n+h-m)-f_{\str}(n+h+Q_n-m)\big|.
\end{align*}
Since $\lim_{s\to\infty} N_s/N_{s+1}=\lim_{s\to\infty} H_s/N_s=0$, we therefore have
\begin{align*}
\frac{1}{N_s}\sum_{n=1}^{N_s} & \frac{1}{H_s} \sum_{h=0}^{H_s-1} |f_{\str}*_{(\mathbf{N},\mathbf{H})}g(n+h)-f_{\str}*_{(\mathbf{N},\mathbf{H})}g(n+h+Q_n)|
\\
&\leq \frac{1}{N_s}\sum_{n=1}^{N_s}\frac{1}{H_s^2}\sum_{h=0}^{H_s-1}\sum_{m=1}^{H_s}  \big|f_{\str}(n+h-m)-f_{\str}(n+h+Q_n-m)\big| +\oh_{s\to\infty}(1)
\\
&=\frac{1}{N_s}\sum_{n=1}^{N_s}\frac{1}{H_s}\sum_{h=0}^{H_s-1} \big|f_{\str}(n+h)-f_{\str}(n+h+Q_n)\big| +\oh_{s\to\infty}(1)
\end{align*}
Combined with \eqref{eqn_convolution_at_good_scale_2_1} and \eqref{eqn_convolution_at_good_scale_2_2}, this now gives
\[
\frac{1}{N_s}\sum_{n=1}^{N_s} \frac{1}{H_s} \sum_{h=0}^{H_s-1} |f*_{(\mathbf{N},\mathbf{H})}g(n+h)-f*_{(\mathbf{N},\mathbf{H})}g(n+h+Q_n)|=\oh_{s\to\infty}(1).
\]
Replacing $\epsilon_s$ with a more slowly decaying sequence if necessary, we get
\[
\frac{1}{N_s}\sum_{n=1}^{N_s} \frac{1}{H_s} \sum_{h=0}^{H_s-1} |f*_{(\mathbf{N},\mathbf{H})}g(n+h)-f*_{(\mathbf{N},\mathbf{H})}g(n+h+Q_n)|\leq\epsilon_s.
\]
In other words, for all but $o(N_s)$ many $n\in \{N_{s-1}+1,\ldots,N_s\}$ we have
\[
\frac{1}{H_s}\sum_{h=0}^{H_s-1} |f*_{(\mathbf{N},\mathbf{H})}g(n+h)-f*_{(\mathbf{N},\mathbf{H})}g(n+h+Q_n)|\leq \epsilon_s.
\]
as desired.
\end{proof}

\subsection{Modulated $\delta$-popular sumsets at intermediate scales}

Given $Q\in\N$, $E,F\subset\{0,1,\ldots,Q-1\}$, and $\delta>0$, we define $E +_{\delta,\bmod{Q}} F$ to be the $\delta$-popular sumset (as defined in \eqref{def_popular_sumset}) of $E$ and $F$ viewed as subsets of $\Z/Q\Z$ with addition taken modulo $Q$.\nomenclature{$+_{\delta,\bmod{Q}}$}{$\delta$-popular sumset mod $Q$}
Observe that 
\begin{align*}
E +_{\delta,\bmod{Q}} F = \bigg\{h\in \{0,1,\ldots,Q-1\}: \frac{1}{Q} &\sum_{m=0}^h \1_E(h-m)\1_F(m) 
\\
&~+ \frac{1}{Q}\sum_{m=h+1}^{Q-1} \1_E(h+Q-m)\1_F(m)>\delta\bigg\}.
\end{align*}

\begin{Theorem}\label{thm: local popular sumset}
Let $A \subseteq \N$ such that $d(A) > 0$, and let $\mathbf{N} = (N_s)_{s \in \N}$ and $\mathbf{H}=(H_s)_{s\in\N}$ be sequences in $\N$ with
$1\prec \mathbf{H}\prec \mathbf{N}$ and $\lim_{s\to\infty} N_s/N_{s+1}=0$.
If $(\mathbf{N},\mathbf{H})$ is a $U^2$~good scale for $A$ then there exist sequences $\epsilon_s \to 0$ and $\delta_s \to 0$ such that for all but $o(N_s)$ many $n\in \{N_{s-1}+1,\ldots,N_s\}$, there exists $Q_n\in\N$ with $H_s\leq Q_n\leq (1+\epsilon_s)H_s$ such that $Q_n$ is highly divisible in the sense that
\[
\lim_{n \to \infty} \min\{m \in \N : m \nmid Q_n\} = \infty,
\]
and for every $\tilde{B}_s \subseteq [H_s]$,
    \begin{multline*}
        \left| \Big( \big( (A-n)\cap \{0,1,\ldots,Q_n-1\} \big) +_{\delta_s,\bmod{Q_n}} \tilde{B}_s \Big) \right.
        \\ \left. \setminus \Big( (A+\tilde{B}_s-n)\cap\{0,1,\ldots,Q_n-1\} \Big) \right| \leq \epsilon_s Q_n.
    \end{multline*}
\end{Theorem}

\begin{proof}
Using \cref{cor_convolution_at_good_scale_2}, we can find a sequence $\epsilon_s' \to 0$ and sets $G_s\subset \{N_{s-1}+1,\ldots,N_s\}$ with $|G_s|\geq 1-\oh(N_s)$ such that for every $n\in G_s$ there exists some highly divisible $Q_n\in\N$ with $H_s\leq Q_n\leq (1+\epsilon_s')H_s$ and
\[
\sup_{B \subseteq \N} \frac{1}{H_s}\sum_{h=0}^{H_s-1}|\1_A*_{(\mathbf{N},\mathbf{H})}\1_B(n+h)-\1_A*_{(\mathbf{N},\mathbf{H})}\1_B(n+Q_n+h)|\leq \epsilon_s'.
\]
This implies that
\[
\sup_{B \subseteq \N} \frac{1}{Q_n}\sum_{h=0}^{Q_n-1}|\1_A*_{(\mathbf{N},\mathbf{H})}\1_B(n+h)-\1_A*_{(\mathbf{N},\mathbf{H})}\1_B(n+Q_n+h)|\leq 4\epsilon_s'.
\]

Fix a set $\tilde{B}_s \subseteq  [H_s]$, and let $D_n$ denote the set of all $h\in \{0,1,\ldots,Q_n-1\}$ such that
\[
|\1_A*_{(\mathbf{N},\mathbf{H})}\1_{\tilde{B}_s}(n+h)-\1_A*_{(\mathbf{N},\mathbf{H})}\1_{\tilde{B}_s}(n+Q_n+h)|\geq \sqrt{\epsilon_s'}/4.
\]
By Markov's inequality we have $|D_n|\leq 16 \sqrt{\epsilon_s'} Q_n$.
We claim that for appropriately chosen $\delta_s > 0$ (which we will choose independently of the set $\tilde{B}_s$),
\begin{multline*}
    \Big( \big( (A-n)\cap \{0,1,\ldots,Q_n-1\} \big) +_{\delta_s,\bmod{Q_n}} \tilde{B}_s \Big)
     \\ \setminus \Big( (A+\tilde{B}_s-n)\cap\{0,1,\ldots,Q_n-1\} \Big)\subseteq D_n.
\end{multline*}
Indeed, if $h\in \big( (A-n)\cap \{0,1,\ldots,Q_n-1\} \big) +_{\delta_s,\bmod{Q_n}} \tilde{B}_s$ then
\begin{multline*}
    \frac{1}{Q_n}\sum_{m=0}^h \1_{(A-n)\cap\{0,1,\ldots,Q_n-1\}}(h-m)\1_{\tilde{B}_s}(m) \\
    + \frac{1}{Q_n}\sum_{m=h+1}^{Q_n-1} \1_{(A-n)\cap\{0,1,\ldots,Q_n-1\}}(h+Q_n-m)\1_{\tilde{B}_s}(m)>\delta_s.
\end{multline*}
We can replace $(A-n)\cap\{0,1,\ldots,Q_n-1\}$ with $(A-n)$, since the ranges of the sums guarantee that we are in $\{0,1,\ldots,Q_n-1\}$.
Therefore,
\[
\frac{1}{Q_n}\sum_{m=0}^h \1_{(A-n)}(h-m)\1_{\tilde{B}_s}(m) + \frac{1}{Q_n}\sum_{m=h+1}^{Q_n-1} \1_{(A-n)}(h+Q_n-m)\1_{\tilde{B}_s}(m)>\delta_s.
\]
If additionally, $h\notin (A+\tilde{B}_s-n)\cap \{0,1,\ldots,Q_n-1\}$ then we have
\[
\frac{1}{Q_n}\sum_{m=0}^h \1_{(A-n)}(h-m)\1_{\tilde{B}_s}(m)=0
\]
and hence
\[
\frac{1}{Q_n}\sum_{m=h+1}^{Q_n-1} \1_{(A-n)}(h+Q_n-m)\1_{\tilde{B}_s}(m)>\delta_s,
\]
which implies
\[
\frac{1}{Q_n}\sum_{m=1}^{Q_n} \1_{(A-n)}(h+Q_n-m)\1_{\tilde{B}_s}(m)>\delta_s.
\]
Using the assumption $H_s\leq Q_n\leq (1+\epsilon_s')H_s$ once more gives
\[
\frac{1}{H_s}\sum_{m=1}^{H_s} \1_{(A-n)}(h+Q_n-m)\1_{\tilde{B}_s}(m)>\delta_s-3\epsilon_s'.
\]
Using the definition of the convolution at scale $(\mathbf{N},\mathbf{H})$, we can rewrite the last inequality as
\[
\1_A*_{(\mathbf{N},\mathbf{H})}\1_{\tilde{B}_s}(n+Q_n+h)>\delta_s-3\epsilon_s'.
\]
Since $h\notin (A+\tilde{B}_s-n)\cap \{0,1,\ldots,Q_n-1\}$ implies $h\notin (A+\tilde{B}_s-n)\cap \{0,1,\ldots,H_s-1\}$ and therefore $\1_A*_{(\mathbf{N},\mathbf{H})}\1_{\tilde{B}_s}(n+h)=0$, we conclude that
\[
|\1_A*_{(\mathbf{N},\mathbf{H})}\1_{\tilde{B}_s}(n+h)-\1_A*_{(\mathbf{N},\mathbf{H})}\1_{\tilde{B}_s}(n+Q_n+h)|>\delta_s-3\epsilon_s'.
\]
Taking $\delta_s=\sqrt{\epsilon_s'}/4+3\epsilon_s'$, we get that 
$h\in D_n$ as claimed. Finally, we pick $\epsilon_s=16\sqrt{\epsilon_s'}$ then $|D_n|\leq \epsilon_s Q_n$, completing the proof.
\end{proof}

\subsection{Modulated sumsets at intermediate scales}

We are now in position to prove \cref{thm: local sumset}.
The proof consists primarily of taking \cref{thm: local popular sumset} and then applying \cref{thm: delta-sum almost contains a sum} to eliminate the usage of $\delta_s$ appearing in \cref{thm: local popular sumset}.

\begin{proof}[Proof of \cref{thm: local sumset}]
Suppose $A \subseteq \N$ with $d(A) > 0$, and let $\mathbf{N} = (N_s)_{s \in \N}$ and $\mathbf{H}=(H_s)_{s\in\N}$ be sequences with $1\prec \mathbf{H}\prec \mathbf{N}$ and $\lim_{s\to\infty} N_s/N_{s+1}=0$.
Assume also that $(\mathbf{N},\mathbf{H})$ is a $U^2$~good scale for $A$.
In light of \cref{thm: local popular sumset}, there exist sequences $\epsilon_{1,s} \to 0$ and $\delta_s \to 0$ such that for all but $o(N_s)$ many $n\in \{N_{s-1}+1,\ldots,N_s\}$, there exists highly divisible $Q_n\in\N$ with $H_s\leq Q_n\leq (1+\epsilon_{1,s})H_s$ such that for every $\tilde{B}_s \subseteq [H_s]$,
\begin{multline}
\label{eqn_proof_local_sumset_1} 
    \left| \Big( \big( (A-n)\cap \{0,1,\ldots,Q_n-1\} \big) +_{\delta_s,\bmod{Q_n}} \tilde{B}_s \Big) \right.
     \\ \left. \setminus \Big( (A+\tilde{B}_s-n)\cap\{0,1,\ldots,Q_n-1\} \Big) \right| \leq \epsilon_{1,s} Q_n.
\end{multline}

Now choose a sequence $\epsilon_{2,s} \to 0$ such that $\delta'_s = \delta(\epsilon_{2,s})$ as given by \cref{thm: delta-sum almost contains a sum} satisfies $\delta'_s \geq \delta_s$.
For each such $n$, we then apply \cref{thm: delta-sum almost contains a sum} to the finite group $\Z/Q_n\Z$ in order to find sets $A_n \subset (A-n)\cap \{0,1,\ldots,Q_n-1\}$ and $B_n \subset \tilde{B}_s$ such that:
\[
|(A-n) \cap \{0,1,\ldots,Q_n-1\} \setminus A_n| \leq \epsilon_{2,s} Q_n,
\qquad
|\tilde{B}_s \setminus B_n| \leq \epsilon_{2,s} Q_n,
\]
and
\[
\Big|(A_n +_{\bmod{Q_n}} B_n) \setminus \Big( \big( (A-n)\cap \{0,1,\ldots,Q_n-1\} \big) +_{\delta_s,\bmod{Q_n}} \tilde{B}_s \Big) \Big|
\leq \epsilon_{2,s} Q_n.
\]
Combining the last part with \eqref{eqn_proof_local_sumset_1}, we get 
\begin{equation*}
\left| \big( A_n +_{\bmod{Q_n}} B_n \big) \setminus \big( (A+\tilde{B}_s-n)\cap\{0,1,\ldots,Q_n-1\} \big) \right| \leq (\epsilon_{1,s} + \epsilon_{2,s}) Q_n.
\end{equation*}
Putting $\epsilon_s = 2(\epsilon_{1,s} + \epsilon_{2,s}) \to 0$, we have
\begin{equation}\label{eqn_proof_local_sumset_2}
    \left| \big( A_n +_{\bmod{Q_n}} B_n \big) \setminus \big( (A+\tilde{B}_s-n)\cap\{0,1,\ldots,Q_n-1\} \big) \right| \leq \frac{\epsilon_s}{2} Q_n.
\end{equation}

Finally, if $\tilde{B}_s = B \cap [H_s]$, we can replace $B_n$ with $B_n\cup (B\cap [J_s])$ for some non-decreasing slow-growing sequence $J_s\to\infty$.
If the sequence $(J_s)_{s\in\mathbb{N}}$ grows sufficiently slowly, then the error introduced by this modification can be absorbed by replacing $\tfrac{\epsilon_s}{2}$ with $\epsilon_s$ in \eqref{eqn_proof_local_sumset_2}.
Consequently, if $B$ meets every residue class in $\N$, then, since the sets $B_n$ contain increasingly large initial segments of $B$, we conclude that there exists a slowly growing sequence $D_s \to \infty$ such that each $B_n$ can be chosen to have non-empty intersection with every residue class modulo $m$ for all $m \leq D_s$.
\end{proof}

%SECTION
\section{A local inverse theorem for sumsets of sets of positive density} \label{sec: local inverse theorem}

The aim of this section is to prove a ``local'' version of \cref{thm: main}, extracting a local structural description of sets $A$ and $B$ for which $\underline{d}(A+B) = d(A) + d(B)$.
The strategy is to use \cref{thm: local sumset} to relate the local properties of the sumset $A+B$ to a sumset in a cyclic group and then apply the inverse theorem for sumsets in cyclic groups expressed in \cref{thm: cyclic sum without obstructions}.

\begin{Theorem} \label{thm: local inverse theorem}
    Let $A, B \subseteq \N$ with $d(A) = \alpha > 0$ and $d(B) = \beta > 0$ such that $\alpha + \beta < 1$.
    Suppose $B$ meets every residue class in $\N$.
    Let $\mathbf{N} = (N_s)_{s \in \N}$ be an increasing sequence with $\lim_{s \to \infty} N_s/N_{s+1} = 0$ such that $d_{\mathbf{N}}(A+B) = \alpha + \beta$.
    Let $\mathbf{H}_G = (H_{G,s})_{s \in \N}$ with $1 \prec \mathbf{H}_G \prec \mathbf{N}$ be a Gowers scale for $A$ as guaranteed by \cref{thm: existence of U1 G and HK scales}\ref{itm: existence of U1 G scales}.
    Let $\mathbf{H} = (H_s)_{s \in \N}$ be a $U^2$~good scale for $A$ with $\mathbf{H}_G \prec \mathbf{H} \prec \mathbf{N}$.
    Then, after passing to a subsequence of $(\mathbf{N},\mathbf{H})$, there exists $H = h\Z$ for some $h \in \N$ and a sequences $k_s, C_s \to \infty$ such that for all $s \in \N$ and all but $\oh(N_s)$ many $n \in \{N_{s-1}+1,\ldots,N_s\}$, there exists a decomposition
    \[
        (A - n) \cap \{0,1,\ldots,H_s-1\} = (A_n \cup E_n) - a_n
    \]
    and
    \[
        B \cap \{0,1,\ldots,H_s-1\} = (B_{s,0} \cup B_{s,1} \cup F_s) - b_s
    \]
    such that $A_n, B_{s,0} \subseteq H$, $a_n, b_s \in \{0,1,\ldots,h-1\}$, $B_{s,1} \subseteq \N \setminus H$ with $|B_{s,1}| = \left( 1 - \frac{1}{h} - \oh(1) \right)H_s$, $|E_n| = \oh(H_s)$,  $|F_s| = \oh(H_s)$, and one of the following two conditions is satisfied:
    \begin{enumerate}[label=(\arabic*), leftmargin=*]
        \item\label{itm local inverse theorem 1}
            There exists $\theta_s \in \T$ with $\min_{q \in [k_s]} \|q\theta_s\|_{\T} \geq \frac{C_s}{H_s}$ and closed intervals $I_n, J_s \subseteq \T$ such that if $\phi_s : H \to \T$ is the map $\phi_s(m) = m\theta_s$ for $m \in H$, then
            \[
                A_n \subseteq \phi_s^{-1}(I_n), \quad B_{s,0} \subseteq \phi_s^{-1}(J_s),
            \]
            and
            \[
                |\phi_s^{-1}(I_n) \cap \{0,1,\ldots,H_s-1\} \setminus A_n|, |\phi_s^{-1}(J_s) \cap \{0,1,\ldots,H_s-1\} \setminus B_{s,0}| = \oh(H_s).
            \]
        \item\label{itm local inverse theorem 2}
            $|B_{s,0}| = \oh(H_s)$ and $|A_n + B_{s,0}| = |A_n| + \oh(H_s)$.
    \end{enumerate}
\end{Theorem}

\begin{proof}
    We combine \cref{thm: local sumset} and \cref{thm: cyclic sum without obstructions}.

    By \cref{thm: local sumset}, there exist sequences $\epsilon_s \to 0$ and $D_s \to \infty$ such that for all but $\oh(N_s)$ many $n \in \{N_{s-1}+1,\ldots,N_s\}$, we can find $Q_n \in \N$ with $H_s \leq Q_n \leq (1 + \epsilon_s)H_s$, $A_n \subseteq (A-n) \cap \{0,1,\ldots,Q_n-1\}$ and $B_n \subseteq B \cap \{0,1,\ldots,Q_n-1\}$ such that
    \begin{enumerate}[label=(\Roman*),leftmargin=*]
        \item\label{itm local inverse theorem I}
            $|(A-n) \cap \{0,1,\ldots,Q_n-1\} \setminus A_n| \leq \epsilon_s Q_n$,
        \item\label{itm local inverse theorem II}
            $|B \cap \{0,1,\ldots,Q_n-1\} \setminus B_n| \leq \epsilon_s Q_n$,
        \item\label{itm local inverse theorem III}
            $|(A_n +_{\bmod{Q_n}} B_n) \setminus ((A+B-n)\cap \{0,1,\ldots,Q_n-1\})| \leq \epsilon_s Q_n$, and
        \item\label{itm local inverse theorem IV}
            $B_n$ intersects every residue class mod $m \leq D_s$.
    \end{enumerate}
    Since $\mathbf{H} \succ \mathbf{H}_G$, we also have that for all but $\oh(N_s)$ many $n \in \{N_{s-1}+1,\ldots,N_s\}$,
    \begin{align*}
        |(A-n) \cap \{0,1,\ldots,Q_n-1\}| & = (\alpha+\oh_{s \to \infty}(1))Q_n
        \intertext{and}
        |B \cap \{0,1,\ldots,Q_n-1\}| & = (\beta+\oh_{s\to\infty}(1))Q_n,
    \end{align*}
    so $|A_n| = (\alpha+\oh_{s\to\infty}(1))Q_n$ and $|B_n|=(\beta+\oh_{s\to\infty}(1))Q_n$.

    Let $\epsilon \in (0, 1-\alpha-\beta)$ and $k \in \N$ be given.
    Put $\epsilon' = \eta = \frac{\epsilon}{2}$, and let $\delta = \delta(\alpha, \beta, \epsilon',\eta,k) > 0$ and $D = D(\alpha,\beta,\epsilon',\eta,k) \in \N$ be given by \cref{thm: cyclic sum without obstructions} for each $k \in \N$.
    Note that since $D_s \to \infty$, we have $D_s \geq D$ for all large $s$, so \cref{thm: cyclic sum without obstructions} applies to the pair $(A_n,B_n)$.
    We consider several cases.

    \bigskip

    We first rule out the possibility that conditions \ref{itm cyclic sum without obstructions i} or \ref{itm cyclic sum without obstructions ii} from \cref{thm: cyclic sum without obstructions} hold for a positive proportion of $n \in \{N_{s-1}+1,\ldots,N_s\}$.
    Suppose for contradiction that there exists $c > 0$ such that conclusion \ref{itm cyclic sum without obstructions i} or \ref{itm cyclic sum without obstructions ii} from \cref{thm: cyclic sum without obstructions} holds for infinitely many $s \in \N$ and at least $cN_s$ many $n \in \{N_{s-1}+1,\ldots,N_s\}$.
    We pass to a subsequence to upgrade from ``infinitely many'' $s \in \N$ to ``every'' $s \in \N$.
    Then for every $s \in \N$, there is a set $L_s \subseteq \{N_{s-1}+1,\ldots,N_s\}$ with $|L_s| \geq c N_s$ with the property that if $n \in L_s$, then
    \begin{equation*}
        |A_n +_{\bmod{Q_n}} B_n| \geq \min \left\{ (1-\epsilon)Q_n, (\alpha + \beta + \delta)Q_n \right\} = (\alpha + \beta + \lambda)Q_n,
    \end{equation*}
    where $\lambda = \min\{\delta, 1 - \epsilon - \alpha - \beta\} > 0$.
    Note that by \cref{thm: cyclic group Kneser} and \ref{itm local inverse theorem IV}, we also have
    \begin{equation*}
        |A_n +_{\bmod{Q_n}} B_n| \geq |A_n| + |B_n| - D_s^{-1} Q_n
    \end{equation*}
    for all but $\oh(N_s)$ many $n \in \{N_{s-1}+1,\ldots,N_s\}$.
    Therefore, applying \ref{itm local inverse theorem III}, we have
    \begin{align*}
        d_{\mathbf{N}}(A+B) 
        & = \lim_{s \to \infty} \frac{1}{N_s} \sum_{n=1}^{N_s} \frac{|A_n +_{\bmod{Q_n}} B_n|}{Q_n} \\
        & = \lim_{s \to \infty} \frac{1}{N_s} \left( \sum_{n \in L_s} \frac{|A_n +_{\bmod{Q_n}} B_n|}{Q_n}
        + \sum_{n \notin L_s} \frac{|A_n +_{\bmod{Q_n}} B_n|}{Q_n} \right) \\
        & \geq \lim_{s \to \infty} \frac{1}{N_s} \left( |L_s|(\alpha+\beta+\lambda) + (N_s-|L_s|)(\alpha+\beta-D_s^{-1}) \right) \\
        & \geq c (\alpha+\beta+\lambda) + (1-c)(\alpha+\beta) \\
        & = \alpha + \beta + c\lambda > \alpha + \beta.
    \end{align*}
    This is a contradiction, so the set of $n$ for which conclusion \ref{itm cyclic sum without obstructions i} or \ref{itm cyclic sum without obstructions ii} holds has size $\oh(N_s)$.
    In other words, for all but $\oh(N_s)$ many $n \in \{N_{s-1}+1,\ldots,N_s\}$, either conclusion \ref{itm cyclic sum without obstructions iii} or conclusion \ref{itm cyclic sum without obstructions iv} holds.

    \bigskip

    We now analyze the consequences of \ref{itm cyclic sum without obstructions iii} and \ref{itm cyclic sum without obstructions iv}.
    Suppose $n \in \{N_{s-1}+1,\ldots,N_s\}$ and \ref{itm cyclic sum without obstructions iii} holds.
    That is, there exists $h_n \leq D$ with $h_n \mid Q_n$, $a_n, b_n \in \{0,1,\ldots,h_n-1\}$, and a decomposition $A_n = A'_n - a_n$ and $B_n = (B'_{n,0} \cup B'_{n,1}) - b_n$ such that
    \begin{enumerate}[label=(iii.\alph*), leftmargin=*]
        \item\label{itm local inverse theorem iii.a} $A'_n, B'_{n,0} \subseteq h_n\Z \cap \{0,1,\ldots,Q_n-1\}$,
        \item\label{itm local inverse theorem iii.b} $B'_{n,1} \subseteq \{0,1,\ldots,Q_n-1\} \setminus h_n\N$ and $|B'_{n,1}| > \left( 1 - \frac{1}{h_n} - \frac{\epsilon}{2h_n} \right) Q_n$, and
        \item\label{itm local inverse theorem iii.c} there exists $t_n, N_n \in \N$ with $N_n > k$, $\gcd(t_n,N_n) = 1$, and $N_n \mid \frac{Q_n}{h_n}$, and there exist intervals $I_n, J_n \subseteq \Z/N_n\Z$ such that if $\phi_n : \left\{ 0,h_n,2h_n,\ldots,Q_n - h_n \right\} \to \Z/N_n\Z$ is given by $\phi_n(h_nm) = mt_n$, then
        \begin{equation*}
            A'_n \subseteq \phi_n^{-1}(I_n), \quad B'_{n,0} \subseteq \phi_n^{-1}(J_n), \quad \text{and} \quad |\phi_n^{-1}(I_n) \setminus A'_n|, |\phi_n^{-1}(J_n) \setminus B'_{n,0}| < \frac{\epsilon}{2} Q_n.
        \end{equation*}
    \end{enumerate}
    Embedding $\Z/N_n\Z$ as the subgroup $\{0,\frac{1}{N_n},\ldots,\frac{N_n-1}{N_n}\} \subseteq \T$ and letting $\theta_n = \frac{t_n}{h_nN_n} \in \T$, we can find intervals $\tilde{I}_n, \tilde{J}_n \subseteq \T$ such that if $\tilde{\phi}_n : \{0,h_n,2h_n,\ldots,Q_n-h_n\} \to \T$ is given by $\tilde{\phi}_n(m) = m\theta_n$, then
    \begin{equation*}
            A'_n \subseteq \tilde{\phi}_n^{-1}(\tilde{I}_n), \quad B'_{n,0} \subseteq \tilde{\phi}_n^{-1}(\tilde{J}_n), \quad \text{and} \quad \left| \tilde{\phi}_n^{-1}(\tilde{I}_n) \setminus A'_n \right|, \left| \tilde{\phi}_n^{-1}(\tilde{J}_n) \setminus B'_{n,0} \right| < \frac{\epsilon}{2} Q_n.
    \end{equation*}
    Since $k < N_n \leq Q_n$, we also have $\|q\theta\|_{\T} \geq \frac{1}{N_n} \geq \frac{1}{Q_n}$ for $q \in [k]$.
    
    Note that by \ref{itm local inverse theorem iii.b},
    \begin{equation*}
        |B_n| \geq |B'_{n,1}| \geq Q_n \left( 1 - \frac{1}{h_n} - \frac{\epsilon}{2h_n} \right),
    \end{equation*}
    so
    \begin{equation} \label{eq: h bounded}
        h_n \leq \frac{1 + \frac{\epsilon}{2}}{1 - \beta - \oh_{s\to\infty}(1)} \leq \frac{2}{1-\beta}
    \end{equation}
    for all large enough $s$.
    
    Combining \ref{itm local inverse theorem iii.c} with \ref{itm local inverse theorem I} and applying the triangle inequality, we have
    \begin{multline*}
        \left| \left( (A-n) \cap \{0,1,\ldots,Q_n-1\} \right) \triangle \left( \tilde{\phi}_n^{-1}(\tilde{I}_n) - a_n \right) \right| \\
        \leq \left| \left( (A-n) \cap \{0,1,\ldots,Q_n-1\} \right) \triangle A_n \right| + \left| A'_n \triangle \tilde{\phi}_n^{-1}(\tilde{I}_n)\right|
        \leq \left( \epsilon_s + \frac{\epsilon}{2} \right) Q_n.
    \end{multline*}
    
    Let $E_n = \left( (A-n+a_n) \cap \{0,1,\ldots,H_s-1\} \right) \triangle \left( \tilde{\phi}_n^{-1}(\tilde{I}_n) \cap \{0,1,\ldots,H_s-1\}\right)$ and $F_n = B \cap \{0,1,\ldots,H_s-1\} \setminus B_n$.
    Put $\tilde{A}_n = A'_n \cap \{0,1,\ldots,H_s-1\}$ and $\tilde{B}_{n,i} = B'_{n,i} \cap \{0,1,\ldots,H_s-1\}$ for $i \in \{0,1\}$.
    This provides a decomposition
    \begin{equation*}
        (A - n) \cap \{0,1,\ldots,H_s-1\} = (\tilde{A}_n \cup E_n) - a_n
    \end{equation*}
    and
    \begin{equation*}
        B \cap \{0,1,\ldots,H_s-1\} = (\tilde{B}_{n,0} \cup \tilde{B}_{n,1} \cup F_n) - b_n
    \end{equation*}
    with the property that (as long as $s \in \N$ is sufficiently large)
    \begin{enumerate}[label=(\alph*), leftmargin=*]
        \item\label{itm local decomposition 1.a} $|\tilde{B}_{n,1}| \geq Q_n \left( 1 - \frac{1}{h_n} - \frac{\epsilon}{2h_n} \right) - (Q_n-H_s) > H_s \left( 1 - \frac{1}{h_n} - \epsilon \right)$,
        \item\label{itm local decomposition 1.b} $|E_n| \leq \left( \epsilon_s + \frac{\epsilon}{2} \right) Q_n < \epsilon H_s$,
        \item\label{itm local decomposition 1.c} $|F_n| \leq \epsilon_s Q_n < \epsilon H_s$, and
        \item\label{itm local decomposition 1.d} there exists $\theta_n \in \T$ with $\min_{q \in [k]} \|q\theta_n\|_{\T} \geq \frac{1}{Q_n} \geq \frac{1}{2H_s}$ and intervals $\tilde{I}_n, \tilde{J}_n \subseteq \T$ such that if $\tilde{\phi}_n : h_n\Z \to \T$ is given by $\tilde{\phi}_n(m) = m\theta_n$, then
        \begin{equation*}
            \tilde{A}_n \subseteq \tilde{\phi}_n^{-1}(\tilde{I}_n),
            \quad \text{and} \quad
            \tilde{B}_{n,0} \subseteq \tilde{\phi}_n^{-1}(\tilde{J}_n),
        \end{equation*}
        \begin{align*}
            \left|\tilde{\phi}_n^{-1}(\tilde{I}_n) \cap \{0,1,\ldots,H_s-1\} \setminus \tilde{A}_n \right| & \leq \frac{\epsilon}{2} Q_n < \epsilon H_s \\
        \intertext{and}
            \left|\tilde{\phi}_n^{-1}(\tilde{J}_n) \cap \{0,1,\ldots,H_s-1\} \setminus \tilde{B}_{n,0} \right| & \leq \frac{\epsilon}{2} Q_n < \epsilon H_s.
        \end{align*}
    \end{enumerate}
    Let us observe that since $(\mathbf{N},\mathbf{H})$ is a $U^2$ good scale for $A$ and the frequency $\theta_n$ satisfies $\min_{q \in [k]} \|q\theta_n\|_{\T} \geq \frac{1}{2H_s}$, the dichotomy provided by condition \ref{itm U^2 structure theorem ii} in \cref{def: U2 good scale} (see the definition of the set $W_{s,C}$) implies that for a given $C \geq 1$, we have $\min_{q \in [k]} \|q\theta_n\|_{\T} \geq \frac{C}{H_s}$ if $s$ is sufficiently large in terms of $C$ and $k$.
    Hence, we in fact have
    \begin{enumerate}[label=(\alph*'), leftmargin=*]
        \setcounter{enumi}{3}
        \item\label{itm local decomposition 1.d'} $\min_{q \in [k]} \|q\theta_n\|_{\T} \geq \frac{C_s}{H_s}$ for some sequence $C_s \to \infty$.
    \end{enumerate}
    When $(A - n) \cap \{0,1,\ldots,H_s-1\}$ and $B \cap \{0,1,\ldots,H_s-1\}$ admit decompositions satisfying \ref{itm local decomposition 1.a}--\ref{itm local decomposition 1.d} and \ref{itm local decomposition 1.d'}, we will say that $(A,B)$ satisfies property $(1_{h_n,\epsilon,k})$ at $n$.

    \bigskip

    Now suppose $n \in \{N_{s-1}+1,\ldots,N_s\}$ and \ref{itm cyclic sum without obstructions iv} holds.
    Then there exists $h_n \leq D$ with $h_n \mid D$, elements $a_n, b_n \in \{0,1,\ldots,h-1\}$, and a decomposition $A_n = A'_n - a_n$ and $B_n = (B'_{n,0} \cup B'_{n,1}) - b_n$ such that
    
    \begin{enumerate}[label=(iv.\alph*), leftmargin=*]
        \item\label{itm local inverse theorem iv.a}
            $A'_n, B'_{n,0} \subseteq h_n\Z \cap \{0,1,\ldots,Q_n-1\}$,
        \item\label{itm local inverse theorem iv.b}
            $B'_{n,1} \subseteq \{0,1,\ldots,Q_n-1\} \setminus h\N$ and $|B'_{n,1}| > \left( 1 - \frac{1}{h_n} - \frac{\epsilon}{2h_n} \right) Q_n$, and
        \item\label{itm local inverse theorem iv.c}
            $|B'_{n,0}| < \frac{\epsilon}{2h_n} Q_n$.
    \end{enumerate}
    From condition \ref{itm local inverse theorem iv.b}, we conclude $h_n \leq \frac{2}{1-\beta}$ as in the previous case.
    Also as above, we may define $\tilde{A}_n = A'_n \cap \{0,1,\ldots,H_s-1\}$, $\tilde{B}_{n,i} = B'_{n,i} \cap \{0,1,\ldots,H_s-1\}$ for $i \in \{0,1\}$, $E_n = (A-n+a_n) \cap \{0,1,\ldots,H_s-1\} \setminus \tilde{A}_n$, and $F_n = B \cap \{0,1,\ldots,H_s-1\} \setminus B_n$.
    Then
    \begin{equation*}
        (A-n) \cap \{0,1,\ldots,H_s-1\} = (\tilde{A}_n \cup E_n) - a_n
    \end{equation*}
    and
    \begin{equation*}
        B \cap \{0,1,\ldots,H_s-1\} = (\tilde{B}_{n,0} \cup \tilde{B}_{n,1} \cup F_n) - b_n,
    \end{equation*}
    and this decomposition satisfies (for large enough $s$)
    \begin{enumerate}[label=(\alph*), leftmargin=*]
        \item $\tilde{A}_n, \tilde{B}_{n,0} \subseteq h_n\Z \cap \{0,1,\ldots,H_s-1\}$,
        \item $\tilde{B}_{n,1} \subseteq \{0,1,\ldots,H_s-1\} \setminus h_n\Z$ and
            \begin{equation*}
                |\tilde{B}_{n,1}| \geq \left( 1 - \frac{1}{h_n} - \frac{\epsilon}{2h_n} \right)Q_n - (Q_n - H_s) > \left( 1 - \frac{1}{h_n} - \epsilon \right) H_s,
            \end{equation*}
        \item $|E_n| \leq \epsilon_s Q_n < \epsilon H_s$,
        \item $|F_n| \leq \epsilon_s Q_n < \epsilon H_s$, and
        \item $|\tilde{B}_{n,0}| \leq \frac{\epsilon}{2h_n} Q_n < \epsilon H_s$.
    \end{enumerate}
    When we have such a decomposition, we say that $(A,B)$ satisfies property $(2_{h_n,\epsilon})$ at $n$.

    \bigskip

    We have shown that for every $\epsilon > 0$, every $k \in \N$, every $s \in \N$, and all but $\oh(N_s)$ many $n \in \{N_{s-1}+1,\ldots,N_s\}$, the pair $(A,B)$ satisfies either property $(1_{h_n,\epsilon,k})$ or $(2_{h_n,\epsilon})$ for some $h_n \leq \frac{2}{1-\beta}$.
    Now taking $\epsilon = \epsilon'_s$ with $\epsilon'_s \to 0$ and $k = k_s$ with $k_s \to \infty$ sufficiently slowly, we have that for every $s \in \N$ and all but $\oh(N_s)$ many $n \in \{N_{s-1}+1,\ldots,N_s\}$, the pair $(A,B)$ satisfies either property $(1_{h_n,\epsilon'_s,k_s})$ or $(2_{h_n,\epsilon'_s})$ for some $h_n \leq \frac{2}{1-\beta}$.
    
    Both of the properties $(1_{h,\epsilon'_s,k_s})$ and $(2_{h,\epsilon'_s})$ provide decompositions of $B \cap \{0,1,\ldots,H_s-1\}$.
    The properties corresponding to different values of $h$ are mutually incompatible, so there can be at most one value of $h$ for each $s \in \N$.
    By the pigeonhole principle, we may pick $h \in \N$ and pass to a subsequence such that for every $s \in \N$ and all but $\oh(N_s)$ many $n \in \{N_{s-1}+1,\ldots,N_s\}$, the pair $(A,B)$ satisfies either property $(1_{h,\epsilon'_s,k_s})$ or $(2_{h,\epsilon'_s})$ at $n$.

    Suppose that for arbitrarily large $s \in \N$, there exists $n_s \in \{N_{s-1}+1,\ldots,N_s\}$ such that $(A,B)$ satisfies property $(2_{h,\epsilon'_s})$ at $n_s$.
    By passing to a subsequence, we may assume $n_s$ exists for every $s \in \N$.
    Then necessarily $\beta = 1 - \frac{1}{h}$, since $|B \cap \{0,1,\ldots,H_s-1\}| = (\beta + \oh(1))H_s$.
    For $n \in \{N_{s-1}+1,\ldots,N_s\}$ for which $(A,B)$ satisfies property $(1_{h,\epsilon'_s,k_s})$, we then also have that $(A,B)$ satisfies property $(2_{h,\epsilon'_s})$: the decomposition of $B$ at $n_s$ using property $(2_{h,\epsilon'_s})$ combined with the decomposition of $A$ at $n$ using property $(1_{h,\epsilon'_s,k_s})$ provides a decomposition of $(A,B)$ at $n$ satisfying property $(2_{h,\epsilon'_s})$.
    Thus, property \ref{itm local inverse theorem 2} holds by taking $B_{s,0}, B_{s,1}, F_s$, and $b_s$ according to the decomposition of $B$ at $n_s$.

    If the hypothesis of the previous paragraph fails, then for every $s \in \N$ and all but $\oh(N_s)$ many $n \in \{N_{s-1}+1,\ldots,N_s\}$, we have that $(A,B)$ satisfies property $(1_{h,\epsilon'_s,k_s})$ at $n$.
    This verifies property \ref{itm local inverse theorem 1}, with the caveat that property \ref{itm local inverse theorem 1} requires a decomposition of $B$ and values of $\theta_s$ depending only on $s$ and not on $n$.
    However, since all of the decompositions of $B$ for different values of $n \in \{N_{s-1}+1,\ldots,N_s\}$ are representing the same set, we may adjust the decompositions with only negligible changes in order to make the decomposition depend only on $s$.
\end{proof}

%SECTION
\section{From local residue classes to global residue classes}\label{sec_loc_to_glob_residue_classes}

The goal of this section is to prove the following proposition, which takes the local information about sets $A$ and $B$ at scale $\mathbf{H}$ as given by conclusion \ref{itm local inverse theorem 2} in \cref{thm: local inverse theorem} and converts it into global information at scale $\mathbf{N}$ to derive condition \ref{itm_main_thm_2} in \cref{thm: main}.

\begin{Proposition} \label{prop: local to global residues}
    Let $A, B \subseteq \N$ with $d(A) = \alpha > 0$ and $d(B) = \beta > 0$ such that $\alpha + \beta < 1$.
    Suppose $B$ meets every residue class in $\N$.
    Let $\mathbf{N} = (N_s)_{s \in \N}$ be an increasing sequence with $\lim_{s \to \infty} N_s/N_{s+1} = 0$ such that $d_{\mathbf{N}}(A+B) = \alpha + \beta$.
    Let $\mathbf{H}_{G,A} = (H_{G,A,s})_{s \in \N}$ with $1 \prec \mathbf{H}_{G,A} \prec \mathbf{N}$ be a Gowers scale for $A$ as guaranteed by \cref{thm: existence of U1 G and HK scales}\ref{itm: existence of U1 G scales} and $\mathbf{H}_{G,B} = (H_{G,B,s})_{s \in \N}$ with $1 \prec \mathbf{H}_{G,B} \prec \mathbf{N}$ a Gowers scale for $B$.
    Let $\mathbf{H} = (H_s)_{s \in \N}$ be a $U^2$~good scale for $A$ and $B$ with $\mathbf{H}_{G,A}, \mathbf{H}_{G,B} \prec \mathbf{H} \prec \mathbf{N}$.
    Suppose there exist $h \in \N$ and $b_0 \in \{0,1,\ldots,h-1\}$ such that
    \begin{itemize}
        \item for all but $\oh(N_s)$ many $n \in \{N_{s-1}+1, \ldots, N_s\}$, there exists $a_n \in \{0,1,\ldots,h-1\}$ such that $(A-n) \cap \{0,1,\ldots,H_s-1\} = (A_n \cup E_n) - a_n$ for some $A_n \subseteq h\Z$ and $|E_n| = \oh(H_s)$, and
        \item $B \sim_{\mathbf{H}} (\N \setminus h\N) - b_0$.
    \end{itemize}
    Then there exists $a_0 \in \{0,1,\ldots,h-1\}$ such that $A \subseteq h\N - a_0$, $A + h \sim_{\mathbf{N}} A$, and $B \sim_{\mathbf{N}} (\N \setminus h\N) - b_0$.
\end{Proposition}

In the proof, we will use the following combinatorial lemma.

\begin{Lemma} \label{lem: density calculuation mod h}
    Let $h, q \in \N$.
    Let $\alpha \in \left( 0, \frac{1}{qh} \right)$, and let $\beta_0, \beta_1, \ldots, \beta_{h-1} \in [0,1]$ such that $\frac{1}{h} \sum_{j=0}^{h-1} \beta_j = 1 - \frac{1}{qh}$.
    If $\min_{0 \leq j \leq h-1} \beta_j \geq 1 - \frac{1}{q} + \eta$ for some $\eta > 0$, then
    \begin{equation*}
        \frac{1}{h} \sum_{j=0}^{h-1} \min\{h\alpha+\beta_j,1\} \geq \min \left\{ 1, 2\alpha + 1 - \frac{1}{qh}, \alpha + 1 - \frac{1}{qh} + \frac{\eta}{h} \right\}.
    \end{equation*}
\end{Lemma}

\begin{proof}
    Let $J = \{0 \leq j \leq h-1 : h\alpha + \beta_j \leq 1\}$.
    Then
    \begin{equation*}
        \frac{1}{h} \sum_{j=0}^{h-1} \min\{h\alpha+\beta_j,1\} = \frac{1}{h} \left( \sum_{j \in J} (h\alpha + \beta_j) + \sum_{j \notin J} 1 \right) = |J|\alpha + 1 - \frac{1}{qh} + \frac{1}{h} \sum_{j \notin J} (1 - \beta_j).
    \end{equation*}
    If $|J| \geq 2$, then
    \begin{equation*}
        \frac{1}{h} \sum_{j=0}^{h-1} \min\{h\alpha+\beta_j,1\} \geq 2\alpha + 1 - \frac{1}{qh},
    \end{equation*}
    and we are done.

    We check the remaining two cases.
    First, if $J = \emptyset$, then
    \begin{equation*}
        \frac{1}{h} \sum_{j=0}^{h-1} \min\{h\alpha+\beta_j,1\} = 1,
    \end{equation*}
    in which case we are again done.

    Finally, suppose $|J| = 1$.
    By symmetry, we may assume $J = \{0\}$.
    Now,
    \begin{equation*}
        \sum_{j = 1}^{h-1} \beta_j = h - \frac{1}{q} - \beta_0 \leq h - 1 - \eta,
    \end{equation*}
    so
    \begin{equation*}
        \frac{1}{h} \sum_{j\notin J} (1 - \beta_j) \geq \frac{h-1}{h} - \frac{1}{h}\sum_{j=1}^{h-1} \beta_j \geq \frac{\eta}{h}.
    \end{equation*}
\end{proof}

We will also make use of the following variant of \cref{lem: subinterval with Schnirelmann density}.

\begin{Lemma} \label{lem: Schnirelmann density at Gowers scales}
    Let $A \subseteq \N$.
    If $\mathbf{N}=(N_s)_{s \in \N}$ is a sequence of natural numbers such that $N_s \to \infty$ and there exists another sequence $\mathbf{H}=(H_s)_{s\in\N}$ with $1 \prec \mathbf{H} \prec \mathbf{N}$ such that
    \begin{equation*}
        d(A,[N_s]) \geq \alpha + \oh_{s\to\infty}(1)
    \end{equation*}
    and
    \begin{equation*}
        \frac{1}{N_s} \sum_{n=1}^{N_s} |d(A,[N_s]) - d\left( A,\{n+1,\ldots,n+H_s\} \right)| = \oh_{s\to\infty}(1),
    \end{equation*}
    then there exists a sequence of natural numbers $(n_s)_{s\in\N}$ with $\lim_{s\to\infty} n_s/N_s=0$ and such that
    \begin{equation*}
        \sigma(A,\{n_s+1,\ldots,N_s\}) \geq \alpha + \oh_{s\to\infty}(1).
    \end{equation*}
\end{Lemma}

\begin{proof}
    Let $\alpha_s = d(A, [N_s]) \geq \alpha + \oh(1)$, and let
    \begin{equation*}
        \delta_s = \frac{1}{N_s} \sum_{n=1}^{N_s} \left| d\left( A,\{n+1,\ldots,n+H_s\} \right) - \alpha_s \right| = \oh(1).
    \end{equation*}
    Let $M_s \leq N_s$ and $\epsilon_s > 0$ be parameters to be chosen later.
    Put $\gamma_s = d(A,[M_s])$.
    By Lemma \ref{lem: subinterval with Schnirelmann density}, there exists $x \leq (1-\epsilon_s)M_s$ such that $\sigma(A, \{x+1, \ldots, M_s\}) \geq \gamma_s - \epsilon_s$.
    Let $y \in \{M_s+1,\ldots,N_s\}$.
    Then
    \begin{align*}
        |A \cap \{x+1,\ldots,y\}| & = |A \cap \{x+1,\ldots,M_s\}| + |A \cap \{M_s+1,\ldots,y\}| \\
        & \geq (\gamma_s - \epsilon_s)(M_s - x) + \sum_{n=M_s+1}^y d(A, \{n+1,\ldots,n+H_s\}) + \Oh(H_s) \\
        & \geq (\gamma_s - \epsilon_s) (M_s - x) + \alpha_s(y-M_s) \\
        & \quad - \sum_{n=M_s+1}^y \left| d(A,\{n+1,\ldots,n+H_s\}) - \alpha_s \right| + \Oh(H_s) \\
        & \geq (\gamma_s - \epsilon_s) (M_s - x) + \alpha_s(y-M_s) - \delta_s N_s + \Oh(H_s).
    \end{align*}
    
    We now select the parameters to achieve the desired outcome.
    In order to overcome the $\delta_s N_s$ term, we should choose $M_s$ large compared with $\delta_s N_s$ and so that $\gamma_s \geq \alpha_s + \oh(1)$.
    On the other hand, we want $n_s =x = \oh(N_s)$, so $M_s$ cannot be too large.
    Put $M_s = \max\{\sqrt{\delta_s} N_s, \sqrt{H_sN_s}\}$.
    Note that
    \begin{multline*}
        \gamma_s = d(A, [M_s]) = \frac{1}{M_s} \sum_{n=1}^{M_s} d(A, \{n+1,\ldots,n+H_s\}) + \Oh \left( \frac{H_s}{M_s} \right) \\
        \geq \alpha_s - \frac{\delta_s N_s}{M_s} + \Oh \left( \frac{H_s}{M_s} \right)
        \geq \alpha_s - \sqrt{\delta_s} + \Oh \left( \sqrt{\frac{H_s}{N_s}} \right) = \alpha_s + \oh(1).
    \end{multline*}
    Thus, for $y \in \{M_s+1,\ldots,N_s\}$,
    \begin{align*}
        d(A, \{x+1,\ldots,y\}) & \geq \alpha_s - \epsilon_s + \oh(1) - \frac{\delta_s N_s}{y-x} + \Oh \left( \frac{H_s}{y-x} \right) \\
        & \geq \alpha_s - \epsilon_s - \frac{\delta_s N_s}{\epsilon_sM_s} + \Oh \left( \frac{H_s}{\epsilon_s M_s} \right) + \oh(1) \\
        & \geq \alpha_s - \epsilon_s - \epsilon_s^{-1} \sqrt{\delta_s} + \Oh \left(  \epsilon_s^{-1} \sqrt{\frac{H_s}{N_s}} \right) + \oh(1).
    \end{align*}
    The quantities $\sqrt{\delta_s}$ and $\sqrt{\frac{H_s}{N_s}}$ both tend to zero as $s \to \infty$, so we can choose $\epsilon_s = \oh(1)$ going to zero sufficiently slowly so that
    \begin{equation*}
        \epsilon_s^{-1} \sqrt{\delta_s} = \oh(1)
    \end{equation*}
    and
    \begin{equation*}
        \epsilon_s^{-1}\sqrt{\frac{H_s}{N_s}} = \oh(1).
    \end{equation*}
    For example, we can take
    \begin{equation*}
        \epsilon_s = \max \left\{ \delta_s^{1/4}, \left( \frac{H_s}{N_s} \right)^{1/4} \right\}.
    \end{equation*}

    Putting everything together and taking $n_s = x$, we have $n_s \leq M_s = \oh(N_s)$ and
    \begin{multline*}
        \sigma(A, \{n_s+1,\ldots,N_s\}) \geq \min \left\{ \sigma(A, \{n_s+1,\ldots,M_s\}), \min_{M_s < y \leq N_s} d(A, \{n_s+1,\ldots,y\}) \right\} \\
        \geq \alpha_s + \oh(1) \geq \alpha + \oh(1).
    \end{multline*}
\end{proof}

\begin{proof}[Proof of Proposition \ref{prop: local to global residues}]
    As a preliminary remark, we note that the condition $B \sim_{\mathbf{H}} (\N \setminus h\N) - b_0$ implies that $\beta = d(B) = d_{\mathbf{H}}(B) = 1 - \frac{1}{h}$. 

    We will deduce the global information about $A$ and $B$ in stages.
    First we show that $A + h \sim_{\mathbf{N}} A$.
    Then we show that in almost every subinterval of length $H_s$ in $\{N_{s-1}+1,\ldots,N_s\}$, $B$ is approximately the union of $h-1$ residue classes mod $h$.
    This local description of $B$ allows us to conclude globally that $A$ belongs to a single residue class mod $h$.
    Hence, taking $B_0 = (B+b_0) \cap h\N$ and $B_1 = (B+b_0) \setminus h\N$, we can write $A+B$ as a disjoint union of $A+B_0$ and $A+B_1$.
    Finally, we leverage this disjointness to show that the $d_{\mathbf{N}}(B_0) = 0$.

    \bigskip
    
    Let $B_0 = (B+b_0) \cap h\N$, and let $b_1, b_2 \in B_0$ with $b_1 < b_2$.
    Then applying \cref{thm: local sumset}, for all but $\oh(N_s)$ many $n \in \{N_{s-1}+1,\ldots,N_s\}$, there exists $Q_n \in \N$ with $H_s \leq Q_n \leq (1 + \oh(1))H_s$ and sets $\tilde{A}_n \subseteq (A-n) \cap \{0,1,\ldots,Q_n-1\}$ and $\tilde{B}_n \subseteq \{0,1,\ldots,Q_n-1\}$ such that
    \begin{enumerate}[label=(\Roman{enumi}),ref=(\Roman{enumi}),leftmargin=*]
        \item\label{itm local sumset I - proof of local to global} $|(A-n) \cap \{0,1,\ldots,Q_n-1\} \setminus \tilde{A}_n| = \oh(Q_n)$,
        \item\label{itm local sumset II - proof of local to global} $|B \cap \{0,1,\ldots,Q_n-1\} \setminus \tilde{B}_n| = \oh(Q_n)$, and
        \item\label{itm local sumset III - proof of local to global} $|(\tilde{A}_n +_{\bmod{Q_n}} \tilde{B}_n) \setminus ((A+B-n)\cap \{0,1,\ldots,Q_n-1\})| = \oh(Q_n)$.
    \end{enumerate}
    Now, by \ref{itm local sumset II - proof of local to global} and the assumption $B \sim_{\mathbf{H}} (\N \setminus h\N) - b_0$, we have that $B'_n = (\tilde{B}_n+b_0) \setminus h\N$ satisfies $|B'_n| = \left( 1 - \frac{1}{h} - \oh(1) \right) H_s$.
    Let $A'_n = (\tilde{A}_n + a_n) \cap h\N$.
    Then by our assumption on $A$, we have $|A'_n| = (\alpha + \oh(1))H_s$.
    Therefore,
    \begin{multline*}
        (\tilde{A}_n +_{\bmod{Q_n}} \tilde{B}_n) \cup (\tilde{A}_n + \{b_1-b_0,b_2-b_0\}) \\
        \supseteq \underbrace{((A'_n - a_n) +_{\bmod{Q_n}} (B'_n-b_0))}_{X} \cup \underbrace{((A'_n-a_n) + \{b_1-b_0,b_2-b_0\})}_{Y},
    \end{multline*}
    where $X$ is all but $\oh(H_s)$ of $\{0,1,\ldots,Q_n-1\} \setminus (h\N - a_n - b_0)$ and $Y \subseteq h\N - a_n - b_0$ has
    \begin{equation*}
        |Y| = |A'_n| + |(A'_n + (b_2-b_1)) \setminus A'_n| = \left( \alpha + \frac{|(A'_n + (b_2-b_1)) \setminus A'_n|}{H_s} + \oh(1) \right) H_s.
    \end{equation*}
    Since $b_1-b_0, b_2-b_0 \in B$ and $X \cap Y = \emptyset$, we have
    \begin{multline*}
        \frac{|(A+B-n) \cap \{0,1,\ldots,H_s-1\}|}{H_s} \geq \frac{|X| + |Y|}{H_s} - \oh(1) \\
         = \left( \alpha + \beta + \frac{\left| \left( (A+b_2-b_1-n) \setminus A \right) \cap \{0,1,\ldots,H_s-1\} \right|}{H_s} - \oh(1) \right).
    \end{multline*}
    Averaging over $n \leq N_s$ and taking a limit as $s \to \infty$, we conclude that
    \begin{align*}
        \alpha + \beta & = d_{\mathbf{N}}(A+B) = \lim_{s \to \infty} \frac{1}{N_s} \sum_{n=1}^{N_s} \frac{|(A+B-n) \cap \{0,1,\ldots,H_s-1\}|}{H_s} \\
        & \geq \alpha + \beta + \limsup_{s \to \infty} \frac{1}{N_s} \sum_{n=1}^{N_s} \frac{\left| \left( (A+b_2-b_1-n) \setminus A \right) \cap \{0,1,\ldots,H_s-1\} \right|}{H_s} \\
        & = \alpha + \beta + \overline{d}_{\mathbf{N}} \left( (A+b_2-b_1) \setminus A \right).
    \end{align*}
    Therefore, $A+b_2-b_1 \sim_{\mathbf{N}} A$.
    Hence, $P_A = \{n \in \Z : A+n \sim_{\mathbf{N}} A\} \supseteq B_0 - B_0$.
    But it is easily checked that $P_A$ is a group, and the assumption that $B$ meets every residue class in $\N$ shows that $B_0 - B_0$ is not contained in any proper subgroup of $h\Z$, so $P_A \supseteq h\Z$.
    That is, $A + h \sim_{\mathbf{N}} A$.

    \bigskip
    
    We will now show that for all but $\oh(N_s)$ many $m \in \{N_{s-1}+1,\ldots,N_s\}$, there exists $b_m \in \{0,1,\ldots,h-1\}$ such that $(B-m) \cap \{0,1,\ldots,h-1\}$ differs from $(\{0,1,\ldots,h-1\} \setminus h\N) - b_m$ by $\oh(H_s)$ many elements.
    To this end, we apply \cref{thm: local sumset} again with the roles of $A$ and $B$ reversed.
    For each $s \in \N$, we may pick $n_s \in \{N_{s-1}+1,\ldots,N_s\}$ with $n_s = \oh(N_s)$ such that
    \begin{itemize}
        \item $|(A-n_s) \cap \{0,1,\ldots,H_s-1\}| = (\alpha + \oh(1))H_s$,
        \item $|(A-n_s+a_{n_s}) \cap \{0,1,\ldots,H_s-1\} \setminus h\N| = \oh(H_s)$, and
        \item $\left| \left( (A - n_s + a_{n_s} + h) \setminus (A-n_s+a_{n_s}) \right) \cap \{0,1,\ldots,H_s-1\}\right| = \oh(H_s)$.
    \end{itemize}
    Therefore, for all but $\oh(N_s)$ many $m \in \{N_{s-1}+1,\ldots,N_s\}$, \cref{thm: local sumset} provides $Q_m \in h\N$ with $H_s \leq Q_m \leq (1 + \oh(1))H_s$ and sets $\tilde{A}_m \subseteq (A-n_s+a_{n_s}) \cap \{0,1,\ldots,Q_m-1\}$ and $\tilde{B}_m \subseteq (B-m) \cap \{0,1,\ldots,Q_m-1\}$ such that
    \begin{enumerate}[label=(\Roman{enumi}),ref=(\Roman{enumi}),leftmargin=*]
        \item $|(A-n_s+a_{n_s}) \cap \{0,1,\ldots,Q_m-1\} \setminus \tilde{A}_m| = \oh(Q_m)$,
        \item $|(B-m) \cap \{0,1,\ldots,Q_m-1\} \setminus \tilde{B}_m| = \oh(Q_m)$,
        \item $|(\tilde{A}_m +_{\bmod{Q_m}} \tilde{B}_m) \setminus ((A+B-n_s+a_{n_s}-m)\cap \{0,1,\ldots,Q_m-1\})| = \oh(Q_m)$,
        \item $\tilde{A}_m \subseteq h\Z \cap \{0,1,\ldots,Q_m-1\}$, and
        \item $|(\tilde{A}_m + h) \setminus \tilde{A}_m| = \oh(Q_m)$.
    \end{enumerate}
    
    We may write
    \begin{equation*}
        \tilde{A}_m +_{\bmod{Q_m}} \tilde{B}_m = \bigcup_{j=0}^{h-1} (\tilde{A}_m +_{\bmod{Q_m}}\tilde{B}_{m,j}),
    \end{equation*}
    where $\tilde{B}_{m,j} = \{x \in \tilde{B}_m : x \equiv j \pmod{h}\}$.
    Since $\tilde{A}_m \subseteq h\Z$, this union is a disjoint one.
    Letting $A'_m = \frac{\tilde{A}_m}{h}$ and $B'_{m,j} = \frac{\tilde{B}_{m,j}-j}{h}$, we have
    \begin{equation*}
        \left| \tilde{A}_m +_{\bmod{Q_m}} \tilde{B}_m \right| = \sum_{j=0}^{h-1} \left| A'_m +_{\bmod{\frac{Q_m}{h}}} B'_{m,j} \right|.
    \end{equation*}
    Now, $|(A'_m + 1) \setminus A'_m| = \oh(Q_m/h)$, so $A'_m$ meets every residue class modulo at most $D_s$ for some $D_s \to \infty$.
    Hence, by Theorem \ref{thm: cyclic group Kneser}, if $B'_{m,j} \ne \emptyset$, then
    \begin{equation*}
        \left| A'_m +_{\bmod{\frac{Q_m}{h}}} B'_{m,j} \right| \geq \min \left\{ |A'_m| + |B'_{m,j}| - \frac{Q_m}{hD_s}, \frac{Q_m}{h} \right\}.
    \end{equation*}
    Now, $|A'_m| = |\tilde{A}_m| = (\alpha - \oh_{s \to \infty}(1))Q_m$, so if $|B'_{m,j}| = \beta_{m,j} \frac{Q_m}{h}$, then
    \begin{equation*}
        \frac{\left| A'_m +_{\bmod{\frac{Q_m}{h}}} B'_{m,j} \right|}{Q_m/h} \geq \min\{h\alpha+\beta_{m,j}, 1\} - \oh_{s \to \infty}(1).
    \end{equation*}
    By Lemma \ref{lem: density calculuation mod h} (with $q = 1$), we conclude that if
    \begin{equation} \label{eq: local density of A+B not too big}
        |(A + B - n_s+a_{n_s}-m) \cap \{0,1,\ldots,H_s-1\}| \leq (\alpha + \beta + \oh_{s \to \infty}(1)) H_s,
    \end{equation}
    then $\min_{0 \leq j \leq h-1} \beta_{m,j} = \oh_{s \to \infty}(1)$.
    In this case, taking $b_m$ such that $\beta_{m,b_m} = \oh(1)$, it follows that $B$ has a local decomposition $(B-m) \cap \{0,1,\ldots,H_s-1\} = (B_{m,0} \cup B_{m,1}) - b_m$ with $|B_{m,0}| = \oh(H_s)$ and $B_{m,1} \subseteq \{0,1,\ldots,H_s-1\} \setminus h\Z$ with $|B_{m,1}| = \left( 1 - \frac{1}{h} - \oh(1) \right)H_s$.
    But \eqref{eq: local density of A+B not too big} holds for all but $\oh(N_s)$ many $m \in \{N_{s-1}+1,\ldots,N_s\}$ (see the argument in Case 2 of the proof of \cref{cor: Kneser without local obstructions}), so we have the desired local structure of $B$ for all but $\oh(N_s)$ many $m \in \{N_{s-1}+1, \ldots, N_s\}$.

    \bigskip

    To summarize our progress so far, we have shown that locally on intervals of length $H_s$, $A$ is nearly contained in a single residue class mod $h$, and $B$ looks like the union of $h-1$ residue classes.
    However, so far we have allowed for the possibility that the specific residue classes involved vary with the choice of interval of length $H_s$.
    We now seek to rule out this possibility, first for $A$ and then for $B$.
    
    Suppose for contradiction that $A$ contains two elements $a_1, a_2 \in A$ with $a_1 \not\equiv a_2 \pmod{h}$.
    Then from the local description of the set $B$, we have
    \begin{equation*}
        d_{\mathbf{N}}(A+B) \geq d_{\mathbf{N}}(\{a_1,a_2\}+B) = 1,
    \end{equation*}
    which contradicts the assumption $d_{\mathbf{N}}(A+B) = \alpha + \beta < 1$.
    Thus, $A \subseteq h\N - a_0$ for some $a_0 \in \{0,\ldots,h-1\}$.
    
    Recall that $B_0 = (B+b_0) \cap h\N$.
    Let $B_1 = (B+b_0) \setminus h\N$ so that $B = (B_0 \cup B_1) - b_0$.
    First we claim that $\underline{d}_{\mathbf{N}}(A+B_1) \geq 1 - \frac{1}{h} = \beta$.
    To see this, note that
    \begin{equation*}
        \underline{d}_{\mathbf{N}}(A+B_1) \geq \liminf_{s \to \infty} \frac{\left| (A \cap [N_s]) + (B \cap [H_s]) \right|}{N_s}.
    \end{equation*}
    Since $\mathbf{H} \succ \mathbf{H}_{G,A}$ and $d(A) = \alpha > 0$, we have that for all but $\oh(N_s)$ many $n \in [N_s]$, there exists $t_n \in \N$ with $t_n = \oh(H_s)$ such that $n + t_n \in A$.
    Adding the elements $n + t_n$ to $B \cap [H_s]$, we have
    \begin{equation*}
        \liminf_{s \to \infty} \frac{\left| (A \cap [N_s]) + (B \cap [H_s]) \right|}{N_s} \geq \beta.
    \end{equation*}
    Since $A + B_0$ is disjoint from $A + B_1$, we deduce that $\overline{d}_{\mathbf{N}}(A+B_0) = d_{\mathbf{N}}(A+B) - \underline{d}_{\mathbf{N}}(A+B_1) \leq \alpha$.
    But $A+B_0$ contains a shifted copy of $A$, so $d_{\mathbf{N}}(A+B_0) = \alpha$.
    We will use this observation to show $d_{\mathbf{N}}(B_0) = 0$, whence $B \sim_{\mathbf{N}} (\N \setminus h\N) - b_0$.

    Suppose for contradiction that $d_{\mathbf{N}}(B_0) = \delta > 0$.
    Let $A' = (A+a_0)/h$ and $B' = B_0/h$.
    Note that $A+B_0 = h(A'+B') - a_0$, so $d_{\mathbf{N}/h}(A'+B') = h \cdot d_{\mathbf{N}}(A+B_0) = h\alpha$.
    Similarly, $d_{\mathbf{N}/h}(A') = h\alpha$ and $d_{\mathbf{N}/h}(B') = h\delta$.
    Then for each $s \in \N$, we may apply Lemma \ref{lem: subinterval with Schnirelmann density} to find $m_s \leq (1 - \frac{h\delta}{2}) N_s/h$ such that $B'$ has Schnirelmann density at least $\frac{h\delta}{2} + \oh(1)$ on the interval $\{m_s+1,\ldots,\lfloor N_s/h \rfloor\}$.
    By Lemma~\ref{lem: Schnirelmann density at Gowers scales}, there exists $n_s = \oh(N_s)$ such that $A'$ has Schnirelmann density $h\alpha + \oh(1)$ on $\{n_s+1,\ldots, n_s + N_s/h - m_s\}$.
    Moreover, since $A' + 1 \sim_{\mathbf{N}/h} A'$, we may pick $n_s$ to itself be an element of $A'$.
    Observe that for a fixed $b' \in B'$,
    \begin{align*}
        (A'+B') & \cap [N_s/h] \\
        \supseteq &~\underbrace{\left( b' + (A' \cap [N_s/h-b']) \right)}_{\sim A' \cap [N_s/h]} \\
        & \cup \underbrace{\left( B' \cap \{m_s+1, \ldots,N_s/h\} + A' \cap \{n_s, n_s+1,\ldots,n_s + \lfloor N_s/h \rfloor - m_s\} \right) \cap [N_s/h]}_{\text{density}~\geq h\alpha + \frac{h\delta}{2}(1 - h\alpha) - \oh(1)~\text{in}~\{n_s+m_s+1,\ldots,n_s+\lfloor N_s/h \rfloor\}~\text{by \cref{cor: Schnirelmann up to N}}},
    \end{align*}
    so $d_{\mathbf{N}/h}(A'+B') \geq h\alpha + \frac{h^2\delta^2}{4}(1 - h\alpha)$.
    But $\alpha + \beta < 1$, so $h\alpha < 1$ and we obtain a strict inequality $d_{\mathbf{N}/h}(A'+B') > h\alpha$, which is a contradiction.
\end{proof}

\begin{Corollary} \label{cor: local inverse theorem upgraded}
    Let $A, B \subseteq \N$ with $d(A) = \alpha > 0$ and $d(B) = \beta > 0$ such that $\alpha + \beta < 1$.
    Suppose $B$ meets every residue class in $\N$.
    Let $\mathbf{N} = (N_s)_{s \in \N}$ be an increasing sequence with $\lim_{s \to \infty} N_s/N_{s+1} = 0$ such that $d_{\mathbf{N}}(A+B) = \alpha + \beta$.
    Let $\mathbf{H}_{G,A} = (H_{G,A,s})_{s \in \N}$ with $1 \prec \mathbf{H}_{G,A} \prec \mathbf{N}$ be a Gowers scale for $A$ as guaranteed by \cref{thm: existence of U1 G and HK scales}\ref{itm: existence of U1 G scales} and $\mathbf{H}_{G,B} = (H_{G,B,s})_{s \in \N}$ with $1 \prec \mathbf{H}_{G,B} \prec \mathbf{N}$ a Gowers scale for $B$.
    Let $\mathbf{H} = (H_s)_{s \in \N}$ be a $U^2$~good scale for $A$ and $B$ with $\mathbf{H}_{G,A}, \mathbf{H}_{G,B} \prec \mathbf{H} \prec \mathbf{N}$.
    Then, after passing to a subsequence of $(\mathbf{N},\mathbf{H})$, there exists $h \in \N$ and $b_0 \in \{0,1,\ldots,h-1\}$ such that one of the following holds:
    \begin{enumerate}[label=(\arabic*'), leftmargin=*]
        \item\label{itm local inverse theorem - upgraded 1}
            There exists $\theta \in \T$ and sequences $k_s, C_s \to \infty$ and $\epsilon_s \in (-1/2,1/2]$ with $\epsilon_s \to 0$ and $\min_{q \in [k_s]} \|q(\theta+\epsilon_s)\|_{\T} \geq \frac{C_s}{H_s}$ such that for all $s \in \N$ and all but $\oh(N_s)$ many $n \in \{N_{s-1}+1,\ldots,N_s\}$, there exists a decomposition
            \begin{equation*}
                (A - n) \cap \{0,1,\ldots,H_s-1\} = (A_n \cup E_n) - a_n
            \end{equation*}
            and
            \begin{equation*}
                B \cap \{0,1,\ldots,H_s-1\} = (B_{s,0} \cup B_{s,1} \cup F_s) - b_0
            \end{equation*}
            such that $A_n, B_{s,0} \subseteq H$, $a_n \in \{0,1,\ldots,h-1\}$, $B_{s,1} \subseteq \N \setminus H$ with $|B_{s,1}| = \left( 1 - \frac{1}{h} - \oh(1) \right)H_s$, $|E_n| = \oh(H_s)$,  $|F_s| = \oh(H_s)$, and there are closed intervals $I_n, J_s \subseteq \T$ such that if $\phi_s : h\N \to \T$ is the map $\phi_s(m) = m(\theta+\epsilon_s)$ for $m \in H$, then
            \[
                A_n \subseteq \phi_s^{-1}(I_n), \quad B_{s,0} \subseteq \phi_s^{-1}(J_s),
            \]
            and
            \[
                |\phi_s^{-1}(I_n) \cap \{0,1,\ldots,H_s-1\} \setminus A_n|, |\phi_s^{-1}(J_s) \cap \{0,1,\ldots,H_s-1\} \setminus B_{s,0}| = \oh(H_s).
            \]
            Moreover, $\alpha = \frac{\alpha_0}{h}$ and $\beta = 1 - \frac{1}{h} + \frac{\beta_0}{h}$ for some $\alpha_0, \beta_0 \in (0,1)$ with $\alpha_0 + \beta_0 < 1$, and $|I_n| = \alpha_0 + \oh(1)$ and $|J_s| = \beta_0 + \oh(1)$.
        \item\label{itm local inverse theorem - upgraded 2}
            There exist $a_0 \in \{0,1,\ldots,h-1\}$ such that $A \subseteq h\N - a_0$, $A + h \sim_{\mathbf{N}} A$, and $B \sim_{\mathbf{N}} (\N \setminus h\N) - b_0$.
    \end{enumerate}
\end{Corollary}

\begin{proof}
    We apply \cref{thm: local inverse theorem}.
    For every $s \in \N$, either \ref{itm local inverse theorem 1} or \ref{itm local inverse theorem 2} from \cref{thm: local inverse theorem} holds.
    We may pass to a subsequence such that $b_s$ is a constant $b_0$ and either \ref{itm local inverse theorem 1} holds for every $s \in \N$ or \ref{itm local inverse theorem 2} holds for every $s \in \N$.

    Suppose \ref{itm local inverse theorem 2} holds for every $s \in \N$.
    Then by \cref{prop: local to global residues}, we get the upgraded conclusion \ref{itm local inverse theorem - upgraded 2}.

    Suppose \ref{itm local inverse theorem 1} holds for every $s \in \N$.
    We may assume that \ref{itm local inverse theorem 2} does not hold.
    That is, $\lim_{s \to \infty} \frac{|B_{s,0}|}{H_s} \ne 0$.
    Therefore, $\beta_0 = 1 - h(1-\beta) > 0$.
    We now pass to a further subsequence so that the limit $\theta = \lim_{s \to \infty} \theta_s$ exists and put $\epsilon_s = \theta_s-\theta$.
    The property \ref{itm local inverse theorem - upgraded 1}, save the final sentence beginning with ``moreover,'' then follows from \ref{itm local inverse theorem 1}.
    For the ``moreover'' statement, we have already shown $\beta_0 \in (0,1)$.
    Letting $\alpha_0 = h\alpha$, we have $\alpha_0 > 0$ and, since $\alpha + \beta < 1$,
    \begin{equation*}
        \alpha_0 + \beta_0 = h\alpha + 1 - h(1-\beta) = 1 - h(1 - \alpha - \beta) < 1.
    \end{equation*}
    It remains only to show that we may take the intervals $I_n$ and $J_s$ to satisfy $|I_n| = \alpha_0 + \oh(1)$ and $|J_s| = \beta_0 + \oh(1)$.
    We will show $|I_n| = \alpha_0 + \oh(1)$.
    The condition $|J_s| = \beta_0 + \oh(1)$ follows by a similar argument.
    Since $\mathbf{H} \succ \mathbf{H}_{G,A}$, we have that $|\left( (A - n) \cap \{0,1,\ldots,H_s-1\} \right)| = \left( \alpha + \oh(1) \right) H_s$ for all but $\oh(N_s)$ many $n \in \{N_{s-1}+1,\ldots,N_s\}$, so
    \begin{equation*}
        \left| \phi_s^{-1}(I_n) \cap \{0,1,\ldots,H_s-1\} \right| = (\alpha+\oh(1))H_s.
    \end{equation*}
    Letting $\tilde{\phi}_s(m) = \phi_s(hm) = hm\theta_s$, we thus have
    \begin{equation*}
        \left| \tilde{\phi}_s^{-1}(I_n) \cap \{0,1,\ldots,\lfloor (H_s-1)/h \rfloor\} \right| = (\alpha_0 + \oh(1)) \frac{H_s}{h}.
    \end{equation*}
    Note that $\tilde{\phi}_s^{-1}(I_n) = \Bohr(h\theta_s, I_n)$, so we may assume $|I_n| = \delta_n$ for $\delta_n = d(\Bohr(h\theta_s, I_n))$.
    Let $k'_s = \min \left\{ k_s, \lfloor e^{\sqrt{C_s}} \rfloor \right\}$ so that $\lim_{s \to \infty} k'_s = \infty$ and $\lim_{s \to \infty} \frac{\log k'_s}{C_s} = 0$.
    Applying \cref{lem: discrepancy for irrational Bohr set} with $Q = k'_s/h$, we have
    \begin{align*}
        |\alpha_0 - \delta_n| & = \left| \frac{\left| \tilde{\phi}_s^{-1}(I_n) \cap \{0,1,\ldots,\lfloor (H_s-1)/h \rfloor\} \right|}{\lfloor (H_s-1)/h \rfloor +1} - \delta_n \right| + \oh(1) \\
        & \ll  \frac{h}{k'_s} + \frac{\log(k'_s/h)}{\left( \lfloor (H_s-1)/h \rfloor +1 \right) C_s/H_s} + \oh(1) \\
        & = \frac{h}{k'_s} + h \frac{\log(k'_s)}{C_s} + \oh(1) = \oh(1),
    \end{align*}
    so $|I_n| = \delta_n = \alpha_0 + \oh(1)$ as desired.
\end{proof}

%SECTION
\section{Eliminating rational frequencies}
\label{sec_eliminating_rational_frequencies}

In this section, we refine the conclusion of \cref{cor: local inverse theorem upgraded}, showing that in case \ref{itm local inverse theorem - upgraded 1}, the frequency $\theta$ must be irrational.

\begin{Proposition} \label{prop: no rational frequencies}
    Let $A, B \subseteq \N$ with $d(A) = \alpha > 0$ and $d(B) = \beta > 0$ such that $\alpha + \beta < 1$.
    Suppose $B$ meets every residue class in $\N$.
    Let $\mathbf{N} = (N_s)_{s \in \N}$ be an increasing sequence with $\lim_{s \to \infty} N_s/N_{s+1} = 0$ such that $d_{\mathbf{N}}(A+B) = \alpha + \beta$.
    Let $\mathbf{H}_{G,A} = (H_{G,A,s})_{s \in \N}$ with $1 \prec \mathbf{H}_{G,A} \prec \mathbf{N}$ be a Gowers scale for $A$ as guaranteed by \cref{thm: existence of U1 G and HK scales}\ref{itm: existence of U1 G scales} and $\mathbf{H}_{G,B} = (H_{G,B,s})_{s \in \N}$ with $1 \prec \mathbf{H}_{G,B} \prec \mathbf{N}$ a Gowers scale for $B$.
    Let $\mathbf{H} = (H_s)_{s \in \N}$ be a $U^2$~good scale for $A$ and $B$ with $\mathbf{H}_{G,A}, \mathbf{H}_{G,B} \prec \mathbf{H} \prec \mathbf{N}$.
    If $A$ and $B$ satisfy \ref{itm local inverse theorem - upgraded 1} from \cref{cor: local inverse theorem upgraded}, then $\theta \notin \Q$.
\end{Proposition}

\begin{proof}
    We prove \cref{prop: no rational frequencies} in several steps.
    Let us fix the notation for the proof.
    In addition to the notation in the statement of \cref{prop: no rational frequencies}, we adopt the notation from \cref{cor: local inverse theorem upgraded}\ref{itm local inverse theorem - upgraded 1}.
    In summary, we have the following:
    \begin{itemize}
        \item sets $A, B \subseteq \N$ with $d(A) = \alpha > 0$, $d(B) = \beta > 0$ and $\alpha + \beta < 1$, and $B$ meets every residue class in $\N$,
        \item an increasing sequence $\mathbf{N} = (N_s)_{s \in \N}$ with $\lim_{s \to \infty} N_s/N_{s+1} = 0$ such that $d_{\mathbf{N}}(A+B) = \alpha + \beta$,
        \item Gowers scales $\mathbf{H}_{G,A}$ and $\mathbf{H}_{G,B}$ for $A$ and $B$ respectively as given by \cref{thm: existence of U1 G and HK scales}\ref{itm: existence of U1 G scales},
        \item a $U^2$ good scale $\mathbf{H}$ for $A$ and $B$ with $\mathbf{H}_{G,A}, \mathbf{H}_{G,B} \prec \mathbf{H} \prec \mathbf{N}$,
        \item a natural number $h \in \N$ with $h < \frac{1}{1-\beta} < \frac{1}{\alpha}$,
        \item an element $b_0 \in \{0,1,\ldots,h-1\}$
        \item positive real numbers $\alpha_0, \beta_0 \in (0,1)$ satisfying $\alpha_0 = h\alpha$, $\beta_0 = 1 - h(1-\beta)$, and $\alpha_0 + \beta_0 < 1$,
        \item an element $\theta \in \T$,
        \item sequences $(k_s)_{s \in \N}, (C_s)_{s \in \N}$ with $\lim_{s \to \infty} k_s = \lim_{s \to \infty} C_s = \infty$,
        \item a sequence $\epsilon_s \in (-1/2,1/2]$ with $\epsilon_s \to 0$ such that $\min_{q\in[k_s]} \|q(\theta+\epsilon_s)\|_{\T} \geq \frac{C_s}{H_s}$,
        \item for $s \in \N$, a decomposition $B \cap \{0,1,\ldots,H_s-1\} = (B_{s,0} \cup B_{s,1} \cup F_s) - b_0$ such that $|F_s| = \oh(H_s)$, $B_{s,1} \subseteq \N \setminus h\N$ with $|B_{s,1}| = \left( 1 - \frac{1}{h} - \oh(1) \right)H_s$, and taking $\phi_s(m) = m(\theta+\epsilon_s)$, there exists a closed interval $J_s \subseteq \T$ of length $|J_s| = \beta_0 + \oh(1)$ such that $B_{s,0} \subseteq \phi_s^{-1}(J_s)$ and
        \begin{equation*}
            \left| \left( \phi_s^{-1}(J_s) \cap h\Z \cap \{0,1,\ldots,H_s-1\} \right) \setminus B_{s,0} \right| = \oh(H_s),
        \end{equation*}
        and
        \item for $s \in \N$ and all but $\oh(N_s)$ many $n \in \{N_{s-1}+1,\ldots,N_s\}$, an element $a_n \in \{0,1,\ldots,h-1\}$ and an interval $I_n \subseteq \T$ of length $|I_n| = \alpha_0 + \oh(1)$ such that
        \begin{equation*}
            \left| \left( (A - n + a_n) \cap \{0,1,\ldots,H_s-1\} \right) \triangle \left( \phi_s^{-1}(I_n) \cap h\Z \cap \{0,1,\ldots,H_s-1\} \right) \right| = \oh(H_s).
        \end{equation*}
    \end{itemize}
    
    Note that for an interval $I \subseteq \T$, we may write $\phi_s^{-1}(I) \cap h\Z = h \cdot \Bohr(h(\theta+\epsilon_s), I)$, which will be useful for applying estimates from \cref{sec: discrepancy}.
    We will write
    \begin{equation*}
        A_n = \left( h \cdot \Bohr(h(\theta+\epsilon_s),I_n) - a_n \right) \cap \{0,1,\ldots,H_s-1\}
    \end{equation*}
    for the set approximating $(A-n) \cap \{0,1,\ldots,H_s-1\}$.

    \bigskip

    Let us suppose for contradiction that $\theta$ is rational.
    Write $h\theta$ in reduced terms as $h\theta = \frac{p}{q}$.
    Observe that $\|q\epsilon_s\|_{\T} = \|q(\theta+\epsilon_s)\|_{\T} \geq \frac{C_s}{H_s}$, so $|\epsilon_s| = \|\epsilon_s\|_{\R/\Q} \geq \frac{C_s}{qH_s}$.
    In particular, $\lim_{s \to \infty} |\epsilon_s| H_s = \infty$.

    \bigskip
    
    \textbf{Claim 1. $\alpha_0 < \frac{1}{q}$.}

    We prove the claim by contradiction, considering the two cases $\alpha_0 > \frac{1}{q}$ and $\alpha_0 = \frac{1}{q}$ separately.
    
    Suppose $\alpha_0 > \frac{1}{q}$.
    Choose $T \in \N$ such that $B \cap [T]$ meets every residue class mod $qh$.
    By Corollary \ref{cor: local periodicity for nearly rational frequency}, for all large $s$ (large enough so that $(q+2T)h|\epsilon_s| < \alpha_0$), the set $\Bohr(h(\theta+\epsilon_s),I_n)$ contains a full residue class mod $q$ in every interval of length at most $2T$.
    Therefore, $A_n = h \cdot \Bohr(h(\theta+\epsilon_s),I_n) - a_n$ contains a full residue class mod $qh$ in every interval of length at most $2hT$.
    Therefore,
    \begin{equation*}
        \left( A_n \cap \{x,x+1,\ldots,x+2T\} \right) + (B \cap [T]) \supseteq \{x+T,x+T+1,\ldots,x+2T\}
    \end{equation*}
    for every $x \in [H_s - 2T]$, so $A_n + (B \cap [T]) \supseteq \{T, T+1,\ldots, H_s\}$.
    It follows that $d_{\mathbf{N}}(A+(B\cap[T])) = 1$, which contradicts the assumption $d_{\mathbf{N}}(A+B) = \alpha + \beta < 1$.

    If $\alpha_0 = \frac{1}{q}$, a similar argument applies but with a small modification.
    Since $|I_n| = \frac{1}{q} + \oh(1)$, the set $\Bohr(h(\theta+\epsilon_s),I_n)$ can be partitioned into intervals of length on the order of $|\epsilon_s|^{-1}$ corresponding to levels sets of $m \mapsto \lfloor mq\epsilon_s - \min(I_n) \rfloor$ such that on each interval, $\Bohr(h(\theta+\epsilon_s),I_n)$ is equal to a single residue class mod $q$ (up to boundedly many elements).
    Adding $B \cap [T]$ to $A_n$ therefore produces a set with density $1 - \Oh(T|\epsilon_s|)$ in every interval of length at least $|\epsilon_s|^{-1}$, so we again conclude $d_{\mathbf{N}}(A+(B\cap[T])) = 1$, which is a contradiction.

    This proves Claim 1.

    \bigskip

    \textbf{Claim 2. $\beta_0 \geq 1 - \frac{1}{q}$.}

    Let $K_s = \min\left\{ \lfloor |\epsilon_s|^{-1/2} \rfloor, \lfloor C_s^{1/2} \rfloor \right\}$.
    Then
    \begin{equation*}
        \lim_{s \to \infty} K_s = \infty \quad \text{and} \quad
        \lim_{s \to \infty} |\epsilon_s| K_s = \lim_{s \to \infty} \frac{K_s}{|\epsilon_s| H_s} = 0.
    \end{equation*}
    Put $M_s = \lfloor (|\epsilon_s|K_s)^{-1} \rfloor$ and $L_s = 4h\lceil |\epsilon_s|^{-1}K_s \rceil$.
    Subdivide $[H_s/h] = J_{s,1} \cup \ldots \cup J_{s,\ell_s}$ into intervals of lengths between $\frac{L_s}{4h}$ and $\frac{L_s}{2h}$.
    By construction,
    \begin{equation*}
        K_s M_s \leq |\epsilon_s|^{-1} \leq K_s^{-1} \frac{L_s}{4h} \leq K_s^{-1} |J_{s,i}|
    \end{equation*}
    and
    \begin{equation*}
        q^2 |\epsilon_s| + \frac{q}{K_s} = \oh(1),
    \end{equation*}
    so by Corollary \ref{cor: local periodicity for nearly rational frequency}, for a $q\alpha_0 - \oh_{s \to \infty}(1)$ proportion of points $x \in E_{n,i} \subseteq J_{s,i}$, the set $\Bohr(h(\theta+\epsilon_s),I_n) \cap \{x+1,\ldots,x+M_s\}$ consists of a single residue class mod $q$ for each $i \in [\ell_s]$.
    Let $E_n = \bigcup_{i=1}^{\ell_s} E_{n,i} \subseteq [H_s/h]$.
    Then $|E_n| = (q\alpha_0 + \oh(1))H_s$ and $A_n$ consists of a single residue class mod $qh$ in each of the intervals $\{y+1,\ldots,y+hM_s\}$ for $y \in E'_n = \{hx-a_n + r : x \in E_n, r \in \{0,1,\ldots,h-1\}\}$.
    Let $A'_n = \bigcup_{y \in E'_n} (A_n \cap \{y+1,\ldots,y+hM_s\})$.
    Note that
    \begin{align*}
        \frac{|A'_n|}{H_s} & = \frac{1}{H_s} \sum_{x=1}^{H_s} \frac{|A'_n \cap \{x+1,\ldots,x+hM_s\}|}{hM_s} + \oh_{s \to \infty}(1) \\
         & \geq \frac{1}{H_s} \sum_{y \in E'_n} \frac{1}{qh} + \oh_{s \to \infty}(1) \\
         & = \alpha + \oh_{s \to \infty}(1).
    \end{align*}
    We conclude that $|A_n \setminus A'_n| = \oh_{s \to \infty}(H_s)$.

    Now let $T \in \N$ such that $B \cap [T]$ meets every residue class mod $qh$.
    Then
    \begin{equation*}
        \left| \big( (n + A'_n) + (B \cap [T]) \big) \setminus \big( \left( A + B \right) \cap \{n+1,\ldots,n+H_s\} \big) \right| = \oh(H_s),
    \end{equation*}
    and
    \begin{align*}
        A'_n + (B \cap [T])
        & = \bigcup_{y \in E'_n} \left( (A_n \cap \{y+1,\ldots,y+hM_s\}) + (B \cap [T]) \right) \\
        & \supseteq \bigcup_{y \in E'_n} \{y+T+1,\ldots,y+hM_s\} \\
        & = T + E'_n + [hM_s-T].
    \end{align*}
    Let $X_n = T + E'_n + [hM_s-T]$, and let $Y_n = [H_s] \setminus X_n$.
    Note that
    \begin{align*}
        \frac{|X_n|}{H_s} & = \frac{1}{H_s} \sum_{x=1}^{H_s} \frac{|X_n \cap \{x+1,\ldots,x+hM_s\}|}{hM_s} + \oh(1) \\
         & \geq \frac{1}{H_s} \sum_{y \in E'_n} \frac{hM_s-T}{hM_s} + \oh(1) \\
         & = qh\alpha + \oh(1).
    \end{align*}
    
    We will now show that $A+B - n$ has density at least $\beta + \oh(1)$ in $Y_n$ for all but $\oh(N_s)$ many $n \in \{N_{s-1}+1,\ldots,N_s\}$.
    Since $\lim_{s \to \infty} L_s/H_s = 0$, we have
    \begin{equation*}
        \left| \big( (n + A_n) + (B \cap [L_s]) \big) \setminus \big( \left( A + B \right) \cap \{n+1,\ldots,n+H_s\} \big) \right| = \oh(H_s),
    \end{equation*}
    so it suffices to compute the density of $A_n + (B \cap [L_s])$ in $Y_n$.
    Let $W \in \N$ be large.
    We may write $Y_n$ as a union of maximal intervals and obtain a decomposition $Y_n = Y_{n,1} \cup Y_{n,2}$, where $Y_{n,1}$ consists of intervals of length less than $W$ in $Y_n$ and $Y_{n,2}$ consists of intervals of length at least $W$ in $Y_n$.
    Since the complement of $Y_n$ is the set $X_n$, which is a union of intervals of length $hM_s-T$, we have
    \begin{equation*}
        \frac{|Y_{n,1}|}{H_s} \leq \frac{W}{hM_s-T+W} = \oh(1),
    \end{equation*}
    so the set $Y_{n,1}$ is negligible.
    Now, by construction, $E'_n$ contains a point in every interval of length $L_s$ in $[H_s]$, so the intervals making up $Y_{n,2}$ have lengths between $W$ and $L_s$.
    Suppose $\{u+1, \ldots, u+S\} \subseteq Y_{n,2}$ is a maximal interval in $Y_{n,2}$.
    Then $u \in X_n \subseteq A'_n + (B \cap [T])$, so there exists $v \in A'_n \subseteq A_n$ with $u-T \leq v \leq u$.
    Therefore,
    \begin{align*}
        \left| \left( A_n + (B\cap [L_s]) \right) \cap \right. & \left. \{u+1,\ldots,u+S\} \right| \\
         & \geq \left| \left( v + (B\cap [L_s])\right) \cap \{u+1,\ldots,u+S\} \right| \\
         & = |B \cap [L_s] \cap \{u-v+1,\ldots,u-v+S\}| \\
         & \geq \min_{\substack{x \in \Z, \\ 0 \leq x \leq T}} |B \cap \{x+1,\ldots,x+S\}| - T.
    \end{align*}
    It follows that the density of $A_n + (B\cap [L_s])$ in $Y_{n,2}$ is at least
    \begin{equation*}
        \frac{\left| \left( A_n + (B\cap [L_s]) \right) \cap Y_{n,2} \right|}{|Y_{n,2}|}
        \geq \min_{\substack{x \in \Z, \\ 0 \leq x \leq T}} \min_{\substack{S \in \Z \\ W \leq S \leq L_s}} \frac{|B \cap \{x+1,\ldots,x+S\}|}{S} - \frac{T}{W}
    \end{equation*}
    Replacing $W$ by a slowly growing sequence $W_s \to \infty$ (for example, $W_s = \lfloor \sqrt{M_s} \rfloor$), we can retain the property $|Y_{n,1}| = \oh(H_s)$, while also having
    \begin{align*}
        \frac{\left| \left( A_n + (B\cap [L_s]) \right) \cap Y_{n,2} \right|}{|Y_{n,2}|}
        & \geq \min_{\substack{x \in \Z, \\ 0 \leq x \leq T}} \min_{\substack{S \in \Z \\ W_s \leq S \leq L_s}} \frac{|B \cap \{x+1,\ldots,x+S\}|}{S} - \oh(1)\\
        & \geq \min_{\substack{x \in \Z, \\ 0 \leq x \leq T}} \underline{d}(B-x) - \oh(1) \\
        & = \beta - \oh(1).
    \end{align*}
    Therefore, $\underline{d}_{\mathbf{N}}(A+B) \geq qh\alpha + \beta(1-qh\alpha) = \beta + qh\alpha(1-\beta)$.
    Since $\underline{d}_{\mathbf{N}}(A+B) = \alpha + \beta$ by assumption, we conclude $qh\alpha(1-\beta) \leq \alpha$.
    That is, $\beta \geq 1 - \frac{1}{qh}$.
    Hence, $\beta_0 = 1 - h(1-\beta) \geq 1 - \frac{1}{q}$ as claimed.

    \bigskip

    \textbf{Claim 3. $\beta_0 = 1 - \frac{1}{q}$.}

    Suppose for contradiction that $\beta_0 > 1 - \frac{1}{q}$.
    We have
    \begin{multline*}
        \left| \phi_s^{-1}(J_s) \cap h\Z \cap \{x+1,\ldots,x+m\} \right| \\
        = \left| \Bohr(h(\theta+\epsilon_s), J_s) \cap \{\lfloor x/h \rfloor + 1, \ldots, \lfloor x/h \rfloor + \lfloor m/h \rfloor\}\right| + \Oh(1)
    \end{multline*}
    Therefore, by \cref{lem: large discrepancy for rational Bohr set}, if $\mathbf{M}' = (M'_s)_{s \in \N}$ and $\mathbf{H}' = (H'_s)_{s \in \N}$ satisfy
    \begin{equation*}
        \lim_{s \to \infty} |\epsilon_s| M'_s = 0 \quad \text{and} \quad \lim_{s \to \infty} M'_s = \lim_{s \to \infty} |\epsilon_s| H'_s = \infty,
    \end{equation*}
    then
    \begin{equation*}
        \sup_{x \in \N} \left| \frac{|\phi_s^{-1}(J_s) \cap h\Z \cap \{x+1,\ldots,x+H'_s\}|}{H'_s} - \frac{\beta_0}{h} \right| = \oh(1)
    \end{equation*}
    and
    \begin{equation} \label{eq: non-Gowers scale}
        \frac{1}{H'_s} \sum_{x=1}^{H'_s} \left| \frac{|\phi_s^{-1}(J_s) \cap h\Z \cap \{x+1,\ldots,x+M'_s\}|}{M'_s} - \frac{\beta_0}{h} \right| \geq \frac{\|q\beta_0\|_{\T}}{qh} - \oh(1).
    \end{equation}
    Now, by \cref{thm: existence of U1 G and HK scales}\ref{itm: existence of U1 G scales} applied to the function $f = \1_B - \beta$, let $(V_N)_{N \in \N}$ with $\lim_{N \to \infty} V_N = \lim_{N \to \infty} \frac{N}{V_N} = \infty$ such that if $(W_N)_{N \in \N}$ satisfies $\lim_{N \to \infty} \frac{W_N}{V_N} = \lim_{N \to \infty} \frac{N}{W_N} = \infty$, then
    \begin{equation} \label{eq: Gowers scale}
        \lim_{N \to \infty} \frac{1}{N} \sum_{n=1}^N \left| \frac{|B \cap \{n+1,\ldots,n+w_N\}|}{W_N} - \beta \right| = 0.
    \end{equation}
    We adjust the scales $\mathbf{M}'$ and $\mathbf{H}'$ above to satisfy the additional property
    \begin{equation*}
        \lim_{s \to \infty} \frac{M'_s}{V_{H'_s}} = \infty.
    \end{equation*}
    From the description of the set $A$ and the property $\lim_{s \to \infty} |\epsilon_s|H'_s = \infty$, we see that $\mathbf{H}'$ is a Gowers scale for $A$, so \cref{thm: local inverse theorem} still applies at the scale $\mathbf{H}'$ to give a decomposition of $B \cap \{0,1,\ldots,H'_s-1\}$.
    Thus, by \eqref{eq: non-Gowers scale},
    \begin{multline*}
        \lim_{s \to \infty} \frac{1}{H'_s} \sum_{h=1}^{H'_s} \left| \frac{|B \cap \{x+1,\ldots,x+M'_s\}|}{M'_s} - \beta \right| \\
        = \lim_{s \to \infty} \frac{1}{H'_s} \sum_{h=1}^{H'_s} \left| \frac{|B_{s,0} \cap \{x+1,\ldots,x+M'_s\}|}{M'_s} - \frac{\beta_0}{h} \right|
        \geq \frac{\|q\beta_0\|_{\T}}{qh} > 0.
    \end{multline*}
    This contradicts \eqref{eq: Gowers scale}.

    \bigskip

    A consequence of Claim 3 is that $q \geq 2$ (since $\beta_0 > 0$).

    Now, as in the proof of \cref{prop: local to global residues}, we may find $n_s \in \{N_{s-1}+1,\ldots,N_s\}$ with $n_s = \oh(N_s)$ such that $|((A-n_s) \cap \{0,1,\ldots,H_s-1\}) \triangle A_n| = \oh(H_s)$ and then apply \cref{thm: local sumset} to find, for all but $\oh(N_s)$ many $m \in \{N_{s-1}+1,\ldots,N_s\}$, a number $Q_m \in h\N$ with $H_s \leq Q_m \leq (1+\oh(1))H_s$ and sets $\tilde{A}_m \subseteq (A-n_s+a_{n_s}) \cap \{0,1,\ldots,Q_m-1\}$ and $\tilde{B}_m \subseteq (B-m) \cap \{0,1,\ldots,Q_m-1\}$ such that
    \begin{enumerate}[label=(\Roman{enumi}),ref=(\Roman{enumi}),leftmargin=*]
        \item\label{itm no rationals I} $|(A-n_s+a_{n_s}) \cap \{0,1,\ldots,Q_m-1\} \setminus \tilde{A}_m| = \oh(Q_m)$,
        \item\label{itm no rationals II} $|(B-m) \cap \{0,1,\ldots,Q_m-1\} \setminus \tilde{B}_m| = \oh(Q_m)$,
        \item\label{itm no rationals III} $|(\tilde{A}_m +_{\bmod{Q_m}} \tilde{B}_n) \setminus ((A+B-n_s+a_{n_s}-m)\cap \{0,1,\ldots,Q_m-1\})| = \oh(Q_m)$, and
        \item\label{itm no rationals IV} $\tilde{A}_m \subseteq h\Z \cap \{0,1,\ldots,Q_m-1\}$.
    \end{enumerate}
    
    Note that by \ref{itm no rationals I},
    \begin{equation*}
        |\tilde{A}_m \triangle (A_{n_s}+a_{n_s})| = \oh(Q_m).
    \end{equation*}
    The set $\frac{A_{n_s} + a_{n_s}}{h} = \Bohr(h(\theta+\epsilon_s),I_n)$ has positive density in every residue class mod $d \leq D_s$ for some $D_s \to \infty$.
    Indeed, for $d \in \N$ and $r \in \{0,1,\ldots,d-1\}$, we may write
    \begin{equation*}
        \Bohr(h(\theta+\epsilon_s),I_n) \cap (d\Z + r) = d \cdot \Bohr(dh(\theta+\epsilon_s),I_n) + r,
    \end{equation*}
    and this set has positive density in every interval of length large compared with $d|\epsilon_s|^{-1}$ by \cref{lem: large discrepancy for rational Bohr set}.
    Thus, we can take, for example, $D_s = \lfloor \sqrt{|\epsilon_s|H_s} \rfloor$.
    Since the intersection $\Bohr(h(\theta+\epsilon_s),I_n) \cap \{0,1,\ldots,\lfloor H_s/h \rfloor-1\} \cap (d\Z+r)$ has positive density in $\{0,1,\ldots,\lfloor H_s/h \rfloor-1\}$, we conclude that $\tilde{A}_m \cap (hd\Z+hr) \ne \emptyset$ for every $d \leq D_s$ and $r \in \{0,1,\ldots,d-1\}$.

    By \ref{itm no rationals III} and the assumption $d_{\mathbf{N}}(A+B) = \alpha + \beta$, we must have
    \begin{equation*}
        |\tilde{A}_m +_{\bmod{Q_m}} \tilde{B}_m| \leq (\alpha + \beta + \oh(1)) Q_m
    \end{equation*}
    for all but $\oh(1)$ many $m \in \{N_{s-1}+1,\ldots,N_s\}$.
    We can therefore apply the inverse theorem for sumsets in finite cyclic groups to deduce structural information about $\tilde{A}_m$ and $\tilde{B}_m$.
    
    As preparation for applying the inverse theorem, we decompose $\tilde{B}_m = \bigcup_{j=0}^{h-1} \tilde{B}_{m,j}$ with $\tilde{B}_{m,j} = \{x \in \tilde{B}_m : x \equiv j \pmod{h}\}$.
    Using \ref{itm no rationals II}, we have $|\tilde{B}_{m,j}| \geq (\frac{\beta_0}{h} - \oh(1))Q_m = \frac{1}{h} \left( 1 - \frac{1}{q} - \oh(1) \right) Q_m$ for each $j \in \{0,1,\ldots,h-1\}$.
    Let $A'_m = \frac{\tilde{A}_m}{h}$ and $B'_{m,j} = \frac{\tilde{B}_{m,j}-j}{h}$.
    We then have
    \begin{equation*}
        |\tilde{A}_m +_{\bmod{Q_m}} \tilde{B}_m| = \sum_{j=0}^{h-1} \left| A'_m +_{\bmod{\frac{Q_m}{h}}} B'_{m,j} \right|.
    \end{equation*}
    But $A'_m$ has nonempty intersection with every subgroup of index at most $D_s/h$, so by \cref{thm: cyclic group Kneser},
    \begin{equation*}
        \frac{|\tilde{A}_m +_{\bmod{Q_m}} \tilde{B}_m|}{Q_m} \geq \frac{1}{h} \sum_{j=0}^{h-1} \min\{\alpha_0+ \beta_{m,j},1\} - \oh(1),
    \end{equation*}
    where $\beta_{m,j} = \frac{B'_{m,j}}{Q_m/h}$.
    Combined with the inequality $\frac{|\tilde{A}_m +_{\bmod{Q_m}} \tilde{B}_m|}{Q_m} \leq \alpha+\beta + \oh(1)$, \cref{lem: density calculuation mod h} implies that $\min_{0 \leq j \leq h-1} \beta_{m,j} = 1 - \frac{1}{q} + \oh(1)$.
    Let $j_m \in \{0,1,\ldots,h-1\}$ such that $\beta_{m,j_m} = 1 - \frac{1}{q} + \oh(1)$.
    For $j \ne j_m$, we then have $\beta_{m,j} = 1 - \oh(1)$ since $\frac{1}{h} \sum_{j=0}^{h-1} \beta_{m,j} = \beta + \oh(1) = 1 - \frac{1}{qh} + \oh(1)$.
    We can now update our bound on
    \begin{equation*}
        \left| A'_m +_{\bmod{\frac{Q_m}{h}}} B'_{m,j} \right|
    \end{equation*}
    by
    \begin{equation*}
        \left| A'_m +_{\bmod{\frac{Q_m}{h}}} B'_{m,j} \right| = (1 - \oh(1)) \frac{Q_m}{h}
    \end{equation*}
    for $j \ne j_m$ and
    \begin{equation*}
        \left| A'_m +_{\bmod{\frac{Q_m}{h}}} B'_{m,j_m} \right| = (\alpha_0 + \beta_{m,j_m} + \oh(1)) \frac{Q_m}{h}.
    \end{equation*}
    (The lower bound comes from \cref{thm: cyclic group Kneser} and the upper bound is then imposed by  the condition $|\tilde{A}_m +_{\bmod{Q_m}} \tilde{B}_m| \leq (\alpha+\beta+\oh(1))Q_m$.)

    We may now apply \cref{thm: cyclic sum without obstructions} to conclude that
    \begin{equation*}
        |B'_{m,j_m} \triangle (\Bohr(h(\theta+\epsilon_s,J'_m)) \cap \{0,1,\ldots,Q_m-1\})| = \oh(Q_m)
    \end{equation*}
    for some interval $J'_m \subseteq \T$ of length $1 - \frac{1}{q}$.
    (The given structure of $A'_m$ as agreeing up to zero density with a Bohr set precludes the other possibilities enumerated in \cref{thm: cyclic sum without obstructions}.)

    We now finish by arguing as in Claim 1.
    As we have shown, the set $(B - m) \cap \{0,1,\ldots,H_s-1\}$ has density $1 - \oh(1)$ in $h-1$ residue classes mod $h$.
    In the remaining residue class, $(B - m) \cap \{0,1,\ldots,H_s-1\}$ is approximately equal to a translate of $h \cdot \Bohr(h(\theta+\epsilon_s),J'_m)$.
    Using the description of $A$, we may choose $x,y \in A$ such that $x \equiv y \pmod{h}$ and $x \not\equiv y \pmod{qh}$.
    (It is important here that $q \geq 2$.)
    Since the interval $J'_m$ has length $|J'_m| = 1 - \frac{1}{q} - \oh(1)$, the set
    \begin{equation*}
        \{x,y\} + h \cdot \Bohr(h(\theta+\epsilon_s),J'_m)
    \end{equation*}
    has density $1 - \oh(1)$ in the multiples of $h$ in every interval whose length is long compared with $|\epsilon_s|^{-1}$.
    Since $|\epsilon_s|H_s \to \infty$, we thus have
    \begin{equation*}
        |\{x,y\} + ((B-m) \cap \{0,1,\ldots,H_s-1\})| = (1 - \oh(1))H_s,
    \end{equation*}
    so $d_{\mathbf{N}}(A+B) \geq d_{\mathbf{N}}(\{x,y\} + B) = 1$.
    This is a contradiction, and the proof is complete.
\end{proof}

%SECTION
\section{Local ergodicity} \label{sec: local ergodicity}

We will now use the description of the sets $A$ and $B$ in case \ref{itm local inverse theorem - upgraded 1} of \cref{cor: local inverse theorem upgraded} (with $\theta \notin \Q$, as provided by \cref{prop: no rational frequencies}) to deduce that $\1_A - \alpha$ and $\1_B - \beta$ are locally ergodic.
Our argument relies on the estimates from \cref{sec: discrepancy}.

Throughout this section, we fix sets $A$ and $B$ satisfying \ref{itm local inverse theorem - upgraded 1} in \cref{cor: local inverse theorem upgraded} for a given frequency $\theta$.

\begin{Proposition} \label{prop: local ergodicity of A irrational frequency}
    If $\theta \notin \Q$, then $\1_A - \alpha$ is locally ergodic along $\mathbf{N}$.
    Moreover, if $q \in \Q \setminus \frac{\Z}{h}$, then $n \mapsto \1_A(n+t_1) \ldots \1_A(n +t_k) e(nq)$ is locally ergodic along $\mathbf{N}$ for every $k \in \N$ and $t_1, \ldots, t_k \in \N$.
\end{Proposition}

\begin{proof}
    For $s \in \N$ and $n \in \{N_{s-1}+1,\ldots,N_s\}$ for which we have local structural information about $A$ at $n$, let $A_n = (\phi_s^{-1}(I_n) \cap h\Z) - a_n = h \cdot \Bohr(h(\theta + \epsilon_s), I_n) - a_n$.
    
    Fix $Q \in \N$.
    Since $\theta \notin \Q$, we have
    \begin{equation*}
        \delta(Q) = \min_{1 \leq q \leq Q} \left\| q\theta \right\| > 0
    \end{equation*}
    and
    \begin{equation*}
        \min_{1 \leq q \leq Q} \left\| q(\theta+\epsilon_s) \right\| \geq \frac{\delta(Q)}{2}
    \end{equation*}
    for all sufficiently large $s$.
    Hence, if $M \geq \frac{Q \log{Q}}{\delta(Q)}$, then by Lemma \ref{lem: discrepancy for irrational Bohr set} and the estimate $|I_n| = \alpha_0 + \oh(1) = h\alpha + \oh(1)$ from above, we have
    \begin{align*}
        \sup_{x \in \N} \Bigg| \frac{1}{M} \sum_{m=1}^M & \1_{A_n}(x+m) - \alpha \Bigg| \\
        & = \sup_{x \in \N} \left| \frac{|(h \cdot \Bohr(h(\theta+\epsilon_s),I_n)) \cap \{x+1,\ldots,x+M\}|}{M} - \alpha \right| \\
        & = \sup_{y \in \N} \left| \frac{|\Bohr(h(\theta+\epsilon_s),I_n) \cap \{y + 1, \ldots, y + \lfloor M/h \rfloor\}|}{M} - \alpha \right| + \Oh \left( \frac{1}{M} \right) \\
        & = \frac{1}{h} \cdot \sup_{y \in \N} \left| \frac{|\Bohr(h(\theta+\epsilon_s),I_n) \cap \{y + 1, \ldots, y + \lfloor M/h \rfloor\}|}{\lfloor M/h \rfloor} - h\alpha \right| + \Oh \left( \frac{1}{M} \right) \\
        & \ll \frac{1}{Q} + \frac{\log{Q}}{\lfloor M/h \rfloor \delta(Q)/2} + \frac{1}{M} + \oh(1) \\
        & \ll \frac{1}{Q} + \oh(1).
    \end{align*}
    Therefore,
    \begin{multline*}
        \frac{1}{N_s} \sum_{n=1}^{N_s} \left| \frac{1}{M} \sum_{m=1}^{M} \1_A(n+m) - \alpha \right|
         = \frac{1}{H_s} \sum_{x=1}^{H_s} \frac{1}{N_s} \sum_{n=1}^{N_s} \left| \frac{1}{M} \sum_{m=1}^{M} \1_A(n+m+x) - \alpha \right| + \Oh \left( \frac{H_s}{N_s} \right) \\
         = \frac{1}{H_s} \sum_{x=1}^{H_s} \frac{1}{N_s} \sum_{n=1}^{N_s} \left| \frac{1}{M} \sum_{m=1}^{M} \1_{A_n}(x+m) - \alpha \right| + \oh(1) \ll \frac{1}{Q} + \oh(1).
    \end{multline*}
    Taking a limit first as $s \to \infty$ and then as $Q \to \infty$, we have
    \begin{equation*}
        \limsup_{M \to \infty} \limsup_{s \to \infty} \frac{1}{N_s} \sum_{n=1}^{N_s} \left| \frac{1}{M} \sum_{m=1}^{M} \1_A(n+m) - \alpha \right| = 0.
    \end{equation*}
    That is, $\1_A - \alpha$ is locally ergodic along $\mathbf{N}$.

    Fix $k \in \N$, $t_1, \ldots, t_k \in \N$, and $q \in \Q \setminus \frac{\Z}{h}$.
    Similarly to the above, we may approximate
    \begin{equation*}
        f(m) = \1_A(m+t_1) \ldots \1_A(m+t_k) e(mq)
    \end{equation*}
    locally (for $m \in \{n, n+1,\ldots,n+H_s-1\}$ with $n \in \{N_{s-1}+1,\ldots,N_s\}$) by
    \begin{equation*}
        f_n(m) = \1_{A_n}(m-n+t_1) \ldots \1_{A_n}(m-n+t_k) e((m-n)q) e(nq).
    \end{equation*}
    We can write
    \begin{equation*}
        f_n(n-a_n+x) = e((n-a_n)q) e(xq)\prod_{j=1}^k \1_{h\N}(x+t_j) \1_{I_n}((x+t_j)(\theta+\epsilon_s)).
    \end{equation*}
    Let $d \in \N$ be minimal such that $dq \in \Z$.
    By assumption, $d \nmid h$.
    Given $M \in \N$, we may decompose $[M]$ into residue classes mod $d$ to estimate
    \begin{align*}
        \left| \frac{1}{M} \sum_{m=1}^M f_n(n-a_n+x+m) \right|
         & = \left| \frac{1}{d} \sum_{r=0}^{d-1} \frac{1}{M/d} \sum_{m=1}^{M/d} f_n(n-a_n+x+md+r) \right| + \Oh \left( \frac{d}{M} \right) \\
         & = \left|  \frac{1}{d} \sum_{r=0}^{d-1} e(rq) \mathcal{E}_r(x,M/d) \right| + \Oh \left( \frac{d}{M} \right)
    \end{align*}
    for $x \in \{a_n,a_n+1,\ldots,a_n+H_s-M-1\}$,
    where
    \begin{equation*}
        \mathcal{E}_r(M/d) = \frac{1}{M/d} \sum_{m=1}^{M/d} \prod_{j=1}^k \1_{h\N}(x+md+r+t_j) \1_{I_n}((x+md+r+t_j)(\theta+\epsilon_s)).
    \end{equation*}
    If there exist $j_1, j_2$ such that $t_{j_1} \not\equiv t_{j_2} \pmod{h}$, then
    \begin{equation*}
        \1_{h\N}(x+md+r+t_{j_1}) \1_{h\N}(x+md+r+t_{j_2}) = 0
    \end{equation*}
    for all $m$, so $\mathcal{E}_r(x,M/d) = 0$.
    Suppose $t_j \equiv t \pmod{h}$ for every $j \in \{1,\ldots,k\}$.
    Then we can write
    \begin{equation*}
        \prod_{j=1}^k \1_{h\N}(x+md+r+t_j) = \1_{h\N-r-t-x}(md)
    \end{equation*}
    Let $\ell = \gcd(d,h)$ and $d' = \frac{d}{\ell}$ and $h' = \frac{h}{\ell}$.
    Since $d \nmid h$, we have $d' \geq 2$.
    Also,
    \begin{equation*}
        \1_{h\N-r-t-x}(md) = \begin{cases}
            \1_{h'\N-\frac{r-t-x}{\ell}}(md'), & \text{if}~\ell \mid r-t; \\
            0, & \text{if}~\ell \nmid r-t
        \end{cases}
    \end{equation*}
    Therefore, if $\ell \nmid r-t-x$, we have $\mathcal{E}_r(x,M/d) = 0$.
    Let $r \in \{0,1,\ldots,d-1\}$ such that $\ell \mid r-t-x$.
    Then taking $r' \equiv \frac{r-t-x}{\ell} (d')^{-1} \pmod{h'}$, we have
    \begin{align*}
        \mathcal{E}_r(M/d) & = \frac{1}{M/d} \sum_{m=1}^{M/d} \1_{h'\N-\frac{r-t-x}{\ell}}(md') \prod_{j=1}^k \1_{I_n}((x+md+r+t_j)(\theta+\epsilon_s)) \\
         & = \frac{1}{M/d} \sum_{m=1}^{M/d} \1_{h'\N-r'}(m) \prod_{j=1}^k \1_{I_n}((x+md+r+t_j)(\theta+\epsilon_s)) \\
         & = \frac{1}{M/d} \sum_{m=1}^{M/(dh')} \prod_{j=1}^k \1_{I_n}((x+mdh' - dr' + r + t_j)(\theta+\epsilon_s)) + \Oh \left( \frac{dh'}{M} \right) \\
         & = \frac{1}{M/d} \sum_{m=1}^{M/(dh')} \1_{I_{n,r}}(x+mdh'(\theta+\epsilon_s)) + \Oh \left( \frac{dh'}{M} \right),
    \end{align*}
    where
    \begin{equation*}
        I_{n,r} = \bigcap_{j=1}^k \left( I_n +(dr'-r-t_j)(\theta+\epsilon_s) \right).
    \end{equation*}
    Note that if we let
    \begin{equation*}
        I'_n = \bigcap_{j=1}^k (I_n - t_j(\theta+\epsilon_s)),
    \end{equation*}
    then $I_{n,r} = I'_n + (dr'-r)(\theta+\epsilon_s)$, so the length $|I_{n,r}| = |I'_n|$ does not depend on $r$.
    Applying \cref{lem: discrepancy for irrational Bohr set} as above with $M \geq \frac{Q\log{Q}}{\delta(Q)}$, we have
    \begin{equation*}
        \left| \mathcal{E}_r(x,M/d) - h'|I'_n| \right| \ll \frac{1}{Q}.
    \end{equation*}
    Putting everything together,
    \begin{align*}
        \left| \frac{1}{M} \sum_{m=1}^M f_n(n-a_n+m) \right|
         & \ll \left| \frac{1}{d} \sum_{r=0}^{d-1} e(rq) \mathcal{E}_r(M/d) \right| + \frac{1}{M} \\
         & = \frac{1}{\ell} \left| \frac{1}{d'} \sum_{r'=0}^{d'-1} e(r'\ell q) \mathcal{E}_{r'\ell+t}(M/d) \right| + \frac{1}{M} \\
         & \ll \frac{h'|I'_n|}{\ell} \left| \frac{1}{d'} \sum_{r'=0}^{d'-1} e(r'\ell q) \right| + \frac{1}{Q}.
    \end{align*}
    But $\ell q$ is a rational number with denominator $d' \geq 2$, so $\sum_{r'=0}^{d'-1} e(r'\ell q) = 0$.
    Thus,
    \begin{equation*}
        \lim_{s \to \infty} \frac{1}{N_s} \sum_{n=1}^{N_s} \left| \frac{1}{M} \sum_{m=1}^{M} f(n+m) \right|
         = \lim_{s \to \infty} \frac{1}{H_s} \sum_{x=1}^{H_s} \frac{1}{N_s} \sum_{n=1}^{N_s} \left| \frac{1}{M} \sum_{m=1}^{M} f_n(n-a_n+x+m) \right|
         \ll \frac{1}{Q},
    \end{equation*}
    so $f$ is locally ergodic along $\mathbf{N}$.
\end{proof}

\begin{Proposition} \label{prop: local ergodicity of B irrational frequency}
    If $\theta \notin \Q$, then $\1_B - \beta$ is locally ergodic along $\mathbf{N}$.
\end{Proposition}

\begin{proof}
    By \cref{prop: local ergodicity of A irrational frequency}, every $U^1$ good scale for $A$ satisfies $\|\1_A - \alpha\|_{U^1(\mathbf{N},\mathbf{H})} = 0$, so we may freely adjust the scale $\mathbf{H}$.
    By taking $\mathbf{K} = (K_s)_{s \in \N}$ to be sufficiently quickly growing and letting $\mathbf{N}' = \mathbf{N}_{\mathbf{K}}$, the pair $(\mathbf{N}',\mathbf{N})$ is a $U^2$~good scale for $A$ by \cref{lem: slow scales are U2 good}.
    Letting this scale play the role of $(\mathbf{N},\mathbf{H})$, we consider the resulting decomposition
    \begin{equation*}
        B \cap [N_s] = (B_{s,0} \cup B_{s,1} \cup F_s) - b_0
    \end{equation*}
    as provided by \cref{thm: local inverse theorem}.
    (The pair $(\mathbf{N},\mathbf{H})$ is not necessarily a Gowers scale for $B$ \emph{a priori}, so we cannot apply \cref{cor: local inverse theorem upgraded} and instead use \cref{thm: local inverse theorem}.)
    We have
    \begin{multline*}
        \limsup_{M \to \infty} \limsup_{s \to \infty} \frac{1}{N_s} \sum_{n=1}^{N_s} \left| \frac{1}{M} \sum_{m=1}^{M} \1_B(n+m) - \beta \right| \\
         \leq \underbrace{\limsup_{M \to \infty} \limsup_{s \to \infty} \frac{1}{N_s} \sum_{n=1}^{N_s} \left| \frac{1}{M} \sum_{m=1}^{M} \1_{B_{s,0}}(n+m) - \frac{\beta_0}{h} \right|}_{[1]} \\
         + \underbrace{\limsup_{M \to \infty} \limsup_{s \to \infty} \frac{1}{N_s} \sum_{n=1}^{N_s} \left| \frac{1}{M} \sum_{m=1}^{M} \1_{B_{s,1}}(n+m) - \left( 1 - \frac{1}{h} \right) \right|}_{[2]}.
    \end{multline*}
    The term [1] is equal to zero by the same argument as in the proof of \cref{prop: local ergodicity of A irrational frequency} using $\theta \notin \Q$, and [2] is also zero because $|([N_s] \setminus h\N) \setminus B_{s,1}| = \oh(N_s)$.
\end{proof}

%SECTION
\section{Completing the proof with ergodic theory}

Combining the results of the preceding sections, we can prove \cref{thm: ergodic or h}, reproduced below.

\begin{named}{\cref{thm: ergodic or h}}{}
Let $A, B \subseteq \N$ with $d(A) = \alpha > 0$ and $d(B) = \beta > 0$ such that $\alpha + \beta < 1$.
Suppose $B$ meets every residue class in $\N$.
Let $\mathbf{N} = (N_s)_{s \in \N}$ be an increasing sequence with $\lim_{s \to \infty} N_s = \infty$ such that $d_{\mathbf{N}}(A+B) = \alpha + \beta$.
Then there exists a subsequence $\mathbf{N}'$ of $\mathbf{N}$ such that either
\begin{enumerate}[label=(\arabic*), leftmargin=*]
    \item\label{itm ergodic or h 1 - second appearance} $\1_A - \alpha$ and $\1_B - \beta$ are locally ergodic along $\mathbf{N}'$ and there exists $h \in \N$ with $h < \frac{1}{1-\beta}$ such that $n \mapsto \1_A(n+t_1) \ldots \1_A(n+t_k) e(nq)$ is locally ergodic along $\mathbf{N}'$ for every $q \in \Q \setminus \frac{\Z}{h}$, every $k \in \N$, and every $t_1, \ldots, t_k \in \N$, or
    \item\label{itm ergodic or h 2 - second appearance} there exist %a subsequence $\mathbf{N}'$ of $\mathbf{N}$, and 
    integers $h \geq 2$ and $a_0, b_0 \in \{0,1,\ldots,h-1\}$ such that $A \subseteq h\N - a_0$, $A + h \sim_{\mathbf{N}'} A$, and $B \sim_{\mathbf{N}'} (\N \setminus h\N) - b_0$.
\end{enumerate}
\end{named}

\begin{proof}
    By \cref{thm: existence of U1 G and HK scales}, let $\mathbf{H}_{G,A}$ and $\mathbf{H}_{G,B}$ be Gowers scales for $A$ and $B$ respectively.
    Define $\mathbf{H}_G = (H_{G,s})_{s \in \N}$ by $H_{G,s} = \max\{H_{G,A,s}, H_{G,B,s}\}$.
    Then by \cref{thm:U^2 structure theorem}, there is a subsequence $(\mathbf{N}',\mathbf{H}_G')$ of $(\mathbf{N},\mathbf{H}_G)$ such that there exist scales $\mathbf{H}^-, \mathbf{H}^+$ such that $\mathbf{H}_G' \preceq \mathbf{H}^- \prec \mathbf{H}^+ \preceq \mathbf{N}'$ and for every $\mathbf{H}$ with $\mathbf{H}^- \prec \mathbf{H} \prec \mathbf{H}^+$, the pair $(\mathbf{N}',\mathbf{H})$ is a $U^2$ good scale for $A$.
    Applying \cref{thm:U^2 structure theorem} again, after replacing $(\mathbf{N}', \mathbf{H}^+,\mathbf{H}^-)$ by a subsequence, there exists $\mathbf{H}$ with $\mathbf{H}^- \prec \mathbf{H} \prec \mathbf{H}^+$ such that $(\mathbf{N}',\mathbf{H})$ is a $U^2$ good scale for $B$.
    Hence, $(\mathbf{N}',\mathbf{H})$ is a $U^2$ good scale and a Gowers scale for both $A$ and $B$.

    By passing to a further subsequence, we may assume $\lim_{s \to \infty} N'_s/N'_{s+1} = 0$.
    Now by \cref{cor: local inverse theorem upgraded}, there exists $h \in \N$ such that either \ref{itm local inverse theorem - upgraded 1} or \ref{itm local inverse theorem - upgraded 2} holds.
    The property \ref{itm local inverse theorem - upgraded 2} gives conclusion \ref{itm ergodic or h 2 - second appearance}.
    Suppose \ref{itm local inverse theorem - upgraded 1} holds.
    Then by \cref{prop: no rational frequencies}, the frequency $\theta$ is irrational.
    \cref{prop: local ergodicity of A irrational frequency} and \cref{prop: local ergodicity of B irrational frequency} then imply that \ref{itm ergodic or h 1 - second appearance} holds.
\end{proof}

In order to prove \cref{thm: main}, it remains only to prove the following, which is a restatement of \cref{thm: 1_A is locally ergodic}.

\begin{Theorem} \label{thm: locally ergodic irrational Bohr}
    Let $A, B \subseteq \N$ with $d(A) = \alpha > 0$ and $d(B) = \beta > 0$ such that $\alpha + \beta < 1$.
    Suppose $B$ meets every residue class in $\N$.
    Let $\mathbf{N} = (N_s)_{s \in \N}$ be an increasing sequence with $\lim_{s \to \infty} N_s/N_{s+1} = 0$ such that $d_{\mathbf{N}}(A+B) = \alpha + \beta$.
    Suppose $\1_A - \alpha$ and $\1_B - \beta$ are locally ergodic along $\mathbf{N}$ and there exists $h \in \N$ with $h < \frac{1}{1-\beta}$ such that $n \mapsto \1_A(n+t_1) \ldots \1_A(n+t_k) e(nq)$ is locally ergodic along $\mathbf{N}$ for every $q \in \Q \setminus \frac{\Z}{h}$, every $k \in \N$, and every $t_1, \ldots, t_k \in \N$.
    Then there exist $h \in \N$, $a_0, b_0 \in \{0,1,\ldots,h-1\}$, an irrational $\theta \in \T$, and closed intervals $I, J \subseteq \T$ such that if
    \begin{equation*}
        A_0 = A + a_0, \quad B_0 = (B + b_0) \cap h\N, \quad \text{and} \quad B_1 = (B+b_0) \setminus h\N,
    \end{equation*}
    then $A_0 \subseteq h\N$, $B_1 \sim \N \setminus h\N$, $A_0 \sim \{n \in h\N : n\theta \in I\}$, and $B_0 \sim \{n \in h\N : n\theta \in J\}$.
\end{Theorem}

The local ergodicity assumption on $\1_A - \alpha$ and $\1_B - \beta$ enables us to transfer the problem into a dynamical setting, where we apply tools from ergodic theory.
For an overview of this strategy, see \cref{sec: proof outline}.

%SUBSECTION
\subsection{Dynamical model}
\label{sec_dyn_model}

As a first step in the dynamical approach to proving \cref{thm: locally ergodic irrational Bohr}, we produce dynamical systems modeling the sets $A$, $B$, and $A+B$.

Given an invertible measure-preserving system $(X, \mu, T)$, a measurable subset $E \subseteq X$, and a set $D \subseteq \N$, we define the sumset $D + E$ by
\begin{equation*}
	D + E = \bigcup_{d \in D} T^dE.
\end{equation*}

\begin{Theorem} \label{thm: Furstenberg system}
	There exist ergodic invertible measure-preserving systems $(X_A, \mu_A, T_A)$ and $(X_B, \mu_B, T_B)$, transitive points $x_A \in X_A$ and $x_B \in X_B$, clopen sets $E_A \subseteq X_A$ and $E_B \subseteq X_B$, and continuous factor maps $\pi_A : X_A \to G$ and $\pi_B : X_B \to G$, where $(G, m, \theta)$ is the maximal common rotational factor, such that:
	\begin{enumerate}[label=(\arabic*),leftmargin=*]
		\item	$\pi_A(x_A) = \pi_B(x_B) = 0$,
		\item	$A = \{n \in \N : T_A^n x_A \in E_A\}$ and $B = \{n \in \N : T_B^n x_B \in E_B\}$,
		\item	$\mu_A(E_A) = \alpha$ and $\mu_B(E_B) = \beta$,
		\item	$\max\{\mu_A(B + E_A), \mu_B(A + E_B)\} \leq \alpha + \beta$, and
        \item the number $C_G$ of connected components of $G$ satisfies $C_G < \frac{1}{1-\beta}$.
	\end{enumerate}
\end{Theorem}

We prove \cref{thm: Furstenberg system} through a sequence of lemmas.
To begin, we have the following version of the Furstenberg correspondence principle.

\begin{Lemma} \label{lem: properties of Furstenberg measures}
	Let $C \subseteq \N$ be a set with positive natural density $d(C) = \gamma > 0$.
	Let $\mathbf{N} = (N_s)_{s \in \N}$ be a sequence along which $C$ admits correlations.
    Suppose $\1_C - \gamma$ is locally ergodic along $\mathbf{N}$.
	Then there exists a topological dynamical system $(Y, S)$, a clopen subset $F \subseteq Y$, a transitive point $c \in Y$, and an $S$-invariant measure $\nu$ such that $c \in \textup{gen}(\nu, \mathbf{N})$ with the following properties:
	\begin{enumerate}[label=(\arabic*),leftmargin=*]
		\item\label{itm Furstenberg measure 1}	$C = \{n \in \N : S^nc \in F\}$.
		\item\label{itm Furstenberg measure 2}	$\nu(F) = \gamma$.
		\item\label{itm Furstenberg measure 3}	for any set $D \subseteq \N$,
			\begin{equation*}
				 \nu(D+F) \leq \underline{d}_{\mathbf{N}}(D+C).
			\end{equation*}
		\item\label{itm Furstenberg measure 4}	$\mathbb{E} \left[ \1_F \mid \mathcal{I} \right] = \gamma$.
	\end{enumerate}
\end{Lemma}

\begin{proof}
    We let $(Y, \nu, S)$ be the Furstenberg system of $C$ along $\mathbf{N}$ (see \cref{sec: Furstenberg systems}).
    Let $\Phi : \mathcal{A}_C \to C(Y)$ be the Gelfand representation, where $\mathcal{A}_C \subseteq \ell^{\infty}(\Z)$ is the translation-invariant *-algebra generated by $\1_C$, and let $F$ be the clopen set defined by $\1_F = \Phi(\1_C)$.
    Let $c \in Y$ be the point such that $\Phi(f)(c) = f(0)$ for $f \in \mathcal{A}_C$.
    Then by the definition of $\nu$, we have for every $f \in C(Y)$,
    \begin{equation*}
        \int_Y f~d\nu = \E_{n \in \mathbf{N}} (\Phi^{-1}(f))(n) = \E_{n \in \mathbf{N}} f(S^nc),
    \end{equation*}
    so $\nu = \E_{n \in \mathbf{N}} \delta_{S^nc}$.
    That is, $c \in \textup{gen}(\nu,\mathbf{N})$.
	Let us check that properties \ref{itm Furstenberg measure 1}--\ref{itm Furstenberg measure 4} hold.
	
	\ref{itm Furstenberg measure 1} By the choice of the set $F$ and the point $c$, we have
	\begin{equation*}
		\{n \in \N : S^nc \in F\} = \{n \in \N : \1_F(S^nc) = 1\} = \{n \in \N : \1_C(n) = 1\} = C.
	\end{equation*}
	
	\ref{itm Furstenberg measure 2} By the definition of the measure $\nu$ and the identity $\1_F = \Phi(\1_C)$, we may compute
	\begin{equation*}
		\nu(F) = \int_Y \Phi(\1_C)~d\nu = \E_{n \in \mathbf{N}} \1_C(n) = d_{\mathbf{N}}(C) = d(C) = \gamma.
	\end{equation*}
	
	\ref{itm Furstenberg measure 3} Note that
	\begin{equation*}
		D + C = \bigcup_{d \in D} \{d + n : S^n c \in F\} = \bigcup_{d \in D} \{n : S^n c \in S^dF\} = \{n : S^n c \in D + F\}.
	\end{equation*}
	The set $D + F$ is open (it is a union of open sets), so by the portmanteau lemma,
	\begin{equation*}
		\underline{d}_{\mathbf{N}}(D+C) = \liminf_{s \to \infty} \frac{1}{N_s} \sum_{n=1}^{N_s} \1_{D+F}(S^nc) \geq \nu(D+F).
	\end{equation*}
	
	\ref{itm Furstenberg measure 4} This follows from Theorem \ref{thm: U^1 uniformity implies constant measure in ergodic components} and the assumption that $\1_C - \gamma$ is locally ergodic.
\end{proof}

We apply Lemma \ref{lem: properties of Furstenberg measures} to obtain systems $(Y_A, \nu_A, S_A)$ and $(Y_B, \nu_B, S_B)$, points $a \in Y_A$ and $b \in Y_B$, and sets $F_A \subseteq Y_A$ and $F_B \subseteq Y_B$ corresponding to $A$ and $B$ respectively.
Applying property \ref{itm Furstenberg measure 3} with $D \in \{A,B\}$, we have
\begin{equation} \label{eq: measures bounded by alpha + beta}
	\max\{\nu_A(B+F_A), \nu_B(A+F_B)\} \leq d_{\mathbf{N}}(A+B) = \alpha + \beta.
\end{equation}

We now consider the ergodic decomposition of the measure $\nu_A$, which we may write as
\begin{equation*}
	\nu_A = \int_{Y_A} \nu_{A,y}~d\nu_A(y),
\end{equation*}
where $\nu_{A,y}$ is $\sigma$-invariant and ergodic for $\nu_A$-almost every $y \in Y_A$, and
\begin{equation*}
	\mathbb{E}_{\nu_A} \left[ f \mid \mathcal{I} \right](y) = \int_{Y_A} f~d\nu_{A,y}
\end{equation*}
for every $f \in L^1(\nu_A)$ and $\nu_A$-almost every $y \in Y_A$.
By \eqref{eq: measures bounded by alpha + beta},
\begin{equation*}
	\int_{Y_A} \nu_{A,y}(B+F_A)~d\nu_A(y) = \nu_A(B+F_A) \leq \alpha + \beta,
\end{equation*}
so the set
\begin{equation*}
	G_1 = \left\{ y : \nu_{A,y}(B+F_A) \leq \alpha + \beta \right\}
\end{equation*}
satisfies $\nu_A(G_1) > 0$.
Moreover, by property \ref{itm Furstenberg measure 4} from Lemma \ref{lem: properties of Furstenberg measures},
\begin{equation*}
	\nu_{A,y}(F_A) = \int_{Y_A} \1_{F_A}~d\nu_{A,y} = \mathbb{E}_{\nu_A} \left[ \1_{F_A} \mid \mathcal{I} \right](y) = \alpha
\end{equation*}
for $\nu_A$-almost every $y \in Y_A$.
That is, the set
\begin{equation*}
	G_2 = \left\{ y : \nu_{A,y}(F_A) = \alpha \right\}
\end{equation*}
has full measure.
Finally, let $G_3$ be the full measure set of $y \in Y_A$ such that $\nu_{A,y}$ is ergodic.
Choose $y \in G_1 \cap G_2 \cap G_3$.
Then $\nu_{A,y}$ is an ergodic $S_A$-invariant measure with $\nu_{A,y}(F_A) = \alpha$ and $\nu_{A,y}(B+F_A) \leq \alpha + \beta$.
We can repeat the same argument to find an $S_B$-invariant ergodic measure $\nu_{B,y'}$ such that $\nu_{B,y'}(F_B) = \beta$ and $\nu_{B,y'}(A+F_B) \leq \alpha + \beta$.

\bigskip

We will now use Lemma \ref{lem: extension} to replace $(Y_A, \nu_{A,y}, S_A)$ and $(Y_B, \nu_{B,y'}, S_B)$ by more convenient extensions.
Let $(G, m, \theta)$ be the maximal common rotational factor of the system $(Y_A, \nu_{A,y}, S_A)$ and $(Y_B, \nu_{B,y'}, S_B)$.
By Lemma \ref{lem: extension}, let $(X_A, \mu_A, T_A)$ be an ergodic system, $x_A \in X_A$ a transitive point, and $\rho : X_A \to Y_A$ a continuous factor map such that $\rho(x_A) = a$, $\rho$ is a measurable isomorphism, and there exists a continuous factor map $\pi_A : X_A \to G$.
By conjugating the factor map $\pi_A$, we may assume for convenience that $\pi_A(x_A) = 0$.
We then let $E_A = \rho^{-1}(F_A)$.
Define an extension $(X_B, \mu_B, T_B)$ of $(Y_B, \nu_{B,y'}, S_B)$ similarly.
To complete the proof of \cref{thm: Furstenberg system}, it remains to show that the number $C_G$ of connected components of $G$ satisfies $C_G < \frac{1}{1-\beta}$.

\begin{Lemma} \label{lem: rational spectrum}
    For $\nu_A$-almost every $y \in Y_A$, the set of eigenvalues
    \begin{equation*}
        \sigma_y = \{t \in \T : \exists f \in L^2(\nu_{A,y})~\text{such that}~|f|=1~\text{and}~S_Af = e(t)f\}
    \end{equation*}
    satisfies $\sigma_y \cap \Q \subseteq \left\{0,\frac{1}{h},\ldots,\frac{h-1}{h} \right\}$.
\end{Lemma}

\begin{proof}
    Let $q \in \Q \setminus \frac{\Z}{h}$.
    By assumption, the function $n \mapsto \1_A(n+t_1) \ldots \1_A(n+t_k) e(-nq)$ is locally ergodic along $\mathbf{N}$ for every $k \in \N$ and $t_1, \ldots, t_k \in \N$.
    Hence, for every $u \in \mathcal{A}_A$, we have that $n \mapsto u(n) e(-nq)$ is locally ergodic.

    Suppose for contradiction that $\nu_A(\{y \in Y_A : q \in \sigma_y\}) > 0$.
    Fix a countable dense set $D \subseteq \mathcal{A}_A$.
    Let $f \in L^2(\nu_{A,y})$ with $|f| = 1$ such that $S_Af = e(q)f$.
    The space $C(Y_A)$ is dense in $L^2(\nu_{A,y})$, so for every $\epsilon > 0$, there exists $u \in D$ such that $\|f-\Phi(u)\|_{L^2(\nu_{A,y})} < \epsilon$.
    Therefore,
    \begin{equation*}
        \|\Phi(\tau^mu) - \Phi(e(mq)u)\|_{L^2(\nu_{A,y})} < 2\epsilon
    \end{equation*}
    for every $m \in \N$.
    By countable additivity of the measure $\nu_A$, there exists $u \in D$ such that the set
    \begin{equation*}
        S_u = \{y \in Y_A : \|\Phi(u)\|_{L^2(\nu_{A,y})} > 1- \epsilon~\text{and}~\|\Phi(\tau^m u) - \Phi(e(mq)u)\|_{L^2(\nu_{A,y})} < 2\epsilon~\text{for all}~m \in \N\}
    \end{equation*}
    has $\nu_A(S_u) > 0$.
    
    Let $\tilde{u}(n) = u(n)e(-nq)$.
    The function $\tilde{u}$ may not belong to the algebra $\mathcal{A}_A$, so let $\tilde{\mathcal{A}}$ be the translation-invariant *-algebra generated by $\mathcal{A}_A$ and $n \mapsto e(-nq)$.
    Let $(\tilde{Y}, \tilde{\nu}, \tilde{S})$ be the Furstenberg system of $\tilde{A}$ along $\mathbf{N}$, and let $\tilde{\Phi} : \tilde{\mathcal{A}} \to C(\tilde{Y})$ be the Gelfand representation.
    There is a factor map $\pi : \tilde{Y} \to Y$ given by $\pi(\tilde{y}) = \tilde{y}|_{\overline{\mathcal{A}}_A}$ when viewing points $\tilde{y} \in \tilde{Y}$ as C*-algebra homomorphisms $\tilde{y} : \overline{\tilde{\mathcal{A}}} \to \C$.
    
    For convenience, let $e_q(n) = e(-nq)$ so that $\tilde{u} = ue_q$.
    We have
    \begin{equation*}
        (\tau^m \tilde{u})(n) = \tilde{u}(n+m) = e(-mq) (\tau^mu)(n) e(-nq) = e(-mq) (e_q\tau^m u)(n),
    \end{equation*}
    so
    \begin{equation*}
        \|\tilde{\Phi}(\tau^m \tilde{u}) - \tilde{\Phi}(\tilde{u})\|_{L^2(\rho)} < 2\epsilon
    \end{equation*}
    for $y \in S_u$ and any measure $\rho$ on $\tilde{Y}$ with $\pi_*\rho = \nu_{A,y}$.
    If we let $(\tilde{\nu}_{\tilde{y}})_{\tilde{y} \in \tilde{Y}}$ be an ergodic decomposition of $\tilde{\nu}$, then $\pi_*\tilde{\nu}_{\tilde{y}} = \nu_{A,\pi(\tilde{y})}$ for $\tilde{\nu}$-almost every $\tilde{y} \in \tilde{Y}$.
    Therefore,
    \begin{equation*}
        \tilde{S}_u = \left\{ \tilde{y} \in \tilde{Y} : \|\tilde{\Phi}(\tilde{u})\|_{L^2(\tilde{\nu}_{\tilde{y}})} > 1- \epsilon~\text{and}~\|\tilde{\Phi}(\tau^m \tilde{u}) - \tilde{\Phi}(\tilde{u})\|_{L^2(\tilde{\nu}_{\tilde{y}})} < 2\epsilon~\text{for all}~m \in \N \right\}
    \end{equation*}
    satisfies $\tilde{\nu}(\pi^{-1}(S_u) \setminus \tilde{S}_u) = 0$, so $\tilde{\nu}(S_u) > 0$.
    By the discrete mean ergodic theorem (\cref{thm: discrete ergodic theorem}),
    \begin{equation*}
        \lim_{M \to \infty} \left\| \frac{1}{M} \sum_{m=1}^M \tau^m \tilde{u} - \tilde{u}_{\inv} \right\|_{2,\mathbf{N}} = 0.
    \end{equation*}
    Hence, if $\tilde{y} \in \tilde{S}_u$, we have
    \begin{equation*}
        \|\tilde{\Phi}(\tilde{u}_{\erg})\|_{L^2(\tilde{\nu}_{\tilde{y}})} = \|\tilde{\Phi}(\tilde{u}) - \tilde{\Phi}(\tilde{u}_{\inv})\|_{L^2(\tilde{\nu}_{\tilde{y}})} \leq 2 \epsilon.
    \end{equation*}
    On the other hand, $\tilde{u}$ is locally ergodic by assumption, so
    \begin{equation*}
        \|\tilde{\Phi}(\tilde{u}_{\erg})\|_{L^2(\tilde{\nu}_{\tilde{y}})} = \|\tilde{\Phi}(\tilde{u})\|_{L^2(\tilde{\nu}_{\tilde{y}})} > 1 - \epsilon
    \end{equation*}
    for $\tilde{\nu}$-almost every $\tilde{y} \in \tilde{S}_u$.
    If $\epsilon \leq \frac{1}{3}$, this gives a contradiction.
\end{proof}

We can now combine the results of this section to prove \cref{thm: Furstenberg system}.

\begin{proof}[Proof of \cref{thm: Furstenberg system}]
    Let $(X_A, \mu_A, T_A)$ and $(X_B, \mu_B, T_B)$ be the systems defined above.
    By construction, we have transitive points $x_A \in X_A$ and $x_B \in X_B$ and continuous factor maps $\pi_A : X_A \to G$ and $\pi_B : X_B \to G$ satisfying $\pi_A(x_A) = \pi_B(x_B) = 0$.

    By \cref{lem: properties of Furstenberg measures},
    \begin{equation*}
        A = \{n \in \N : S_A^ny_A \in F_A\} = \{n \in \N : T_A^nx_A \in E_A\}
    \end{equation*}
    and similarly $B = \{n \in \N : T_B^nx_B \in E_B\}$.

    The ergodic measures $\mu_A$ and $\mu_B$ were chosen so that $\mu_A(E_A) = \alpha$, $\mu_B(E_B) = \beta$, $\mu_A(B+E_A) \leq \alpha + \beta$ and $\mu_B(A+E_B) \leq \alpha + \beta$.

    Finally, by \cref{lem: rational spectrum} we may choose $\mu_A$ so that
    \begin{equation*}
        \sigma_A = \{t \in \T : \exists f \in L^2(\mu_A)~\text{such that}~|f|=1~\text{and}~T_Af = e(t)f\}
    \end{equation*}
    satisfies $\sigma_A \cap \Q \subseteq \left\{ 0, \frac{1}{h}, \ldots, \frac{h-1}{h} \right\}$.
    By the Halmos--von Neumann theorem (see \cite[p. 48]{Halmos60}), the Kronecker factor $(Z_A, m_{Z_A}, \theta_A)$ of $(X_A, \mu_A, T_A)$ is (isomorphic to) a rotation on the compact dual group of $\sigma_A$ when viewing $\sigma_A$ as a discrete group.
    The fact that all torsion elements of $\sigma_A$ are of the form $\frac{k}{h}$ for some $k \in \{0,1,\ldots,h-1\}$ implies that $Z_A = \hat{\sigma}_A$ has at most $h$ connected components.
    Since $(G,m,\theta)$ is a factor of $(Z_A, m_{Z_A},\theta_A)$, we conclude that $C_G \leq h < \frac{1}{1-\beta}$.
\end{proof}

%SUBSECTION
\subsection{Reducing to the rotational factor}
\label{sec_reducing_to_rotationl_factor}

Our next step is to replace the dynamical sumsets $B+E_A$ and $A+E_B$ with convolutions of functions on the compact group $G$, which will allow us to reduce \cref{thm: locally ergodic irrational Bohr} to \cref{thm_inverse_theorem_sumsets_compact_groups}.

Let $\varphi_A = \E_{\mu_A}[\1_{E_A} \mid G]$ and $\varphi_B = \E_{\mu_B}[\1_{E_B} \mid G]$ be the projections of $\1_{E_A}$ and $\1_{E_B}$ onto $G$.
Consider the set $S = \{x \in G : (\varphi_A * \varphi_B)(x) > 0\}$.
The following proposition describes the relationship between $S$ and the sets $A + E_B$ and $B + E_A$.
The notation $U \subseteq_{\mu} V$ means $\mu(U \setminus V) = 0$.\nomenclature{$\subseteq_{\mu},\supseteq_{\mu},=_{\mu}$}{set relations up to $\mu$-null sets}

\begin{Proposition} \label{prop: correspondence with convolution}
	The following properties hold:
	\begin{enumerate}[label=(\arabic*),leftmargin=*]
		\item\label{itm correspondence with convolution 1}	$\int_G \varphi_A~dm = \mu_A(E_A) = \alpha$ and $\int_G \varphi_B~dm = \mu_B(E_B) = \beta$,
		\item\label{itm correspondence with convolution 2}	$E_A \subseteq_{\mu_A} \pi_A^{-1} \left( \{\varphi_A > 0\} \right)$ and $E_B \subseteq_{\mu_B} \pi_B^{-1} \left( \{\varphi_B > 0\} \right)$,
		\item\label{itm correspondence with convolution 3}	$B + E_A \supseteq_{\mu_A} \pi_A^{-1}(S)$ and $A + E_B \supseteq_{\mu_B} \pi_B^{-1}(S)$, and
		\item\label{itm correspondence with convolution 4}	$m(S) \leq \alpha + \beta$.
	\end{enumerate}
\end{Proposition}

We carry out the proof in stages.
First, we recall a fundamental tool in the analysis of multiple ergodic averages: the van der Corput lemma.

\begin{Lemma}[{\cite[Theorem 1.5]{Bergelson87PET}}] \label{lem: van der Corput}
    Let $(u_n)_{n \in \N}$ be a bounded sequence in a Hilbert space.
    If the limit
    \begin{equation*}
        \lim_{N \to \infty} \frac{1}{N} \sum_{n=1}^N \innprod{u_{n+h}}{u_n}
    \end{equation*}
    exists for every $h \in \N$ and
    \begin{equation*}
        \limsup_{H \to \infty} \frac{1}{H} \sum_{h=1}^H \left| \lim_{N \to \infty} \frac{1}{N} \sum_{n=1}^N \innprod{u_{n+h}}{u_n} \right| = 0,
    \end{equation*}
    then
    \begin{equation*}
        \lim_{N \to \infty} \left\| \frac{1}{N} \sum_{n=1}^N u_n \right\| = 0.
    \end{equation*}
\end{Lemma}

Now, we utilize the van der Corput lemma to show that certain bilinear averages can be reduced to convolutions on $G$.

\begin{Lemma} \label{lem: average equals convolution}
	Let $f \in L^2(\mu_A)$ and $g \in L^2(\mu_B)$.
	Then for $(\mu_A \times \mu_B)$-a.e. $(x_1, x_2) \in X_A \times X_B$,
	\begin{equation*}
		\lim_{N \to \infty} \frac{1}{N} \sum_{n=1}^N f(T^nx_1) g(T^{-n}x_2) = (\tilde{f} * \tilde{g})(\pi_A(x_1) + \pi_B(x_2)),
	\end{equation*}
	where $\tilde{f} = \E_{\mu_A}[f \mid G]$ and $\tilde{g} = \E_{\mu_B}[g \mid G]$ are the projections of $f$ and $g$ onto $G$.
\end{Lemma}

\begin{proof}
	The average on the left hand side converges almost everywhere by the pointwise ergodic theorem.
	For ease of computation, we will establish equality with the right hand side in $L^2$.
	Let $\mathcal{Z}_A$ and $\mathcal{Z}_B$ be the Kronecker factors of $(X_A, \mu_A, T_A)$ and $(X_B, \mu, T_B)$ respectively.
	We break the proof into two main steps.
	First, we show that
	\begin{equation} \label{eq: characteristic factor}
		\lim_{N \to \infty} \frac{1}{N} \sum_{n=1}^N \left( T^nf \otimes T^{-n}g - T^n \mathbb{E}_{\mu_A}[f \mid \mathcal{Z}_A] \otimes T^{-n} \mathbb{E}_{\mu_B}[g \mid \mathcal{Z}_B] \right) = 0.
	\end{equation}
	Then we prove
	\begin{equation} \label{eq: limit formula}
		\lim_{N \to \infty} \frac{1}{N} \sum_{n=1}^N \mathbb{E}_{\mu_A}[f \mid \mathcal{Z}_A](T^nx_1) \cdot \mathbb{E}_{\mu_B}[g \mid \mathcal{Z}_B](T^{-n}x_2) = (\tilde{f} * \tilde{g})(\pi_A(x_1) + \pi_B(x_2)).
	\end{equation}
	
	For the first step, we may expand
	\begin{multline*} \label{eq: characteristic factor}
		T^nf \otimes T^{-n}g - T^n \mathbb{E}_{\mu_A}[f \mid \mathcal{Z}_A] \otimes T^{-n} \mathbb{E}_{\mu_B}[g \mid \mathcal{Z}_B] \\
		 = T^n \left( f - \mathbb{E}_{\mu_A}[f \mid \mathcal{Z}_A] \right) \otimes T^{-n}g + T^n \mathbb{E}_{\mu_A}[f \mid \mathcal{Z}_A] \otimes T^{-n} \left( g - \mathbb{E}_{\mu_B}[g \mid \mathcal{Z}_B] \right),
	\end{multline*}
	and it suffices to show that the average of each term converges to 0.
	In both cases, one of the functions is orthogonal to the Kronecker factor of the corresponding system.
	Thus, it suffices to prove: if $f$ is orthogonal to $\mathcal{Z}_A$ or $g$ is orthogonal to $\mathcal{Z}_B$, then
	\begin{equation} \label{eq: average converges to 0}
		\lim_{N \to \infty} \frac{1}{N} \sum_{n=1}^N T^nf \otimes T^{-n}g = 0.
	\end{equation}
	To this end, we use the van der Corput lemma.
	Let $u_n = T^nf \otimes T^{-n}g$.
	Then
	\begin{equation*}
		\langle u_{n+h}, u_n \rangle = \int_{X \times X} \left( T^{n+h}f \otimes T^{-n-h}g \right) \left( T^n \overline{f} \otimes T^{-n} \overline{g} \right)~d(\mu \times \mu).
	\end{equation*}
	Applying $(T \times T^{-1})$-invariance of $\mu \times \mu$, we have
	\begin{equation*}
		\langle u_{n+h}, u_n \rangle = \left( \int_X \overline{f} T^h f~d\mu \right) \left( \int_X \overline{g} T^{-h}g~d\mu \right).
	\end{equation*}
	Thus,
	\begin{equation*}
		\limsup_{H \to \infty} \frac{1}{H} \sum_{h=1}^H \left| \lim_{N \to \infty} \frac{1}{N} \sum_{n=1}^N \langle u_{n+h}, u_n \rangle \right| \leq \|g\|_2^2 \cdot \limsup_{H \to \infty} \frac{1}{H} \sum_{h=1}^H \left| \int_X \overline{f} T^h f~d\mu \right| \leq \|g\|_2^2 \cdot \|f\|_{U^2}^2
	\end{equation*}
	and similarly with the roles of $f$ and $g$ reversed.
    Therefore, by the van der Corput lemma, \eqref{eq: average converges to 0} holds.
	This proves \eqref{eq: characteristic factor}.
	
	Now let us prove \eqref{eq: limit formula}.
    We can realize the Kronecker factors as ergodic group rotations $(Z_A, m_{Z_A}, \theta_A)$ and $(Z_B, m_{Z_B}, \theta_B)$, and the factor maps to $(G,m,\theta)$ can be obtained as surjective group homomorphisms $\pi^{Z_A}_G : Z_A \to G$ and $\pi^{Z_B}_G : Z_B \to G$ such that $\pi^{Z_A}_G(\theta_A) = \theta$ and $\pi^{Z_B}_G(\theta_B) = \theta$.
    Taking $f' = \E_{\mu_A}[f \mid Z_A] \in L^2(Z_A)$ and $g' = \E_{\mu_B}[g \mid Z_B] \in L^2(Z_B)$, \eqref{eq: limit formula} reduces to showing that
    \begin{equation} \label{eq: limit formula - group rotations}
        \lim_{N \to \infty} \frac{1}{N} \sum_{n=1}^N f'(u + n\theta_A) g'(v - n\theta_B) = (\tilde{f}' * \tilde{g}')(\pi^{Z_A}_G(u) + \pi^{Z_B}_G(v))
    \end{equation}
    in $L^2(m_A \times m_B)$.
	Both sides of \eqref{eq: limit formula - group rotations} are bilinear in the functions $f'$ and $g'$.
    Since $L^2(Z_A)$ and $L^2(Z_B)$ are spanned by group characters on $Z_A$ and $Z_B$ respectively, it suffices to prove that \eqref{eq: limit formula - group rotations} holds in the case that $f'$ and $g'$ are characters.
    If $\chi \in \hat{Z}_A$, then by orthogonality of characters,
    \begin{equation*}
        \tilde{\chi} = \begin{cases}
            \lambda, & \text{if}~\chi = \lambda \circ \pi^{Z_A}_G, \lambda \in \hat{G}, \\
            0, & \text{otherwise}.
        \end{cases}
    \end{equation*}
    A similar argument applies to characters on $Z_B$.
    Now, for the convolution of characters on $G$, the orthogonality relations give for $\lambda_1, \lambda_2 \in \hat{G}$,
    \begin{equation*}
        (\lambda_1 * \lambda_2)(x) = \begin{cases}
            \lambda(x), & \text{if}~\lambda_1 = \lambda_2 = \lambda, \\
            0, & \text{if}~\lambda_1 \ne \lambda_2.
        \end{cases}
    \end{equation*}
    Thus, it suffices to show that for $\chi_A \in \hat{Z}_A$ and $\chi_B \in \hat{Z}_B$,
    \begin{equation*}
        \lim_{N \to \infty} \frac{1}{N} \sum_{n=1}^N \chi_A(n\theta_A) \chi_B(-n\theta_B) = \begin{cases}
            1, & \text{if}~\chi_A = \lambda \circ \pi^{Z_A}_G~\text{and}~\chi_B = \lambda \circ \pi^{Z_B}_G~\text{for some}~\lambda \in \hat{G}, \\
            0, & \text{otherwise}.
        \end{cases}
    \end{equation*}
    We may write
    \begin{equation*}
        \chi_A(n\theta_A)\chi_B(-n\theta_B) = \left( \frac{\chi_A(\theta_A)}{\chi_B(\theta_B)} \right)^n,
    \end{equation*}
    so we want to show $\chi_A(\theta_A) = \chi_B(\theta_B)$ if and only if there exists $\lambda \in \hat{G}$ such that $\chi_A = \lambda \circ \pi^{Z_A}_G$ and $\chi_B = \lambda \circ \pi^{Z_B}_G$.
    If $\lambda \in \hat{G}$ and $\chi_A = \lambda \circ \pi^{Z_A}_G$ and $\chi_B = \lambda \circ \pi^{Z_B}_G$, then
    \begin{equation*}
        \chi_A(\theta_A) = \lambda(\theta) = \chi_B(\theta_B).
    \end{equation*}
    Conversely, if $\chi_A(\theta_A) = \chi_B(\theta_B)$, then $\chi_A$ and $\chi_B$ are eigenfunctions of $(Z_A, m_{Z_A}, \theta_A)$ and $(Z_B, m_{Z_B}, \theta_B)$ respectively with the same eigenvalue $\chi_A(\theta_A) = \chi_B(\theta_B)$.
    Since $(G,m,\theta)$ is maximal among the common rotational factors of these two systems, we can find eigenfunctions $\psi_A, \psi_B$ of $(G,m,\theta)$ with eigenvalue $\chi_A(\theta_A) = \chi_B(\theta_A)$ such that $\chi_A = \psi_A \circ \pi^{Z_A}_G$ and $\chi_B = \psi_B \circ \pi^{Z_B}_G$.
    The eigenfunctions of $(G,m,\theta)$ are constant multiples of characters, and ergodicity of the system implies that each eigenspace is one-dimensional, so we can write $\psi_A = c_A \lambda$ and $\psi_B = c_B \lambda$ for some $\lambda \in \hat{G}$ and constants $c_A, c_B \in \C$.
    Finally, $c_A = c_A\lambda(0) = \psi_A(0) = \chi_A(0) = 1$ and similarly $c_B = 1$.
\end{proof}

\begin{Proposition} \label{prop: average equals convolution at given point}
	There exist F{\o}lner sequences $\Phi, \Psi$ such that for any $f \in C(X_A)$ and $g \in C(X_B)$, we have the following properties:
	\begin{itemize}
		\item	for almost every $x \in X_B$,
			\begin{equation*}
				\lim_{N \to \infty} \frac{1}{|\Phi_N|} \sum_{n \in \Phi_N} f(T_A^n x_A) g(T_B^{-n}x) = (\tilde{f} * \tilde{g})(\pi_B(x))
			\end{equation*}
		\item	for almost every $x \in X_A$,
			\begin{equation*}
				\lim_{N \to \infty} \frac{1}{|\Psi_N|} \sum_{n \in \Psi_N} f(T_A^{-n}x) g(T_B^nx_B) = (\tilde{f} * \tilde{g})(\pi_A(x)).
			\end{equation*}
	\end{itemize}
\end{Proposition}

\begin{proof}
	We will construct the F{\o}lner sequence $\Phi$.
	The construction of $\Psi$ is analogous.
	
	We follow the argument in \cite[Lemma 3.12]{KMRR24b}.
	By Lemma \ref{lem: average equals convolution} and Fubini's theorem, let $x_0 \in X_A$ such that for almost every $x \in X_B$,
	\begin{equation*}
		\lim_{N \to \infty} \frac{1}{N} \sum_{n=1}^N f(T_A^nx_0) g(T_B^{-n}x) = (\tilde{f} * \tilde{g})(\pi_A(x_0) + \pi_B(x))
	\end{equation*}
	for $f \in C(X_A)$ and $g \in C(X_B)$.
	
	Let $(f_i)_{i \in \N}$ and $(g_j)_{j \in \N}$ be countable dense subsets of $C(X_A)$ and $C(X_B)$ respectively.
	For $i, j \in \N$, let
	\begin{equation*}
		F_{i,j}(x_1, x_2) = (\tilde{f}_i * \tilde{g}_j)(\pi_A(x_1) + \pi_B(x_2)).
	\end{equation*}
	Then $F_{i,j} \in C(X_A \times X_B)$ is $(T_A \times T_B^{-1})$-invariant.
	
	The point $x_A \in X_A$ is transitive, so there exists a sequence $(k_m)_{m \in \N}$ with $T_A^{k_m} x_A \to x_0$.
	By refining the sequence $(k_m)_{m \in \N}$ so that the convergence occurs very quickly, we can ensure that
	\begin{equation*}
		\max_{1 \leq i, j \leq m} \|g_j\| \cdot \|f_i(x_0) - f_i(T_A^{k_m}x_A)\| \leq 2^{-m}
	\end{equation*}
	and
	\begin{equation*}
		\max_{1 \leq i, j \leq m} \sup_{x_2 \in X_B} \|F_{i,j}(x_0, x_2) - F_{i,j}(T_A^{k_m} x_A, x_2) \| \leq 2^{-m}.
	\end{equation*}
	By the choice of the point $x_0$, there exists $N_m$ such that
	\begin{equation*}
		H_m(x_2) = \max_{1 \leq i,j \leq m} \left| \frac{1}{N_m} \sum_{n=1}^{N_m} f_i(T_A^n x_0) g_j(T_B^{-n} x_2) - F_{i,j}(x_0, x_2) \right|
	\end{equation*}
	satisfies $\|H_m\|_{L^1(\mu_B)} \leq 2^{-m}$.
	By the triangle inequality, we then have that
	\begin{equation*}
		\tilde{H}_m(x_2) = \max_{1 \leq i,j \leq m} \left| \frac{1}{N_m} \sum_{n=1}^{N_m} f_i(T_A^{n+k_m} x_A) g_j(T_B^{-n} x_2) - F_{i,j}(T_A^{k_m} x_A, x_2) \right|
	\end{equation*}
	satisfies $\|\tilde{H}_m\|_{L^1(\mu_B)} \leq 3 \cdot 2^{-m}$.
	
	Let $\Phi_m = \{k_m+1, \ldots, k_m+N_m\}$.
	Then since $F_{i,j}$ is $(T_A \times T_B^{-1})$-invariant, we have
	\begin{align*}
		H'_m(x_2) & = \tilde{H}_m(T_B^{-k_m}x_2) \\
        & = \max_{1 \leq i,j \leq m} \left| \frac{1}{|\Phi_m|} \sum_{n \in \Phi_m} f_i(T_A^n x_A) g_j(T_B^{-n} x_2) - F_{i,j}(T_A^{k_m} x_A, T_B^{-k_m} x_2) \right| \\
		 &= \max_{1 \leq i,j \leq m} \left| \frac{1}{|\Phi_m|} \sum_{n \in \Phi_m} f_i(T_A^n x_A) g_j(T_B^{-n} x_2) - F_{i,j}(x_A, x_2) \right|.
	\end{align*}
	Moreover, since $\mu_B$ is $T_B$-invariant, we have $\|H'_m\|_{L^1(\mu_B)} = \|\tilde{H}_m\|_{L^1(\mu_B)} \leq 3 \cdot 2^{-m}$.
	
	Let $H = \sum_{m=1}^{\infty} H'_m$.
	Then $\|H\|_{L^1(\mu_B)} \leq \sum_{m=1}^{\infty} \|H'_m\|_{L^1(\mu_B)} < \infty$, so the set $X_0 = \{x \in X_B : H(x) < \infty\}$ has full measure.
	Suppose $x \in X_0$.
	Then $H'_m(x) \to 0$ as $m \to \infty$, so
	\begin{equation*}
		\lim_{m \to \infty} \frac{1}{|\Phi_m|} \sum_{n \in \Phi_m} f_i(T_A^n x_A) g_j(T_B^{-n} x) = F_{i,j}(x_A, x)
	\end{equation*}
	for every $i, j \in \N$.
	By density of the families $(f_i)_{i \in \N}$ and $(g_j)_{j \in \N}$, this completes the proof.
\end{proof}

We can now complete the proof of Proposition \ref{prop: correspondence with convolution}.

\begin{proof}[Proof of \cref{prop: correspondence with convolution}]
	Property \ref{itm correspondence with convolution 1} is immediate by basic properties of conditional expectation.

    \bigskip
	
	\ref{itm correspondence with convolution 2} We will prove $E_A \subseteq_{\mu_A} \pi_A^{-1} \left( \{\varphi_A > 0\} \right)$.
	The corresponding statement for $B$ follows by the same argument.
    Let $F = \{\varphi_A = 0\} \subseteq G$.
    Then by the definition of the conditional expectation,
    \begin{equation*}
        \mu_A(E_A \cap \pi_A^{-1}(F)) = \int_{\pi_A^{-1}(F)} \1_{E_A}~d\mu_A = \int_F \varphi_A~dm = 0. 
    \end{equation*}
    But $\pi_A^{-1}(\{\varphi_A > 0\}) = X \setminus \pi_A^{-1}(F)$, so we conclude that $E_A \subseteq_{\mu_A} \pi_A^{-1}(\{\varphi_A > 0\})$.

    \bigskip
	
	\ref{itm correspondence with convolution 3} We prove $B + E_A \supseteq_{\mu_A} \pi_A^{-1}(S)$.
	The other part of \ref{itm correspondence with convolution 3} holds by swapping the roles of $A$ and $B$.
	
	Let $\Psi$ be the F{\o}lner sequence given by Proposition \ref{prop: average equals convolution at given point}.
	Then for almost every $x \in X_A$,
	\begin{align*}
		\lim_{N \to \infty} \frac{1}{|\Psi_N|} \sum_{n \in \Psi_N} \1_B(n) \1_{E_A}(T_A^{-n}x) 
        & = \lim_{N \to \infty} \frac{1}{|\Psi_N|} \sum_{n \in \Psi_N} \1_{E_B}(T_B^n x_B) \1_{E_A}(T_A^{-n}x) \\
        & = (\varphi_A * \varphi_B)(\pi_A(x)).
	\end{align*}
	Thus,
	\begin{align*}
		\pi_A^{-1}(S) & = \{x \in X_A : (\varphi_A * \varphi_B)(\pi_A(x)) > 0\} \\
         & =_{\mu_A} \left\{ x \in X_A : \limsup_{N \to \infty} \frac{1}{|\Psi_N|} \sum_{n \in \Psi_N} \1_B(n) \1_{E_A}(T_A^{-n} x) > 0 \right\}.
	\end{align*}
	We want to show that the set on the right hand side is contained in $B + E_A$.
	Suppose $x \in X_A$ and
	\begin{equation*}
		\limsup_{N \to \infty} \frac{1}{|\Psi_N|} \sum_{n \in \Psi_N} \1_B(n) \1_{E_A}(T_A^{-n}x) > 0.
	\end{equation*}
	Then there exist many values of $n \in B$ such that $T_A^{-n}x \in E_A$.
	Hence, $x \in \bigcup_{t \in B} T^tE_A = B + E_A$.

    \bigskip
	
	\ref{itm correspondence with convolution 4} This is a consequence of \ref{itm correspondence with convolution 3} combined with \cref{thm: Furstenberg system}:
	\begin{equation*}
		m(S) = \mu_A(\pi_A^{-1}(S)) \leq \mu_A(B + E_A) \leq \alpha + \beta.
	\end{equation*}
\end{proof}

%SUBSECTION
\subsection{Finishing the proof}
\label{sec_finishing_the_proof}

We now want to take the set $S$ defined by convolution of functions and replace it with a sumset of measurable subsets of $G$.
This process is enabled by the following lemma.

\begin{Lemma}[{\cite[Lemma 3.15]{Griesmer26}}] \label{lem: replace convolution by sumset}
	Let $f, g : G \to [0,1]$ be measurable functions.
	Let $C = \{f > 0\}$, $D = \{g > 0\}$, and $E = \{f * g > 0\}$.
	Then there exist sets $C'$ and $D'$ such that $C' =_m C$, $D' =_m D$, and the sumset $C'+D'$ is a measurable subset of $E$.
\end{Lemma}

We apply \cref{lem: replace convolution by sumset} to the set $S$ from \cref{prop: correspondence with convolution} to obtain measurable sets $C_0, D_0 \subseteq G$ such that $C_0 =_m \{\varphi_A > 0\}$, $D_0 =_m \{\varphi_B > 0\}$, the sumet $C_0+D_0$ is measurable, and $C_0+D_0 \subseteq S$.
The sets $C_0$ and $D_0$ have the following properties:

\begin{itemize}
	\item	$m(C_0) \geq \int_G \varphi_A~dm = \alpha$ and $m(D_0) \geq \int_G \varphi_B~dm = \beta$.
	\item	$m(C_0+D_0) \leq m(S) \leq \alpha+\beta$.
	\item	$E _A \subseteq_{\mu_A} \pi_A^{-1}(C_0)$ and $E_B \subseteq_{\mu_B} \pi_B^{-1}(D_0)$.
	\item	$B + E_A \supseteq_{\mu_A} \pi_A^{-1}(C_0 + D_0)$ and $A + E_B \supseteq_{\mu_B} \pi_B^{-1}(C_0 + D_0)$.
\end{itemize}

Combining the inequalities for the measures of $C_0$, $D_0$, and $C_0+D_0$, we have $m(C_0+D_0) \leq \alpha+\beta \leq m(C_0)+m(D_0)$.
If we have a strict inequality $m(C_0+D_0) < m(C_0) + m(D_0)$, then by Kneser's theorem for compact abelian groups (\cref{thm_Kneser_for_compact_groups}), the sumset $C_0+D_0$ is contained in a finite union of cosets of a compact open subgroup $H < G$ and $m(C_0+D_0) = m(C_0+H) + m(D_0+H) - m(H)$.
The cosets of $H$ partition $G$ into disjoint open sets, so we must have $[G:H] \leq C_G < \frac{1}{1-\beta}$.
The set $D_0$ satisfies $m(D_0) \geq \beta > 1 - \frac{1}{[G:H]}$, so $D_0$ has non-empty intersection with every coset of $H$, i.e. $D_0 + H = G$.
Thus,
\begin{equation*}
    m(C_0+D_0) = m(C_0+H) + m(D_0+H) - m(H) \geq m(H) + 1 - m(H) = 1.
\end{equation*}
This contradicts the inequality $m(C_0+D_0) \leq \alpha + \beta < 1$.
Therefore, $m(C_0+D_0) = m(C_0) + m(D_0)$, which forces several inequalities to become equalities.
First, from the string of inequalities
\begin{equation*}
	\alpha + \beta \leq m(C_0) + m(D_0) = m(C_0+D_0) \leq \alpha + \beta,
\end{equation*}
we conclude $m(C_0) = \alpha$, $m(D_0) = \beta$, and $m(C_0 + D_0) = \alpha + \beta$.
We can then deduce
\begin{equation*}
	\mu_A(\pi_A^{-1}(C_0) \setminus E_A) = m(C_0) - \mu_A(E_A) = \alpha - \alpha = 0,
\end{equation*}
and similarly, $\mu_B(\pi_B^{-1}(D_0) \setminus E_B) = 0$, so we in fact have almost equalities of sets $E_A =_{\mu_A} \pi_A^{-1}(C_0)$ and $E_B =_{\mu_B} \pi_B^{-1}(D_0)$.
Moreover,
\begin{equation*}
	\mu_B\left( (A + E_B) \setminus \pi_B^{-1}(C_0+D_0) \right) = \mu_B(A+E_B) - m(C_0+D_0) \leq \alpha + \beta - (\alpha+\beta) = 0.
\end{equation*}
Similarly, $\mu_A\left( (B + E_A) \setminus \pi_A^{-1}(C_0+D_0) \right) = 0$, so $A + E_B =_{\mu_B} \pi_B^{-1}(C_0+D_0)$ and $B + E_A =_{\mu_A} \pi_A^{-1}(C_0+D_0)$.

To summarize, we have shown:
\begin{itemize}
	\item	$m(C_0) = \alpha$ and $m(D_0) = \beta$.
	\item	$m(C_0+D_0) = \alpha + \beta$.
	\item	$E _A =_{\mu_A} \pi_A^{-1}(C_0)$ and $E_B =_{\mu_B} \pi_B^{-1}(D_0)$.
	\item	$A + E_B =_{\mu_B} \pi_B^{-1}(C_0 + D_0)$ and $B + E_A =_{\mu_A} \pi_A^{-1}(C_0 + D_0)$.
\end{itemize}
Let $C = C_0 \cup (A\theta)$ and $D = D_0 \cup (B\theta)$.
Since $A$ and $B$ are countable sets, we have $C_0 =_m C$ and $D_0 =_m D$.
Moreover,
\begin{multline*}
	C + D = C_0 + D_0 \cup (A\theta + D_0) \cup (C_0 + B\theta) \cup (A\theta + B\theta) \\
    =_m (C_0 + D_0) \cup (\underbrace{A+D_0}_{=_m C_0+D_0}) \cup (\underbrace{B+C_0}_{=_m C_0+D_0}) =_m C_0+D_0.
\end{multline*}
So $C$ and $D$ satisfy:
\begin{itemize}
	\item	$m(C) = \alpha$ and $m(D) = \beta$.
	\item	$m(C+D) = m(C) + m(D)$.
	\item	$A \subseteq \{n \in \N : n\theta \in C\}$ and $B \subseteq \{n \in \N : n\theta \in D\}$.
\end{itemize}
Now we apply \cref{thm_inverse_theorem_sumsets_compact_groups} to obtain, for some finite index subgroup $H \leq G$, a decomposition
\begin{equation*}
    C = C' + x \quad \text{and} \quad D = (D'\cup D'') + y
\end{equation*}
with $C',D' \subseteq H$, $x,y \in G$, $D'' \subseteq G \setminus H$ with $m((G\setminus H) \setminus D'') = 0$ such that there exists a surjective homomorphism $\phi : H \to \T$ and closed intervals $I, J \subseteq \T$ such that
\begin{equation*}
    C' \subseteq \phi^{-1}(I), \quad D' \subseteq \phi^{-1}(J), \quad \text{and} \quad m(\phi^{-1}(I) \setminus C') = m(\phi^{-1}(J) \setminus D') = 0.
\end{equation*}

Let $k = [G:H] \le C_G < \frac{1}{1-\beta}$.
Since $(G,m,\theta)$ is ergodic, there exists $a_0, b_0 \in \{0,1,\ldots,k-1\}$ such that $a_0\theta  + x \equiv b_0\theta + y \equiv 0 \pmod{H}$.
Let $\tilde{\theta} \in \T$ such that $k\tilde{\theta}= \phi(k\theta)$.
We may then write $A = A_0 - a_0$ with
\begin{multline*}
    A_0 = A + a_0
    \subseteq \{n \in \N : (n-a_0)\theta \in C\}
    = \{n\in \N : (n-a_0)\theta \in C'+x\} \\
    \subseteq \{n \in \N : n \theta \in \underbrace{C' + x + a_0\theta}_{\subseteq H} \}
    \subseteq \{n \in k\N : n\tilde{\theta} \in I + \phi(x) + a_0 \tilde{\theta}\}.
\end{multline*}
Now,
\begin{equation*}
    B + b_0 = \{n \in \N : (n-b_0)\theta \in D\} = \{n \in \N : n\theta \in (D' \cup D'') + \underbrace{(y+b_0)}_{\in H} \theta\},
\end{equation*}
so we may write
\begin{equation*}
    B_0 = (B+b_0) \cap \{n \in \N : n\theta \in D' + (y+b_0)\theta\}
\end{equation*}
and
\begin{equation*}
    B_1 = (B+b_0) \cap \{n \in \N : n\theta \in D'' + (y+b_0)\theta\}
\end{equation*}
to get a decomposition $B = (B_0 \cup B_1) - b_0$ with
\begin{equation} \label{eq: B_0 containment}
    B_0 \subseteq \{n \in k\N : n\tilde{\theta} \in J + \phi(y) + b_0 \tilde{\theta}\}
\end{equation}
and
\begin{equation} \label{eq: B_1 containment}
    B_1 \subseteq \{n \in \N : n\theta \notin H\} = \N \setminus k\N.
\end{equation}

Let $\tilde{I} = I + \phi(x) + a_0\tilde{\theta}$ and $\tilde{J} = J + \phi(y) + b_0\tilde{\theta}$.
Then $|\tilde{I}| = |I| = m(C') = m(C) = \alpha$ and $|\tilde{J}| = |J| = m(D') = m(D) - m(G\setminus H) = \beta - \left( 1 - \frac{1}{k} \right)$.
Let $\tilde{\phi} : k\N \to \T$ be the map $\tilde{\phi}(n) = n\tilde{\theta} = \phi(n\theta)$.
Then
\begin{equation*}
    d(\tilde{\phi}^{-1}(\tilde{I}) \setminus A_0) = |\tilde{I}| - d(A) = \alpha - \alpha = 0.
\end{equation*}
From \eqref{eq: B_0 containment} and \eqref{eq: B_1 containment}, the sets $B_0$ and $B_1$ have upper density bounded by
\begin{equation*}
    \overline{d}(B_0) \leq |\tilde{J}| = \beta - \left( 1 - \frac{1}{k} \right)
\end{equation*}
and
\begin{equation*}
    \overline{d}(B_1) \leq 1 - \frac{1}{k}.
\end{equation*}
Taking complements, we then have
\begin{equation*}
    \underline{d}(B_0) = \underline{d}((B+b_0)\setminus B_1) = \beta - \overline{d}(B_1) \geq \beta - \left( 1 - \frac{1}{k} \right)
\end{equation*}
and
\begin{equation*}
    \underline{d}(B_1) = \underline{d}((B+b_0) \setminus B_0) = \beta - \overline{d}(B_0) \geq 1 - \frac{1}{k}.
\end{equation*}
Thus, $B_0$ and $B_1$ have density, and
\begin{equation*}
    d(\tilde{\phi}^{-1}(\tilde{J}) \setminus B_0) = |\tilde{J}| - d(B_0) = \beta - \left( 1 - \frac{1}{k} \right) - \left( \beta - \left( 1 - \frac{1}{k} \right) \right) = 0
\end{equation*}
and
\begin{equation*}
    d((\N \setminus k\N) \setminus B_1) = 1 - \frac{1}{k} - d(B_1) = 0.
\end{equation*}
This proves \cref{thm: locally ergodic irrational Bohr}.

%SECTION
\section{Relevant examples} \label{sec: examples}

In this section, we present instructive examples to help clarify features of our main result that may not be immediately clear otherwise.

\subsection{Examples of sets satisfying $d(A+B)=d(A)$}

In light of case~\ref{itm_main_thm_2} of \cref{thm: main}, it is natural to wonder what type of sets $A,B\subset\N$ have the property that $d(A)>0$ and $d(A+B)=d(A)$. The following proposition guarantees the existence of such pairs even under the additional restriction that $B$ meets every residue class in $\N$.

\begin{Proposition}
\label{prop_ctmn}
Let $\alpha\in(0,1)$. There exist a set $A\subset \N$ with $d(A)=\alpha$ and a set $B\subset \N$ that meets every residue class in $\N$ such that $d(A+B)=d(A)$.
\end{Proposition}

\begin{proof}
Given $t\in\N$ and $M<N\in\N$, consider the set
\[
C(t,M,N)=(t\Z+\{1,\ldots,\lfloor \alpha t\rfloor\})\cap [M,N).
\]
It is straightforward to check that
\begin{align}
\label{eqn_ctmn_p1}
 \big|C(t,M,N)\big|&= (N-M) \frac{\lfloor \alpha t\rfloor}{t}+ \Oh(t),
\\
\label{eqn_ctmn_p3}
\big|(C(t,M,N)+\{1,\ldots,b\})\triangle C(t,M,N)\big|&=\Oh\bigg(\frac{b N}{t}\bigg),\qquad\forall b,t\in\N.
\end{align}
Let $(b_i)_{i\in\N}$, $(t_i)_{i\in\N}$, and $(K_i)_{i\in\N}$ be any sequences of positive integers such that
\begin{align}
\label{eqn_ctmn_con_1}
\lim_{i\to\infty} \frac{t_{i}}{K_{i}}=0,
\\
\label{eqn_ctmn_con_2}
\lim_{i\to\infty} \frac{i^2K_i}{b_i}=0,
\\
\label{eqn_ctmn_con_3}
\lim_{i\to\infty} \frac{b_i}{t_{i+1}}=0,
\\
\label{eqn_ctmn_con_4}
b_i~\text{is a multiple of}~t_i.
\end{align}
We claim that if
\[
A=\bigcup_{i\in\N} C(t_i,K_{i}, K_{i+1})\qquad\text{and}\qquad B=\bigcup_{i\in\N} (b_i+[b_{i-2}])
\]
then $d(A)=\alpha$ and $d(A+B)=d(A)$.

From \eqref{eqn_ctmn_p1} and \eqref{eqn_ctmn_con_1} it is straightforward to conclude that the density of $A$ exists and equals $\alpha$. Also, from \eqref{eqn_ctmn_con_1}, \eqref{eqn_ctmn_con_2}, and \eqref{eqn_ctmn_con_3} we see that $d(B)=0$. Since $B$ contains arbitrarily long intervals, we also conclude that $B$ meets every residue class.

It remains to show that the density of $A+B$ exists and equals $\alpha$.
To verify this claim, let $N\in\N$ be arbitrary, and let $i\in\N$ be such that $N\in [K_i, K_{i+1})$.
%Since $b_i\in [K_i, K_{i+1})$, we either have $N\leq b_i$ or $N> b_i$. Let us first deal with the case $N\leq b_i$. 
In this case, 
\begin{equation}
\label{eqn_ctmn_m_1}
\begin{aligned}
\big|\big((A+B) \cap [N]\big) & \setminus (A\cap [N])\big|
\\
&\leq \sum_{s\leq i} \big|\big(\big(C(t_s,K_s,K_{s+1})+B\big)\cap[N]\big)\setminus (C(t_s,K_s,K_{s+1})\cap[N]) \big|.
\end{aligned}
\end{equation}
Note that $b_{i+1}>N$. So we have
\begin{equation}
\label{eqn_ctmn_m_2}
\begin{aligned}
\big|\big(\big(C(t_s,K_s,K_{s+1})+B\big)\cap[N]\big)\setminus (C( & t_s,K_s,K_{s+1})\cap[N])\big|
\\
&\leq \big|\big(C(t_s,K_s,K_{s+1}) +B\big)\cap[N]\big|
\\
&\leq \sum_{t\leq i}\big|\big(C(t_s,K_s,K_{s+1}) +b_t+[b_{t-2}]\big)\big|
\\
&\leq \sum_{t\leq i} (K_{s+1}+b_{t+2}-K_s)
\\
&=\Oh(iK_{s+1}).
\end{aligned}
\end{equation}
Since $b_i\in [K_i, K_{i+1})$, we either have $N\leq b_i$ or $N> b_i$. Let us first deal with the case $N> b_i$. In this case, we get from \eqref{eqn_ctmn_m_1} and \eqref{eqn_ctmn_m_2} that
\begin{align*}
\big|\big( & (A+B) \cap [N]\big)  \setminus (A\cap [N])\big|
\\
&\leq \big|\big(\big(C(t_i,K_i,K_{i+1})+B\big)\cap[N]\big)\setminus (C(t_i,K_i,K_{i+1})\cap[N])\big| + \Oh(i^2 K_{i}).
\end{align*}
Since $B\cap [N]\subset \{b_1,\ldots,b_i\}+[b_{i-2}]$ and $C(t_i,K_i,K_{i+1})\cap[N]=C(t_i,K_i,N)$ we get
\begin{align*}
\big|\big( & (A+B) \cap [N]\big)  \setminus (A\cap [N])\big|
\\
&\leq \big|\big(\big(C(t_i,K_i,K_{i+1})+\{b_1,\ldots,b_i\}+[b_{i-2}]\big)\cap[N]\big)\setminus (C(t_i,K_i,N))\big| + \Oh(i^2 K_{i}).
\end{align*}
Since $b_i$ is a multiple of $t_i$, we have
\begin{align*}
\big(C(t_i,K_i,K_{i+1})+\{b_1,\ldots,b_i\} & +[b_{i-2}]\big)\cap[N]
\\
&=
\big(C(t_i,K_i,K_{i+1})+\{b_1,\ldots,b_{i-1}\}+[b_{i-2}]\big)\cap[N]
\\
&\subset
\big(C(t_i,K_i,K_{i+1})+[b_{i-1}]\big)\cap[N].
\end{align*}
We obtain
\begin{align*}
\big|\big( & (A+B) \cap [N]\big)  \setminus (A\cap [N])\big|
\\
&\leq \big|\big(\big(C(t_i,K_i,K_{i+1})+[b_{i-1}]\big)\cap[N]\big)\setminus (C(t_i,K_i,N))\big| + \Oh(i^2 K_{i})
\\
&= \big|\big(C(t_i,K_i,N)+[b_{i-1}]\big)\setminus (C(t_i,K_i,N))\big| + \Oh(i^2 K_{i}+b_{i-1})
\\
&= \Oh\bigg(\frac{b_{i-1} N}{t_i}+ i^2 K_{i}+b_{i-1}\bigg).
\end{align*}
Using \eqref{eqn_ctmn_con_2} and \eqref{eqn_ctmn_con_3} and $N>b_i$, we see that
\[
\Oh\bigg(\frac{b_{i-1} N}{t_i}+ i^2 K_{i}+b_{i-1}\bigg)=\oh(N)
\]
and hence
\[
\big|\big( (A+B) \cap [N]\big)  \setminus (A\cap [N])\big|=\oh(N).
\]

It remains to deal with the case $N\leq b_i$. In this case, we get from \eqref{eqn_ctmn_m_1} and \eqref{eqn_ctmn_m_2} that
\begin{align*}
\big|\big( & (A+B) \cap [N]\big)  \setminus (A\cap [N])\big|
\\
&\leq 
\big|\big(\big(C(t_{i-1},K_{i-1},K_i)+B\big)\cap[N]\big)\setminus (C(t_{i-1},K_{i-1},K_i)\cap[N])\big| 
\\
&\qquad +\big|\big(\big(C(t_i,K_i,K_{i+1})+B\big)\cap[N]\big)\setminus (C(t_i,K_i,K_{i+1})\cap[N])\big| + \Oh(i^2 K_{i-1}).
\end{align*}
The second term in the sum can be estimated as in the previous case to get that
\[
\big|\big(\big(C(t_i,K_i,K_{i+1})+B\big)\cap[N]\big)\setminus (C(t_i,K_i,K_{i+1})\cap[N])\big|=
\Oh\bigg(\frac{b_{i-1} N}{t_i}+ b_{i-1}\bigg).
\]
For the first term we use $b_{i}\geq N$ and have
\begin{align*}
\big|\big(\big(C(t_{i-1}, & K_{i-1},K_i)+B\big)\cap[N]\big)\setminus (C(t_{i-1},K_{i-1},K_i)\cap[N])\big| 
\\
&=
\big|\big(\big(C(t_{i-1},  K_{i-1},K_i)+\{b_1,\ldots,b_{i-1}\}+[b_{i-3}]\big)\cap[N]\big)\setminus (C(t_{i-1},K_{i-1},K_i)\cap[N])\big|
\\
&=
\big|\big(C(t_{i-1},  K_{i-1},K_i)+\{b_1,\ldots,b_{i-1}\}+[b_{i-3}]\big)\setminus C(t_{i-1},K_{i-1},K_i)\big| +\Oh(b_{i-1}).
\end{align*}
Using that $b_{i-1}$ is a multiple of $t_{i-1}$ we get
\begin{align*}
\big|C(t_{i-1},  K_{i-1},K_i)+b_{i-1}\setminus C(t_{i-1},K_{i-1},K_i)\big|
=\Oh(b_{i-1})
\end{align*}
and hence
\begin{align*}
\big|\big(\big(C(t_{i-1}, & K_{i-1},K_i)+B\big)\cap[N]\big)\setminus (C(t_{i-1},K_{i-1},K_i)\cap[N])\big| 
\\
&=
\big|\big(C(t_{i-1},  K_{i-1},K_i)+\{b_1,\ldots,b_{i-2}\}+[b_{i-3}]\big)\setminus C(t_{i-1},K_{i-1},K_i)\big| +\Oh(b_{i-1}).
\end{align*}
Now we can again argue as before to find that
\[
\big|\big(C(t_{i-1},  K_{i-1},K_i)+\{b_1,\ldots,b_{i-2}\}+[b_{i-3}]\big)\setminus C(t_{i-1},K_{i-1},K_i)\big| =
\Oh\bigg(\frac{b_{i-2} N}{t_{i-1}}+ b_{i-2}\bigg).
\]
Overall, this implies that
\begin{align*}
\big|\big( (A+B) \cap [N]\big)  \setminus (A\cap [N])\big|
= \Oh\bigg(\frac{b_{i-1} N}{t_i}+\frac{b_{i-2} N}{t_{i-1}}+ b_{i-2}+b_{i-1}+i^2 K_{i-1}\bigg)=\oh(N).
\end{align*}
This finishes the proof.
\end{proof}

\subsection{Examples of sets satisfying $\underline{d}(A+B)=d(A)+d(B)$ and $\overline{d}(A+B)>d(A)+d(B)$}
\label{sec_non-existence_of_density_of_sumset}

Note that it follows from \cref{thm_Kneser_for_integers} that if $\underline{d}(A+B)<d(A)+d(B)$ then the density $d(A+B)$ exists. A natural question arising from case~\ref{itm_main_thm_2} of \cref{thm: main} is whether it can happen that
$\underline{d}(A+B)=d(A)+d(B)$ while $\overline{d}(A+B)>d(A)+d(B)$.
The following proposition implies that this phenomenon can indeed occur.

\begin{Proposition}
\label{prop_dens_sumset_DNE}
Suppose that $\mathbf{N} = (N_s)_{s \in \N}$ and $\mathbf{M} = (M_s)_{s \in \N}$ are sequences in $\N$ with $\lim_{s\to\infty} N_{s-1}/M_{s}=\lim_{s\to\infty} M_s/N_{s}=0$. Assume $\alpha,\beta\in(0,1)$ with $\alpha<\frac{1}{h}$ and $\beta=1-\frac{1}{h}$ for some $h\in\N$.
Let $A^*\subset h\N$ and $B^*\subset\N$ be sets with  $d(A^*)=\alpha$ and $d(B^*)=\beta$.
Then there exist sets $A\subset h\N$ and $B\subset \N$ with:
\begin{enumerate}[label=(\roman{enumi}),ref=(\roman{enumi}),leftmargin=*]
\item
\label{itm_dens_sumset_DNE_i}
$d(A)=\alpha$ and $d(B)=\beta$;
\item
\label{itm_dens_sumset_DNE_ii}
$(A+h)\sim_{\mathbf{N}} A$ and $B\sim_{\mathbf{N}} (\N\setminus h\N)$;
\item
\label{itm_dens_sumset_DNE_iii}
$\underline{d}(A+B)=d_{\mathbf{N}}(A+B)=\alpha+\beta$;
\item 
\label{itm_dens_sumset_DNE_iv}
$A\sim_{\mathbf{M}} A^*$ and $B\sim_{\mathbf{M}} B^*$.
\end{enumerate}
\end{Proposition}

Before we provide a proof of \cref{prop_dens_sumset_DNE}, let us show how it can be used to deduce the existence of sets with the property that $\underline{d}(A+B)=d(A)+d(B)$ and $\overline{d}(A+B)>d(A)+d(B)$.

Let $A^*$ and $B^*$ be any sets with 
$d(A^*)=\alpha$, $d(B^*)=\beta$, and 
\[
\lim_{K_2\to\infty}\lim_{K_1\to\infty} \frac{(A^*\cap [K_1])+(B^*\cap [K_2])}{K_1+K_2} =1.
\]
Such sets are easy to construct. For example, one can take $A^*$ to be a random subset of $\mathbb{N}$ obtained by including $n\in\N$ in $A^*$ independently with probability $\alpha$, and $B^*$ to be the set constructed analogously with probability $\beta$ instead of $\alpha$. Using \cref{prop_dens_sumset_DNE}, we can then find sets $A,B\subset\N$ satisfying properties \ref{itm_dens_sumset_DNE_i}--\ref{itm_dens_sumset_DNE_iv}. It follows that
\[
\underline{d}(A+B)=d_{\mathbf{N}}(A+B)=\alpha+\beta\qquad\text{and}\qquad \overline{d}(A+B)=d_{\mathbf{M}}(A+B)=d_{\mathbf{M}}(A^*+B^*)=1. 
\]

\begin{proof}[Proof of \cref{prop_dens_sumset_DNE}]
Let $(K_s)_{s\in\N}$ and $(t_s)_{s\in\N}$ be any sequences such that
\[
\lim_{s\to\infty} \frac{M_s}{t_s}=
\lim_{s\to\infty} \frac{t_s}{K_s}=
\lim_{s\to\infty} \frac{K_s}{N_s}=0.
\]
Let
\begin{align*}
A_s'&=A^*\cap (N_{s-1}, K_{s}],
\\
A_s''&=(t_s\N+\{1,\ldots,\lfloor \alpha h t_s\rfloor\})\cap (K_{s}, N_s]\cap h\N, 
\end{align*}
and
\begin{align*}
B_s'&= B^*\cap (N_{s-1}, M_s],
\\
B_s''&=(N\setminus h\N)\cap (M_s, N_s],
\end{align*}
Taking
\[
A=\bigcup_{s\in\N} A_s'\cup A_s''\qquad\text{and}
\qquad B=\bigcup_{s\in\N} B_s'\cup B_s''
\]
it is straightforward to check that \ref{itm_dens_sumset_DNE_i}, \ref{itm_dens_sumset_DNE_ii}, and \ref{itm_dens_sumset_DNE_iv} are satisfied.

It remains to verify \ref{itm_dens_sumset_DNE_iii}.
We have
\begin{align*}
\big| (A+B)\cap [1,N_s] \big|
&= \big| (A+B)\cap (2K_s,N_s] \big| + \Oh(K_s)
\\
&= \Big| \Big(\big((A\setminus A_s'')\cup A_s''\big) + \big((B\setminus B_s'')\cup B_s''\big)\Big) \cap (2K_s,N_s]\Big| +\Oh(K_s).
\end{align*}
Note that $(A\setminus A_s'')\cap [1,N_s]\subset [1,K_s]$ and $(B\setminus B_s'')\cap [1,N_s]\subset [1,M_s]$.
Hence
\[
\big((A\setminus A_s'') +(B\setminus B_s'')\big)\cap (2K_s,N_s]=\emptyset.
\]
It follows that
\begin{align*}
\big| (A & +B)\cap [1,N_s] \big|
\\
&=
\Big|\Big(
\big((A\setminus A_s'') +B_s''\big)
\cup
\big(A_s'' +(B\setminus B_s'')\big)
\cup
\big(A_s''+B_s''\big)\Big)
\cap (2K_s,N_s] \Big| + \Oh(K_s)
\\
&=
\Big|\big(
(A +B_s'')
\cup
(A_s'' +(B\setminus B_s''))\big)
\cap (2K_s,N_s] \Big| + \Oh(K_s).
\end{align*}
Since $A\subset h\N$ and $B_s''=(N\setminus h\N)\cap (M_s, N_s]$, we see that $(A +B_s'')\cap (2K_s,N_s] = (\N\setminus h\N)\cap (2K_s,N_s]$. This yields
\begin{align*}
\big| (A & +B)\cap [1,N_s] \big|
=
\Big|\big(
(\N\setminus h\N)
\cup
(A_s'' +(B\setminus B_s''))\big)
\cap (2K_s,N_s] \Big| + \Oh(K_s).
\end{align*}
Finally, observe that since $A_s''\subset h\N$ we have
\[
(\N\setminus h\N)
\cup
(A_s'' +(B\setminus B_s''))
=
(\N\setminus h\N)
\cup
(A_s'' +((B\setminus B_s'')\cap h\N))
\]
and
\[
\big|(A_s'' +((B\setminus B_s'')\cap h\N))\cap (2K_s,N_s]\big| = \big|A_s'' \cap (2K_s,N_s]\big| + \Oh\Bigg(\frac{M_s N_s}{t_s}\Bigg).
\]
In conclusion, we have
\begin{align*}
\frac{\big|(A  +B)\cap [1,N_s] \big|}{N_s}
&= \frac{\big|(\N\setminus h\N)\cap (2K_s,N_s]\big|}{N_s}+
\frac{|A_s''\cap (2K_s,N_s]|}{N_s}+
\Oh\Bigg(\frac{K_s}{N_s}+\frac{M_s}{t_s}\Bigg).
\\
&= \beta+\alpha+\oh_{s\to\infty}(1).
\end{align*}
\end{proof}

%SECTION
\section{Further explorations} \label{sec: questions}

We conclude this paper with an assortment of observations, questions and conjectures concerning sumsets in the integers and other discrete groups.

\subsection{More on direct theorems for sumsets in the integers}

Consider the \define{lower logarithmic density} %and \define{upper logarithmic density} defined respectively as
defined as
\[
\underline{\delta}(A)=\liminf_{N\to\infty}\frac{1}{\log N}\sum_{n\in A\cap [N]} \frac{1}{n}. % 
\]
Just as Kneser's theorem (\cref{thm_Kneser_for_integers}) characterizes all cases when $\underline{d}(A+B)<\underline{d}(A)+\underline{d}(B)$ occurs,  we can ask for a similar characterization in the case of logarithmic density. 
\begin{Question}[Kneser's theorem for logarithmic density]
\label{q_logarithmic_density_Kneser}
Can one classify all instances in which $\underline{\delta}(A+B)<\underline{\delta}(A)+\underline{\delta}(B)$ holds?
\end{Question}

The following example provides a pair of sets $A,B\subset\N$ satisfying $\underline{\delta}(A+B)<\underline{\delta}(A)+\underline{\delta}(B)$, but $\underline{d}(A+B)>\underline{d}(A)+\underline{d}(B)$.
This illustrates that the use of lower asymptotic density in Kneser's theorem is essential for the conclusion to hold, and an answer to \cref{q_logarithmic_density_Kneser} involves new phenomena not present in Kneser's theorem.

\begin{Example}
    Fix a base $b \geq 7$.
    Let $C = \bigcup_{n \geq 0}[b^n, 3b^n) \cap \N$ be the set of positive integers whose base $b$ expansion has leading digit 1 or 2.
    Then $\delta(C) = \log_b(3)$ and
    \begin{equation*}
        \delta(C + C) = \log_b(6) = \log_b(3) + \log_b(2) < 2 \delta(C).
    \end{equation*}
    On the other hand, $\underline{d}(C)=\frac{2}{b-1}$ and $\underline{d}(C+C)=\frac{5}{b-1}$. 
\end{Example}

A recent result of Griesmer \cite[Theorem 2.7]{Griesmer26arXiv} provides information about sets satisfying $m(A+B) < m(A) + m(B)$ for an arbitrary \emph{invariant mean} $m$.
Applying this result to the logarithmic density produces a partial answer to \cref{q_logarithmic_density_Kneser} but falls short of a full characterization.

It is natural to wonder if there exists a quantitative version of \cref{thm_Kneser_for_integers}. We attempt a possible statement below.

\begin{Question}[Finitary form of Kneser's theorem]
Is it true that for every $\epsilon>0$ and $N_0\in\N$ there exist $D=D(\epsilon,N_0)>0$ and $N_1=N_1(\epsilon,N_0)\in\N$ such that the following holds: 
For any $\alpha,\beta>0$, any $N\geq N_1$, and any sets $A,B\subset [N]$ such that 
\[
\min_{N_0< n\leq N} \frac{|A\cap [n]|}{n}\geq \alpha\qquad\text{and}\qquad \min_{N_0< n\leq N} \frac{|B\cap [n]|}{n}\geq \beta,
\]
and $B$ intersects every residue class mod $D$, we have
\[
\frac{|(A+B)\cap[N]|}{N}\geq \alpha+\beta-\epsilon.
\]
\end{Question}

The next question concerns an extension of Kneser's theorem to higher dimensions. For simplicity, we restrict to the $2$-dimensional case, but from there it is straightforward to infer the analogous question for higher dimensions.
We define
\[
\liminf_{n,m\to\infty} a_{n,m}= \lim_{N\to\infty}\Bigg(\inf_{\substack{n\geq N\\m\geq N}} a_{n,m}\Bigg).
\]
The \define{lower density} of a set $A\subset \N^2$ is given by
\[
\underline{d}(A)=\liminf_{n,m\to\infty} \frac{|A\cap ([n]\times [m])|}{nm}.
\]
Note that all finite-index subgroups of $\Z^2$ are of the form $H_{a,b,c}=\{(an+bm,cm): m,n\in\Z\}$ for some $a,c\in\N$ and $b\in\{0,1,\ldots,a-1\}$.

\begin{Question}[Kneser's theorem in $\N^2$]
If $A, B \subseteq \N^2$ are non-empty sets then either
\begin{equation*}
\underline{d}(A+B)\geq \underline{d}(A)+\underline{d}(B),
\end{equation*}
or there exists $H=H_{a,b,c}\cap (\N\times \N)$ for some $a,c\in\N$ and $b\in\{0,1,\ldots,a-1\}$ such that $A+B$ equals, up to finitely many elements, a finite union of translates of $H$ and $\underline{d}(A+B)=d(A+H)+d(A+H)-d(H)$.
\end{Question}

As with \cref{q_logarithmic_density_Kneser} above, results of Griesmer \cite[Theorem 2.7 and Theorem 3.6]{Griesmer26arXiv} give nontrivial information about sets $A$ and $B$ satisfying $\underline{d}(A+B) < \underline{d}(A) + \underline{d}(B)$.
However, it does not appear that these results are sufficient to reach the very strong conclusion that $A+B$ is equal to a finite union of translates of a subgroup $H$ up to finitely many elements.

\subsection{More on inverse theorems for sumsets in the integers}

We begin with a question regarding a possible generalization of \cref{thm: main}.

\begin{Question}
\label{q_1}
Suppose $A,B\subset\N$ with $\underline{d}(A)>0$ and assume that $B$ meets every residue class in $\N$.
If $\underline{d}(A+B) = \underline{d}(A) + \underline{d}(B) < 1$, what can be said about the structure of $A$ and $B$?
\end{Question}

The difference between \cref{thm: main} and \cref{q_1} is that in the latter we do not assume that $A$ and $B$ have density. This brings the setting of \cref{q_1} closer to that of Kneser's theorem in the integers (\cref{thm_Kneser_for_integers}), where no assumptions on the existence of the densities are made.
The following example illustrates that without this assumption on the existence of the density, new phenomena emerge that did not show up in \cref{thm: main}.

\begin{Example}
Define sequences $(a_n)_{n \geq 0}$ and $(H_n)_{n \geq 0}$ recursively as follows:
    \begin{itemize}
        \item $a_0 = 0, H_0 = 1$;
        \item For each $n \geq 0$, let $a_{n+1} = 3(H_0 + H_1 + \dots + H_n) + 1$ so that
            \begin{equation} \label{eq: a_n}
                \frac{H_0 + H_1 + \dots + H_n}{a_{n+1}-1} = \frac{1}{3}.
            \end{equation}
        \item Choose $H_{n+1} \in \N$ such that $a_{n+1} = o(H_{n+1})$.
    \end{itemize}
Then the set $C = \N \cap \bigcup_{n \geq 0} [a_n, a_n + H_n)$ has $\underline{d}(C) = \frac{1}{3}$ and $\underline{d}(C + C) = \frac{2}{3}$.
Indeed, the fact that $\underline{d}(C) = \frac{1}{3}$ follows immediately from \eqref{eq: a_n}.
    By \cref{thm_Kneser_for_integers}, we know $\underline{d}(C + C) \geq 2 \underline{d}(C) = \frac{2}{3}$.
    To see that $\underline{d}(C + C) \leq \frac{2}{3}$, observe that
    \begin{equation*}
        C + C = \N \cap \bigcup_{n \geq 0} \left( \bigcup_{0 \leq m \leq n} [a_n + a_m, a_n + a_m + H_n + H_m) \right)
        \subset \N \cap \bigcup_{n \geq 0} [a_n, 2a_n + 2H_n),
    \end{equation*}
    so
    \begin{equation*}
        \underline{d}(C + C) \leq \lim_{N \to \infty} \frac{1}{a_{N+1}-1} \sum_{n=0}^N (a_n + 2H_n) = \lim_{N \to \infty} \frac{(2 + o(1))(H_0 + H_1 + \dots + H_N)}{a_{N+1}-1} = \frac{2}{3}.
    \end{equation*}
\end{Example}

Whilst our main result deals with the equality case $d_{\mathbf{N}}(A+B) = d(A) + d(B)$, it is natural to ask what can be said in the regime of near equality $d_{\mathbf{N}}(A+B) = d(A) + d(B) + \Oh(\epsilon)$. The following question inquires about this stability version of our result.

\begin{Question}[Stability of Theorem~\ref{thm: main}]
Is it true that for every $\alpha, \beta > 0$ with $\alpha + \beta < 1$ and every $\epsilon > 0$ there exists $\delta > 0$ such that the following holds:

Let $A, B \subset \N$ with $d(A) =\alpha$, $d(B)=\beta$, and suppose that $B$ meets every residue class in $\N$.
If $\mathbf{N} = (N_s)_{s \in \N}$ is a sequence of natural numbers with $N_s \to \infty$ and
\[
d(A) + d(B) \leq d_{\mathbf{N}}(A+B) \leq d(A) + d(B) +\delta,
\]
then there exist sets $A',B'\subset\N$ and a subsequence $\mathbf{N}'$ of $\mathbf{N}$ with
\[
d_{\mathbf{N}'}(A'+B') =d(A')+d(B'),\quad d_{\mathbf{N}'}(A'\triangle A)\leq\epsilon \quad\text{and}\quad d_{\mathbf{N}'}(B'\triangle B)\leq\epsilon.
\]
\end{Question}

We continue with a question closer to, but still weaker than, the original problem posed by \Erdos{}--Graham (\cref{proplem_335}).

\begin{Question}
\label{q_2}
Let $A, B \subseteq \N$ be subsets of the positive integers with the property that for all $a,b\in\N$ the densities $d(A \cap (a\N+b))$ and $d(B \cap (a\N+b))$ exist, and assume that one of the sets has positive density, say $d(A)>0$. 
If $\underline{d}(A+B) = d(A) + d(B)<1$, what can be said about the structure of $A$ and $B$?
\end{Question}

Note that imposing the existence of $d(A \cap (a\N+b))$ and $d(B \cap (a\N+b))$ is a different, and arguably less restrictive, regularity assumption than our assumption in \cref{thm: main} that $B$ meets every residue class. An illustrative example of a set $A\subset\N$ such that
$d(A \cap (a\N+b))$ exists for all $a,b\in\N$ and  $d(A+A)=2d(A)$ was constructed by Griesmer \cite[Example~3.2]{Griesmer13}.

\bigskip

In \cref{prop_ctmn}, we encounter sets $A, B \subseteq \N$ such that $A$ has positive density $d(A) \in (0,1)$, $B$ meets every residue class in $\N$, and $d(A+B)=d(A)$.
Our next question concerns potential additional properties the set $B$ must exhibit in such a situation.
Say that a set $B \subseteq \N$ is a \emph{$d$-essential component} if for every $A \subseteq \N$ with $d(A) \in (0,1)$, one has $\underline{d}(A+B) > d(A)$.
Essential components with respect to Schnirelmann density and lower density have been studied in the past, and we refer to the reader to \cite{Ruzsa87} for results on the asymptotics of essential components in the sense of Schnirelmann and lower density.
The sets $B$ appearing in \cref{prop_ctmn} are \emph{not} $d$-essential components.

\begin{Question}
\label{question_d-essential_component}
    Is every set with full upper density $\overline{d}(B) = 1$ a $d$-essential component?
\end{Question}

Finally, we inquire about an analogue of Freiman's theorem for infinite sets in the integers.

\begin{Question}[Freiman's theorem for density]
    Let $\alpha \in (0,1)$ and $K \in [1, \alpha^{-1})$.
    Suppose $A \subseteq \N$ is a set with density $d(A) = \alpha$ satisfying $d(A+A) \leq K \alpha$.
    Must $A$ be contained in a Bohr set of dimension $\ll_K 1$ and density $\ll_K \alpha$?
\end{Question}

%==========================================================
%==========================================================

%BIBLIOGRAPHY
%\bibliographystyle{plain}
%\bibliographystyle{aomplain}
\bibliographystyle{aomalphanomr}
\bibliography{mynewlibrary}

%\printbibliography

%\begin{thebibliography}{9}
%\bibitem{}
%\end{thebibliography} 

%==========================================================
%==========================================================

%AFFILIATION
\bigskip
\footnotesize
\noindent
Ethan Ackelsberg\\
\textsc{{\'E}cole Polytechnique F{\'e}d{\'e}rale de Lausanne (EPFL)}\\
\href{mailto:ethan.ackelsberg@epfl.ch}
{\texttt{ethan.ackelsberg@epfl.ch}}

\bigskip
\footnotesize
\noindent
Florian K.\ Richter\\
\textsc{{\'E}cole Polytechnique F{\'e}d{\'e}rale de Lausanne (EPFL)}\\
\href{mailto:f.richter@epfl.ch}
{\texttt{f.richter@epfl.ch}}

%END DOCUMENT
\end{document}